\pdfoutput=1
\documentclass[11pt,a4paper]{article}

\usepackage[margin=1.05in]{geometry}
\usepackage{amsmath,amssymb,amsthm,mathtools}
\usepackage{aliascnt}
\usepackage{mathrsfs}
\usepackage{xcolor}
\usepackage[colorlinks=true,linkcolor=blue!60!black,citecolor=blue!60!black,urlcolor=blue!60!black]{hyperref}
\hypersetup{pdftitle={A Compressible Miles Instability for Wind-Wave Generation in a Neutral Atmosphere},
            pdfauthor={Tianyi Li}}
\usepackage{microtype}
\usepackage{enumitem}
\usepackage[capitalize,nameinlink]{cleveref}

\newtheorem{theorem}{Theorem}[section]
\newcommand{\newsharedtheorem}[3]{%
  \newaliascnt{#1}{theorem}%
  \newtheorem{#1}[#1]{#2}%
  \aliascntresetthe{#1}%
  \crefname{#1}{#2}{#3}}
\newsharedtheorem{proposition}{Proposition}{Propositions}
\newsharedtheorem{lemma}{Lemma}{Lemmas}
\newsharedtheorem{corollary}{Corollary}{Corollaries}
\theoremstyle{definition}
\newsharedtheorem{definition}{Definition}{Definitions}
\newsharedtheorem{example}{Example}{Examples}
\newsharedtheorem{hypothesis}{Hypothesis}{Hypotheses}
\theoremstyle{remark}
\newsharedtheorem{remark}{Remark}{Remarks}

\newcommand{\R}{\mathbb{R}}
\newcommand{\C}{\mathbb{C}}
\newcommand{\T}{\mathbb{T}}
\newcommand{\Z}{\mathbb{Z}}
\newcommand{\eps}{\varepsilon}
\newcommand{\rhobar}{\bar{\rho}_0}
\newcommand{\Abar}{\bar{A}}
\newcommand{\qbar}{\bar{q}}
\newcommand{\chibar}{\bar{\chi}}
\newcommand{\Imc}{\operatorname{Im}}
\newcommand{\Rec}{\operatorname{Re}}
\newcommand{\dd}{\,\mathrm{d}}
\newcommand{\im}{\mathrm{i}}
\newcommand{\Za}{Z_{\mathrm{a}}}
\newcommand{\Dw}{D_{\mathrm{w}}}
\newcommand{\ha}{h_{\mathrm{a}}}
\newcommand{\hw}{h_{\mathrm{w}}}
\newcommand{\rhow}{\rho_{\mathrm{w}}}
\newcommand{\Wc}{\mathcal{W}}
\newcommand{\Jf}{\mathcal{J}}
\newcommand{\Edet}{\mathfrak{E}}
\newcommand{\Lop}{\mathbb{L}}

\title{\bfseries A Compressible Miles Instability for Wind-Wave Generation in a Neutral Atmosphere}
\author{Tianyi Li\thanks{Lawrence Livermore National Laboratory.
Corresponding author: \texttt{li100@llnl.gov}.}}
\date{September 2026}

\begin{document}
\maketitle

\begin{abstract}
We study the normal-mode stability of a flat interface between compressible Euler flow in the air and incompressible Euler flow in the water, including gravity and surface tension.
The base state is quiescent water beneath a hydrostatic, constant-entropy shear flow $U(z)$ with $U(0)=0$, so the atmosphere is neutrally stratified.
For a small ratio $\varepsilon$ of the air density at the interface to the water density, and for uniformly subsonic wave-relative flow, we establish a compressible analog of the critical-layer instability of Miles, extending the rigorous incompressible theory of B\"uhler, Shatah, Walsh, and Zeng.
For a real phase speed $c_r$, the singular coefficient at a regular critical level $U(z_c)=c_r$ has a residue proportional to $Q=U''+gU'/a^2=\bar\rho_0(U'/\bar\rho_0)'$, with gravitational acceleration $g$, sound speed $a$, and normalized base-state air density $\bar\rho_0$.
Thus the gradient of specific vorticity replaces the curvature $U''$ of the incompressible theory.

We prove a limiting absorption principle at regular, non-resonant phase speeds and a residue formula for the imaginary part of the boundary value $Z_{\mathrm a}^+(k,c_r)$ of the air-side pressure response at wavenumber $k>0$.
Let $c_k>0$ be the phase speed of the free gravity--capillary wave.
If $\mathrm{Im}\,Z_{\mathrm a}^+(k,c_k)>0$, then, for sufficiently small $\varepsilon$, the dispersion relation has exactly one unstable root $c_\varepsilon$ in a neighborhood of $c_k$ that does not depend on $\varepsilon$.
A sufficient condition is $Q\le0$ at every critical level and $Q<0$ at the highest one.

As the Mach number $M\to0$, $Z_{\mathrm a}^+$ converges to its incompressible counterpart at rate $O(M^2)$.
Under the corresponding incompressible sign condition, the unstable root satisfies $|c_\varepsilon(M)-c_\varepsilon(0)|\le C\varepsilon M^2$ with $C$ independent of $\varepsilon$ and $M$.
\end{abstract}



\setcounter{tocdepth}{2}
\tableofcontents

\section{Introduction}\label{sec:intro}

Wind waves are the surface waves of the ocean that the wind generates.
When wind blows over the sea, it transfers momentum and energy to the water, and the surface waves grow.
Understanding this transfer is a classical problem of geophysical fluid dynamics, and it matters in practice.
The wind input is the source term that drives operational wave forecasting models \cite{Komen1994,Janssen2004}, the wave field sets the roughness and the drag of the sea surface felt by the atmosphere \cite{SullivanMcWilliams2010}, and wave breaking controls the exchange of gases, including carbon dioxide, between the ocean and the atmosphere \cite{Deike2022}.

The generation process has several stages, and the classical theories treat them separately.
Phillips \cite{Phillips1957} described the initial excitation of a calm surface by the turbulent pressure fluctuations of the wind.
Miles \cite{Miles1957} described the subsequent exponential growth of a wave through its coupling to the sheared mean wind.
Benjamin \cite{Benjamin1959} computed the normal and tangential stresses that a laminar shear flow exerts on a wavy boundary.
Lighthill \cite{Lighthill1962} interpreted Miles's mechanism physically through the vortex force at the critical layer.
Townsend \cite{Townsend1972} included the variation of the turbulent Reynolds stresses in a deep turbulent boundary layer over waves, and Belcher and Hunt \cite{BelcherHunt1993} derived an asymptotic theory with four layers for turbulent shear flow over slowly moving waves, using rapid-distortion theory for the outer flow.
These theories rest on formal linearization and asymptotic matching.
A rigorous mathematical foundation for the incompressible Miles problem was given only recently \cite{BSWZ2016}.

In Phillips's theory a pressure component traveling at the phase speed of a free surface wave excites that wave resonantly.
The theory distinguishes an initial stage from a principal stage of linear growth of the mean-square displacement.
Li and Shen \cite{LiShen2022Principal,LiShen2023Phillips} confirmed both stages by direct numerical simulation, developed a random-sweeping pressure--wave model for the principal stage, and obtained direct evidence of the initial-stage resonance.

Miles's theory describes what happens once the wave feeds back on the air.
Because the ratio $\eps\approx1.2\times10^{-3}$ of the air density at the sea surface to the water density is small, a water wave of horizontal wavenumber $k$ travels, to leading order, at the phase speed $c_k$ of the free gravity--capillary wave, that is, of the wave that the water would carry if the air density were zero.
The wave interacts most strongly with the wind at the height $z_c$ where the wind speed $U(z)$ equals the phase speed, $U(z_c)=c_k$.
This height is the critical level, and the thin region around it is the critical layer.
Miles's formula gives an $O(\eps)$ growth rate whose sign, for incompressible air, is determined by the curvature $U''(z_c)$ of the wind profile at the critical level, with growth when $U''(z_c)<0$.
Here $O(\eps)$ denotes a quantity whose absolute value is at most a constant multiple of $\eps$, with the constant independent of $\eps$, and we use $O(\cdot)$ in the same way for the other small quantities below.
The mechanism has been confirmed by field and laboratory measurements \cite{Hristov2003,Carpenter2022,Plant1982} and enters the wind-input source terms of operational wave models \cite{Janssen2004,Komen1994}.
The present paper treats this mechanism and asks how atmospheric compressibility modifies it.

B\"uhler, Shatah, Walsh, and Zeng \cite{BSWZ2016} (BSWZ) gave the first rigorous proof of Miles's criterion.
They derived the linearization of the two-phase Euler equations about an arbitrary shear flow and proved the criterion in the presence of surface tension and of an interfacial vortex sheet, a jump of the tangential velocity across the interface.
Their result therefore covers both the Kelvin--Helmholtz theory, in which such a jump drives the growth of the waves, and the quasi-laminar theory.
Here quasi-laminar means that turbulence sets the mean wind profile, while the wave-induced turbulent stresses and the instantaneous turbulent fluctuations are neglected.
Subsequent rigorous work treats the full spectral picture of the two-phase problem \cite{Liu2024TwoPhase}, the linearized dynamics of one-phase free-surface shear flows \cite{LiuZeng2025CMP,LiuZeng2025Memoirs}, and circular geometries \cite{GuiLiu2025}.

All of these rigorous results take the air to be incompressible with constant density.
At the Mach numbers of winds over the sea this is a natural approximation, but the size of the correction it neglects has not been quantified within the wind-wave eigenvalue problem.
We are not aware of a rigorous normal-mode theory that combines a compressible gas, an incompressible liquid, a sheared wind, and a Miles critical layer.

We consider the full compressible Euler equations in the air and the incompressible Euler equations in water of finite depth, coupled through a free interface, with gravitational acceleration $g>0$ and surface tension coefficient $\sigma\ge0$.
The horizontal coordinate is $x$ and the vertical coordinate is $z$, measured upward from the undisturbed interface $z=0$.
The air has a general equation of state $p=P(\rho,S)$, which gives the pressure $p$ as a function of the density $\rho$ and the specific entropy $S$.
In an inviscid gas without heat conduction, each fluid parcel carries its specific entropy unchanged.
Below, entropy means specific entropy.
The undisturbed air and water depths are $\ha>0$ and $\hw>0$, and the top and bottom boundaries $z=\ha$ and $z=-\hw$ are flat and impermeable, so the vertical velocity vanishes there.
The base state, the undisturbed flow about which we linearize, consists of water of constant density $\rhow$ at rest beneath a flat interface and a wind $U(z)$ in hydrostatic air with $U(0)=0$ and constant entropy.
Its air density is $\rho_0(z)$, and the small parameter of the problem is the ratio $\eps=\rho_0(0)/\rhow$ of the air density at the interface to the water density.

Constant entropy is the main physical restriction of this paper, and we explain what it means.
A parcel of air displaced vertically expands or contracts adiabatically.
If it arrives lighter or heavier than its new surroundings, buoyancy pushes it back and it oscillates.
The frequency of this oscillation is the Brunt--V\"ais\"al\"a frequency $N$, and such oscillations propagating through the air are internal gravity waves.
When the entropy of the base state is constant with height, a displaced parcel always has the density of its surroundings, so there is no restoring force, $N$ vanishes, and the atmosphere is neutrally stratified.
In a hydrostatic atmosphere, neutrality is equivalent to a density profile that decreases at the rate set by the sound speed,
\begin{equation}\label{eq:neutral-intro}
N^2 \;=\; -g\Big(\frac{\rho_0'}{\rho_0}+\frac{g}{a^2}\Big)\;=\;0
\qquad\Longleftrightarrow\qquad
\frac{\rho_0'}{\rho_0}=-\frac{g}{a^2},
\end{equation}
where $a(z)$ is the sound speed of the base state and the prime denotes $\dd/\dd z$.
The restriction has a mathematical and a physical justification.
A neutral atmosphere carries no internal gravity waves, so the only singular behavior of the normal-mode equation at a critical level is the one already present in Rayleigh's equation for incompressible shear flow, a simple pole.
In a stratified atmosphere, the normal modes instead satisfy an equation of Taylor--Goldstein type, whose solutions near a critical level behave like powers of the distance to it, with exponents that depend on the Richardson number $N^2/U'^2$ at the critical level (\cref{rem:general-strat}).
Physically, the logarithmic wind profile that wave models use for the wind input is the profile of a neutral surface layer \cite{Janssen2004}, so neutral stratification is the standard description of the lowest tens of meters over the sea, where the critical levels of wind waves lie.
Compared with the incompressible theory, two effects then remain in the air equations.
The sound speed is finite, and the base-state density decreases with height at the hydrostatic rate $\rho_0'/\rho_0=-g/a^2$.

\subsection{Main results}\label{subsec:mainresults}

The central question of the present study is how Miles's critical-layer instability and its criterion change when the air is a compressible gas.
We answer it for uniformly subsonic wave-relative flow, which includes the low-Mach-number regime of winds over the sea without being restricted to it.
Uniformly subsonic means that the speed of the wind relative to the wave stays below a constant fraction, smaller than one, of the sound speed (Hypothesis \ref{hyp:subsonic}).
The instability persists, with the curvature $U''(z_c)$ replaced by $Q(z_c)=U''(z_c)+(g/a^{2}(z_c))\,U'(z_c)$, the density-weighted gradient of specific vorticity given in Equation~\eqref{eq:Q-intro}.
The incompressible theory is recovered as the Mach number $M$ tends to zero, where $M$ is the ratio of a reference wind speed to the sound speed at the interface.
This convergence holds for the unstable root itself, at the rate $M^2$ (\S\ref{sec:lowmach}).
In a model regime of strong compressibility, $Q(z_c)>0$ although $U''<0$, and the Miles bifurcation is absent, that is, no growing mode emerges from the free water wave as the density ratio $\eps$ increases from zero.

A normal mode is a solution of the linearized equations in which every perturbation field is the product of a function of the height $z$ and the factor $e^{\im k(x-ct)}$, with horizontal wavenumber $k>0$ and complex phase speed $c$.
The mode grows exponentially in time exactly when $\Imc c>0$, at the rate $k\Imc c$.
Here $\Rec$ and $\Imc$ denote the real and imaginary parts of a complex number.
The proof reduces each normal mode to an ordinary differential equation for the vertical-velocity amplitude, which becomes singular at critical levels as the imaginary part of $c$ tends to zero.
It then defines the solution at a real phase speed as the limit of the solutions with $\Imc c>0$, which is a limiting absorption principle, and it locates the unstable root near the free wave speed $c_k$ with Rouch\'e's theorem.
This theorem states that two functions that are holomorphic, that is, complex differentiable, on an open set containing a simple closed curve and its interior have equally many zeros inside the curve when their difference is smaller in modulus than one of them on the curve.
Here zeros are counted with multiplicity, so that a simple zero, at which the derivative does not vanish, counts once.

Write $\rhobar(z)=\rho_0(z)/(\eps\rhow)=\rho_0(z)/\rho_0(0)$ for the base-state air density normalized by its value at the interface, so that $\rhobar(0)=1$.
For a horizontal wavenumber $k>0$ let $c_k$ be the positive free gravity--capillary wave speed defined in Equation~\eqref{eq:ck}, which depends on $g$, $\sigma$, $\rhow$, $k$, and the water depth $\hw$.
A critical level $z_c$ is regular when $U'(z_c)\neq0$, so that the wind speed crosses the phase speed with nonzero shear.
The main theorem treats a monotone wind with one regular critical level.

\begin{theorem}
\label{thm:headline}
Let $U\in C^4([0,\ha])$ with $U(0)=0$ and $U'>0$. Let the neutrally
stratified atmosphere $(\rhobar,a)$ satisfy the regularity and
hydrostatic assumptions of \S\ref{subsec:hypotheses}. Let $k>0$ be a horizontal wavenumber such that the free wave speed $c_k$ at this wavenumber
lies in $U((0,\ha))$, so that there is exactly one critical level
$z_c\in(0,\ha)$ with $U(z_c)=c_k$, and such that the wave-relative
flow $U-c$ is uniformly subsonic for $c$ near $c_k$. Assume
\begin{equation}
Q(z_c)=U''(z_c)+\frac{g}{a^{2}(z_c)}\,U'(z_c)<0 .
\end{equation}
Then there are $\eps_0>0$ and $r_0>0$ with the following properties
for every air--water density ratio $\eps\in(0,\eps_0)$. The linearized
problem has exactly one unstable phase-speed root $c_\eps$ in the set $\{c\in\C:\ |c-c_k|<r_0,\ \Imc c>0\}$, and the space of normal
modes at this root is one-dimensional. The root satisfies
$c_\eps=c_k+\eps c_1+O(\eps^{2}(1+\log(1/\eps)))$ with an
explicit $c_1$, and the growth rate is
$\gamma_\eps=k\Imc c_\eps>0$. The number
$\lambda_\eps=-\im kc_\eps$ is an isolated, algebraically simple
eigenvalue of the closed, densely defined operator $\Lop_k$ of
\cref{thm:operator}, the linearized evolution operator at wavenumber
$k$.
If, in addition, the equation of state satisfies the thermodynamic regularity in Equation~\eqref{eq:thermo-domain}, then $\Lop_k$ is the linearization of the nonlinear two-phase problem in the sense of \cref{thm:operator}.
\end{theorem}

\begin{proof}
This is \cref{cor:monotone}.
\end{proof}

The last part of the theorem addresses a point that formal linearization does not settle.
In a free-boundary problem the fluid domains move with the solution, so the term ``linearized operator'' is not well defined until the problem is posed on fixed function spaces.
BSWZ \cite[\S2]{BSWZ2016} raised this point for the incompressible problem.
The dynamically correct object is the Fr\'echet derivative of the equations after they have been transported to fixed reference domains.
Here the Fr\'echet derivative at a state is the linear map that approximates how the expressions in the equations change under a small perturbation of the state, with an error that is small compared with the perturbation.
Appendix~\ref{app:operator} carries out this transport, shows that the derivative exists under the thermodynamic regularity in Equation~\eqref{eq:thermo-domain}, and shows that, after a change of variables which removes the apparent perturbation produced by moving the coordinate system, this derivative coincides with the linearized system obtained by differentiating the equations at fixed points.
The new unknown after this change of variables is called the good unknown \cite{Alinhac1989}.
The operator $\Lop_k$ is then the closure of this linear system in $L^2$, which is the standard way to turn a differential operator into a closed operator with a well-defined spectrum.
Here $L^2$ denotes the space of square-integrable functions of the height, defined precisely in \cref{rem:spaces}.
Densely defined means that the domain of $\Lop_k$, the subspace of states on which $\Lop_k$ is defined, is dense in the $L^2$-based state space, that is, every state is a limit, in the norm of this space, of states in the domain.
Section~\ref{subsec:operator} defines closed operators, the closure, the spectrum, the resolvent, the Riesz projection, and algebraic simplicity.
The theorem therefore says that the growing mode is an eigenvalue of the actual derivative of the nonlinear problem.
A formally linearized system alone would not give this.
The eigenvalue is isolated, so no other point of the spectrum lies near it, and algebraically simple, that is, the Riesz projection onto it has a one-dimensional range, so the mode is unique up to a constant factor.
\cref{thm:operator} and \cref{prop:resolvent} also give an explicit bound on the resolvent $(\lambda-\Lop_k)^{-1}$ near it.

We now describe the objects in the theorem.
Consider a normal mode with $\Imc c\neq0$ and write the perturbation of the vertical air velocity as $w(z)e^{\im k(x-ct)}$, so that $w$ is the vertical-velocity amplitude, a function of height alone.
Because the base entropy is constant and each parcel conserves its entropy, the entropy perturbation of such a mode vanishes, and the remaining air amplitudes, the horizontal velocity, the density, and the pressure, are recovered from $w$ (\cref{lem:entropy,prop:SLform}).
The amplitude $w$ satisfies
\begin{equation}\label{eq:SL-intro}
\big(\Abar\,w'\big)' \;=\;\Big(\rhobar k^{2}+\frac{G}{V}\Big) w,
\qquad
V=U-c,\quad
\Abar=\frac{\rhobar\,a^{2}}{a^{2}-V^{2}},
\end{equation}
where $G=\bigl(\Abar(U'+bV)\bigr)'+b\Abar(U'+bV)$ and $b=g/a^{2}$.
For real $c$ away from critical levels and from sonic levels, the heights where $V^2=a^2$, Equation~\eqref{eq:SL-intro} is a formally self-adjoint equation in divergence form.
Divergence form means that the derivatives of $w$ appear only in $(\Abar w')'$.
The coefficients are then real, and the equation is formally self-adjoint, that is, two integrations by parts over an interval free of critical and sonic levels move the differential operator from one function to the other when both functions vanish at the ends of the interval.
It is not a Sturm--Liouville eigenvalue problem in the usual sense, in which the eigenvalue multiplies a given weight function, because the phase speed $c$ enters the coefficients $\Abar$ and $G/V$ nonlinearly through $V=U-c$.
At a regular critical level, one has
\begin{equation}\label{eq:Q-intro}
G(z_c)\;=\;\rhobar(z_c)\,Q(z_c),
\qquad
Q\;=\;U''+\frac{g}{a^{2}}\,U'
\;=\;\rhobar\Big(\frac{U'}{\rhobar}\Big)' .
\end{equation}
Thus the coefficient $G/V$ has at most a simple pole in $z$ at the critical level.
Its residue at $z_c$, the coefficient of $1/(z-z_c)$ in its expansion near $z_c$, is $\rhobar(z_c)Q(z_c)/U'(z_c)$.
The identity $Q=\rhobar(U'/\rhobar)'$ identifies the critical-layer coefficient with $\rhobar$ times the vertical gradient of the specific vorticity $U'/\rhobar$, the vorticity per unit mass.
In compressible boundary-layer stability theory the corresponding coefficient is $(\rho_0U')'$, the quantity in the generalized inflection-point criterion of Lees and Lin \cite{LeesLin1946,Mack1984}.
The two density weightings differ because gravity enters the present problem twice, and \cref{rem:leeslin} explains the relation.
The same specific-vorticity gradient controls the momentum that a wave deposits at a critical level in compressible wave--mean-flow theory \cite{EliassenPalm1961,AndrewsMcIntyre1978,Buhler2014}.

The air acts on the interface through a single complex number for each pair $(k,c)$.
Suppose the interface is displaced by $\hat\zeta\,e^{\im k(x-ct)}$, so that $\hat\zeta$ is the interface displacement amplitude.
The air-pressure perturbation evaluated on the displaced interface includes the change $-\rho_0(0)g\hat\zeta$ of the base-state pressure over the displacement.
This is the Lagrangian pressure perturbation, and it is proportional to $\hat\zeta$.
The air-side pressure response $\Za(k,c)$ of Equation~\eqref{eq:Za-def} is this Lagrangian pressure perturbation divided by the surface air density $\rho_0(0)=\eps\rhow$ and by the displacement amplitude $\hat\zeta$.
It is the pressure that the air exerts on the interface per unit of interface displacement, divided by the surface air density, and it depends on the normalized profiles $(\rhobar,a,U)$ but not on $\eps$.
The water side enters through the function $\Dw(k,c)=kc^{2}\coth(k\hw)-g-\sigma k^{2}/\rhow$ of Equation~\eqref{eq:Dw-def}, whose positive zero is the free wave speed $c_k$.
For $\Imc c>0$, a normal mode exists exactly when the water-wave dynamics at phase speed $c$ is balanced by the air pressure,
$\Dw(k,c)=\eps\Za(k,c)$ (\cref{prop:reduction}).
Since the air is light, the right-hand side is $O(\eps)$, so roots near $\pm c_k$ are $O(\eps)$ perturbations of $\pm c_k$, and the growth rate of the root near $c_k$ is set by the imaginary part of $\Za$ there.

For $\Imc c>0$, Equation~\eqref{eq:SL-intro} has no singularity, because $V=U-c$ does not vanish, and $\Za(k,\cdot)$ is a holomorphic function of $c$ except at the values of $c$ for which the air solution vanishing at the top boundary also vanishes at the interface, $w(0;k,c)=0$.
At such values the normalization by the interface displacement is impossible and $\Za$ has a pole, so $\Za(k,\cdot)$ is meromorphic on the upper half-plane, that is, holomorphic except at isolated poles.
The unstable root we seek has a small positive imaginary part, so we need the behavior of $\Za(k,c)$ as $c=c_r+\im\delta$ tends to the real phase speed $c_r$, with $\delta>0$ decreasing to zero.
At real phase speeds, the equation for $w$ is singular at every critical level, where $V=0$, and its solution is not defined directly.
The limiting absorption principle of scattering theory resolves this.
It defines the boundary value $\Za^{+}(k,c_r)$ as the limit of $\Za(k,c_r+\im\delta)$ as $\delta\downarrow0$.
Physically, the small imaginary part $\delta$ is a weak exponential growth in time, and the limit selects the solution that a slowly growing wave sees.
In the classical theory this is the rule of Lin and Miles for the branch of the logarithm at the critical level \cite{Miles1957,DrazinReid2004}.
A real phase speed $c_r$ is regular when every zero of $U-c_r$ is simple, and it is non-resonant, in the sense of Hypothesis \ref{hyp:nonres}, when the limiting air solution does not vanish at the interface.
Resonance would mean that $c_r$ is also the phase speed of a free oscillation of the air layer alone, between a flat interface held at rest and the top boundary, and then $\Za$ is singular there.
For a regular, non-resonant $c_r$ satisfying uniform subsonicity, \cref{thm:limabs} proves that the limit exists, at the rate $O\bigl(\delta(1+|\log(\delta/V_\ast)|)\bigr)$, where $V_\ast$ is the reference speed of \S\ref{subsec:local}.
The limiting solution $w_\ast$ is continuous, and near each critical level $z_j$ it has the expansion $w_\ast(z_j)\bigl[1+\tfrac{Q(z_j)}{U'(z_j)}(z-z_j)\log|z-z_j|\bigr]$ up to a $C^1$ remainder, with the logarithm acquiring the imaginary part $-\im\pi$ on the side where $U<c_r$.
The imaginary part of the boundary value is a sum of critical-level contributions,
\begin{equation}\label{eq:ImZ-intro}
\Imc \Za^{+}(k,c_r)
\;=\;-\frac{\pi}{k^{2}}\sum_{j=1}^{m}
\frac{\rhobar(z_j)\,Q(z_j)}{|U'(z_j)|}\,|w(z_j)|^{2},
\end{equation}
where $z_1<\dots<z_m$ are the critical levels, the solutions of $U(z_j)=c_r$, and $w$ is normalized so that the interface displacement amplitude equals one.

Equation~\eqref{eq:ImZ-intro} has a physical meaning.
Write $\xi$ for the vertical displacement amplitude of the air parcels, $\qbar$ for the normalized pressure amplitude, and an asterisk for complex conjugation.
The quantity $\Jf=\Imc(\qbar\xi^{\ast})$, the correlation between the pressure and the slope of the displaced material surfaces, is the vertical flux of horizontal wave momentum, called the wave-induced form stress in the wind-wave literature and the Eliassen--Palm flux in the wave--mean-flow literature \cite{EliassenPalm1961}.
Here a material surface is a surface that consists of the same air parcels at all times, such as the surface formed by the parcels that lie at one height in the base state.
For a neutral mode, one with real phase speed $c$ that neither grows nor decays, this flux is constant in height except at the critical levels, where it jumps.
The $j$th summand in Equation~\eqref{eq:ImZ-intro} is the jump of $\Jf$ across $z_j$ (\cref{lem:flux,thm:limabs}), that is, the momentum that the wave deposits in the mean wind at that critical level.
The balance has the same form as in the incompressible theory, with $U''$ replaced by $Q$.
For the criterion to be effective, the limiting solution must not vanish at a critical level, and a monotonicity argument shows that it does not vanish at the highest one (\cref{lem:positivity}).

The instability statement follows from the dispersion relation by a perturbation argument in the complex $c$-plane.
Dividing by the interface trace $w(0;k,c)$ introduces the poles of $\Za$ in $c$, so we work with the equivalent holomorphic form $\Edet(k,c,\eps)=0$ of \cref{def:determinant}, where the function $\Edet$ is the relation multiplied through by that trace.
Here the interface trace $w(0;k,c)$ is the value at $z=0$ of the air solution with $w(\ha)=0$ and $w'(\ha)=1$.
If $\Imc\Za^+(k,c_k)>0$, \cref{thm:main} gives a unique root $c_\eps$ in the set $\{|c-c_k|<r_0,\ \Imc c>0\}$ with $r_0$ independent of $\eps$, and every root in that set lies within $O(\eps)$ of $c_k$.
Moreover,
\begin{equation}\label{eq:expansion-intro}
c_\eps=c_k+\eps\,c_1+O\bigl(\eps^{2}(1+\log(1/\eps))\bigr),
\qquad
c_1=\frac{\Za^{+}(k,c_k)}{2kc_k\coth(k\hw)} ,
\end{equation}
and the growth rate is
$\gamma_\eps=k\Imc c_\eps
=\eps\,\tfrac{\tanh(k\hw)}{2c_k}\,\Imc\Za^{+}(k,c_k)
+O(\eps^{2}(1+\log(1/\eps)))$.
The sign condition $c_kQ(z_j)\le0$ at every critical level, with strict inequality at the highest one, is sufficient for instability (\cref{cor:sufficient}).
For a monotone increasing wind with one critical level this becomes $Q(z_c)<0$, as stated in \cref{thm:headline}.
If $c_k$ lies outside the range of $U$, so that there is no critical level, the root near $c_k$ is instead real and simple, and it is given by a convergent power series in $\eps$ (\cref{thm:necessity}).
Hence the critical layer is the only source of growth near $c_k$.

To quantify the incompressible approximation, Section~\ref{sec:lowmach} considers neutral atmospheres parametrized by the Mach number $M$, with sound speed $a_M=a_\ast/M$, where $a_\ast$ is a reference speed of the order of the wind and wave speeds.
\cref{thm:lowmach} proves
$|\Za^{+,M}(k,c_r)-\Za^{+,\mathrm{inc}}(k,c_r)|\le CM^2$
uniformly for $c_r$ in compact sets of regular, non-resonant phase speeds and on a closed neighborhood of $c_k$ in the upper half-plane that does not depend on $M$.
Here $C$ is independent of $M$, and the superscripts $M$ and $\mathrm{inc}$ mark the atmosphere with Mach number $M$ and the incompressible problem.
A compact set of real phase speeds is a closed and bounded set.
Hence the leading coefficient $c_1(M)$ of the compressible problem satisfies $c_1(M)=c_1(0)+O(M^2)$, with $c_1(0)$ its incompressible value.
Under the incompressible sign condition, the unstable roots of the compressible and incompressible problems satisfy
$|c_\eps(M)-c_\eps(0)|\le C\eps M^2$ at each wavenumber $k$, with $C$ independent of $\eps$ and $M$.
This is the quantitative form of the statement that the incompressible theory is accurate for wind waves.

For wind profiles that increase with height with shear decreasing upward, $U'>0$ and $U''<0$, as for the logarithmic profile, the hydrostatic term $(g/a^2)U'$ in $Q$ is positive and shifts the criterion from $U''<0$ to $Q<0$.
For critical levels at the heights typical of wind waves over the sea, meters to tens of meters above the surface, the shift is negligible (\S\ref{subsec:polytropic-wind}).
For the bounded shear-layer profile $U=U_\infty(1-e^{-z/L})$, with limiting speed $U_\infty$ and shear depth $L$, in the neutral constant-sound-speed atmosphere, the condition $a^2<gL$ gives $Q>0$ although $U''<0$.
At each admissible wavenumber the local $O(\eps)$ Miles bifurcation from $c_k$ is then absent (\cref{prop:suppression}).
This regime requires a sound speed far below that of air.
The example shows that the sign criterion of the incompressible theory does not extend to a compressible atmosphere in general.

\subsection{Methods and relation to prior work}\label{subsec:literature}

The critical-layer argument is built around the variable $\chibar=\im k\qbar$, the horizontal pressure-gradient amplitude scaled by the surface air density.
Unlike the equation for $w$, its first-order equation is regular at $U=c_r$.
Variation of constants then reconstructs $w$ from $\chibar$ and isolates a single logarithmic term with coefficient $Q/U'$.
A contraction argument for a single integral equation gives the local expansion near each critical level and estimates that are uniform as $\Imc c\downarrow0$.
Such an argument writes the solution as the unique fixed point of a map that multiplies every distance by at most a constant factor smaller than one, the contraction constant.
The compressible theory rests on two exact cancellations.
The hydrostatic terms cancel in the flux identity of \cref{lem:flux}, and the singular terms cancel in the derivative of the monotonicity functional of \cref{lem:H-monotone}.
These identities yield the residue formula and the sign criterion.
Rouch\'e's theorem on a neighborhood of $c_k$ in the upper half-plane, applied after an estimate that places all nearby roots within $O(\eps)$ of $c_k$, proves existence, uniqueness, and simplicity of the unstable root.
For $\Abar=\rhobar=1$ and $b=0$, the same argument recovers the incompressible critical-layer theory of \cite[\S4]{BSWZ2016}.
Appendix~\ref{app:operator} supplies the formulation on fixed reference domains, its linearization, and the construction of the operator $\Lop_k$.

Maleewong and Grimshaw \cite{MaleewongGrimshaw2024} study internal-wave radiation in a density-stratified incompressible atmosphere, a mechanism absent when $N^2=0$.
Compressible vortex-sheet results \cite{Miles1958VortexSheet,CoulombelSecchi2008,XieZhao2024} treat piecewise-constant flows without critical layers.
Lees and Lin \cite{LeesLin1946} and Mack \cite{Mack1984} derive the local compressible critical-layer algebra for boundary layers over solid walls, with no free surface or gravity.
Other related results include compressible Rayleigh--Taylor normal modes \cite{GuoTice2011}, compressible Miles--Howard and semicircle theorems \cite{Chimonas1970,Eckart1963}, and viscous compressible--incompressible free-boundary theory \cite{Denisova2000}.
The incompressible necessity and positivity arguments have precedents in \cite{BSWZ2016,Liu2024TwoPhase}, while the low-Mach result is related to the incompressible-limit theory \cite{KlainermanMajda1981,MetivierSchochet2001,LindbladLuo2018,DisconziLuo2020,LuoZhang2026}.

A complementary numerical approach keeps atmospheric compressibility in wave-phase-resolved simulations.
Zhu et al.\ \cite{ZhuEtAl2023WRF} introduced a moving lower boundary in WRF--LES and validated the method for laminar and turbulent flow over prescribed monochromatic waves.
Li et al.\ \cite{LiEtAl2025WRF} extended the method to prescribed broadband wave fields with different propagation directions and directional spreading.
These calculations solve the compressible atmospheric equations with turbulence modeling while prescribing the wave surface.
The present paper treats the inviscid two-phase free-boundary eigenvalue problem, in which the interface motion and the air pressure are coupled.

\subsection{Modeling assumptions and scope}

Neutral stratification removes internal gravity waves (\cref{rem:general-strat}).
The impermeable top boundary replaces the outgoing acoustic radiation condition, which lets sound waves leave the air layer upward, by a reflecting boundary, so the free acoustic oscillations of the air layer occur only at a discrete set of frequencies.
In the uniform-wind example its effect is exponentially small when the top boundary lies several decay lengths above the interface, and this decay length grows as the speed of the wind relative to the wave approaches the sound speed (\S\ref{subsec:KH}).
Uniform subsonicity excludes sonic levels and hence the higher Mack modes of compressible boundary layers, which are acoustic waves trapped in a region where the flow relative to the wave is supersonic \cite{Mack1984}.
The condition $U(0)=0$ excludes the vortex-sheet Kelvin--Helmholtz mechanism at the interface \cite{BSWZ2016}.
Surface tension is included and enters the analysis at each wavenumber through $c_k$.
Viscosity, turbulence closures, oblique three-dimensional modes, and nonlinear instability are not considered.

\subsection{Outline}

We begin in Section~\ref{sec:formulation} with the nonlinear two-phase problem and its laminar states.
Section~\ref{sec:linearization} derives the linearized equations and reduces them to a normal-mode problem.
In Section~\ref{sec:structure}, we obtain the displacement--pressure formulation, the flux identities, and the holomorphic form of the dispersion relation used in the instability analysis.

Section~\ref{sec:criticallayer} establishes limiting absorption at the critical layers and derives the residue formula for the boundary value of the pressure response.
Section~\ref{sec:instability} first proves the necessary condition for instability and then uses these critical-layer results to prove the main instability theorem.
Section~\ref{sec:lowmach} recovers the incompressible theory in the low-Mach limit.
Section~\ref{sec:examples} discusses uniform wind, logarithmic and shear-layer profiles, and a regime of strong compressibility in which the Miles bifurcation does not occur.

Appendix~\ref{app:thermo} contains the thermodynamic identities and the general stratified equations.
Appendix~\ref{app:proofs} gives the technical estimates for the critical-layer analysis and the low-Mach limit.
Appendix~\ref{app:operator} develops the fixed-domain formulation and establishes the operator interpretation of the normal-mode results.
\section{The nonlinear two-phase system and its laminar states}
\label{sec:formulation}

This section formulates the problem studied in this paper.
We write the nonlinear equations for compressible air over incompressible water, identify the steady laminar states about which we linearize, introduce the air--water density ratio as the small parameter, and list the hypotheses used in the rest of the paper.

\subsection{Geometry and variables}

Let $\ha>0$ and $\hw>0$ denote the undisturbed air and water depths.
We work in two spatial dimensions $(x,z)\in\T_{k_0}\times\R$, where $\T_{k_0}=\R/(2\pi/k_0)\Z$ and $k_0>0$.
In the normal-mode analysis below, $k>0$ denotes the horizontal wavenumber of the surface wave at the air--water interface and the associated perturbations in the air and water.
In the present study, we assume that the horizontal domain is periodic, and hence we consider a single wavenumber $k=nk_0$ with a positive integer $n$.
The same analysis may apply when the horizontal domain lies in $x\in\R$, in which case $k$ can be any positive real number.
Gravity is $-g\mathbf e_z$, where $g>0$ and $\mathbf e_z=(0,1)$ is the upward unit vector.
At time $t$ the interface is the graph
\begin{equation}
\Gamma(t)=\{(x,z):z=\eta(x,t)\},
\end{equation}
with air in $\Omega_{\mathrm a}(t)=\{\eta(x,t)<z<\ha\}$ and water in $\Omega_{\mathrm w}(t)=\{-\hw<z<\eta(x,t)\}$.
The boundaries at $z=\ha$ and $z=-\hw$ are flat and impermeable, so the vertical velocity vanishes there, see Equation~\eqref{eq:rigid} below.
We call the upper boundary $z=\ha$ the top boundary and $z=-\hw$ the bottom boundary.
The air velocity is denoted by $\mathbf u_{\mathrm a}=(u_{\mathrm a},w_{\mathrm a})$.
The thermodynamic state of the air is described by the density $\rho_{\mathrm a}>0$ and specific entropy $S_{\mathrm a}$, which determine the pressure $p_{\mathrm a}$ through the equation of state.
The water is assumed to be incompressible and has constant density $\rhow>0$, velocity $\mathbf u_{\mathrm w}=(u_{\mathrm w},w_{\mathrm w})$, and pressure $p_{\mathrm w}$.
The interface elevation is $\eta$.

\subsection{Equations of motion}

In the air phase $\Omega_{\mathrm a}(t)$, we assume that air is compressible and its motion satisfies the full compressible Euler equation system:
\begin{subequations}\label{eq:air-nonlinear}
\begin{align}
&\partial_t\rho_{\mathrm a}+\nabla\cdot(\rho_{\mathrm a}\mathbf u_{\mathrm a})=0,\\
&\rho_{\mathrm a}\bigl(\partial_t\mathbf u_{\mathrm a}
 +\mathbf u_{\mathrm a}\cdot\nabla\mathbf u_{\mathrm a}\bigr)
 +\nabla p_{\mathrm a}=-\rho_{\mathrm a}g\,\mathbf e_z,\\
&\partial_tS_{\mathrm a}+\mathbf u_{\mathrm a}\cdot\nabla S_{\mathrm a}=0,\label{eq:entropy-transport}\\
&p_{\mathrm a}=P(\rho_{\mathrm a},S_{\mathrm a}),
\end{align}
\end{subequations}
where the equation of state $P(\rho,S)$ satisfies the thermodynamic stability condition $P_\rho(\rho,S)>0$, where $P_\rho=\partial_\rho P$, on an open set containing the thermodynamic states of the laminar flow defined below.
On this set, we also assume that $P$ is four times continuously differentiable with respect to $\rho$ and $S$.
Equation~\eqref{eq:entropy-transport} describes entropy transport, and the adiabatic sound speed is defined by $a^2=P_\rho$.
We write $C^k$ for the class of $k$ times continuously differentiable functions, and $C^k([0,\ha])$ for such functions on $[0,\ha]$.
For a $C^3$ laminar density profile, this regularity of $P$ gives a $C^3$ sound-speed profile, as required in the critical-layer analysis.

In the water phase $\Omega_{\mathrm w}(t)$, the flow is assumed to be incompressible and inviscid, so the motion satisfies the incompressible Euler equations:
\begin{equation}\label{eq:water-nonlinear}
\rhow\bigl(\partial_t\mathbf u_{\mathrm w}
+\mathbf u_{\mathrm w}\cdot\nabla\mathbf u_{\mathrm w}\bigr)
+\nabla p_{\mathrm w}=-\rhow g\,\mathbf e_z,
\qquad
\nabla\cdot\mathbf u_{\mathrm w}=0 .
\end{equation}
The interface $\eta$ is material for both air and water phases, that is, fluid parcels on the interface remain on it,
\begin{equation}\label{eq:kinematic-nonlinear}
\partial_t\eta+u_{\mathrm a}\,\partial_x\eta=w_{\mathrm a},
\qquad
\partial_t\eta+u_{\mathrm w}\,\partial_x\eta=w_{\mathrm w}
\qquad\text{on }\Gamma(t),
\end{equation}
and the normal stress balance is applied at the interface,
\begin{equation}\label{eq:dynamic-nonlinear}
p_{\mathrm w}-p_{\mathrm a}=\sigma\kappa
\qquad\text{on }\Gamma(t),
\qquad
\kappa:=\nabla\cdot\mathbf n,
\qquad
\mathbf n:=\frac{(-\partial_x\eta,1)}{\sqrt{1+(\partial_x\eta)^2}},
\end{equation}
where $\sigma\ge0$ is the surface tension coefficient and the normal vector $\mathbf n$ points from water into air.
With this convention $\kappa=-\eta_{xx}(1+\eta_x^2)^{-3/2}$, so at linear order about a flat interface $\kappa\simeq-\eta_{xx}$.
At the impermeable top and bottom boundaries the normal velocity vanishes, and we impose the no-penetration condition:
\begin{equation}\label{eq:rigid}
w_{\mathrm a}=0\ \text{ at } z=\ha,
\qquad
w_{\mathrm w}=0\ \text{ at } z=-\hw .
\end{equation}
Inviscid theory imposes no condition on the tangential velocity, so the tangential velocity may jump across $\Gamma(t)$, which then forms a vortex sheet.
In Hypothesis~\ref{hyp:noslip} below, we assume that the interfacial velocity is continuous in the laminar states.

\subsection{Laminar states}\label{subsec:laminar}

We consider a steady parallel wind over water at rest.
Let $U\in C^4([0,\ha])$ and $\rho_0\in C^3([0,\ha])$, with $\rho_0>0$, and let $S_\ast$ be a constant specific entropy.
The candidate base state is
\begin{equation}\label{eq:laminar}
\eta\equiv0,\qquad
\mathbf u_{\mathrm w}\equiv0,\qquad
\mathbf u_{\mathrm a}=(U(z),0),\qquad
\rho_{\mathrm a}=\rho_0(z),\qquad
S_{\mathrm a}=S_\ast,
\end{equation}
with pressures $p_0(z)=P(\rho_0(z),S_\ast)$ in the air and $p_{\mathrm w,0}(z)=p_0(0)-\rhow g z$ in the water.
The water pressure satisfies hydrostatic balance, and its value $p_0(0)$ at $z=0$ equals the air pressure, as the stress balance requires at the flat interface, where $\kappa=0$.

Substituting Equation~\eqref{eq:laminar} into Equations~\eqref{eq:air-nonlinear}--\eqref{eq:rigid}, we find that all equations and boundary conditions are satisfied except possibly the vertical momentum equation in the air.
Thus Equation~\eqref{eq:laminar} is an exact steady solution if and only if the air is also in hydrostatic balance,
\begin{equation}\label{eq:hydrostatic}
p_0'(z)=-\rho_0(z)\,g
\qquad\Longleftrightarrow\qquad
a^2(z)\,\rho_0'(z)=-\rho_0(z)\,g,
\end{equation}
where $a^2(z):=P_\rho(\rho_0(z),S_\ast)$ and primes denote differentiation with respect to $z$.
The equivalence follows by differentiating the equation of state at constant entropy, which gives $p_0'=a^2\rho_0'$.
For a prescribed equation of state and surface density $\rho_0(0)$, Equation~\eqref{eq:hydrostatic} determines the density profile $\rho_0$.
It places no further restriction on the wind profile $U$.

The stratification is characterized by the Brunt--V\"ais\"al\"a frequency $N$, the frequency of buoyancy oscillations of a vertically displaced air parcel.
Using Equation~\eqref{eq:hydrostatic}, we obtain
\begin{equation}\label{eq:neutral}
\frac{\rho_0'}{\rho_0}=-\frac{g}{a^2}=:-b(z),
\qquad
N^2:=-g\Big(\frac{\rho_0'}{\rho_0}+\frac{g}{a^2}\Big)\equiv0 .
\end{equation}
Thus the constant-entropy hydrostatic atmosphere is neutrally stratified.

\begin{example}[Polytropic atmosphere]\label{ex:polytropic}
Consider an ideal polytropic gas with equation of state $P(\rho,S)=K(S)\rho^\gamma$, where $\gamma>1$ is the adiabatic exponent and $K(S_\ast)>0$.
At constant specific entropy $S_\ast$, hydrostatic balance gives
\begin{equation}\label{eq:polytropic}
a^2(z)=a_0^2-(\gamma-1)gz,
\qquad
\rho_0(z)=\rho_0(0)\Bigl(\frac{a^2(z)}{a_0^2}\Bigr)^{\frac{1}{\gamma-1}},
\end{equation}
where $a_0:=a(0)$.
We require $\ha<a_0^2/((\gamma-1)g)$ so that the sound speed and density are positive on $[0,\ha]$.

Both the sound speed and density decrease with height.
For air with $a_0\approx340\,\mathrm{m\,s^{-1}}$ and $\gamma=1.4$, the limiting height is approximately $29.5\,\mathrm{km}$.
Over the lowest $100\,\mathrm m$, the sound speed decreases by approximately $0.2\,\%$.
\end{example}

\begin{example}[Constant-sound-speed atmosphere]\label{ex:constant-sound-speed}
Consider the equation of state $P(\rho,S)=a_{\mathrm e}^2\rho+p_{\mathrm e}(S)$, where $a_{\mathrm e}>0$ is constant and $p_{\mathrm e}$ is smooth.
At constant specific entropy $S_\ast$, hydrostatic balance gives
\begin{equation}
a^2(z)=a_{\mathrm e}^2,
\qquad
\rho_0(z)=\rho_0(0)e^{-gz/a_{\mathrm e}^2}.
\end{equation}
The sound speed and density are positive on $[0,\ha]$ for every finite $\ha>0$.

The sound speed is constant, while the density decreases exponentially with scale height $a_{\mathrm e}^2/g$, the height over which it decreases by the factor $e$.
Thus varying $a_{\mathrm e}$ changes both the sound speed and the density scale height.
This equation of state admits a neutrally stratified atmosphere with constant sound speed.
For an ideal gas with constant adiabatic exponent, constant sound speed implies an isothermal atmosphere, which is stably stratified.
For uniform wind, the present model leads to a normal-mode equation with constant coefficients, as shown in \S\ref{sec:examples}.
\end{example}

\subsection{The small parameter and scaling of the equation of state}
\label{subsec:eps-family}

We study wind-induced instability in the regime where the air density is small compared with the water density.
The corresponding small parameter is the surface density ratio
\begin{equation}
\eps:=\frac{\rho_0(0)}{\rhow}\ll1 .
\end{equation}
This ratio measures the atmospheric contribution to the water-wave dispersion relation.
We use its smallness to study how the air-side pressure response perturbs a free gravity--capillary wave and under what conditions this perturbation produces growth.

To examine the dependence on $\eps$, we let the wind profile $U$, the normalized density profile $\rhobar$, and the sound-speed profile $a$ be given functions that do not depend on $\eps$.
For each $\eps>0$, we define the physical air density by
\begin{equation}
\rho_0(z)=\eps\rhow\,\rhobar(z),
\qquad
\rhobar(0)=1,
\end{equation}
and scale the equation of state according to
\begin{equation}\label{eq:eos-family}
P_\eps(\rho,S):=\eps\rhow\,\bar P\bigl(\rho/(\eps\rhow),\,S\bigr),
\end{equation}
where $\bar P$ is a function of the normalized density and specific entropy that does not depend on $\eps$.
Evaluating its density derivative at the base state gives
\begin{equation}
a^2(z)
=\partial_\rho P_\eps\bigl(\eps\rhow\rhobar(z),S_\ast\bigr)
=\partial_1\bar P\bigl(\rhobar(z),S_\ast\bigr),
\end{equation}
where $\partial_1$ denotes differentiation with respect to the first argument.
Thus the sound-speed profile is independent of $\eps$.
Likewise, substituting the scaled density into Equation~\eqref{eq:hydrostatic} gives
\begin{equation}
a^2(z)\rhobar'(z)=-g\rhobar(z),
\end{equation}
which is also independent of $\eps$.
Thus changing $\eps$ changes the overall air density without changing its normalized vertical distribution or the wind and sound-speed profiles.
In particular, $b=g/a^2$ is independent of $\eps$, while the total air mass is proportional to $\eps$.

In the following analysis, we normalize the air-side pressure perturbation and momentum flux by the surface air density $\eps\rhow$.
The normalized momentum flux and the air-side pressure response $\Za$ are therefore expressed in terms of $\rhobar$.
This normalization leaves the formulation dimensional, as in \cite{BSWZ2016}.

\subsection{Hypotheses}\label{subsec:hypotheses}

We introduce the hypotheses under which we study perturbations of a water wave at a given horizontal wavenumber $k>0$.
The dispersion relation for a free water wave uncoupled from the air is $\Dw(k,c)=0$, where
\begin{equation}\label{eq:Dw-def}
\Dw(k,c):=kc^2\coth(k\hw)-g-\sigma k^2/\rhow.
\end{equation}
Its positive root,
\begin{equation}\label{eq:ck}
c_k:=\Bigl[\Bigl(\frac gk+\frac{\sigma k}{\rhow}\Bigr)\tanh(k\hw)\Bigr]^{1/2},
\end{equation}
serves as the reference phase speed for the following instability analysis.

Hypothesis~\ref{hyp:all} below summarizes the assumptions used in the present study.
In Hypothesis~\ref{hyp:all}, Hypotheses \ref{hyp:reg}--\ref{hyp:noslip} specify the regularity and interface velocity of the base state.
Hypotheses \ref{hyp:critical}--\ref{hyp:nonres} are used to analyze the air-side pressure response at real phase speeds near $c_k$.
The surface-tension condition in Hypothesis \ref{hyp:sigma} allows both gravity waves and gravity--capillary waves in the current analysis.

\begin{hypothesis}\label{hyp:all}
\leavevmode
\begin{enumerate}[label=\textup{(H\arabic*)},leftmargin=3.2em]
\item\label{hyp:reg} The base-state wind velocity $U$, normalized air density $\rhobar$, and sound speed $a$ satisfy
\begin{equation}
U\in C^4([0,\ha]),
\qquad
\rhobar,a\in C^3([0,\ha]).
\end{equation}
The density and sound speed are positive on $[0,\ha]$ and satisfy the neutrality relation in Equation~\eqref{eq:neutral}.
\item\label{hyp:noslip} The horizontal velocity of the base wind vanishes at the interface,
\begin{equation}
U(0)=0.
\end{equation}
Because the water is at rest in the base state, this condition makes the base velocity continuous across the air--water interface.
\item\label{hyp:critical} The critical levels associated with the reference phase speed $c_k$ are the heights at which the base wind velocity satisfies $U(z)=c_k$.
Every such level in $[0,\ha]$ lies in $(0,\ha)$ and satisfies $U'(z)\neq0$.
We call these \emph{regular critical levels} and, when present, denote them by $z_1<\dots<z_m$.
\item\label{hyp:subsonic} The wind velocity relative to a wave of real phase speed $c$ is uniformly subsonic near $c_k$.
More precisely, there exists a constant $\delta_0>0$ such that
\begin{equation}
a^2(z)-(U(z)-c)^2\ge\delta_0\,a^2(z)
\end{equation}
for every height $z\in[0,\ha]$ and every real phase speed $c$ in a neighborhood of $c_k$.
We call this condition \emph{uniform subsonicity}.
\item\label{hyp:nonres} Let $w_\ast(z;c_r)$ denote the limiting vertical air-velocity amplitude constructed in \cref{prop:limit-solutions} as the complex phase speed $c=c_r+\im\delta$ approaches $c_r$ with $\delta\downarrow0$.
The solution is normalized at the upper boundary by
\begin{equation}
w_\ast(\ha;c_r)=0,
\qquad
\partial_z w_\ast(\ha;c_r)=1.
\end{equation}
We require its value at the air--water interface to satisfy
\begin{equation}
w_\ast(0;c_r)\neq0
\end{equation}
for every real phase speed $c_r$ in a neighborhood of $c_k$.
We call this condition \emph{non-resonance of the air boundary-value problem}.
It excludes a nontrivial limiting air solution, that is, one that is not identically zero, with both the interface and the upper boundary held fixed at the phase speeds under consideration.
\item\label{hyp:sigma} The surface tension coefficient at the air--water interface satisfies
\begin{equation}
\sigma\ge0.
\end{equation}
\end{enumerate}
\end{hypothesis}

\begin{remark}
These hypotheses allow us to determine the air-side pressure response near the free water-wave speed $c_k$ and use it to study the effect of wind on the wave.
The regularity of $U$, $\rhobar$, and $a$ in Hypothesis \ref{hyp:reg} is needed to differentiate the coefficients of the air equations and to estimate how their solutions vary with the wave speed.
The neutrality relation ensures hydrostatic balance and gives $N^2=0$.
Since the water is at rest in the base state, the condition $U(0)=0$ in Hypothesis \ref{hyp:noslip} makes the horizontal air and water velocities equal at the undisturbed interface, excluding a jump in tangential velocity.
The critical levels in Hypothesis \ref{hyp:critical} are the heights where the wind speed matches the reference wave speed, $U(z_j)=c_k$.
The requirement $U'(z_j)\neq0$ ensures that these heights move smoothly as the wave speed varies.
It also ensures that the coefficient $\rhobar k^2+G/V$ in Equation~\eqref{eq:SL-intro} has at most a simple pole at each critical level, with residue proportional to $Q(z_j)$, which may be zero.
Hypothesis \ref{hyp:subsonic} prevents the denominator $a^2-(U-c)^2$ in the air equations from vanishing for real phase speeds near $c_k$.
Hypothesis \ref{hyp:nonres} concerns the limiting vertical velocity amplitude $w_\ast$, obtained by normalizing the air solution at the upper boundary.
Since $w_\ast(0;c_r)\neq0$, the solution can be rescaled to satisfy the kinematic condition at the moving water surface, and the resulting pressure response is finite.
If $w_\ast(0;c_r)=0$, the limiting air problem instead admits a nontrivial solution with the interface held fixed.
When this zero occurs at $c_r=c_k$, the uncoupled air and water problems share the phase speed $c_k$ at $\eps=0$, which is the resonance that the present perturbation argument excludes (\cref{rem:determinant}).
Non-resonance still allows critical levels where $U(z)=c_r$.
Finally, Hypothesis \ref{hyp:sigma} includes both $\sigma=0$, corresponding to gravity waves, and $\sigma>0$, corresponding to gravity--capillary waves.
\end{remark}

\begin{remark}\label{rem:hyp-axis}
For a complex phase speed $c$ with $\Imc c\neq0$, the quantities $V=U-c$ and $a^2-V^2$ are nonzero at every height, so the air normal-mode equation has no singular coefficients (\cref{lem:regular-solve}).
Hypotheses \ref{hyp:reg}--\ref{hyp:noslip} are therefore sufficient to construct the holomorphic dispersion function $\Edet$ of \cref{def:determinant} and identify its zeros with the eigenvalues $\lambda=-\im kc$ of the operator $\Lop_k$ of \cref{thm:operator} with $\Rec\lambda\neq0$.
The additional Hypotheses \ref{hyp:critical}--\ref{hyp:nonres} are used to control the air solution as $c$ approaches a real phase speed near $c_k$.
This limiting analysis gives the estimates needed for the application of Rouch\'e's theorem that proves instability in \S\ref{sec:instability}.
In some cases, the non-resonance condition in Hypothesis \ref{hyp:nonres} follows from the other assumptions and need not be checked separately.
Specifically, if the critical-level values $Q(z_j)$ are all nonnegative or all nonpositive, with $Q(z_m)\neq0$ at the highest critical level, then \cref{lem:positivity} implies Hypothesis \ref{hyp:nonres}.
This criterion is used in \cref{cor:sufficient,cor:monotone,prop:suppression}.
Hypotheses \ref{hyp:subsonic} and \ref{hyp:nonres} need only be checked at $c_k$.
Since $[0,\ha]$ is compact and $U$ and $a$ are continuous, Hypothesis \ref{hyp:subsonic} holds exactly when $(U(z)-c_k)^2<a^2(z)$ for every $z\in[0,\ha]$.
Hypothesis \ref{hyp:nonres} holds when $w_\ast(0;c_k)\neq0$, because $w_\ast(0;c_r)$ depends continuously on $c_r$ near $c_k$ by \cref{prop:limit-solutions}, which does not use Hypothesis \ref{hyp:nonres}.
\end{remark}
\section{Linearization and normal-mode reduction}
\label{sec:linearization}

We linearize Equations~\eqref{eq:air-nonlinear}--\eqref{eq:rigid} about the laminar state in Equation~\eqref{eq:laminar} and then seek normal modes of the resulting equations.
This reduces the coupled air--water problem to a boundary-value problem in the air and a dispersion relation for the phase speed.
Following \cite[\S2]{BSWZ2016}, we derive the linearized equations by taking Eulerian variations along a one-parameter family of solutions.
These variations describe changes at fixed spatial points.
Vertical advection of the base-state density contributes a term involving $\rho_0'$ to the linearized mass conservation equation.
At the air--water interface, the pressure balance is evaluated at the displaced surface, so its linearization also includes the change in base pressure caused by the interface displacement.
This change equals the base-state pressure gradient times the displacement, and adding it to the Eulerian pressure perturbation gives the Lagrangian pressure variation on each side of the interface.

By Hypothesis \ref{hyp:noslip}, the base air and water velocities agree at the undisturbed interface.
The linearized kinematic conditions therefore require the vertical velocity perturbations on both sides of the interface to agree.
For a base state with an interfacial velocity jump, the treatment of the velocity variations involves the Hodge projection described in \cite[Remark 2.5]{BSWZ2016}, which removes the gradient part of a vector field and keeps its divergence-free part.
In the present study, the water velocity perturbation must satisfy incompressibility, whereas the air velocity perturbation has no such constraint.

We also seek to identify the growing normal modes obtained from the dispersion relation with eigenfunctions of the linearized fluid equations.
For this purpose, we introduce the operator $\Lop_k$, which expresses the time derivatives of the perturbation amplitudes at wavenumber $k$ in terms of those amplitudes.
A normal mode with temporal factor $e^{-\im kct}$ then corresponds to an eigenfunction with eigenvalue $\lambda=-\im kc$, and exponential growth corresponds to $\Rec\lambda>0$.
Establishing this correspondence requires a precise choice of function spaces and operator domain that incorporates the boundary conditions.
Appendix~\ref{app:operator} provides this construction on fixed reference domains and, under the additional thermodynamic regularity stated there, identifies the linearized equations with the Fr\'echet derivative of the nonlinear formulation.
The theorem at the end of this section states the resulting correspondence between the dispersion relation and the eigenvalues of $\Lop_k$.

\subsection{The linearized system}

We derive the first-order equations for perturbations of the laminar state, including the displacement of the air--water interface.
We differentiate the governing equations in the air and water at fixed spatial points.
When differentiating the interface conditions, we also account for the displacement of the air--water interface.

Let $\vartheta$ be a real perturbation parameter, and consider a family of solutions of Equations~\eqref{eq:air-nonlinear}--\eqref{eq:rigid} that are defined on a common time interval, whose air and water fields are twice continuously differentiable as functions of $(\vartheta,t,x,z)$ up to the boundary of each phase, and whose interface elevation $\eta$ is three times continuously differentiable as a function of $(\vartheta,t,x)$.
This regularity allows us to exchange the derivative with respect to $\vartheta$ with the derivatives with respect to $t$, $x$, and $z$, including the second derivative $\partial_x^2\eta$ in the curvature term of the dynamic condition.
At $\vartheta=0$, the solution is the laminar state in Equation~\eqref{eq:laminar}.
Differentiating with respect to $\vartheta$ at $\vartheta=0$ gives the first-order perturbations of the interface and the fluid variables.
We denote the first variation of the interface elevation by
\begin{equation}
\zeta:=\left.\partial_\vartheta\eta\right|_{\vartheta=0}.
\end{equation}
For the fluid variables, we take variations at fixed interior points that remain in their respective phases for sufficiently small $\vartheta$, and define
\begin{align}
\begin{aligned}
r&:=\left.\partial_\vartheta\rho_{\mathrm a}\right|_{\vartheta=0},
&\qquad
s&:=\left.\partial_\vartheta S_{\mathrm a}\right|_{\vartheta=0},
\\
(u,w)&:=\left.\partial_\vartheta\mathbf u_{\mathrm a}\right|_{\vartheta=0},
&\qquad
\mathbf v&:=\left.\partial_\vartheta\mathbf u_{\mathrm w}\right|_{\vartheta=0}.
\end{aligned}
\end{align}
Here $r$ and $s$ are the density and entropy perturbations in the air, and $(u,w)$ and $\mathbf v$ are the air and water velocity perturbations, respectively.
The Eulerian pressure variations are
\begin{equation}
q:=\left.\partial_\vartheta p_{\mathrm a}\right|_{\vartheta=0}
=a^2r+P_Ss,
\qquad
q_{\mathrm w}:=\left.\partial_\vartheta p_{\mathrm w}\right|_{\vartheta=0},
\end{equation}
where $P_S(z):=\partial_S P(\rho_0(z),S_\ast)$, which is of class $C^3$ on $[0,\ha]$ because $P$ is four times continuously differentiable and $\rho_0\in C^3([0,\ha])$.
The expression for $q$ follows from differentiating the equation of state at the laminar state.

\begin{proposition}\label{prop:linearized}
In the undisturbed air region $\Omega_{\mathrm a}=\{0<z<\ha\}$, the variations satisfy
\begin{subequations}\label{eq:linear-air}
\begin{align}
&\partial_t r+U\partial_x r+\rho_0'w
 +\rho_0(\partial_xu+\partial_zw)=0,
 \label{eq:lin-cont}\\
&\rho_0\bigl(\partial_tu+U\partial_xu+U'w\bigr)+\partial_xq=0,
 \label{eq:lin-xmom}\\
&\rho_0\bigl(\partial_tw+U\partial_xw\bigr)+\partial_zq=-gr,
 \label{eq:lin-zmom}\\
&\partial_ts+U\partial_xs=0,
 \qquad q=a^2r+P_Ss.
 \label{eq:lin-ent}
\end{align}
\end{subequations}
In the undisturbed water region $\Omega_{\mathrm w}=\{-\hw<z<0\}$, the variations satisfy
\begin{equation}\label{eq:linear-water}
\rhow\,\partial_t\mathbf v+\nabla q_{\mathrm w}=0,
\qquad
\nabla\cdot\mathbf v=0.
\end{equation}
The linearized kinematic and dynamic conditions at $z=0$ are
\begin{align}
&\partial_t\zeta
=w|_{z=0}
=\left.\mathbf v\cdot\mathbf e_z\right|_{z=0},
 \label{eq:lin-kinematic}\\
&\bigl(q_{\mathrm w}|_{z=0}-\rhow g\,\zeta\bigr)
-\bigl(q|_{z=0}-\rho_0(0)g\,\zeta\bigr)
=-\sigma\,\partial_x^2\zeta.
 \label{eq:lin-dynamic}
\end{align}
At the fixed top and bottom boundaries, impermeability gives
\begin{equation}
w|_{z=\ha}=0,
\qquad
\left.\mathbf v\cdot\mathbf e_z\right|_{z=-\hw}=0.
\end{equation}
\end{proposition}

\begin{proof}
We first differentiate the governing equations for the air and water in Equations~\eqref{eq:air-nonlinear}--\eqref{eq:water-nonlinear} with respect to $\vartheta$, keeping the spatial coordinates and time fixed, and evaluate at $\vartheta=0$.
In the air phase, the base velocity is $(U(z),0)$, so transport by the base wind produces the operator $\partial_t+U\partial_x$.
The terms $\rho_0'w$ and $U'w$ arise from vertical advection of the base density and the base wind velocity, respectively.
The gravitational force contributes the term $-gr$ to the vertical momentum equation.
The base entropy is constant, so its gradient adds no term to the entropy equation.
Together, these observations give Equation~\eqref{eq:linear-air}.
In the water phase, the base velocity is zero and the density is constant, so the same differentiation gives Equation~\eqref{eq:linear-water}.

We next differentiate the kinematic conditions, taking the interface displacement into account.
The air-side kinematic condition in Equation~\eqref{eq:kinematic-nonlinear} is
\begin{equation}
\partial_t\eta(\vartheta)
+u_{\mathrm a}(\vartheta)\bigl(x,\eta(\vartheta),t\bigr)
 \,\partial_x\eta(\vartheta)
=w_{\mathrm a}(\vartheta)\bigl(x,\eta(\vartheta),t\bigr).
\end{equation}
At the laminar state, the interface is flat and the vertical air velocity is zero at every height.
Differentiation at $\vartheta=0$ therefore gives
\begin{equation}
\partial_t\zeta+U(0)\partial_x\zeta=w|_{z=0}.
\end{equation}
The corresponding calculation in the water phase, whose base velocity is zero, gives
\begin{equation}
\partial_t\zeta
=\left.\mathbf v\cdot\mathbf e_z\right|_{z=0}.
\end{equation}
Hypothesis \ref{hyp:noslip} sets $U(0)=0$, so the air-side condition reduces to $\partial_t\zeta=w|_{z=0}$.
Together with the water-side condition, this shows that the vertical air velocity perturbation $w|_{z=0}$ and the vertical water velocity perturbation $\left.\mathbf v\cdot\mathbf e_z\right|_{z=0}$ both equal the time derivative of the interface displacement, which gives Equation~\eqref{eq:lin-kinematic}.
If $U(0)\neq0$, these two vertical velocity perturbations instead differ by $U(0)\partial_x\zeta$, as in the vortex-sheet setting discussed in \cite[\S2]{BSWZ2016}.

We then differentiate the dynamic condition in Equation~\eqref{eq:dynamic-nonlinear}, which imposes the pressure balance at the moving air--water interface.
Because the base pressures vary with height, their values at the interface change when the interface moves.
The pressure variation on each side is therefore the sum of the Eulerian pressure perturbation and the base pressure gradient multiplied by $\zeta$.
Differentiating Equation~\eqref{eq:dynamic-nonlinear} at $\vartheta=0$ yields
\begin{align}
\begin{aligned}
q_{\mathrm w}|_{z=0}
+p_{\mathrm w,0}'(0)\zeta
-q|_{z=0}
-p_0'(0)\zeta
&=\sigma\left.\partial_\vartheta\kappa(\vartheta)\right|_{\vartheta=0}\\
&=-\sigma\,\partial_x^2\zeta,
\end{aligned}
\end{align}
where $\kappa(\vartheta)$ is the curvature of the interface $\eta(\vartheta)$.
The last equality follows from the curvature convention in Equation~\eqref{eq:dynamic-nonlinear}.
Substituting the hydrostatic gradients
\begin{equation}
p_{\mathrm w,0}'(0)=-\rhow g,
\qquad
p_0'(0)=-\rho_0(0)g
\end{equation}
gives Equation~\eqref{eq:lin-dynamic}.
We call the combinations $q-\rho_0(0)g\zeta$ and $q_{\mathrm w}-\rhow g\zeta$ the \emph{Lagrangian pressure variations} at the interface.
Each one includes both the pressure perturbation at a fixed height and the change in base pressure caused by the interface displacement.

Finally, the top and bottom boundaries are fixed, so differentiating Equation~\eqref{eq:rigid} gives the stated impermeability conditions directly.
\end{proof}

\begin{remark}\label{rem:spaces}
By instability at wavenumber $k$, we mean that the closed, densely defined operator $\Lop_k$ of \S\ref{subsec:operator} has an eigenvalue $\lambda$ with $\Rec\lambda>0$.
This notion of instability concerns only the eigenvalues of $\Lop_k$, as in \cref{thm:operator}.
By \cref{thm:pointspec}, these eigenvalues correspond exactly to growing normal modes of the form given by Equation~\eqref{eq:normalmode}, with $\lambda=-\im kc$ and $\Rec\lambda=k\Imc c>0$.
The spatial regularity of each mode depends on the regularity of the coefficients.
Its temporal factor $e^{-\im kct}$ gives exponential growth at rate $k\Imc c$ in every $L^2$-based norm in which the mode amplitude is finite.
Here, for an interval $I$, $L^2(I)$ denotes the space of square-integrable functions $\psi$ on $I$, with the norm $\|\psi\|_{L^2(I)}:=\bigl(\int_I|\psi|^2\dd z\bigr)^{1/2}$.
For an integer $m\ge1$, the Sobolev space $H^m(I)$ consists of the functions in $L^2(I)$ whose weak derivatives up to order $m$ lie in $L^2(I)$, with the norm $\|\psi\|_{H^m(I)}:=\bigl(\sum_{i=0}^{m}\|\psi^{(i)}\|_{L^2(I)}^2\bigr)^{1/2}$.
The weak derivative $\psi'$ is the function, when it exists, that satisfies $\int_I\psi'\varphi\,\dd z=-\int_I\psi\varphi'\,\dd z$ for every smooth function $\varphi$ that vanishes near the endpoints of $I$, and higher weak derivatives are defined by repeating this step.
A mode amplitude or a state lies in $L^2$ or $H^m$ when each of its component functions does, where the air components are defined on $[0,\ha]$ and the water components on $[-\hw,0]$.
The $L^2$-based norms are the norms of these spaces.
The analysis does not establish semigroup generation, well-posedness of the initial-value problem, or nonlinear instability.
Semigroup generation would mean that the linear evolution in which the time derivative of the state is its image under $\Lop_k$ has exactly one solution for every initial state in the domain of $\Lop_k$, and that the norm of each solution at time $t$ is at most the norm of its initial state times an exponential function of $t$ that is the same for all initial states.
Well-posedness means that solutions exist, are unique, and depend continuously on the initial state.
For well-posedness results in related settings, see \cite{Denisova2000} for the viscous two-phase problem and \cite{Lindblad2005,Trakhinin2009,CoutandShkoller2012,LuoZhang2026} for inviscid one-phase compressible problems.
\end{remark}

\subsection{Normal modes}

The coefficients of Equations~\eqref{eq:linear-air}--\eqref{eq:lin-dynamic} depend only on $z$, so different horizontal Fourier modes decouple.
For $k>0$ and $c\in\C$, we therefore seek solutions of the form
\begin{equation}\label{eq:normalmode}
(u,w,r,q,s,\mathbf v,q_{\mathrm w})
=\bigl(\hat u,\hat w,\hat r,\hat q,\hat s,\hat{\mathbf v},\hat q_{\mathrm w}\bigr)(z)\,
e^{\im k(x-ct)},
\qquad
\zeta=\hat\zeta\,e^{\im k(x-ct)}.
\end{equation}
A solution of this form is a \emph{normal mode} with wavenumber $k$ and phase speed $c$.
The functions of $z$ with hats are its amplitudes, and the constant $\hat\zeta$ is the amplitude of the interface displacement.
Since $|e^{-\im kct}|=e^{k\Imc c\,t}$, a mode with $\Imc c>0$ grows exponentially in time at rate $k\Imc c$.

We substitute Equation~\eqref{eq:normalmode} into the linearized system and set
\begin{equation}\label{eq:V-def}
V(z):=U(z)-c.
\end{equation}
Dropping the hats on the air amplitudes, we find that the air system in Equation~\eqref{eq:linear-air} becomes
\begin{subequations}\label{eq:mode-air}
\begin{align}
&\im kV r+\rho_0'w+\rho_0(\im ku+w')=0,\label{eq:m-cont}\\
&\rho_0(\im kVu+U'w)+\im kq=0,\label{eq:m-xmom}\\
&\rho_0\,\im kV w+q'=-gr,\label{eq:m-zmom}\\
&\im kV s=0,\qquad q=a^2r+P_Ss.\label{eq:m-ent}
\end{align}
\end{subequations}
In the water, we keep the hats and write $\hat{\mathbf v}=(\hat v_x,\hat v_z)$.
The water system in Equation~\eqref{eq:linear-water} becomes
\begin{equation}\label{eq:water-mode-eqs}
-\im kc\,\rhow\hat v_x+\im k\,\hat q_{\mathrm w}=0,
\qquad
-\im kc\,\rhow\hat v_z+\hat q_{\mathrm w}'=0,
\qquad
\im k\hat v_x+\hat v_z'=0.
\end{equation}
For $c\neq0$, eliminating $\hat v_x$ and $\hat q_{\mathrm w}$ gives $\hat v_z''-k^2\hat v_z=0$.
Together with the bottom condition $\hat v_z(-\hw)=0$, this yields
\begin{subequations}\label{eq:water-mode}
\begin{align}
&\hat v_z(z)=\hat v_z(0)\,\frac{\sinh k(z+\hw)}{\sinh k\hw},
\qquad
\hat q_{\mathrm w}(z)=\frac{\im\rhow c}{k}\,\hat v_z'(z),
\label{eq:water-profile}\\
&\hat q_{\mathrm w}(0)=\rhow k c^2\coth(k\hw)\,\hat\zeta,
\label{eq:water-trace}
\end{align}
\end{subequations}
where Equation~\eqref{eq:water-trace} uses the kinematic condition $\hat v_z(0)=-\im kc\,\hat\zeta$ from Equation~\eqref{eq:lin-kinematic}.
We did not assume that the water motion is irrotational.
Taking the curl of the momentum equation of the water system in Equation~\eqref{eq:linear-water} shows that the linearized water vorticity does not depend on time.
For a mode with $c\neq0$, the time factor $e^{-\im kct}$ is not constant, so the vorticity must vanish.

\begin{proposition}\label{prop:reduction}
Assume Hypotheses \ref{hyp:reg} and \ref{hyp:noslip}, and let $k>0$ and $\Imc c\neq0$.
Every nontrivial normal mode of the form given by Equation~\eqref{eq:normalmode} has $\hat\zeta\neq0$.
Moreover, a nontrivial normal mode with wavenumber $k$ and phase speed $c$ exists if and only if the air system in Equation~\eqref{eq:mode-air} has a solution on $[0,\ha]$ that satisfies
\begin{equation}\label{eq:air-bcs}
w(\ha)=0,
\qquad
w(0)=\im kV(0),
\end{equation}
and whose pressure satisfies the \emph{dispersion relation}
\begin{equation}\label{eq:dispersion0}
\rhow\Bigl[kc^2\coth(k\hw)-g-\frac{\sigma k^2}{\rhow}\Bigr]
=q(0)-\rho_0(0)g.
\end{equation}
The second condition in Equation~\eqref{eq:air-bcs} is the normalization $\hat\zeta=1$.
In terms of the vertical displacement amplitude $\xi:=w/(\im kV)$ of the air, it reads $\xi(0)=1$.
\end{proposition}

\begin{proof}
Since $U$ is real and $\Imc c\neq0$, we have $c\neq0$ and $V(z)\neq0$ for every $z\in[0,\ha]$.
By Hypothesis \ref{hyp:noslip}, $U(0)=0$, so $V(0)=-c$ and the kinematic conditions in Equation~\eqref{eq:lin-kinematic} become
\begin{equation}\label{eq:mode-kinematic}
w(0)=\hat v_z(0)=-\im kc\,\hat\zeta=\im kV(0)\,\hat\zeta.
\end{equation}

We first show that a mode with $\hat\zeta=0$ is trivial.
In this case, Equation~\eqref{eq:mode-kinematic} gives $w(0)=0$ and $\hat v_z(0)=0$.
Since $\hat v_z''-k^2\hat v_z=0$ and $\hat v_z(-\hw)=\hat v_z(0)=0$, we have $\hat v_z\equiv0$.
Incompressibility in Equation~\eqref{eq:water-mode-eqs} with $k>0$ then gives $\hat v_x\equiv0$, and the horizontal momentum equation gives $\im k\,\hat q_{\mathrm w}\equiv0$, so $\hat q_{\mathrm w}\equiv0$.
No free constant remains in the water pressure, because an amplitude that is constant in $z$ still multiplies $e^{\im k(x-ct)}$ and therefore has the nonzero horizontal gradient $\im k\,\hat q_{\mathrm w}e^{\im k(x-ct)}$ unless $\hat q_{\mathrm w}=0$.
The dynamic condition in Equation~\eqref{eq:lin-dynamic} with $\hat\zeta=0$ then gives $q(0)=\hat q_{\mathrm w}(0)=0$.

In the air, since $V\neq0$, Equation~\eqref{eq:m-ent} gives $s\equiv0$ and hence $r=q/a^2$.
Equation~\eqref{eq:m-xmom} then determines $u$ in terms of $w$ and $q$.
Substituting $r$ and $u$ into Equations~\eqref{eq:m-cont} and \eqref{eq:m-zmom} gives a first-order linear system for $(w,q)$.
Its coefficients are continuous on $[0,\ha]$ by Hypothesis \ref{hyp:reg}, because the only quantities we divide by are $V$, $a^2$, and $\rho_0$, and none of them vanishes on $[0,\ha]$.
This elimination is carried out in \cref{prop:xiq}, and it does not use any subsonicity assumption.
Since $w(0)=q(0)=0$, uniqueness for the initial-value problem gives $(w,q)\equiv0$.
Then $u$, $r$, and $s$ also vanish, and since the water fields vanish as shown above, the whole mode is trivial.

Now let $\hat\zeta\neq0$.
Since the system is linear, we may rescale the mode so that $\hat\zeta=1$.
Then Equation~\eqref{eq:mode-kinematic} gives $w(0)=\im kV(0)$, and impermeability at the top gives $w(\ha)=0$, so the air amplitudes satisfy Equations~\eqref{eq:mode-air} and \eqref{eq:air-bcs}.
For a mode, $-\sigma\partial_x^2\zeta=\sigma k^2\zeta$, so the dynamic condition in Equation~\eqref{eq:lin-dynamic} with $\hat\zeta=1$ reads
\begin{equation}
\bigl(\hat q_{\mathrm w}(0)-\rhow g\bigr)-\bigl(q(0)-\rho_0(0)g\bigr)=\sigma k^2.
\end{equation}
Inserting Equation~\eqref{eq:water-trace} and rearranging gives Equation~\eqref{eq:dispersion0}.

Conversely, suppose that the air amplitudes solve Equation~\eqref{eq:mode-air} and satisfy Equations~\eqref{eq:air-bcs} and \eqref{eq:dispersion0}.
Set $\hat\zeta=1$, define $\hat v_z$ and $\hat q_{\mathrm w}$ by Equation~\eqref{eq:water-profile} with $\hat v_z(0)=-\im kc$, and define $\hat v_x$ by incompressibility.
These water fields satisfy Equation~\eqref{eq:water-mode-eqs} and the bottom condition.
The kinematic conditions hold by Equation~\eqref{eq:air-bcs} and $U(0)=0$, and the dynamic condition holds by Equation~\eqref{eq:dispersion0}.
The resulting normal mode is nontrivial because $\hat\zeta=1$.
\end{proof}

We now express the dispersion relation through a single quantity computed from the air.

\begin{definition}\label{def:air-response}
Let $\Imc c>0$.
By \cref{lem:regular-solve} below, the system in Equation~\eqref{eq:mode-air} has a unique solution with $w(\ha)=0$ and $w'(\ha)=1$.
We denote it by $w(\cdot;k,c)$ and call it the \emph{top-normalized solution}.
Its value $w(0;k,c)$ at the interface is the \emph{interface trace}.
Here and below, a dot marks the free argument and the arguments after a semicolon are parameters, so $w(\cdot;k,c)$ is a function of the height $z$ and $w(0;k,\cdot)$ is a function of $c$.
For the limiting solution $w_\ast(\cdot;c_r)$ of \cref{prop:limit-solutions}, the wavenumber $k$ is omitted from the arguments.
As in Equation~\eqref{eq:air-bcs}, $\xi=w/(\im kV)$ is the vertical displacement amplitude of the air, and $\xi(0)$ is the interface displacement of the solution.
Where the interface trace does not vanish, the \emph{air-side pressure response} is
\begin{equation}\label{eq:Za-def}
\Za(k,c):=\frac{\qbar(0)-\rhobar(0)\,g\,\xi(0)}{\xi(0)},
\qquad
\qbar:=\frac{q}{\eps\rhow},
\end{equation}
where $q$ is the pressure of the top-normalized solution.
\end{definition}

Thus $\Za$ is the Lagrangian pressure perturbation at the interface, normalized by the surface air density $\eps\rhow$ and divided by the interface displacement.
By \cref{lem:regular-solve}, the top-normalized solution spans the solutions of Equation~\eqref{eq:mode-air} with $w(\ha)=0$, and it depends holomorphically on $c$.
The quotient in Equation~\eqref{eq:Za-def} does not change when the solution is multiplied by a nonzero constant, so $\Za$ does not depend on the choice of normalization.
Since $V(0)=-c\neq0$, the condition $\xi(0)\neq0$ is equivalent to $w(0;k,c)\neq0$.
Where $w(0;k,c)\neq0$, we may therefore rescale the solution so that $\xi(0)=1$ as in Equation~\eqref{eq:air-bcs}, and then $\Za(k,c)=\qbar(0)-\rhobar(0)g$.
The function $\Za$ depends only on the normalized profiles $(\rhobar,a,U)$ and not on $\eps$ (\S\ref{subsec:eps-family}).
By \cref{lem:regular-solve}, $\Za(k,\cdot)$ is meromorphic on the open upper half-plane, with poles exactly at the zeros of $w(0;k,\cdot)$.
Here meromorphic means holomorphic except at isolated poles.

With $\Dw$ defined in Equation~\eqref{eq:Dw-def}, the dispersion relation in Equation~\eqref{eq:dispersion0} becomes
\begin{equation}\label{eq:dispersion}
D(k,c,\eps):=\Dw(k,c)-\eps\,\Za(k,c)=0
\end{equation}
on the set where $w(0;k,c)\neq0$.
No normal mode is lost by this restriction.
Indeed, by \cref{prop:reduction}, every mode has $\hat\zeta\neq0$, so its air amplitude satisfies $w(0)=\im kV(0)\hat\zeta\neq0$.
This amplitude is a nonzero multiple of the top-normalized solution, so $w(0;k,c)\neq0$ as well.
Multiplying the relation in Equation~\eqref{eq:dispersion} by the interface trace $w(0;k,c)$ gives the holomorphic dispersion function $\Edet$ of \cref{def:determinant}.

At $\eps=0$, the roots of Equation~\eqref{eq:dispersion} are the roots $\pm c_k$ of $\Dw(k,\cdot)$ given in Equation~\eqref{eq:ck}.
Since $\partial_c\Dw(k,c_k)=2kc_k\coth(k\hw)>0$, the free wave speed $c_k$ is a simple root.
To follow this root for small $\eps>0$, we need $\Za(k,c)$ as the imaginary part of $c$ tends to zero with $c$ near $c_k$.
Section~\ref{sec:criticallayer} constructs the boundary value of $\Za(k,c)$, its limit as $c=c_r+\im\delta$ with $\delta\downarrow0$, for real $c_r$ near $c_k$.

\subsection{The linearized operator at one wavenumber}
\label{subsec:operator}

For each $k$, the operator $\Lop_k$ maps the spatial Fourier amplitudes at wavenumber $k$ to their time derivatives, after the water pressure has been recovered by solving Laplace's equation in the water, with its interface value given by the linearized dynamic condition and a zero vertical derivative at the bottom.
Its definition uses three notions.
The linearized constraint space is the set of states that satisfy the linearized constraints, namely incompressibility in the water, the two kinematic conditions, the dynamic condition, and the conditions at the top and bottom boundaries.
The $H^3$ core $\mathfrak D_k$ consists of the states in this space with $H^3$ regularity whose water pressure solves Laplace's equation in the water with a zero vertical derivative at the bottom, and the linearized equations can be applied to these states directly.
The graph of an operator is the set of pairs formed by a state and its image.
An operator is closed when its graph contains every limit, in the norm of the $L^2$-based state space, of a sequence of such pairs.
If adding to the graph all limits of such pairs in this norm gives again the graph of an operator, that operator is the graph closure.
It is the smallest closed operator that agrees with the original operator on the domain of the original operator, and it is the standard way to make a differential operator closed.
The operator $\Lop_k$ is the graph closure of the linearized evolution on the $H^3$ core, so $\mathfrak D_k$ is dense in the domain of $\Lop_k$ in the graph norm, the sum of the $L^2$ norms of a state and of its image under $\Lop_k$.

Under the additional thermodynamic regularity in Equation~\eqref{eq:thermo-domain}, the flattened equations of \cref{subsec:flattening}, which are the equations transformed to fixed reference domains on which the interface is flat, are Fr\'echet differentiable.
After the change of variables to the good unknown in Equation~\eqref{eq:goodunknown}, their derivative becomes the Eulerian linearized system.
Restricting this evolution to perturbations proportional to $e^{\im kx}$ gives the spatial Fourier component at wavenumber $k$.
A normal mode adds the temporal factor $e^{-\im kct}$.

The theorem uses the following spectral notions for $\Lop_k$, which acts on an $L^2$ space of states at wavenumber $k$ (\cref{def:modal-operator}).
A complex number $\lambda$ belongs to the resolvent set of $\Lop_k$ when, for every state $\mathbb F$, the equation $(\lambda-\Lop_k)\hat{\mathbb U}=\mathbb F$ has exactly one solution $\hat{\mathbb U}$ in the domain of $\Lop_k$.
The map $\mathbb F\mapsto\hat{\mathbb U}$ is then the resolvent $(\lambda-\Lop_k)^{-1}$, and the spectrum of $\Lop_k$ is the set of complex numbers that do not belong to the resolvent set.
Every eigenvalue belongs to the spectrum.
An eigenvalue $\lambda$ is an isolated point of the spectrum when some neighborhood of $\lambda$ contains no other point of the spectrum.
For such an eigenvalue, the Riesz projection, also called the spectral projection, is $1/(2\pi\im)$ times the integral of the resolvent over a small circle around $\lambda$, traversed counterclockwise (\cref{cor:algsimple}).
The eigenvalue is algebraically simple when this projection has rank one, that is, a one-dimensional range.
Then every solution of $(\Lop_k-\lambda)^m\hat{\mathbb U}=0$ with an integer $m\ge1$ is a multiple of one eigenvector.
Here $(\Lop_k-\lambda)^m\hat{\mathbb U}$ is defined when $\hat{\mathbb U}$ and its images under $(\Lop_k-\lambda)^j$ for $j=1,\dots,m-1$ lie in the domain of $\Lop_k$.
The following theorem states the construction and its spectral consequences.
Its proof is given in Appendix~\ref{app:operator}.

\begin{theorem}\label{thm:operator}
Assume Hypotheses \ref{hyp:reg} and \ref{hyp:noslip}.
For parts \ref{it:op-frechet} and \ref{it:op-realized}, assume in addition the thermodynamic regularity in Equation~\eqref{eq:thermo-domain}, which requires an open set $\mathscr U$ whose closure $\overline{\mathscr U}$ is a compact subset of $(0,\infty)\times\R$ and which contains the base range $\{(\rho_0(z),S_\ast):z\in[0,\ha]\}$, with $P\in C^5(\overline{\mathscr U})$ and $\inf_{\overline{\mathscr U}}P_\rho>0$.
Here $P\in C^5(\overline{\mathscr U})$ means that the partial derivatives of $P$ of order at most five are continuous on $\mathscr U$ and extend continuously to $\overline{\mathscr U}$.
Then the following statements hold.
\begin{enumerate}[label=\textup{(\roman*)}]
\item\label{it:op-frechet}
In the flattened formulation of \cref{subsec:flattening}, the two-phase residual, whose components are the right-hand sides of the flattened evolution equations and the expressions set to zero in the flattened constraint equations, is continuously Fr\'echet differentiable, that is, with a derivative that depends continuously on the state, on an open neighborhood of the laminar state.
After the change of variables to the good unknown in Equation~\eqref{eq:goodunknown}, its derivative at the laminar state is the linearized system of \cref{prop:linearized} (\cref{prop:residual,prop:goodunknown}).
\item\label{it:op-realized}
For each $k>0$, take the part of the linearized evolution in \ref{it:op-frechet}, written in terms of the good unknown, that is proportional to $e^{\im kx}$.
Use pressure rather than density as the thermodynamic state variable, recover the water pressure by solving Laplace's equation in the water, and restrict to the linearized constraint space.
The resulting evolution defines the closed, densely defined operator $\Lop_k$ of \cref{def:modal-operator}, which acts on $L^2$.
Precisely, $\Lop_k$ is the graph closure of this evolution on the $H^3$ core $\mathfrak D_k$ in the constraint space (\cref{prop:realization}).
\item\label{it:op-spectrum}
The operator $\Lop_k$ itself is defined under Hypotheses \ref{hyp:reg} and \ref{hyp:noslip} alone, without the thermodynamic regularity in Equation~\eqref{eq:thermo-domain}.
For $\Imc c\neq0$, $\lambda=-\im kc$ is an eigenvalue of $\Lop_k$ if and only if the holomorphic dispersion function $\Edet(k,c,\eps)$ of \cref{def:determinant} vanishes, and the eigenspace is then one-dimensional (\cref{thm:pointspec}).
At a simple zero of $\Edet(k,\cdot,\eps)$, the eigenvalue is an isolated point of the spectrum of $\Lop_k$ and is algebraically simple.
In this case, the resolvent satisfies an explicit bound and the Riesz projection has rank one (\cref{prop:resolvent,cor:algsimple}).
\end{enumerate}
\end{theorem}
\section{Equivalent forms of the air equations and the holomorphic dispersion function}
\label{sec:structure}

Unless stated otherwise, we assume $\Imc c\neq0$ in this section, so $V=U-c$ is nonzero on $[0,\ha]$.
We derive three equivalent forms of the air equations.
The displacement--pressure system exposes the role of neutrality, the vertical-velocity equation identifies the critical-layer coefficient, and the mass-weighted displacement equation yields the flux and energy identities.
We then construct a holomorphic dispersion function.

\subsection{Isentropic reduction and the displacement--pressure system}

The entropy equation first reduces the air system to displacement and pressure.
This reduction is isentropic, that is, the entropy perturbation of a normal mode vanishes (\cref{lem:entropy}).
We write $S_0$ for the entropy of the base state, so $S_0\equiv S_\ast$ and $S_0'\equiv0$ for the laminar states of \S\ref{subsec:laminar}.
We use the vertical Lagrangian displacement amplitude $\xi=w/(\im kV)$ of \cref{prop:reduction}, so that $w=\im kV\xi$.
This relation expresses the vertical velocity as the material derivative of the displacement along the base wind.
This material derivative is the time derivative following the base wind, $\partial_t+U\partial_x$, and it acts on a normal mode as multiplication by $\im k(U-c)=\im kV$.
The next lemma shows that the Lagrangian pressure and density variations satisfy a barotropic relation, that is, a relation that does not involve the entropy.

\begin{lemma}\label{lem:entropy}
Let $\Imc c\neq0$ and let $(u,w,r,q,s)$ solve Equation~\eqref{eq:mode-air} on $[0,\ha]$ with $S_0'\equiv0$.
Then $s\equiv0$, and the Lagrangian pressure and density variations satisfy the barotropic relation
\begin{equation}\label{eq:lagrangian}
\Delta p:=q-\rho_0g\,\xi
=a^2\Delta\rho,
\qquad
\Delta\rho:=r+\rho_0'\xi .
\end{equation}
\end{lemma}

\begin{proof}
Since $V\neq0$, the entropy equation $\im kVs=0$ gives $s\equiv0$.
To obtain the Lagrangian relation, use the hydrostatic relation $p_0'=-\rho_0g$ and the identity $p_0'=a^2\rho_0'+P_SS_0'$, and compute $\Delta p-a^2\Delta\rho =(q-a^2r)+(p_0'-a^2\rho_0')\xi =P_S(s+S_0'\xi)=0$, by the equation of state and $s=S_0'=0$.

The same Lagrangian relation also holds when $S_0'\neq0$.
Indeed, the linearized entropy equation is then $\im kVs+S_0'w=0$, which gives $s=-S_0'\xi$, so the Lagrangian entropy perturbation still vanishes.
In that case the reduction includes the Taylor--Goldstein-type $N^2$ terms discussed in \cref{rem:general-strat}.
\end{proof}

\begin{proposition}\label{prop:xiq}
Under the assumptions of \cref{lem:entropy}, with $b=g/a^2$ and the neutrality relation in Equation~\eqref{eq:neutral}, the pair $(\xi,q)$ satisfies
\begin{equation}\label{eq:xiq}
\xi'=b\,\xi+\frac{a^2-V^2}{\rho_0\,a^2V^2}\,q,
\qquad
q'=-b\,q+\rho_0k^2V^2\,\xi ,
\end{equation}
and conversely $(u,w,r,s)$ are recovered from $(\xi,q)$ by $w=\im kV\xi$, $s=0$, $u=-\bigl(U'w+\im kq/\rho_0\bigr)/(\im kV)$, $r=q/a^2-\rho_0(\rho_0'/\rho_0+b)\xi=q/a^2$.
For a general stratified base state with arbitrary $S_0'$, the same elimination uses $s=-S_0'\xi$ and Equation~\eqref{eq:lagrangian}, as in the proof of \cref{lem:entropy}, and yields
\begin{equation}\label{eq:xiq-general}
\xi'=b\,\xi+\frac{a^2-V^2}{\rho_0a^2V^2}\,q,
\qquad
q'=-b\,q+\rho_0\bigl(k^2V^2-N^2\bigr)\,\xi .
\end{equation}
\end{proposition}

\begin{proof}
We derive the stratified system first and then impose neutrality.
Solving Equation~\eqref{eq:m-xmom} for $u$ gives $\im ku=-(U'w+\im kq/\rho_0)/V$.
Substitute this expression and $w=\im kV\xi$, $w'=\im k(U'\xi+V\xi')$ into Equation~\eqref{eq:m-cont}.
The two shear terms $\rho_0U'\xi$ cancel, and division by $\im kV$ gives
\begin{equation}
r+\rho_0'\xi+\rho_0\xi'-\frac{q}{V^{2}}=0 .
\end{equation}
Now eliminate $r$ using Equation~\eqref{eq:lagrangian} in the form $r=q/a^2-\rho_0(\rho_0'/\rho_0+b)\xi$.
The $\rho_0'\xi$ terms cancel, leaving the first equation of the system in Equation~\eqref{eq:xiq-general}.

For the second equation, substitute $w=\im kV\xi$ into Equation~\eqref{eq:m-zmom} to obtain $q'=-gr+\rho_0k^2V^2\xi$.
The same expression for $r$ gives $-gr=-bq-\rho_0N^2\xi$, where $N^2=-g(\rho_0'/\rho_0+b)$ as in Equation~\eqref{eq:neutral}.
This proves Equation~\eqref{eq:xiq-general}.
Under neutrality, $N^2=0$ and $\rho_0'/\rho_0+b=0$, so the system reduces to Equation~\eqref{eq:xiq}.
\end{proof}

In this system, the critical singularity $V=0$ and the sonic degeneracy $V^2=a^2$ enter separately, through the denominator and the numerator of the coefficient $(a^2-V^2)/(\rho_0a^2V^2)$.
A nonzero buoyancy term $N^2$ would further change the local critical-level equation from Rayleigh to Taylor--Goldstein type.
Uniform subsonicity and neutrality exclude the sonic and buoyancy effects, leaving the critical singularity as the object of the analysis.

\subsection{Second-order equations and the critical-layer coefficient}

Two second-order equations in divergence form, each for one unknown, will be useful.
The equation for vertical velocity isolates the critical singularity in its zeroth-order coefficient, the factor that multiplies $w$, whereas the equation for mass-weighted displacement places the degeneracy in its leading coefficient, the factor that multiplies the highest derivative.
For the first form, define
\begin{equation}\label{eq:A-def}
A(z;c):=\frac{\rho_0\,a^2}{a^2-V^2},
\qquad
\chi:=A\bigl[Vw'-(U'+bV)w\bigr] .
\end{equation}
Hypothesis \ref{hyp:subsonic} makes $A$ well defined, with $\Rec A>0$, for $c$ in a complex neighborhood of $c_k$.
Indeed, $\Rec(a^2-V^2)=a^2-(U-\Rec c)^2+(\Imc c)^2\ge\delta_0a^2$ and $\Rec A=\rho_0a^2\,\Rec(a^2-V^2)/|a^2-V^2|^2$.
For real $c$ in this neighborhood, $A\ge\rho_0>0$.

\begin{proposition}\label{prop:SLform}
Under the assumptions of \cref{prop:xiq} the following hold.
\begin{enumerate}[label=\textup{(\roman*)}]
\item\label{it:chi} $\chi=\im k\,q$, and $(w,\chi)$ satisfies the first-order system
  \begin{equation}\label{eq:wchi}
  w'=\frac{U'+bV}{V}\,w+\frac{\chi}{AV},
  \qquad
  \chi'=-b\,\chi+\rho_0k^2V\,w .
  \end{equation}
\item\label{it:SL} $w$ satisfies the \emph{compressible Rayleigh equation} in Sturm--Liouville form,
  \begin{equation}\label{eq:SL}
  \bigl(A\,w'\bigr)'
  =\Bigl(\rho_0k^2+\frac{G}{V}\Bigr)w,
  \qquad
  G:=\bigl(A(U'+bV)\bigr)'+b\,A(U'+bV) .
  \end{equation}
  For real $c$, the coefficient $\rho_0k^2+G/V$ has at most a simple pole at each critical level, with residue proportional to $Q(z_j)$ (\cref{prop:Q}).
  Under Hypothesis \ref{hyp:subsonic}, it has no other singularity on $[0,\ha]$.
\item\label{it:X} The mass-weighted displacement $X:=\rho_0\xi$ satisfies the divergence-form equation
  \begin{equation}\label{eq:Xform}
  \bigl(FX'\bigr)'=k^2\,\frac{V^2}{\rho_0}\,X,
  \qquad
  F:=\frac{a^2V^2}{\rho_0\,(a^2-V^2)},
  \qquad
  q=\rho_0F\,X' .
  \end{equation}
\end{enumerate}
As $a\to\infty$ with $\rho_0\to\mathrm{const}$, Equation~\eqref{eq:SL} reduces to the classical Rayleigh equation $V(w''-k^2w)=U''w$ and Equation~\eqref{eq:Xform} to $(V^2\xi')'=k^2V^2\xi$.
\end{proposition}

\begin{proof}
The first equation of the system in Equation~\eqref{eq:xiq} gives $q=\rho_0a^2V^2(\xi'-b\xi)/(a^2-V^2)$.
Substituting $\xi=w/(\im kV)$ and expanding its derivative yields $q=\chi/(\im k)$.
Rewriting Equation~\eqref{eq:xiq} in these variables gives Equation~\eqref{eq:wchi} and proves \ref{it:chi}.

To obtain \ref{it:SL}, differentiate $\chi=A[Vw'-(U'+bV)w]$ and use Equation~\eqref{eq:wchi}.
The expansion $(AVw')'=V(Aw')'+U'Aw'$ shows that the terms containing $w'$ cancel, leaving $V(Aw')'=\bigl(G+\rho_0k^2V\bigr)w$.
Division by $V$ gives Equation~\eqref{eq:SL}.

For \ref{it:X}, the choice $X=\rho_0\xi$ absorbs the density gradient.
By neutrality, $\rho_0'=-b\rho_0$, and therefore the first equation of the system in Equation~\eqref{eq:xiq} gives $X'=\rho_0(\xi'-b\xi)=\rho_0\,\dfrac{a^2-V^2}{\rho_0a^2V^2}\,q =\dfrac{q}{\rho_0F}$.
Thus $q=\rho_0FX'$.
Substitute this into $q'+bq=\rho_0k^2V^2\xi$ and use the second neutrality cancellation, $(\rho_0F X')'+b\rho_0FX'=\rho_0(FX')'$.
The result is $(FX')'=k^2(V^2/\rho_0)X$, as required.
The stated incompressible limits follow directly as $a\to\infty$.
\end{proof}

We use Equation~\eqref{eq:SL} for $w$ and its first-order companion, Equation~\eqref{eq:wchi}, for the critical-layer construction in \S\ref{sec:criticallayer}.
Equation~\eqref{eq:Xform} for $X$ instead yields the integral formula in Equation~\eqref{eq:herglotz} for $\Imc\Za$, the coercivity estimate of \cref{lem:coercive}, and the uniform-wind example of \S\ref{subsec:KH}.
Although both equations are formally self-adjoint for real $c$ away from critical and sonic levels, they depend nonlinearly on the phase speed $c$ through $V$, $A$, and $F$.

At a regular critical level, the leading coefficient in Equation~\eqref{eq:Xform} vanishes quadratically, $F\sim U'(z_c)^2(z-z_c)^2/\rho_0(z_c)$, where $\sim$ means that the ratio of the two sides tends to $1$ as $z\to z_c$.
The Frobenius exponents, that is, the powers $s$ for which the equation has a solution that behaves like $(z-z_c)^s$ near $z_c$, are $0,-1$ for $X$ and $1,0$ for $w$.
Since these exponents differ by an integer, the second solution may contain a logarithmic term.
The next proposition identifies its coefficient.
The construction in \S\ref{sec:criticallayer} also covers the case in which this coefficient vanishes.
We write $\operatorname{Log}$ for the principal logarithm,
\begin{equation}\label{eq:Log-def}
\operatorname{Log}\zeta=\log|\zeta|+\im\operatorname{Arg}\zeta,
\qquad
-\pi<\operatorname{Arg}\zeta\le\pi,
\end{equation}
which is holomorphic on $\C\setminus(-\infty,0]$.

\begin{proposition}\label{prop:Q}
Let $z_c$ be a regular critical level, $U(z_c)=c\in\R$, $U'(z_c)\neq0$.
Then
\begin{equation}\label{eq:Gc}
A(z_c)=\rho_0(z_c),
\qquad
G(z_c)=\rho_0(z_c)\,Q(z_c),
\qquad
Q:=U''+\frac{g}{a^2}U'
\;=\;\rho_0\Bigl(\frac{U'}{\rho_0}\Bigr)' ,
\end{equation}
the last identity using the neutrality relation in Equation~\eqref{eq:neutral}.
Moreover the logarithmic basis solution of Equation~\eqref{eq:SL} at $z_c$ (\cref{prop:local}\ref{it:basis-log}) has the local form
\begin{equation}\label{eq:frobenius}
w(z)=w_c\Bigl[1+\frac{Q(z_c)}{U'(z_c)}\,(z-z_c)\,
\Lambda(z)\Bigr]
+O(|z-z_c|),
\qquad
\Lambda(z):=\operatorname{Log}\frac{U(z)-c-\im0}{V_{\mathrm{ref}}},
\end{equation}
where $w_c:=w(z_c)$, $\operatorname{Log}(\cdot-\im0)$ denotes the boundary value from $\Imc c>0$ of the principal logarithm in Equation~\eqref{eq:Log-def}, given explicitly in Equation~\eqref{eq:branch}, and $V_{\mathrm{ref}}>0$ is a reference speed.
Changing $V_{\mathrm{ref}}$ changes only the $O(|z-z_c|)$ term.
\end{proposition}

\begin{proof}
Since $V(z_c)=0$, the definition of $A$ gives $A(z_c)=\rho_0(z_c)$.
Differentiate $\log A=\log\rho_0+\log a^2-\log(a^2-V^2)$.
At $z_c$, the terms involving $a^2$ cancel because $(V^2)'(z_c)=0$, and neutrality gives $A'(z_c)/A(z_c)=\rho_0'/\rho_0=-b$.
Consequently,
\begin{align*}
G(z_c)&=A'(U'+bV)\big|_{z_c}+A\,(U''+b'V+bU')\big|_{z_c}
        +bA(U'+bV)\big|_{z_c}\\
&=\rho_0(z_c)\bigl[-bU'+U''+bU'+bU'\bigr](z_c).
\end{align*}
This proves Equation~\eqref{eq:Gc}.
The logarithmic expansion in Equation~\eqref{eq:frobenius} is established in \cref{prop:local}\ref{it:basis-log}, where the coefficient is derived from the residue at $z_c$ of the singular integrand in the local representation of the first-order system in Equation~\eqref{eq:wchi}, computed with Equation~\eqref{eq:residue-r}.
That construction uses the reference speed $V_\ast$ of \S\ref{subsec:local} and gives a remainder $h\in C^1$ vanishing at $z_c$, hence the stated $O(|z-z_c|)$ remainder.
\end{proof}

\begin{remark}
\label{rem:leeslin}
The density correction in Equation~\eqref{eq:Gc} differs from that in the classical Lees--Lin coefficient.
To see the source of this difference, repeat the reduction of \cref{prop:xiq,prop:SLform} for a base state with $N^2=0$, without imposing hydrostatic neutrality.
The resulting self-adjoint equation for $X_{\mathrm{gen}}:=e^{-\int_0^zb}\,\xi$ has the leading coefficient $F_{\mathrm{gen}}:=e^{2\int_0^zb}\,\rho_0a^2V^2/(a^2-V^2)$ and the critical-layer coefficient
\begin{equation}
Q_{\mathrm{gen}}
:=U''+\Bigl(\frac{2g}{a^2}+\frac{\rho_0'}{\rho_0}\Bigr)U'
\end{equation}
(\cref{app:thermo}).
For a temperature-stratified gas with $g=0$, one has $N^2=0$ automatically, while $\rho_0'$ is arbitrary.
In this case, $Q_{\mathrm{gen}}=U''+(\rho_0'/\rho_0)U'=(\rho_0U')'/\rho_0$ is the \emph{generalized inflection point} quantity of compressible boundary-layer theory \cite{LeesLin1946,LeesReshotko1962}.
Its role in the jump of the Reynolds stress across the critical layer, that is, of the vertical flux of horizontal momentum transported by the wave, is classical \cite[\S9.3]{Mack1984}.

For a neutral hydrostatic atmosphere, instead, $\rho_0'/\rho_0=-g/a^2$ and $Q_{\mathrm{gen}}=U''+(g/a^2)U'=\rho_0(U'/\rho_0)'$, as in Equation~\eqref{eq:Gc}.
The opposite density corrections in $(\rho_0U')'$ and $\rho_0(U'/\rho_0)'$ arise because gravity contributes both to the buoyancy of the pressure perturbation and to the base-state density gradient.
In meteorological variables, the relevant quantity is the specific vorticity $U'/\rho_0$, whose critical-level gradient controls momentum deposition in compressible wave--mean-flow theory \cite{EliassenPalm1961,AndrewsMcIntyre1978,Buhler2014}.
With the opposite curl convention, the specific vorticity is $-U'/\rho_0$, but $Q$ and the instability criterion remain unchanged.
\end{remark}

\begin{remark}\label{rem:general-strat}
When $N^2\neq0$, Equation~\eqref{eq:xiq-general} gives $(F_{\mathrm{gen}}X_{\mathrm{gen}}')'=e^{2\int_0^zb}\rho_0\bigl(k^2V^2-N^2\bigr)X_{\mathrm{gen}}$, with $X_{\mathrm{gen}}$ and $F_{\mathrm{gen}}$ as in \cref{rem:leeslin}.
The coefficient on the right side no longer vanishes at $z_c$, and the Frobenius exponents become $-\tfrac12\pm\sqrt{\tfrac14-\mathrm{Ri}_c}$, where $\mathrm{Ri}_c=N^2(z_c)/U'(z_c)^2$.
This is the compressible Taylor--Goldstein structure.
Its critical-layer theory, including the absorption factors $e^{-2\pi\sqrt{\mathrm{Ri}_c-1/4}}$ of \cite{BookerBretherton1967} and the compressible Miles--Howard bound $\mathrm{Ri}\ge\frac14$ of \cite{Chimonas1970}, lies outside the present paper.
\end{remark}

\subsection{The wave-momentum flux and the imaginary part of \texorpdfstring{$\Za$}{Za}}
\label{subsec:flux-identities}

The sign of the growth rate is determined by the air-side pressure response.
We relate its imaginary part to a flux that is conserved between critical levels.
First normalize by the surface air density, setting $\rhobar=\rho_0/(\eps\rhow)$, $\Abar=A/(\eps\rhow)$, $\chibar=\chi/(\eps\rhow)$, and $\qbar=q/(\eps\rhow)$.
The quantities $A$, $\chi$, $q$, and $G$ scale linearly in the density, whereas $F$ scales inversely.
Thus Equations~\eqref{eq:wchi}, \eqref{eq:SL}, and \eqref{eq:Xform} have the same form with $\rho_0$ replaced by $\rhobar$.
We henceforth use normalized quantities, with the same symbols $G$, $F$, and $X$ for their normalized versions.
In particular, $X=\rhobar\,\xi$, $X(0)=\xi(0)$, and $G(z_c)=\rhobar(z_c)Q(z_c)$.
An overbar always denotes normalization, and complex conjugation is written with an asterisk, as in $w^{\ast}$.

Define the flux
\begin{equation}\label{eq:flux-def}
\Jf(z):=\Imc\bigl(\qbar(z)\,\xi^{\ast}(z)\bigr),
\end{equation}
which represents, up to the conventional multiplicative factor, the wave-induced form stress transmitted across height $z$.
It is associated with the correlation of pressure and the slope of displaced material surfaces, and hence with the Eliassen--Palm flux of the mode.

\begin{lemma}\label{lem:flux}
Let $(\xi,\qbar)$ solve Equation~\eqref{eq:xiq}.
Then
\begin{equation}\label{eq:flux-identity}
\frac{\dd}{\dd z}\,\Imc\bigl(\qbar\,\xi^{\ast}\bigr)
=k^2\rhobar\,\Imc(V^2)\,|\xi|^2
-\Imc\!\Bigl(\frac{a^2-V^2}{\rhobar a^2V^2}\Bigr)|\qbar|^2 ,
\end{equation}
the $b$-terms canceling identically.
Consequently the following hold.
\begin{enumerate}[label=\textup{(\roman*)}]
\item\label{it:flux-real} For $c\in\R$, $\Jf$ is constant on every interval free of critical levels.
\item\label{it:flux-herglotz} For $\Imc c=\delta>0$, assume $\xi(0)=1$ and $w(\ha)=0$.
Then $X=\rhobar\xi$ satisfies $X(0)=1$ because $\rhobar(0)=1$, and
  \begin{equation}\label{eq:herglotz}
  \Imc\qbar(0)
  =2\delta\int_0^{\ha}
  \bigl(U-\Rec c\bigr)
  \Bigl[\frac{a^4\,|X'|^2}{\rhobar\,|a^2-V^2|^2}
  +\frac{k^2}{\rhobar}\,|X|^2\Bigr]\dd z .
  \end{equation}
  In particular $\Imc\Za(k,c)=\Imc\qbar(0)$ has the sign of a weighted average of $U-\Rec c$.
\end{enumerate}
\end{lemma}

\begin{proof}
Differentiate $\qbar\,\xi^{\ast}$ using Equation~\eqref{eq:xiq}.
Since $b$ is real, the terms $-b\qbar\,\xi^{\ast}$ and $+\qbar\,b\,\xi^{\ast}$ cancel, leaving $\rhobar k^2V^2|\xi|^2 +\bigl((a^2-V^2)/(\rhobar a^2V^2)\bigr)^{\!\ast}\,|\qbar|^2$.
Taking imaginary parts gives Equation~\eqref{eq:flux-identity}.
For real $c$ away from critical levels, both imaginary parts vanish, which proves \ref{it:flux-real}.

For \ref{it:flux-herglotz}, multiply Equation~\eqref{eq:Xform} by $X^{\ast}$ and integrate by parts over $(0,\ha)$.
The top boundary term vanishes because $w(\ha)=0$ and $V(\ha)\neq0$ imply $X(\ha)=0$.
Hence
\begin{equation}
-F(0)X'(0)X^{\ast}(0)
=\int_0^{\ha}\Bigl[F|X'|^2+k^2\frac{V^2}{\rhobar}|X|^2\Bigr]\dd z .
\end{equation}
By \cref{prop:SLform}\ref{it:X}, $\qbar=\rhobar FX'$ and thus $\qbar(0)=F(0)X'(0)$.
The coefficient identities $\Imc F=-2\delta(U-\Rec c)\,a^4/\bigl(\rhobar|a^2-V^2|^2\bigr)$ and $\Imc V^2=-2\delta(U-\Rec c)$ now give Equation~\eqref{eq:herglotz} upon taking imaginary parts.
\end{proof}

The limit $\delta\downarrow0$ links this identity to the critical-layer calculation below.
If $w_\ast(z_c)\neq0$, the limiting solution has $|X'|\sim C_c|z-z_c|^{-2}$ for a constant $C_c>0$, and the integral concentrates near the critical levels.
If $w_\ast(z_c)=0$, near $z_c$ the limiting solution is a multiple of the regular basis solution of \cref{prop:local}\ref{it:basis-reg}, which has no logarithmic term, and the flux does not jump at $z_c$.
In either case, twice $\delta$ times the local integral converges to the flux loss $\Jf(z_j^-)-\Jf(z_j^+)$, the negative of the upward jump in \cref{thm:limabs}.
Here and below, $f(z_j^\pm)$ denotes the one-sided limit of a function $f$ at $z_j$ from above and from below.
If $c_k$ lies outside the range of $U$, the integrand is uniformly bounded instead, giving $\Imc\Za=O(\delta)$ as used in \cref{thm:necessity}.

For the critical-layer analysis, it is more convenient to express the same flux in the velocity variables.
With the Wronskian convention $w_1w_2'-w_2w_1'$, the weighted Wronskian $\Abar(w_1w_2'-w_2w_1')$ of two solutions of the divergence-form equation is constant on every interval free of critical levels.
Define
\begin{equation}\label{eq:W-def}
\Wc(z):=\Imc\bigl(\Abar\,w'\,w^{\ast}\bigr)(z).
\end{equation}
For real $\Abar$, $\Imc(\Abar w'w^{\ast})=\tfrac1{2\im}\Abar\,(w'w^{\ast}-w\,w^{\ast\prime})$.
Thus $\Wc$ is $-1/(2\im)$ times the weighted Wronskian of $w$ and $w^{\ast}$, in that order.
When the coefficients are real, $w^{\ast}$ is also a solution, so $\Wc$ is conserved.
We call it the Wronskian-type flux.

For real $c$ away from critical levels, Equation~\eqref{eq:A-def} gives $\Imc(\chibar w^{\ast}) =\Imc\bigl(\Abar Vw'w^{\ast}-\Abar(U'+bV)|w|^2\bigr)=V\,\Wc$, because the second term is real.
Combining this with $\qbar\,\xi^{\ast}=\chibar w^{\ast}/(k^2V)$ for real $V$ yields
\begin{equation}\label{eq:J-W}
\Jf=\frac{\Wc}{k^2},
\qquad\text{so}\qquad
\Imc\Za^+(k,c_r)=\frac{\Wc(0)}{k^2}
\end{equation}
whenever the boundary values exist, with the interface normalization $\xi(0)=1$.
The jump formula for $\Wc$ across a critical level is the compressible generalization of Miles's Reynolds-stress jump, proved in \cref{thm:limabs}.

\subsection{A holomorphic form of the dispersion relation}
\label{subsec:determinant}

The quotient dispersion relation $\Dw(k,c)=\eps\Za(k,c)$ is convenient near the water-wave root, provided the interface trace is nonzero.
For the spectral characterization, we also need a holomorphic function defined at the poles of $\Za$.
We obtain it by multiplying the dispersion relation by the interface trace $w(0;k,c)$.

For $\Imc c\neq0$, let $(w,\chibar)(\cdot;k,c)$ solve Equation~\eqref{eq:wchi} with
\begin{equation}\label{eq:canonical}
  w(\ha;k,c)=0,
  \qquad
  \chibar(\ha;k,c)=\Abar(\ha;c)\,V(\ha;c).
  \end{equation}
The first equation of the system in Equation~\eqref{eq:wchi} gives $w'(\ha)=1$, so $w$ satisfies the normalization of \cref{def:air-response}.
The interface trace $w(0;k,c)$ vanishes precisely when the air equation has a nontrivial solution with $w(0)=w(\ha)=0$.

When $w(0;k,c)\neq0$, division by $w(0;k,c)/(\im kV(0))$ imposes the interface normalization $\xi(0)=1$ in Equation~\eqref{eq:Za-def}.
For real $c_r$, we will also use the relative velocity profile $\hat y=w_\ast/w_\ast(0)$, where $w_\ast$ is the limiting solution of \cref{prop:limit-solutions}.
Clearing the interface trace from the quotient dispersion relation gives the following function.

\begin{definition}\label{def:determinant}
For $k>0$, $\eps\ge0$, and $\Imc c\neq0$ set
\begin{equation}\label{eq:Edet-def}
\Edet(k,c,\eps):=\bigl[\Dw(k,c)+\eps g\bigr]\,w(0;k,c)
-\eps\,V(0;c)\,\chibar(0;k,c),
\end{equation}
with $(w,\chibar)(\cdot;k,c)$ the top-normalized solution with the values at $z=\ha$ given in Equation~\eqref{eq:canonical}.
\end{definition}

\begin{lemma}
\label{lem:determinant}
Let $k>0$ and $\eps>0$.
Then the following hold.
\begin{enumerate}[label=\textup{(\roman*)}]
\item\label{it:det-hol} $\Edet(k,\cdot,\eps)$ is holomorphic on the open upper half-plane and on the open lower half-plane, with the reflection symmetry $\Edet(k,c^{\ast},\eps)=\Edet(k,c,\eps)^{\ast}$.
\item\label{it:det-modes} For $\Imc c\neq0$, the linearized problem has a nontrivial normal mode of the form given by Equation~\eqref{eq:normalmode} at $(k,c)$ if and only if $\Edet(k,c,\eps)=0$.
\item\label{it:det-nonvanish} If $\Edet(k,c,\eps)=0$ with $\Imc c\neq0$, then $w(0;k,c)\neq0$.
Wherever $w(0;k,c)\neq0$,
  \begin{equation}\label{eq:Edet-factor}
  \Edet(k,c,\eps)=w(0;k,c)\,\bigl[\Dw(k,c)-\eps\Za(k,c)\bigr]
  =w(0;k,c)\,D(k,c,\eps),
  \end{equation}
  so for $\Imc c\neq0$ the zeros of $\Edet(k,\cdot,\eps)$ and of $D(k,\cdot,\eps)$ coincide, with multiplicity.
\item\label{it:det-eps0} $\Edet(k,c,0)=\Dw(k,c)\,w(0;k,c)$.
\end{enumerate}
\end{lemma}

\begin{proof}
Holomorphy in \ref{it:det-hol} follows from \cref{lem:regular-solve}.
For the reflection symmetry, the coefficient matrix $\mathcal M(z;c)$ of Equation~\eqref{eq:wchi} and the values at $z=\ha$ in Equation~\eqref{eq:canonical} are rational in $c$ with real $z$-dependent coefficients, so $\mathcal M(z;c^{\ast})=\mathcal M(z;c)^{\ast}$, and likewise for these values.
Hence $z\mapsto(w,\chibar)(z;k,c)^{\ast}$ solves the system at $c^{\ast}$ and satisfies Equation~\eqref{eq:canonical} at $c^{\ast}$, and uniqueness gives $(w,\chibar)(z;k,c^{\ast})=(w,\chibar)(z;k,c)^{\ast}$.
Finally $\Dw(k,c^{\ast})=\Dw(k,c)^{\ast}$ since $\Dw$ is a real polynomial in $c$.

We next prove \ref{it:det-nonvanish}, which is needed for the mode correspondence.
Suppose $w(0;k,c)=0$.
Then $\Edet(k,c,\eps)=-\eps V(0)\chibar(0;k,c)$.
Since $V(0)=-c\neq0$ and $\eps>0$, the equality $\Edet=0$ would imply $\chibar(0;k,c)=0$.
Uniqueness for the regular system with $(w,\chibar)(0)=(0,0)$ would then force $(w,\chibar)\equiv0$.
This contradicts Equation~\eqref{eq:canonical}, whose second component $\Abar(\ha)V(\ha)$ is nonzero for $\Imc c\neq0$ by \cref{lem:regular-solve}.
Substitution of Equation~\eqref{eq:Za-formula} into Equation~\eqref{eq:Edet-def} gives the factorization in Equation~\eqref{eq:Edet-factor}.
Since $w(0;k,\cdot)$ is nonzero on the open set under consideration, the two dispersion functions have the same orders of vanishing there.
Part~\ref{it:det-eps0} follows by setting $\eps=0$ in Equation~\eqref{eq:Edet-def}.

To prove \ref{it:det-modes}, let a nontrivial normal mode be given.
By \cref{prop:reduction}, it has $\hat\zeta\neq0$ and $w(0)=\im kV(0)\hat\zeta\neq0$, since $c\neq0$.
Its air component solves Equation~\eqref{eq:wchi} with $w(\ha)=0$.
This solution space is one-dimensional by \cref{lem:regular-solve}, so the air component is a nonzero multiple $\mu\,(w,\chibar)(\cdot;k,c)$ of the top-normalized solution.
The dynamic condition in Equation~\eqref{eq:dispersion0}, written for general $\hat\zeta$ in normalized variables, is $\Dw\hat\zeta=\eps\bigl(\qbar(0)-g\hat\zeta\bigr)$ with $\rhobar(0)=1$.
Inserting $\hat\zeta=\mu\,w(0;k,c)/(\im kV(0))$ and $\qbar(0)=\mu\,\chibar(0;k,c)/(\im k)$ and multiplying by $\im kV(0)/\mu$ gives $\Edet(k,c,\eps)=0$.

Conversely, if $\Edet(k,c,\eps)=0$ then $w(0;k,c)\neq0$ by \ref{it:det-nonvanish}.
Set $\hat\zeta:=w(0;k,c)/(\im kV(0))\neq0$ and reconstruct the water fields by Equation~\eqref{eq:water-mode}.
Then the two kinematic conditions and the dynamic condition hold by construction, and \cref{prop:reduction} produces the mode.
\end{proof}

\begin{remark}
\label{rem:determinant}
The zeros of the interface trace and of the dispersion function have different meanings.
A zero of $w(0;k,\cdot)$ is a pole of $\Za(k,\cdot)$, because $\chibar(0;k,c)\neq0$ there by uniqueness of the top-normalized solution and the quotient in Equation~\eqref{eq:Za-formula} diverges.
For $\Imc c\neq0$, a zero of $\Edet$ instead corresponds to a normal mode and cannot be a zero of the trace, by \ref{it:det-nonvanish}.
Proposition~\ref{prop:reduction} excludes the possibility of a decoupled air mode with $\hat\zeta=0$ for $k>0$ and $\Imc c\neq0$.

At $\eps=0$, the factorization \ref{it:det-eps0} explains the non-resonance assumption in Hypothesis \ref{hyp:nonres}.
If this assumption fails at $c_r=c_k$, the free wave speed $c_k$ is a zero of the first factor $\Dw(k,\cdot)$ and also of the limit $w_\ast(0;\cdot)$ of the second factor as the imaginary part of $c$ tends to zero.
The proof of \cref{thm:main} uses $\Dw(k,c)-\eps\Za(k,c)=0$, and therefore requires division by the interface trace.
It does not apply at such a common zero.
Near such a common zero, the function $\Edet$ is still holomorphic for $\Imc c\neq0$, but we do not study its behavior at real phase speeds and leave the resonant case open.
Under Hypothesis \ref{hyp:nonres}, the roots of $D$ near $c_k$ are equivalently the zeros of $\Edet$.
\end{remark}
\section{Limiting absorption and the residue formula for the pressure response}
\label{sec:criticallayer}

In this section, we construct the boundary value of the air-side pressure response $\Za(k,c)$ as the positive imaginary part of the complex phase speed $c$ tends to zero.
We first analyze the vertical-velocity equation near each critical level, where the limiting solution has a logarithmic term with coefficient $Q(z_j)/U'(z_j)$, and then combine these local solutions over the whole air layer.
The main result, \cref{thm:limabs}, shows that the boundary value $\Za^{+}(k,c_r)$ exists with a rate and depends $C^1$ on $c_r$, and that its imaginary part is a sum of contributions from the critical levels, each proportional to $Q(z_j)$.
We also extend $\Za$ continuously to a closed neighborhood of $c_k$ in the closed upper half-plane (\cref{prop:Zhalfdisc}) and show that the limiting solution does not vanish at the highest critical level (\cref{lem:positivity}).
These results are used in \S\ref{sec:instability} to prove the main instability theorem.

We study the boundary-value problem
\begin{equation}\label{eq:bvp}
\bigl(\Abar w'\bigr)'=\Bigl(\rhobar k^2+\frac{G}{V}\Bigr)w
\quad\text{on }(0,\ha),
\qquad
w(\ha)=0,
\end{equation}
as $c=c_r+\im\delta$ approaches a regular real value $c_r$ in the range of $U$ from the upper half-plane.
Regularity of $c_r$ means that every zero of $U-c_r$ is simple.
This is a limiting absorption principle.
At real phase speeds, the equation is singular at every critical level, and we take the limit from $\Imc c>0$ because slowly growing modes have small positive $\Imc c$.
We assume Hypotheses \ref{hyp:reg}--\ref{hyp:subsonic}.
We write the critical levels as $z_1<\cdots<z_m$ and take $0\le\delta\le\delta_1$, where $\delta_1>0$ is a small constant chosen in \S\ref{subsec:local}.
In this section, $C$ denotes a positive constant that may change from line to line.
The constants depend on $k$, the profiles, the subsonic gap $\delta_0$ of Hypothesis \ref{hyp:subsonic}, and $\min_j|U'(z_j)|$.
The non-resonance condition in Hypothesis \ref{hyp:nonres} is imposed only where stated.

Recall the first-order system in Equation~\eqref{eq:wchi}, written in the normalized variables of \S\ref{subsec:flux-identities}, where $\chibar=\im k\qbar$,
\begin{equation}\label{eq:wchi-recall}
w'=\Bigl(\frac{V'}{V}+b\Bigr)w+\frac{\chibar}{\Abar V},
\qquad
\chibar'=-b\,\chibar+\rhobar k^2V\,w .
\end{equation}
Here $V'=U'$, since $c$ does not depend on $z$, so the first coefficient equals $(U'+bV)/V$ as in Equation~\eqref{eq:wchi}.
The second equation of this system has no singularity at $V=0$, because $V$ appears in it only in the numerator.
Variation of constants in the first equation then gives an integral equation for $w$, and one integration by parts isolates the logarithmic term.
A contraction argument for a single integral equation gives the local solutions, their critical-level expansion, and estimates uniform as $\delta\downarrow0$.
For $\Abar=\rhobar=1$ and $b=0$, the construction reduces to the incompressible analysis of \cite[\S4]{BSWZ2016}.

Heuristically, near a critical level $z_j$ the solution behaves like $w(z_j)\bigl[1+\tfrac{Q(z_j)}{U'(z_j)}\,(z-z_j)\log|z-z_j|\bigr]$ up to a $C^1$ remainder, and letting $\Imc c\downarrow0$ adds $-\im\pi$ to the logarithm on the side where $U<c_r$.
\cref{prop:local} makes this exact.
That one-sided constant produces the jump in the derivative given in Equation~\eqref{eq:wjump}, and through it the sign in the residue formula in Equation~\eqref{eq:residue}.

\subsection{The pressure response for complex phase speeds}

For complex phase speeds with $\Imc c\neq0$, the equation has no singularity, and the top-normalized solution depends holomorphically on the phase speed.
The next lemma records this and the resulting form of $\Za$.

\begin{lemma}\label{lem:regular-solve}
For $\Imc c\neq0$ the coefficients of Equation~\eqref{eq:wchi-recall} are continuous on $[0,\ha]$ and holomorphic in $c$, because $V$ has no zero, and neither does $a^2-V^2$, since $\Imc(V^2)=2\,\Rec V\,\Imc V$ can vanish only where $\Rec V=0$, and there $a^2-V^2=a^2+(\Imc c)^2>0$.
No subsonicity assumption is needed in this case.
The top-normalized solution $(w,\chibar)(\cdot;k,c)$ of \cref{subsec:determinant}, with the values at $z=\ha$ given in Equation~\eqref{eq:canonical}, exists and is unique.
This solution spans the solution space of Equation~\eqref{eq:wchi-recall} subject to $w(\ha)=0$ and is holomorphic in $c$ on each open half-plane.
Consequently $\Za(k,\cdot)$ is meromorphic on the open upper half-plane, with poles exactly at the zeros of $w(0;k,\cdot)$.
Moreover,
\begin{equation}\label{eq:Za-formula}
\Za(k,c)=\frac{V(0)\,\chibar(0;k,c)}{w(0;k,c)}-\rhobar(0)\,g .
\end{equation}
\end{lemma}

\begin{proof}
For $\Imc c=\delta\neq0$ we have $|V|\ge|\delta|$ on $[0,\ha]$, and $a^2-V^2\neq0$ as shown in the statement.
Hence the coefficients of Equation~\eqref{eq:wchi-recall} are continuous in $z$ and holomorphic in $c$, and the solvability and analyticity claims follow from the standard theory of linear ODEs with parameters \cite[Ch.~1]{CoddingtonLevinson1955}.
The solution space with $w(\ha)=0$ is one-dimensional, because a solution is determined uniquely by its values $(w,\chibar)(\ha)$, which can be prescribed arbitrarily in $\C^2$, so the solutions with $w(\ha)=0$ are the multiples of the solution with $(w,\chibar)(\ha)=(0,1)$.
If $w(0;k,c)=0$ then $\chibar(0;k,c)\neq0$, since otherwise $(w,\chibar)\equiv0$ by uniqueness, contradicting Equation~\eqref{eq:canonical}.
Thus every zero of $w(0;k,\cdot)$ is a pole of $\Za$.
Equation~\eqref{eq:Za-formula} restates Equation~\eqref{eq:Za-def} using $\qbar=\chibar/(\im k)$, $\xi(0)=w(0)/(\im kV(0))$.
\end{proof}

\subsection{Logarithmic solutions and the flux jump at a critical level}
\label{subsec:local}

In this subsection, we study the solutions of Equation~\eqref{eq:wchi-recall} on a small interval around one critical level $z_j$.
For $\delta>0$ the equation is regular there, but the limit of its solutions as $\delta\downarrow0$ contains a logarithmic term at $z_j$.
To control this limit, we rewrite the system as an integral equation in which the logarithm appears explicitly, and we solve it by a contraction argument with bounds that hold uniformly for $\delta\in[0,\delta_1]$.
At $\delta=0$, this gives two basis solutions, a regular one that vanishes at $z_j$ and one that contains the logarithmic term with coefficient $Q(z_j)/U'(z_j)$ (\cref{prop:local}).
By Equation~\eqref{eq:Wjump}, the logarithmic term produces a jump of the Wronskian-type flux $\Wc$ across $z_j$, proportional to $Q(z_j)\,|w(z_j)|^2$.
\cref{prop:convergence} then shows that the solutions for $\delta>0$ converge to these limits with an explicit rate.

Consider a critical level $z_j$ and an interval $I_j=[z_j-\varsigma,z_j+\varsigma]\subset(0,\ha)$ containing no other critical level, with $m_U:=\min_{I_j}|U'|>0$.
Choose $\varsigma$ small from the profiles, as the results below require, and impose $\delta_1\le m_U\varsigma$.
Set the reference speed $V_\ast=m_U\varsigma$.
Every logarithm below is made dimensionless with this scale, and $O(\delta\log(1/\delta))$ abbreviates $O\bigl(\delta(1+|\log(\delta/V_\ast)|)\bigr)$.
We measure the constants $(\alpha,\beta)$ of the local representation in Equation~\eqref{eq:rep} below by the norm $|\alpha|+|\beta|$.
Since all norms on $\C^2$ are equivalent, another choice of norm changes only the constants in the estimates.
For a function $\psi$ on $I_j$, we use the norms
\begin{equation}\label{eq:norms}
\|\psi\|_{C^0(I_j)}:=\sup_{z\in I_j}|\psi(z)|,
\qquad
\|\psi\|_{C^k(I_j)}:=\sum_{i=0}^{k}\bigl\|\psi^{(i)}\bigr\|_{C^0(I_j)},
\end{equation}
where $C^k(I_j)$ is the space of $k$ times continuously differentiable functions on $I_j$.
The Lipschitz constant of a function $\varphi$ on an interval $J$ is
\begin{equation}\label{eq:Lip-def}
\mathrm{Lip}(\varphi):=\sup_{s,t\in J,\ s\neq t}\frac{|\varphi(s)-\varphi(t)|}{|s-t|},
\end{equation}
and $\varphi$ is Lipschitz if this constant is finite.
We set $\|\psi\|_{C^{0,1}(I_j)}:=\|\psi\|_{C^0(I_j)}+\mathrm{Lip}(\psi)$, write $\|\psi\|_\infty$ for $\|\psi\|_{C^0(I_j)}$, and omit $I_j$ from the norms when the interval is clear.
Write
\begin{equation}
\tau(z):=U(z)-c_r,\qquad
E(z):=\exp\Bigl(\int_{z_j}^{z}b\Bigr),\qquad
f:=\frac1{\Abar E},
\end{equation}
so that $V=U-c=\tau-\im\delta$, and introduce the two regular coefficient functions
\begin{equation}\label{eq:g-def}
g_1:=\Bigl(\frac{f}{U'}\Bigr)'-\frac{b\,f}{U'},
\qquad
g_2:=\frac{\rhobar k^2 f}{U'} .
\end{equation}
Under Hypotheses \ref{hyp:reg} and \ref{hyp:subsonic}, $f\in C^2(I_j)$ and $g_1\in C^1(I_j)$, $g_2\in C^2(I_j)$, uniformly for $\delta\in[0,\delta_1]$ and for $c_r$ in a compact set of regular values.
These functions depend on the phase speed $c=c_r+\im\delta$ in two ways, holomorphically through $V$ and $\Abar$, and through the critical level $z_j$, where the integral in $E$ starts.
Since $z_j$ is determined by $U(z_j)=c_r$ with $U'(z_j)\neq0$, it is a $C^1$ function of the real part $c_r$ alone.
Consequently the local objects constructed below depend $C^1$ on the real parameter $c_r$ rather than holomorphically on $c$, and we use holomorphy in $c$ only where $\Imc c\neq0$, as in \cref{lem:regular-solve}.
The local analysis uses the values of $f$, $f'$, and $g_1$ at the critical level $z_j$, which we now compute.
In $\Abar'/\Abar$, the term $\rhobar'/\rhobar$ equals $-b$ exactly, and the other terms contain $V^2$ or $(V^2)'$, which are $O(\delta)$ at $z_j$ because $V(z_j)=-\im\delta$.
Hence $\Abar'(z_j)/\Abar(z_j)=-b(z_j)+O(\delta)$, with equality when $\delta=0$, and one has
\begin{equation}\label{eq:residue-r}
f'(z_j)=O(\delta),
\qquad
g_1(z_j)=-\,\frac{f(z_j)\,Q(z_j)}{U'(z_j)^2}+O(\delta),
\qquad
f(z_j)=\frac1{\rhobar(z_j)}+O(\delta) .
\end{equation}
Thus the coefficient $Q=U''+bU'$ enters the local theory through $g_1(z_j)$, where the $U''$ part comes from the curvature of $V$ inside $(f/U')'$ and the $bU'$ part from the term $-bf/U'$.
In the proof of \cref{lem:representation}, this term comes from the term $-b\chibar$ in the equation $\chibar'=-b\chibar+\rhobar k^2Vw$.

\begin{lemma}\label{lem:representation}
Let $\delta>0$ and let $(w,\chibar)$ solve Equation~\eqref{eq:wchi-recall} on $I_j$.
Then there exist unique constants $(\alpha,\beta)\in\C^2$ such that, for all $z\in I_j$,
\begin{subequations}\label{eq:rep}
\begin{align}
\chibar(z)&=e^{-B(z)}\beta
+\int_{z_j}^{z}e^{B(s)-B(z)}\,\rhobar k^2V(s)\,w(s)\,\dd s,
\qquad B(z):=\int_{z_j}^z b,
\label{eq:rep-chi}\\
w(z)&=\alpha\,V(z)E(z)-\frac{E(z)f(z)\,\chibar(z)}{U'(z)}
+V(z)E(z)\int_{z_j+\varsigma}^{z}
\Bigl[\frac{g_1(s)\,\chibar(s)}{V(s)}+g_2(s)\,w(s)\Bigr]\dd s .
\label{eq:rep-w}
\end{align}
\end{subequations}
Conversely, for any $(\alpha,\beta)$ and any $\delta\ge0$, a continuous pair $(w,\chibar)$ on $I_j$, with $w'$ taken on $I_j\setminus\{z_j\}$ when $\delta=0$, satisfying Equation~\eqref{eq:rep} solves Equation~\eqref{eq:wchi-recall} away from $V=0$.
Moreover $\beta=\chibar(z_j)$, and at $\delta=0$,
\begin{equation}\label{eq:wc-beta}
w(z_j)=-\frac{f(z_j)\,\beta}{U'(z_j)}
=-\frac{\beta}{\rhobar(z_j)\,U'(z_j)} .
\end{equation}
\end{lemma}

\begin{proof}
Equation~\eqref{eq:rep-chi} is the variation-of-constants solution of the $\chibar$-equation with $\beta:=\chibar(z_j)$.
Since $V$ appears in the $\chibar$-equation only in the numerator, this equation is regular at $z_j$ even when $\delta=0$, so we can integrate it from $z_j$ itself, whereas the integral in Equation~\eqref{eq:rep-w} starts at $z_j+\varsigma$.
For Equation~\eqref{eq:rep-w}, the homogeneous $w$-equation $w'=(V'/V+b)w$ has solution $VE$, so variation of constants gives $w=VE\bigl[\tilde\alpha+\int_{z_j+\varsigma}^{z}\chibar f/V^2\,\dd s\bigr]$ for some constant $\tilde\alpha$.
The integrand is $\chibar/(\Abar V)$ divided by the homogeneous solution $VE$, and its factor $f/V^2$, which multiplies the unknown $\chibar$, is the kernel of this integral.
One integration by parts using $1/V^2=-(1/V)'/U'$ turns the double pole into a boundary term and a simple pole,
\begin{equation}
\int\frac{\chibar f}{V^2}\,\dd s
=-\Bigl[\frac{\chibar f}{U'V}\Bigr]
+\int\Bigl(\frac{\chibar f}{U'}\Bigr)'\frac{\dd s}{V},
\end{equation}
and eliminating $\chibar'$ with the $\chibar$-equation, $(\chibar f/U')'=(-b\chibar+\rhobar k^2Vw)f/U'+\chibar(f/U')' =g_1\chibar+g_2\,Vw$, whose second term cancels the pole entirely.
Absorbing the boundary term at $z_j+\varsigma$ into $\alpha$ yields Equation~\eqref{eq:rep-w}.
This determines $\alpha$ uniquely in terms of the values $(w,\chibar)(z_j+\varsigma)$.
For the converse, differentiate Equation~\eqref{eq:rep-w} and use Equation~\eqref{eq:rep-chi}, and the terms reassemble into Equation~\eqref{eq:wchi-recall}.
The two coefficient identities this requires,
\begin{equation}
-\Bigl(\frac{Ef}{U'}\Bigr)'+\frac{bEf}{U'}+Eg_1
+\Bigl(\frac{U'}{V}+b\Bigr)\frac{Ef}{U'}-\frac1{\Abar V}=0,
\qquad
-\frac{Ef}{U'}\,\rhobar k^2V+VEg_2=0,
\end{equation}
hold pointwise.
The first reduces, after expanding $g_1$ and $E'=bE$, to $Ef=1/\Abar$, and the second restates the definition of $g_2$.
Equation~\eqref{eq:wc-beta} follows by evaluating Equation~\eqref{eq:rep-w} at $z=z_j$ where $V(z_j)=0$, $E(z_j)=1$, and using Equation~\eqref{eq:residue-r}.
\end{proof}

In Equation~\eqref{eq:rep} the only singular integrand left is $g_1\chibar/V$, which has a simple pole at $z_j$ when $\delta=0$, with the Lipschitz numerator $g_1\chibar$, so the logarithm it produces can be split off explicitly.
For a Lipschitz function $\psi$ on $I_j$ and $z\in I_j$, writing $\dd s=\dd\tau/U'$,
\begin{equation}\label{eq:logextract}
\int_{z_j+\varsigma}^{z}\frac{\psi(s)}{V(s)}\,\dd s
=\frac{\psi(z_j)}{U'(z_j)}
\Bigl[\operatorname{Log}\bigl(\tau(z)-\im\delta\bigr)
-\operatorname{Log}\bigl(\tau(z_j+\varsigma)-\im\delta\bigr)\Bigr]
+\mathcal R[\psi](z),
\end{equation}
where $\mathcal R[\psi]$ denotes the remainder defined by this identity.
For $\delta>0$, $V=\tau-\im\delta$ lies in the open lower half-plane, where the principal logarithm $\operatorname{Log}$ of Equation~\eqref{eq:Log-def} is holomorphic.
For $\delta=0$, $\operatorname{Log}(\tau-\im0):=\lim_{\delta\downarrow0}\operatorname{Log}(\tau-\im\delta)$ denotes the boundary value from below.
Since $\operatorname{Arg}(\tau-\im\delta)$ tends to $-\pi$ for $\tau<0$ and to $0$ for $\tau>0$ as $\delta\downarrow0$, Equation~\eqref{eq:Log-def} gives
\begin{equation}\label{eq:branch}
\operatorname{Log}(\tau-\im0)=\log|\tau|-\im\pi\,\mathbf 1_{\{\tau<0\}}.
\end{equation}
Here $\mathbf 1_{\{\tau<0\}}$ equals $1$ where $\tau<0$ and $0$ elsewhere.
We use this branch in the critical-layer analysis below.
The remainder obeys $\|\mathcal R[\psi]\|_{C^0(I_j)}\le C\varsigma\,\mathrm{Lip}(\mathfrak g)$ uniformly in $\delta\in[0,\delta_1]$, where $\mathfrak g:=(\psi/U')\circ\tau^{-1}$ is the function $\psi/U'$ written in the variable $\tau$.
Since $|U'|\ge m_U>0$ on $I_j$, the map $\tau$ is strictly monotone there, and the difference quotient $[\mathfrak g(\tau)-\mathfrak g(0)]/(\tau-\im\delta)$ is bounded by $\mathrm{Lip}(\mathfrak g)$.
If moreover $\psi\in C^1(I_j)$, then at $\delta=0$ the remainder is continuously differentiable across $z_j$, with
\begin{equation}\label{eq:remainder-C1}
\mathcal R[\psi]'(z)
=\frac{\psi(z)-\psi(z_j)\,U'(z)/U'(z_j)}{\tau(z)}
\;\xrightarrow[z\to z_j^\pm]{}\;
\frac{\psi'(z_j)U'(z_j)-\psi(z_j)U''(z_j)}{U'(z_j)^{2}},
\end{equation}
the same limit from both sides (\cref{app:proofs}).
Therefore the only logarithmic term in the local solutions is the explicit first term of Equation~\eqref{eq:logextract}.

\begin{lemma}
\label{lem:fixedpoint}
There exist $\varsigma_0>0$ and $C_\ast\ge1$, depending only on the quantities listed at the beginning of this section, such that for every $\varsigma\in(0,\varsigma_0]$, every $\delta\in[0,\delta_1]$ with $\delta_1\le m_U\varsigma$, and every $(\alpha,\beta)\in\C^2$ the system in Equation~\eqref{eq:rep} has a unique solution $(w,\chibar)=(w^{\alpha,\beta,\delta},\chibar^{\alpha,\beta,\delta}) \in C^0(I_j)^2$, depending linearly on $(\alpha,\beta)$, with
\begin{equation}\label{eq:fp-bounds}
\|w\|_{C^0(I_j)}+\|\chibar\|_{C^0(I_j)}\le C_\ast(|\alpha|+|\beta|),
\qquad
\mathrm{Lip}(\chibar)\le C_\ast(|\alpha|+|\beta|),
\end{equation}
and, for $\delta=0$, $w\in C^1(I_j\setminus\{z_j\})$ with $|w'(z)|\le C_\ast(|\alpha|+|\beta|)\bigl(1+\bigl|\log(|z-z_j|/\varsigma)\bigr|\bigr)$.
For $\delta>0$ these are exactly the solutions of Equation~\eqref{eq:wchi-recall} in the coordinates of \cref{lem:representation}.
\end{lemma}

\begin{proof}
In the system in Equation~\eqref{eq:rep}, the first equation gives $\chibar$ in terms of $w$, so we substitute it into the second equation and apply the contraction argument to $w$ alone.
For given $w\in C^0(I_j)$, Equation~\eqref{eq:rep-chi} defines a function $\chibar[w]\in C^1(I_j)$ with $\chibar[w](z_j)=\beta$, and, since $|V|\le|\tau|+\delta\le C\varsigma$ on $I_j$ (using $\delta\le\delta_1\le m_U\varsigma$),
\begin{equation}\label{eq:chi-of-w}
\|\chibar[w]\|_{\infty}\le C\bigl(|\beta|+\varsigma^{2}\|w\|_\infty\bigr),
\qquad
\mathrm{Lip}(\chibar[w])\le C\bigl(|\beta|+\varsigma\|w\|_\infty\bigr),
\end{equation}
the second bound because $\chibar[w]'=-b\,\chibar[w]+\rhobar k^2V\,w$.
Substituting $\chibar[w]$ into Equation~\eqref{eq:rep-w} defines the affine map $\mathcal B:C^0(I_j)\to C^0(I_j)$,
\begin{equation}
\mathcal B[w]:=\alpha\,VE-\frac{Ef\,\chibar[w]}{U'}
+VE\int_{z_j+\varsigma}^{z}
\Bigl[\frac{g_1(s)\,\chibar[w](s)}{V(s)}+g_2(s)\,w(s)\Bigr]\dd s ,
\end{equation}
and a pair $(w,\chibar)\in C^0(I_j)^2$ satisfies Equation~\eqref{eq:rep} if and only if $w=\mathcal B[w]$ and $\chibar=\chibar[w]$.
Here affine means that $\mathcal B[w]$ is the sum of a linear map of $w$ and a function that does not depend on $w$.
The map $\mathcal B$ is well defined on all of $C^0(I_j)$, including at $\delta=0$, because the simple pole is integrated against $g_1\chibar[w]$, which is Lipschitz by Equation~\eqref{eq:chi-of-w} and $g_1\in C^1$, so the decomposition in Equation~\eqref{eq:logextract} applies with $\psi=g_1\chibar[w]$.
The map $\mathcal B$ requires only continuity of $w$ itself.
At $z=z_j$ we define the singular term of $\mathcal B[w]$ by continuity, since the prefactor $V(z)E(z)$ multiplies the logarithm and $V\operatorname{Log}V\to0$.

For the contraction, let $w,\tilde w\in C^0(I_j)$ and set $\Delta\chibar:=\chibar[w]-\chibar[\tilde w]$.
Both functions take the value $\beta$ at $z_j$, so $\Delta\chibar(z_j)=0$, and by Equation~\eqref{eq:chi-of-w}, $\|\Delta\chibar\|_\infty\le C\varsigma^{2}\|w-\tilde w\|_\infty$ and $\mathrm{Lip}(\Delta\chibar)\le C\varsigma\|w-\tilde w\|_\infty$.
In
\begin{equation}
\mathcal B[w]-\mathcal B[\tilde w]
=-\frac{Ef\,\Delta\chibar}{U'}
+VE\int_{z_j+\varsigma}^{z}
\Bigl[\frac{g_1\Delta\chibar}{V}+g_2\,(w-\tilde w)\Bigr]\dd s
\end{equation}
the decomposition in Equation~\eqref{eq:logextract} applied to $\psi=g_1\Delta\chibar$ has vanishing residue coefficient, $\psi(z_j)=0$, so no logarithm arises in the difference, only the remainder, bounded by $C\varsigma\,\mathrm{Lip}\bigl((\psi/U')\circ\tau^{-1}\bigr)\le C\varsigma^{2}\|w-\tilde w\|_\infty$.
Collecting the three terms,
\begin{equation}
\|\mathcal B[w]-\mathcal B[\tilde w]\|_\infty
\le C_5\,\varsigma^{2}\,\|w-\tilde w\|_\infty
\le\tfrac12\,\|w-\tilde w\|_\infty
\qquad\bigl(\varsigma\le\varsigma_0:=(2C_5)^{-1/2}\bigr),
\end{equation}
uniformly in $\delta\in[0,\delta_1]$, in $(\alpha,\beta)$, and in the parameters.
Banach's fixed-point theorem gives a unique fixed point $w$.
Setting $\chibar:=\chibar[w]$ produces the unique solution of Equation~\eqref{eq:rep}, linear in $(\alpha,\beta)$ since $\mathcal B[w]$ and $\chibar[w]$ depend linearly on $(w,\alpha,\beta)$ jointly.

For the bounds, $\mathcal B[0]$ consists of $\alpha VE$, of $-Ef\chibar[0]/U'$ with $\chibar[0]=e^{-B}\beta$, and of the singular term with $\psi=g_1e^{-B}\beta$, whose logarithmic part is multiplied by the prefactor $VE$ and is bounded, since $\sup_{z\in I_j}|V(z)|\bigl(1+\bigl|\log|V(z)|\bigr|\bigr) \le C\varsigma\log(1/\varsigma)$.
Thus $\|\mathcal B[0]\|_\infty\le C(|\alpha|+|\beta|)$.
Hence $\|w\|_\infty\le2\|\mathcal B[0]\|_\infty\le C_\ast(|\alpha|+|\beta|)$, and Equation~\eqref{eq:chi-of-w} converts this into Equation~\eqref{eq:fp-bounds}.
The derivative bound for $\delta=0$ follows by differentiating Equation~\eqref{eq:rep-w}, all terms being bounded except the product of $(VE)'$ with the integral in Equation~\eqref{eq:logextract}, which is logarithmic.
The last statement is \cref{lem:representation}.
\end{proof}

\begin{proposition}\label{prop:local}
Let $\delta=0$ and denote by $(w^{(1)},\chibar^{(1)})$ and $(w^{(2)},\chibar^{(2)})$ the fixed points of \cref{lem:fixedpoint} with $(\alpha,\beta)=(1,0)$ and $(0,1)$ respectively.
Then the following hold.
\begin{enumerate}[label=\textup{(\roman*)}]
\item\label{it:basis-reg} $w^{(1)}\in C^1(I_j)$, $\chibar^{(1)}\in C^2(I_j)$ are real-valued, with $w^{(1)}=V(z)E(z)\bigl(1+O(\varsigma)\bigr)$, $w^{(1)}(z_j)=0$, and $\chibar^{(1)}(z)=O(|z-z_j|^{3})$.
No logarithm occurs.
\item\label{it:basis-log} $w^{(2)}(z_j)=-1/(\rhobar U')(z_j)\neq0$, and for $z\in I_j$,
\begin{equation}\label{eq:local-expansion}
w^{(2)}(z)
=w^{(2)}(z_j)\Bigl[1+\frac{Q(z_j)}{U'(z_j)}\,(z-z_j)\,
\operatorname{Log}\bigl(\tau(z)-\im0\bigr)\Bigr]
+h(z),
\end{equation}
where the remainder satisfies $h\in C^{1}(I_j)$ and $h(z_j)=0$.
In particular $w^{(2)}$ is continuous.
When $Q(z_j)\neq0$ its one-sided derivatives diverge logarithmically at $z_j$, so we use instead the \emph{renormalized} one-sided derivatives
\begin{equation}\label{eq:renorm-deriv}
\mathcal D_\pm:=\lim_{z\to z_j^\pm}
\Bigl[w^{(2)\prime}(z)
-\frac{Q(z_j)}{U'(z_j)}\,w^{(2)}(z_j)\,\log\bigl|\tau(z)\bigr|\Bigr],
\end{equation}
which exist and satisfy the jump formula
\begin{equation}\label{eq:wjump}
\mathcal D_+-\mathcal D_-
=\lim_{s\downarrow0}
\bigl[w^{(2)\prime}(z_j+s)-w^{(2)\prime}(z_j-s)\bigr]
=\im\pi\,\frac{Q(z_j)}{|U'(z_j)|}\,w^{(2)}(z_j),
\end{equation}
where the middle expression equals $\mathcal D_+-\mathcal D_-$ because the logarithmic terms at $z_j+s$ and $z_j-s$ cancel as $s\downarrow0$.
\item\label{it:basis-wronskian} The modified Wronskian $\Xi:=w^{(1)}\chibar^{(2)}-w^{(2)}\chibar^{(1)}$, the determinant of the two solution vectors of the first-order system in Equation~\eqref{eq:wchi-recall}, satisfies $\Xi(z)=\varkappa\,V(z)$ on $I_j$ for a constant $\varkappa=1+O(\varsigma^2)$, in particular $\varkappa\neq0$ for $\varsigma$ small.
Hence $\{(w^{(1)},\chibar^{(1)}),(w^{(2)},\chibar^{(2)})\}$ is a basis of the solution space on each of $I_j^\pm:=I_j\cap\{\pm(z-z_j)>0\}$, and every solution on $I_j\setminus\{z_j\}$ arising as a limit of $\delta>0$ solutions is a unique combination $\alpha(w^{(1)},\chibar^{(1)})+\beta(w^{(2)},\chibar^{(2)})$ globally on $I_j$.
\item\label{it:quadratic} For any such solution, $u_1:=|w|^2$ is continuous on $I_j$ and $\Wc=\Imc(\Abar w'w^{\ast})$ extends continuously to each of $I_j^\pm$ with the jump
\begin{equation}\label{eq:Wjump}
\Wc(z_j^+)-\Wc(z_j^-)
=\pi\,\frac{\rhobar(z_j)\,Q(z_j)}{|U'(z_j)|}\,u_1(z_j),
\end{equation}
the limit taken from $\Imc c>0$.
For $\Imc c<0$ the sign reverses.
Moreover $u_1(z_j)=0$ if and only if $\beta=0$, in which case the solution is a complex multiple of the real solution $(w^{(1)},\chibar^{(1)})$, hence $C^1$ across $z_j$ with $\Wc\equiv0$ on $I_j$.
\end{enumerate}
\end{proposition}

\begin{proof}
For \ref{it:basis-reg}, with $\beta=0$, Equation~\eqref{eq:rep-chi} and the bound $\|w^{(1)}\|_\infty\le C_\ast$ give $|\chibar^{(1)}(z)|\le C\!\int_{z_j}^z|\tau|\le C(z-z_j)^2$.
Then the singular integrand $g_1\chibar^{(1)}/V$ is $O(|z-z_j|)$, so the residue in Equation~\eqref{eq:logextract} vanishes ($\psi(z_j)=0$) and no logarithm arises in Equation~\eqref{eq:rep-w}, whose derivative $(VE)'\!\int[\cdot]+E[g_1\chibar^{(1)}+Vg_2w^{(1)}]$ is then continuous.
Hence $w^{(1)}\in C^1(I_j)$ with $w^{(1)}=VE+O(\varsigma)\,O(|z-z_j|)$.
Substituting this into Equation~\eqref{eq:rep-chi} improves the bound to $\chibar^{(1)}=O(|z-z_j|^3)$, and $\chibar^{(1)\prime}=-b\chibar^{(1)}+\rhobar k^2Vw^{(1)}\in C^1$ gives $\chibar^{(1)}\in C^2$.
For $\delta=0$, all kernels are real, as are the constants $(\alpha,\beta)=(1,0)$, so the solution is real.

For \ref{it:basis-log}, $w^{(2)}(z_j)$ is given by Equation~\eqref{eq:wc-beta} with $\beta=1$.
Apply the decomposition in Equation~\eqref{eq:logextract} to $\psi=g_1\chibar^{(2)}$.
By Equation~\eqref{eq:residue-r} the residue coefficient is
\begin{equation}
\frac{(g_1\chibar^{(2)})(z_j)}{U'(z_j)}
=-\frac{f(z_j)Q(z_j)}{U'(z_j)^3}\,\beta
=\frac{Q(z_j)}{U'(z_j)^2}\;w^{(2)}(z_j),
\end{equation}
so Equation~\eqref{eq:rep-w} becomes
\begin{equation}
w^{(2)}
=VE\,[\,\text{const}\,]-\frac{Ef\chibar^{(2)}}{U'}
+VE\,\Bigl[\frac{Q(z_j)}{U'(z_j)^{2}}\,w^{(2)}(z_j)\,
\operatorname{Log}(\tau-\im0)+\mathcal R[\psi]\Bigr]
+VE\!\int_{z_j+\varsigma}^{z}\!g_2\,w^{(2)} .
\end{equation}
Every term except the explicit logarithm is $C^1(I_j)$.
The first is smooth.
The second is $C^1$ because $\chibar^{(2)}\in C^1(I_j)$, its derivative $-b\chibar^{(2)}+\rhobar k^2Vw^{(2)}$ being continuous.
The last is $C^1$ because its integrand is continuous, and $VE\,\mathcal R[\psi]$ is $C^1$ because $\psi=g_1\chibar^{(2)}\in C^1(I_j)$, so that Equation~\eqref{eq:remainder-C1} gives $\mathcal R[\psi]\in C^1(I_j)$.
Replacing the prefactor $VE$ of the logarithm by its linearization $U'(z_j)(z-z_j)$ changes the expression by $[VE-U'(z_j)(z-z_j)]\operatorname{Log}(\tau-\im0) =O\bigl((z-z_j)^2\log|z-z_j|\bigr)$, which is again $C^1$ with derivative vanishing at $z_j$.
The constant term $w^{(2)}(z_j)$ in the bracket of Equation~\eqref{eq:local-expansion} is the value of $-Ef\chibar^{(2)}/U'$ at $z_j$, by Equation~\eqref{eq:wc-beta}.
Collecting all $C^1$ terms into $h$ gives Equation~\eqref{eq:local-expansion} with $h\in C^1(I_j)$, $h(z_j)=0$.

Differentiating Equation~\eqref{eq:local-expansion},
\begin{equation}
w^{(2)\prime}(z)
=\frac{Q(z_j)}{U'(z_j)}\,w^{(2)}(z_j)
\Bigl[\operatorname{Log}\bigl(\tau(z)-\im0\bigr)
+\frac{(z-z_j)\,U'(z)}{\tau(z)}\Bigr]+h'(z),
\end{equation}
where $(z-z_j)U'(z)/\tau(z)\to1$ from both sides, while by Equation~\eqref{eq:branch} the difference $\operatorname{Log}(\tau-\im0)-\log|\tau|=-\im\pi\mathbf 1_{\{\tau<0\}}$ is constant on each side of $z_j$.
Hence the renormalized limits in Equation~\eqref{eq:renorm-deriv} exist, each equal to $(Q/U')(z_j)\,w^{(2)}(z_j)\,(1-\im\pi\mathbf 1_{\{\tau<0\}})+h'(z_j)$ with the indicator evaluated on the corresponding side.
Crossing $z_j$ upward, $\mathbf 1_{\{\tau<0\}}$ drops from $1$ to $0$ if $U'(z_j)>0$ and rises from $0$ to $1$ if $U'(z_j)<0$.
In both cases $\mathcal D_+-\mathcal D_- =\pm\im\pi\,\tfrac{Q}{U'}(z_j)\,w^{(2)}(z_j) =\im\pi\,\tfrac{Q(z_j)}{|U'(z_j)|}\,w^{(2)}(z_j)$.
The formula for the difference evaluated at $z_j+s$ and $z_j-s$ follows since $\log|\tau(z_j+s)|-\log|\tau(z_j-s)|\to0$ as $s\downarrow0$ and $h'$ is continuous.
This is Equation~\eqref{eq:wjump}.

For \ref{it:basis-wronskian}, differentiating $\Xi$ and using Equation~\eqref{eq:wchi-recall} gives $\Xi'=(V'/V)\,\Xi$ on $I_j\setminus\{z_j\}$, because the $b$- and $k^2$-terms cancel between the two solutions.
Hence $\Xi/V$ is constant on each side of $z_j$.
Equivalently, from Equation~\eqref{eq:A-def}, $\Xi=\Abar V\,(w^{(1)}w^{(2)\prime}-w^{(2)}w^{(1)\prime})$, so $\Xi/V$ is the $\Abar$-weighted Wronskian of Equation~\eqref{eq:bvp}, which is constant on each side.
By \ref{it:basis-reg}, $w^{(1)}\in C^1(I_j)$ vanishes at $z_j$ and $\chibar^{(1)}(z)=O(|z-z_j|^{3})$, while $w^{(2)}$ and $\chibar^{(2)}$ are continuous, so $\Xi/V$ has the same limit from both sides of $z_j$, and $\Xi=\varkappa\,V$ on $I_j$ with a single constant $\varkappa$.
To compute $\varkappa$, we evaluate $\Xi$ at the endpoint $z_j+\varsigma$, where the integral in Equation~\eqref{eq:rep-w} is empty.
There $w^{(1)}=VE-Ef\chibar^{(1)}/U'$ and $w^{(2)}=-Ef\chibar^{(2)}/U'$, so the terms with $Ef/U'$ cancel and $\Xi=VE\chibar^{(2)}$.
Equation~\eqref{eq:rep-chi} with $\beta=1$ and $E=e^{B}$ then gives $\varkappa=E(z_j+\varsigma)\,\chibar^{(2)}(z_j+\varsigma)=1+\int_{z_j}^{z_j+\varsigma}E\,\rhobar k^2V\,w^{(2)}\,\dd s$.
The integral is $O(\varsigma^2)$, because $|V|\le\varsigma\max_{I_j}|U'|$ on $I_j$ and $w^{(2)}$ is bounded, so $\varkappa=1+O(\varsigma^2)$.
The combination is unique on all of $I_j$ because $(\alpha,\beta)$ are determined linearly and invertibly by the values at $z_j+\varsigma$, and the representation in Equation~\eqref{eq:rep} extends the solution across $z_j$.

For \ref{it:quadratic}, $u_1=|w|^2$ is continuous since $w$ is.
For $\Wc$, write the solution as $\alpha(w^{(1)},\chibar^{(1)})+\beta(w^{(2)},\chibar^{(2)})$.
By \ref{it:basis-reg} and \ref{it:basis-log} it has the form of Equation~\eqref{eq:local-expansion} with $w^{(2)}(z_j)$ replaced by $w(z_j)=\beta w^{(2)}(z_j)$ and a $C^1$ remainder, so its renormalized derivatives $\mathcal D_\pm$, defined as in Equation~\eqref{eq:renorm-deriv} with $w^{(2)}$ replaced by $w$, exist and satisfy Equation~\eqref{eq:wjump} with $w(z_j)$ in place of $w^{(2)}(z_j)$.
Set $\varpi:=(Q/U')(z_j)\,w(z_j)$ and split
\begin{equation}
\Wc(z)=\Imc\bigl(\Abar w'w^{\ast}\bigr)
=\Abar(z)\,\log|\tau(z)|\;\Imc\bigl(\varpi\,w^{\ast}(z)\bigr)
+\Imc\Bigl(\Abar(z)\bigl[w'(z)-\varpi\log|\tau(z)|\bigr]w^{\ast}(z)\Bigr).
\end{equation}
As $z\to z_j^\pm$ the second term converges to $\Imc\bigl(\Abar(z_j)\,\mathcal D_\pm\,w^{\ast}(z_j)\bigr)$.
The first term vanishes in the limit although $\log|\tau|$ diverges, because the coefficient $(Q/U')(z_j)$ is real and $w(z_j)\,w^{\ast}(z_j)=|w(z_j)|^2$, so $\Imc(\varpi\,w^{\ast}(z_j))=0$, while $w^{\ast}(z)-w^{\ast}(z_j)=O(|z-z_j|\log|z-z_j|)$ by Equation~\eqref{eq:local-expansion}, making the first term $O(|z-z_j|\log^2|z-z_j|)\to0$.
In words, the logarithmically divergent part of $\Abar w'w^{\ast}$ is real and drops out of the imaginary part.
Hence $\Wc$ extends continuously to each of $I_j^\pm$ and
\begin{equation}
\Wc(z_j^+)-\Wc(z_j^-)
=\Imc\bigl(\Abar(z_j)\,(\mathcal D_+-\mathcal D_-)\,w^{\ast}(z_j)\bigr)
=\Abar(z_j)\,\pi\,\frac{Q(z_j)}{|U'(z_j)|}\,|w(z_j)|^2 ,
\end{equation}
which is Equation~\eqref{eq:Wjump} since $\Abar(z_j)=\rhobar(z_j)$ and $u_1(z_j)=|w(z_j)|^2$.
Finally $u_1(z_j)=0\iff w(z_j)=0 \iff\beta=0$ by Equation~\eqref{eq:wc-beta}, and then the solution is $\alpha\,(w^{(1)},\chibar^{(1)})$ with $w^{(1)}$ real, so $\Wc=|\alpha|^2\Imc(\Abar w^{(1)\prime}w^{(1)})=0$ on $I_j$.
\end{proof}

\begin{remark}\label{rem:mack-compare}
Equation~\eqref{eq:local-expansion} shows that the compressible critical layer has the same type of singularity as Rayleigh's equation, with $U''$ replaced by $Q$.
For the zero-gravity, temperature-stratified gas the same computation, with the unified coefficient of \cref{rem:leeslin}, reproduces the classical local expansions of compressible boundary-layer theory \cite[Eqs.~(9.13)--(9.14)]{Mack1984}, including the $-\im\pi$ branch rule.
The derivation makes the physical origin of the two pieces of $Q=U''+bU'$ transparent.
The curvature $U''$ enters through the curvature of the profile in the variation-of-constants kernel, and $bU'$ through the hydrostatic decay $\chibar'=-b\chibar+O(V)$ of the scaled pressure variable across the layer.
\end{remark}

\begin{proposition}
\label{prop:convergence}
Let $(w_\delta,\chibar_\delta)$ denote the fixed point of \cref{lem:fixedpoint} at the parameter $\delta\in[0,\delta_1]$ with constants $(\alpha,\beta)$.
Here and below, a subscript $\delta$ marks the value of a quantity at the parameter $\delta$, for example $V_\delta=\tau-\im\delta$.
\begin{enumerate}[label=\textup{(\roman*)}]
\item\label{it:conv-fixed}
\begin{equation}\label{eq:delta-rate}
\|w_\delta-w_0\|_{C^0(I_j)}+\|\chibar_\delta-\chibar_0\|_{C^0(I_j)}
\;\le\;C_\ast\,(|\alpha|+|\beta|)\;
\delta\bigl(1+\bigl|\log(\delta/V_\ast)\bigr|\bigr),
\qquad \delta\in(0,\delta_1],
\end{equation}
uniformly on $I_j$, including at $z_j$.
\item\label{it:conv-matrices} Let $\mathcal E^{\pm}_{\delta}:(\alpha,\beta)\mapsto (w_\delta,\chibar_\delta)(z_j\pm\varsigma)$ be the endpoint evaluation matrices.
Here $\|\cdot\|$ denotes the operator norm of a linear map from $\C^2$ to itself, where $\C^2$ carries the norm $|\alpha|+|\beta|$ chosen at the beginning of this subsection, and the operator norm of a linear map $\mathcal G$ is the maximum of the norm of $\mathcal G(\alpha,\beta)$ over the $(\alpha,\beta)$ with $|\alpha|+|\beta|=1$.
Then $\|\mathcal E^{\pm}_{\delta}\|\le C$ and $\|\mathcal E^{\pm}_{\delta}-\mathcal E^{\pm}_{0}\|= O\bigl(\delta(1+|\log(\delta/V_\ast)|)\bigr)$ by \ref{it:conv-fixed}.
For every $\delta\in[0,\delta_1]$,
\begin{equation}
\det\mathcal E^{\pm}_{\delta}
=\varkappa_\delta\,V_\delta(z_j\pm\varsigma)
\end{equation}
with a constant $\varkappa_\delta$ independent of the endpoint, $|\varkappa_\delta-1|\le C\varsigma^2$ for every $\delta\in[0,\delta_1]$, and $|\varkappa_\delta-\varkappa_0|\le C\,\delta\bigl(1+|\log(\delta/V_\ast)|\bigr)/(m_U\varsigma)$.
Hence, for $\varsigma$ small, $|\varkappa_\delta|\ge\tfrac12$ for every $\delta\in[0,\delta_1]$, and the inverses exist with $\|(\mathcal E^{\pm}_{\delta})^{-1}\|\le C$ uniformly in $\delta\in[0,\delta_1]$.
\item\label{it:conv-transfer} The maps $T_\delta:=\mathcal E^{-}_{\delta}\circ(\mathcal E^{+}_{\delta})^{-1}$, which for $\delta>0$ coincide with the solution maps $(w,\chibar)(z_j+\varsigma)\mapsto(w,\chibar)(z_j-\varsigma)$ of the regular system in Equation~\eqref{eq:wchi-recall}, are defined for every $\delta\in[0,\delta_1]$, are bounded with bounded inverses uniformly, and converge in norm to $T_0$ at rate $O\bigl(\delta(1+|\log(\delta/V_\ast)|)\bigr)$.
Moreover $T_0$, together with the critical values $w_0(z_j)$, depends $C^1$ on $c_r$ on compact sets of regular values, with uniformly bounded derivatives.
\end{enumerate}
\end{proposition}

The proof compares the fixed-point maps at $\delta$ and at $0$ applied to the same function, and treats the singular integral through its logarithmic decomposition, since the kernel difference itself is not integrable across the critical level.

\begin{proof}
Since the contraction constant of $\mathcal B$ in \cref{lem:fixedpoint} is $\tfrac12$ uniformly in $\delta$, it suffices to bound $\|(\mathcal B_\delta-\mathcal B_0)[w_0]\|_\infty$, where $\mathcal B_\delta$ denotes the map of \cref{lem:fixedpoint} at the value $\delta$.
Then $\|w_\delta-w_0\|_\infty\le2\|(\mathcal B_\delta-\mathcal B_0)[w_0]\|_\infty$, and the $\chibar$-component follows from $\|\chibar_\delta[w_\delta]-\chibar_0[w_0]\|_\infty \le C\varsigma^{2}\|w_\delta-w_0\|_\infty+C\delta$, the kernel of Equation~\eqref{eq:rep-chi} being Lipschitz in $\delta$.

The map $\mathcal B_\delta$ depends on $\delta$ in two ways.
It contains $V$ explicitly, in the prefactor $VE$, in the kernel of Equation~\eqref{eq:rep-chi}, and in the simple pole $1/V$.
It also contains the coefficients $f_\delta=1/(\Abar_\delta E)$, $g_{1,\delta}$, and $g_{2,\delta}$, which depend on $\delta$ through $V^2$ in $\Abar_\delta$.
Since $V_\delta^2-V_0^2=-2\im\delta\tau-\delta^2$, these coefficients are Lipschitz in $\delta$ in the norms we need, namely
\begin{equation}\label{eq:coef-delta}
\|f_\delta-f_0\|_{C^2(I_j)}
+\|g_{1,\delta}-g_{1,0}\|_{C^{0,1}(I_j)}
+\|g_{2,\delta}-g_{2,0}\|_{C^0(I_j)}\le C\delta .
\end{equation}
To see this, note that $f_\delta-f_0=(2\im\delta\,\tau+\delta^{2})/(\rhobar a^2E)$ exactly.
The numerator has $C^2(I_j)$ norm $O(\delta)$, because its derivatives up to order two involve only $\tau$, $U'$, and $U''$, and the denominator does not depend on $\delta$ and is bounded away from zero.
Hence $\|f_\delta-f_0\|_{C^2}\le C\delta$.
The definitions in Equation~\eqref{eq:g-def}, linear in $f$, then give the $g_1$-estimate, whose Lipschitz seminorm involves $f''$, and the $g_2$-estimate.
Likewise $\|\chibar_\delta[w_0]-\chibar_0[w_0]\|_\infty +\mathrm{Lip}\bigl(\chibar_\delta[w_0]-\chibar_0[w_0]\bigr) \le C\delta\,\|w_0\|_\infty$, with $(\chibar_\delta[w_0]-\chibar_0[w_0])(z_j)=0$, the kernel of Equation~\eqref{eq:rep-chi} being Lipschitz in $\delta$.
The regular parts of $(\mathcal B_\delta-\mathcal B_0)[w_0]$ are the prefactor difference $\alpha(V_\delta-V_0)E=-\im\delta\alpha E$, the terms $-E(f_\delta\chibar_\delta[w_0]-f_0\chibar_0[w_0])/U'$, and the $g_2$-integrals.
Together they contribute $O(\delta)$ in $C^0(I_j)$.

The singular part must be compared through the logarithmic decomposition, not by bounding the kernel difference in absolute value.
When $z$ lies on the far side of $z_j$, the integration interval crosses the pole, where $|1/V_\delta-1/V_0|=\delta/(|\tau||\tau-\im\delta|)$ is not integrable, and the $\delta=0$ integral exists only as the boundary value in Equation~\eqref{eq:branch}.
Write $\psi_\delta:=g_{1,\delta}\chibar_\delta[w_0]$, so that $\|\psi_\delta-\psi_0\|_{C^{0,1}}\le C\delta$ by Equation~\eqref{eq:coef-delta} while $\|\psi_0\|_{C^{0,1}}\le C(|\alpha|+|\beta|)$, and apply the decomposition in Equation~\eqref{eq:logextract} at the parameter $\delta$ and at $0$ separately.
The difference of the two singular terms splits into the following four pieces.
\begin{enumerate}[label=\textup{(\alph*)},leftmargin=2.2em]
\item The logarithmic part with the common residue $\psi_0(z_j)/U'(z_j)$, namely $E\,\tfrac{\psi_0(z_j)}{U'(z_j)} \bigl[V_\delta\operatorname{Log}(\tau-\im\delta) -V_0\operatorname{Log}(\tau-\im0)\bigr]$, is bounded by $C\delta(1+\log(1/\delta))$ through the model estimate in Equation~\eqref{eq:log-model} of \cref{app:proofs}.
\item The part from the change of residue, $E\,\tfrac{(\psi_\delta-\psi_0)(z_j)}{U'(z_j)}\, V_\delta\operatorname{Log}(\tau-\im\delta)$, is bounded by $C\delta\cdot\varsigma\log(1/\varsigma)$, since $\sup_{I_j}|V_\delta|\bigl(1+|\log|V_\delta||\bigr)\le C\varsigma\log(1/\varsigma)$.
\item The endpoint terms at $z_j+\varsigma$ differ by $O(\delta)$.
With $\tau_+:=\tau(z_j+\varsigma)$, which satisfies $|\tau_+|\ge m_U\varsigma$, we have $|\operatorname{Log}(\tau_+-\im\delta)-\operatorname{Log}(\tau_+-\im0)| \le\delta/|\tau_+|$ and $|V_\delta-V_0|=\delta$.
\item The remainders are $(V_\delta-V_0)E\,\mathcal R_\delta[\psi_\delta] +V_0E\bigl(\mathcal R_\delta[\psi_\delta]-\mathcal R_\delta[\psi_0]\bigr) +V_0E\bigl(\mathcal R_\delta[\psi_0]-\mathcal R_0[\psi_0]\bigr)$.
The first is $O(\delta\varsigma)$ by the remainder bound of Equation~\eqref{eq:logextract}.
The second is $O(\varsigma\delta)$ by the same bound applied to $\psi_\delta-\psi_0$.
The third is $O(\delta\log(1/\delta))$ by the comparison estimate in Equation~\eqref{eq:remainder-compare} of \cref{app:proofs}, in which the Lipschitz factor $|\mathfrak g(t)-\mathfrak g(0)|\le \mathrm{Lip}(\mathfrak g)|t|$, with $\mathfrak g=(\psi/U')\circ\tau^{-1}$ the function $\psi/U'$ written in the variable $\tau$, cancels the nonintegrable part of the kernel difference.
\end{enumerate}
Collecting, $\|(\mathcal B_\delta-\mathcal B_0)[w_0]\|_\infty\le C(|\alpha|+|\beta|)\,\delta\log(1/\delta)$, and Equation~\eqref{eq:delta-rate} follows.
This proves \ref{it:conv-fixed}.

For \ref{it:conv-matrices}, the columns of $\mathcal E^{\pm}_{\delta}$ are the values at $z_j\pm\varsigma$ of the basis fixed points with $(\alpha,\beta)=(1,0)$ and $(0,1)$, so the bound and convergence follow from \cref{lem:fixedpoint} and \ref{it:conv-fixed}.
For the determinant, the identity $\Xi'=(V'/V)\,\Xi$ from \cref{prop:local}\ref{it:basis-wronskian} uses only the system in Equation~\eqref{eq:wchi-recall} and holds for every $\delta\in[0,\delta_1]$.
For $\delta>0$ it holds on all of $I_j$, because $V_\delta$ has no zero there.
For $\delta=0$ it holds on $I_j\setminus\{z_j\}$, where $\Xi_0/V_0$ is constant on each side of $z_j$, and the two one-sided constants coincide because both equal $\lim_{z\to z_j}\Xi_0/V_0=\varkappa_0$ by \cref{prop:local}\ref{it:basis-wronskian}.
Hence $\Xi_\delta=\varkappa_\delta V_\delta$ on $I_j$ with $\varkappa_\delta$ constant, and $\det\mathcal E^{\pm}_{\delta}=\Xi_\delta(z_j\pm\varsigma)$ because the columns of $\mathcal E^{\pm}_{\delta}$ are the values of the two basis solutions at $z_j\pm\varsigma$.
The endpoint computation in the proof of \cref{prop:local}\ref{it:basis-wronskian} uses only the representation in Equation~\eqref{eq:rep}, which the fixed points satisfy for every $\delta\in[0,\delta_1]$.
For the fixed point $(w^{(2)}_\delta,\chibar^{(2)}_\delta)$ with $(\alpha,\beta)=(0,1)$, it gives $\varkappa_\delta=1+\int_{z_j}^{z_j+\varsigma}E\,\rhobar k^2V_\delta\,w^{(2)}_\delta\,\dd s$.
Since $|V_\delta|\le|\tau|+\delta\le(\max_{I_j}|U'|+m_U)\,\varsigma$ on $I_j$ by $\delta_1\le m_U\varsigma$, and $|w^{(2)}_\delta|\le C_\ast$ by \cref{lem:fixedpoint}, this gives $|\varkappa_\delta-1|\le C\varsigma^2$.
Writing $\varkappa_\delta=\Xi_\delta(z_j+\varsigma)/V_\delta(z_j+\varsigma)$ and using $|V_\delta(z_j+\varsigma)|\ge|\tau(z_j+\varsigma)|\ge m_U\varsigma$, $|V_\delta-V_0|=\delta$, and the endpoint convergence of the basis solutions from \ref{it:conv-fixed} gives $|\varkappa_\delta-\varkappa_0|\le C\delta(1+|\log(\delta/V_\ast)|)/(m_U\varsigma)$.
For $\varsigma$ small, $C\varsigma^2\le\tfrac12$, so $|\varkappa_\delta|\ge\tfrac12$ for all $\delta\in[0,\delta_1]$, $|\det\mathcal E^{\pm}_{\delta}|\ge\tfrac12 m_U\varsigma$, and the inverse bound follows from the formula for the inverse of a $2\times2$ matrix, which swaps the diagonal entries, changes the sign of the other two entries, and divides by the determinant.

For \ref{it:conv-transfer}, recall that for $\delta>0$ the fixed points are the solutions of Equation~\eqref{eq:wchi-recall} (\cref{lem:representation}), so $T_\delta$ carries the values at $z_j+\varsigma$ to the values at $z_j-\varsigma$.
For every $\delta\in[0,\delta_1]$ it is defined and bounded with bounded inverse by \ref{it:conv-matrices}.
Convergence follows from
\begin{gather*}
T_\delta-T_0
=(\mathcal E^{-}_{\delta}-\mathcal E^{-}_{0})(\mathcal E^{+}_{\delta})^{-1}
+\mathcal E^{-}_{0}
\bigl[(\mathcal E^{+}_{\delta})^{-1}-(\mathcal E^{+}_{0})^{-1}\bigr],\\
(\mathcal E^{+}_{\delta})^{-1}-(\mathcal E^{+}_{0})^{-1}
=(\mathcal E^{+}_{\delta})^{-1}
(\mathcal E^{+}_{0}-\mathcal E^{+}_{\delta})
(\mathcal E^{+}_{0})^{-1}.
\end{gather*}
The $C^1$-dependence of $T_0$ and $w_0(z_j)$ on $c_r$ needs a change of variables that follows the critical level.
The critical level $z_j(c_r)$ and the interval $I_j(c_r)$ move as $c_r$ changes, and differentiating $\tau\operatorname{Log}(\tau-\im0)$ with respect to $c_r$ at fixed $z$ gives a term that is singular at $z_j$.
We therefore write $z=z_j(c_r)+x$ and differentiate at fixed $x$.
In this variable, $\tilde\tau(x)=U(z_j+x)-c_r$ satisfies $\partial_{c_r}\tilde\tau=U'(z_j+x)/U'(z_j)-1=O(x)$, which cancels the singularity.
The details are in \cref{lem:pullback} of \cref{app:proofs}.
\end{proof}

\subsection{The limiting solution and the boundary value of \texorpdfstring{$\Za$}{Za}}

We now assemble the local solutions across all critical levels into a global limit and identify the boundary value of $\Za$.

\begin{proposition}\label{prop:limit-solutions}
Let $c_r$ be a regular value as above and define the \emph{limit problem} on $[0,\ha]$ as the system in Equation~\eqref{eq:wchi-recall} with $\delta=0$ on $[0,\ha]\setminus\{z_1,\dots,z_m\}$, continued across each $z_j$ by the representation in Equation~\eqref{eq:rep}, equivalently by the basis of \cref{prop:local}.
Then the solution $(w,\chibar)(\cdot;k,c_r+\im\delta)$ with values $(0,\Abar_\delta(\ha)V_\delta(\ha))$ at $z=\ha$ converges as $\delta\downarrow0$, uniformly on $[0,\ha]$ at rate $O\bigl(\delta(1+|\log(\delta/V_\ast)|)\bigr)$, to the unique solution $(w_\ast,\chibar_\ast)$ of the limit problem with the limiting values $(0,\Abar_0(\ha)V_0(\ha))$ at the top boundary.
The limit $w_\ast$ is continuous on $[0,\ha]$ and $C^1$ away from the critical levels, and at each $z_j$ its derivative has the logarithmic behavior and the jump in the renormalized derivative in Equations~\eqref{eq:renorm-deriv}--\eqref{eq:wjump} of \cref{prop:local}\ref{it:basis-log}.
The function $\chibar_\ast$ is Lipschitz on $[0,\ha]$, and $(w_\ast,\chibar_\ast)(0)$, $w_\ast(z_j)$ depend $C^1$ on $c_r$.
\end{proposition}

\begin{proof}
Partition $[0,\ha]$ into the local intervals $I_j$ and the closed intervals between them that contain no critical level.
On the latter intervals the coefficients of Equation~\eqref{eq:wchi-recall} are $C^2$ uniformly in $\delta\in[0,\delta_1]$ and Lipschitz in $\delta$, so the solution maps are invertible linear maps that converge at rate $O(\delta)$ by Gr\"onwall's inequality and depend $C^1$ on $c_r$.
The values at the top boundary also depend on $\delta$, and $Y_\delta(\ha):=(0,\Abar_\delta(\ha)V_\delta(\ha))$ satisfies $Y_\delta(\ha)-Y_0(\ha)=O(\delta)$ by uniform subsonicity.
Since the solution is obtained from $Y_\delta(\ha)$ by composing finitely many uniformly bounded maps, as described below, this $O(\delta)$ difference does not change the stated rate.
Across each $I_j$, \cref{lem:fixedpoint} identifies the solution with a fixed point, whose constants $(\alpha,\beta)$ depend linearly and invertibly on the values at $z_j+\varsigma$, with bounds uniform in $\delta$ by \cref{prop:convergence}\ref{it:conv-matrices}.
The same proposition gives convergence of $T_\delta$ at rate $O(\delta\log(1/\delta))$, together with $C^1$-dependence on $c_r$.
Composing the finitely many maps from $\ha$ down to $0$ proves the claims.
The interior description of $w_\ast$ is \cref{prop:local}\ref{it:basis-reg}--\ref{it:basis-log}.
\end{proof}

Recall the flux $\Jf=\Imc(\qbar\,\xi^{\ast})$ of Equation~\eqref{eq:flux-def}, the Wronskian-type flux $\Wc=\Imc(\Abar w'w^{\ast})$ of Equation~\eqref{eq:W-def}, and the relation $\Jf=\Wc/k^2$ for real $c$ from Equation~\eqref{eq:J-W}.
The theorem below is the compressible form of Miles's result.
The imaginary part of the boundary value is a sum of critical-level contributions, each proportional to $Q(z_j)|w_\ast(z_j)|^2/|U'(z_j)|$.

\begin{theorem}
\label{thm:limabs}
Let $c_r$ be a regular value of $U$ satisfying Hypotheses \ref{hyp:critical}--\ref{hyp:nonres} with $c_k$ replaced by $c_r$.
Then the following hold.
\begin{enumerate}[label=\textup{(\roman*)}]
\item\label{it:limabs-Z} $\Za(k,\cdot)$ extends continuously to the closed upper half-plane locally at $c_r$, that is, the limit
  \begin{equation}
  \Za^{+}(k,c_r):=\lim_{\delta\downarrow0}\Za(k,c_r+\im\delta)
  \end{equation}
  exists, is finite, is attained at rate $O\bigl(\delta(1+|\log(\delta/V_\ast)|)\bigr)$, written $O(\delta\log(1/\delta))$ in the convention of \S\ref{subsec:local}, and is $C^1$ in $c_r$ on the set of regular values satisfying Hypothesis \ref{hyp:nonres}.
\item\label{it:limabs-flux} the limiting flux $\Jf_\ast=\Wc_\ast/k^2$ of the solution normalized by $\xi(0)=1$, that is, by $w_\ast(0)=-\im kc_r$, is piecewise constant, vanishes on $(z_m,\ha)$, and jumps at each critical level by
  \begin{equation}\label{eq:Jjump}
  \Jf_\ast(z_j^+)-\Jf_\ast(z_j^-)
  =\frac{\pi}{k^2}\,
  \frac{\rhobar(z_j)\,Q(z_j)}{|U'(z_j)|}\;\bigl|w_\ast(z_j)\bigr|^2 .
  \end{equation}
\item\label{it:limabs-residue} consequently
  \begin{equation}\label{eq:residue}
  \Imc\Za^{+}(k,c_r)
  =\Jf_\ast(0)
  =-\frac{\pi}{k^{2}}\sum_{j=1}^{m}
  \frac{\rhobar(z_j)\,Q(z_j)}{|U'(z_j)|}\bigl|w_\ast(z_j)\bigr|^{2}
  =-\pi c_r^{2}\sum_{j=1}^{m}
  \frac{\rhobar(z_j)\,Q(z_j)}{|U'(z_j)|}\bigl|\hat y(z_j)\bigr|^{2},
  \end{equation}
  where $\hat y:=w_\ast/w_\ast(0)$ is the interface-normalized limiting solution.
\end{enumerate}
\end{theorem}

\begin{proof}
For \ref{it:limabs-Z}, by Hypothesis \ref{hyp:nonres}, $w_\ast(0)\neq0$, so by \cref{prop:limit-solutions} $|w(0;k,c_r+\im\delta)|\ge \tfrac12|w_\ast(0)|$ for small $\delta$.
Equation~\eqref{eq:Za-formula} and the uniform convergence of $(w,\chibar)(0)$ give existence and the rate, and together with the $C^1$-dependence of $(w_\ast,\chibar_\ast)(0)$ on $c_r$ from \cref{prop:limit-solutions} they give the smoothness in $c_r$.

For \ref{it:limabs-flux}, on each interval free of critical levels the limit solution solves a real-coefficient system, so $\Wc=\Imc(\Abar w_\ast'w_\ast^{\ast})$ is constant there because $\Wc'=\Imc\bigl((\Abar w')'w^{\ast}\bigr)+\Imc\bigl(\Abar|w'|^2\bigr) =\Imc\bigl((\rhobar k^2+G/V)|w|^2\bigr)+0=0$, all coefficients being real.
This is \cref{lem:flux}\ref{it:flux-real} combined with Equation~\eqref{eq:J-W}.
At the top boundary, $w_\ast(\ha)=0$ gives $\Wc(\ha)=0$, so $\Wc_\ast\equiv0$ on $(z_m,\ha)$.
The jumps are \cref{prop:local}\ref{it:quadratic}.
The normalization rescales the solution by a complex constant, which multiplies $\Wc$ and $u_1$ by the same positive factor, so Equation~\eqref{eq:Wjump} still holds after normalization, and $\Jf=\Wc/k^2$ by Equation~\eqref{eq:J-W}.

For \ref{it:limabs-residue}, write $[\Jf_\ast]_{z_j}:=\Jf_\ast(z_j^+)-\Jf_\ast(z_j^-)$ for the jump at $z_j$.
Since $\Jf_\ast$ vanishes on $(z_m,\ha)$ and is constant between critical levels, summing the jumps from $\ha$ down to $0$ gives $\Jf_\ast(0)=-\sum_j[\Jf_\ast]_{z_j}$.
The relation $\Imc\Za^+=\Imc\qbar^+(0)=\Jf_\ast(0)$, where $\qbar^+(0)$ is the limit of $\qbar(0)$ as $\delta\downarrow0$, holds because $\xi(0)=1$ is real and the hydrostatic term in Equation~\eqref{eq:Za-def} is real.
The final expression substitutes $|w_\ast(z_j)|^2=k^2c_r^2|\hat y(z_j)|^2$.
\end{proof}

\begin{remark}\label{rem:EP}
Equation~\eqref{eq:Jjump} gives the signed change of wave-momentum flux across each critical level.
When the sign condition for Miles growth holds, the flux is absorbed at the critical levels in proportion to the specific-vorticity gradient.
Equation~\eqref{eq:residue} shows that the air-pressure component in phase with the interface slope, equivalently shifted by a quarter wavelength from the interface displacement, is determined by the total critical-layer flux change.
This is the rigorous, compressible form of Miles's resonance argument \cite{Miles1957,Lighthill1962} and of the Eliassen--Palm theory, in which steady waves without dissipation leave the mean flow unchanged except where their momentum flux is absorbed, as at critical levels \cite{EliassenPalm1961,BookerBretherton1967,AndrewsMcIntyre1978}.
For $N^2=0$ there is no internal-gravity-wave radiation, so the flux changes only across the critical levels and vanishes above the highest one because $w(\ha)=0$.
\end{remark}

\begin{remark}\label{rem:resonance}
The map $c_r\mapsto w_\ast(0;c_r)$ is $C^1$ on intervals of regular values with the same number of critical levels (\cref{prop:limit-solutions}).
Hypothesis \ref{hyp:nonres} requires this trace to be nonzero.
It follows from \cref{lem:positivity} under the sign condition of \cref{cor:sufficient}, and it can otherwise be checked for a prescribed profile.
A zero of $w_\ast(0;c_r)$ means that the atmospheric boundary-value problem with the interface held at rest has a nontrivial solution at $c_r$.
When the equation for $w$ has solutions that oscillate in height, such zeros include resonances of the air layer between the interface held at rest and the top boundary, for example acoustic standing waves when the flow relative to the wave is supersonic.
In the uniform-wind example of \S\ref{subsec:KH}, the corresponding solutions in the uniformly subsonic regime are evanescent, that is, they do not oscillate in height, and they decay exponentially with height.
For a general wind profile, such solutions need not decay monotonically or exponentially with height.
Such a zero is a property of the limiting air problem alone.
\cref{prop:reduction} excludes coupled normal modes with $\hat\zeta=0$ only for $\Imc c\neq0$, and we do not apply it at the real phase speed $c_r$.
At $c_r=c_k$, such a zero means that $c_k$ is a zero of both $\Dw(k,\cdot)$ and $w_\ast(0;\cdot)$, the limit of the second factor of $\Edet(k,c,0)=\Dw(k,c)\,w(0;k,c)$, as described in \cref{rem:determinant}.
\end{remark}

The proof of the instability theorem in \S\ref{sec:instability} requires $\Za$ on the set $\mathbb D^+_{r_0}$ defined below, a closed neighborhood of $c_k$ in the closed upper half-plane.
Its boundary values at real phase speeds alone do not suffice.
The next proposition supplies the extension and the quantitative non-resonance bound.

\begin{proposition}
\label{prop:Zhalfdisc}
Assume Hypotheses \ref{hyp:reg}--\ref{hyp:nonres}.
There exist $r_0>0$ and $C,M_Z>0$ such that, with $\mathbb D^+_{r_0}:=\{c\in\C:\,|c-c_k|\le r_0,\ \Imc c\ge0\}$, the following hold.
\begin{enumerate}[label=\textup{(\roman*)}]
\item\label{it:Z-nopole} With $m_0:=\min_{|c_r-c_k|\le r_0}|w_\ast(0;c_r)|>0$, every $c\in\mathbb D^+_{r_0}$ satisfies $|w(0;k,c)|\ge m_0/2$.
Hence $\Za(k,\cdot)$ is holomorphic in the interior of $\mathbb D^+_{r_0}$, extends continuously to $\mathbb D^+_{r_0}$ with boundary value $\Za^{+}$ on the real interval $[c_k-r_0,c_k+r_0]$, and $|\Za(k,c)|\le M_Z$ there.
\item\label{it:Z-modulus} For all $c\in\mathbb D^+_{r_0}$,
\begin{equation}
\bigl|\Za(k,c)-\Za^{+}(k,c_k)\bigr|\le\omega\bigl(|c-c_k|\bigr),
\qquad
\omega(r):=C\,r\bigl(1+\bigl|\log(r/V_\ast)\bigr|\bigr).
\end{equation}
\item\label{it:Z-deriv} For $c$ in the interior of $\mathbb D^+_{r_0}$ with $|c-c_k|\le r_0/2$,
\begin{equation}
\bigl|\partial_c\Za(k,c)\bigr|
\le C\bigl(1+\bigl|\log(\Imc c/V_\ast)\bigr|\bigr).
\end{equation}
\end{enumerate}
\end{proposition}

\begin{proof}
For \ref{it:Z-nopole}, by Hypothesis \ref{hyp:critical} and the implicit function theorem there is $r_0>0$ such that every $c_r\in[c_k-r_0,c_k+r_0]$ is a regular value of $U$ with exactly $m$ critical levels $z_j(c_r)$ depending $C^1$ on $c_r$, and Hypotheses \ref{hyp:subsonic} and \ref{hyp:nonres} persist on this segment.
\cref{prop:limit-solutions} gives continuity of $w_\ast(0;\cdot)$, so $m_0>0$.
By \cref{prop:limit-solutions}, uniformly for $c_r$ in the segment, $|w(0;k,c_r+\im\delta)-w_\ast(0;c_r)|\le C\delta\bigl(1+|\log(\delta/V_\ast)|\bigr)$ for $\delta\in(0,\delta_1]$, where the uniformity in $c_r$ is part of \cref{lem:pullback}.
Shrinking $r_0$ so that $r_0\le\delta_1$ and $Cr_0\bigl(1+|\log(r_0/V_\ast)|\bigr)\le m_0/2$, every $c\in\mathbb D^+_{r_0}$ satisfies $|w(0;k,c)|\ge m_0/2$, since every such $c$ has $\Imc c\le r_0\le\delta_1$.
Poles of $\Za(k,\cdot)$ are zeros of $w(0;k,\cdot)$ (\cref{lem:regular-solve}), so $\Za$ is holomorphic in the interior of $\mathbb D^+_{r_0}$.
The same uniform convergence of $(w,\chibar)(0)$, together with the lower bound $|w(0;k,c)|\ge m_0/2$ for the denominator in Equation~\eqref{eq:Za-formula}, shows $|\Za(k,c_r+\im\delta)-\Za^{+}(k,c_r)|\le C\delta\bigl(1+|\log(\delta/V_\ast)|\bigr)$ uniformly in $c_r$.
Since $\Za^{+}$ is $C^1$ on the segment by \cref{thm:limabs}\ref{it:limabs-Z}, this gives continuity on the compact set $\mathbb D^+_{r_0}$, hence the bound $M_Z$.

For \ref{it:Z-modulus}, let $c=c_r+\im\delta\in\mathbb D^+_{r_0}$.
Then
\begin{align*}
\bigl|\Za(k,c)-\Za^{+}(k,c_k)\bigr|
&\le\bigl|\Za(k,c)-\Za^{+}(k,c_r)\bigr|
+\bigl|\Za^{+}(k,c_r)-\Za^{+}(k,c_k)\bigr|\\
&\le C\delta\bigl(1+|\log(\delta/V_\ast)|\bigr)+C|c_r-c_k| ,
\end{align*}
the second term by the uniform bound $|\partial_{c_r}\Za^{+}|\le C$ from \cref{lem:pullback}.
Since $\delta\le|c-c_k|$, $|c_r-c_k|\le|c-c_k|$, and $r\mapsto r(1+|\log(r/V_\ast)|)$ is increasing on $(0,V_\ast]$, both terms are dominated by $\omega(|c-c_k|)$ after shrinking $r_0\le V_\ast$.

For \ref{it:Z-deriv}, let $c=c_r+\im\delta$ lie in the interior of $\mathbb D^+_{r_0}$ with $|c-c_k|\le r_0/2$ and set $r_1:=\delta/2$, so the circle $|c'-c|=r_1$ lies in that interior ($r_1\le r_0/4$).
For $c'=c_r'+\im\delta'$ on that circle, $\delta'\in[\delta/2,3\delta/2]$ and $|c_r'-c_r|\le\delta/2$, so by the split used in \ref{it:Z-modulus},
\begin{align*}
|\Za(k,c')-\Za(k,c)|
&\le C\delta'\bigl(1+|\log(\delta'/V_\ast)|\bigr)
+C|c_r'-c_r|
+C\delta\bigl(1+|\log(\delta/V_\ast)|\bigr)\\
&\le C'\delta\bigl(1+|\log(\delta/V_\ast)|\bigr).
\end{align*}
Cauchy's formula for the derivative, $\partial_c\Za(k,c)=\frac1{2\pi\im}\oint_{|c'-c|=r_1} \frac{\Za(k,c')-\Za(k,c)}{(c'-c)^2}\,\dd c'$, then gives $|\partial_c\Za(k,c)|\le r_1^{-1}\sup_{|c'-c|=r_1}|\Za(k,c')-\Za(k,c)| \le C\bigl(1+|\log(\delta/V_\ast)|\bigr)$.
\end{proof}

\subsection{Nonvanishing of the limiting solution at the top critical levels}

Equation~\eqref{eq:residue} yields instability only if some $|w_\ast(z_j)|>0$.
Following the idea of \cite[Lemma~4.5]{BSWZ2016}, a monotonicity argument guarantees this at the top critical level, and at the next one down when $Q(z_m)\neq0$.
The new step in the compressible case is the exact cancellation in \cref{lem:H-monotone}.
For a solution of the limit problem and real $c$, introduce on each critical-level-free interval the quasi-derivative $w^{[1]}:=\Abar w'$, which is the flux variable in the divergence form of Equation~\eqref{eq:bvp}, and the real quadratic quantities
\begin{equation}\label{eq:uW-def}
u_2:=\Rec\bigl(w^{[1]}w^{\ast}\bigr),\qquad
u_3:=\bigl|w^{[1]}\bigr|^2 .
\end{equation}
Together with $u_1=|w|^2$ from \cref{prop:local}\ref{it:quadratic} and $\Wc=\Imc(w^{[1]}w^{\ast})$ from Equation~\eqref{eq:W-def}, they satisfy the pointwise identity
\begin{equation}\label{eq:conservation}
u_2^2+\Wc^2\equiv u_1u_3,
\end{equation}
and, since the coefficients are real there,
\begin{equation}\label{eq:u-system-real}
u_1'=\frac{2u_2}{\Abar},
\qquad
u_2'=\Bigl(\rhobar k^2+\frac GV\Bigr)u_1+\frac{u_3}{\Abar},
\qquad
u_3'=2\Bigl(\rhobar k^2+\frac GV\Bigr)u_2,
\qquad
\Wc'=0 .
\end{equation}

\begin{lemma}\label{lem:H-monotone}
On an interval free of critical levels define
\begin{equation}\label{eq:H-def}
H:=u_2-\frac{\Abar\,(U'+bV)}{V}\,u_1
\qquad(c=c_r\in\R,\ V=U-c_r).
\end{equation}
Then along solutions of Equation~\eqref{eq:u-system-real},
\begin{equation}\label{eq:H-derivative}
\begin{aligned}
H'
&=\rhobar k^2u_1
+\frac{u_3}{\Abar}
+\Abar\Bigl(\frac{U'+bV}{V}\Bigr)^2u_1
-2\,\frac{U'+bV}{V}\,u_2\\
&\ge
\rhobar k^2u_1
+\Bigl(\sqrt{\tfrac{u_3}{\Abar}}
-\sqrt{\Abar}\,\Bigl|\tfrac{U'+bV}{V}\Bigr|\sqrt{u_1}\Bigr)^{2}
\;\ge\;\rhobar k^2 u_1\;\ge\;0.
\end{aligned}
\end{equation}
Equality holds identically only for the trivial solution.
In particular $H$ is nondecreasing between critical levels.
\end{lemma}

\begin{proof}
When we differentiate Equation~\eqref{eq:H-def} using Equation~\eqref{eq:u-system-real}, the singular and $b$-dependent terms cancel identically by the definition of $G$ in Equation~\eqref{eq:SL}, since
\begin{equation}
\frac GV-\Bigl(\frac{\Abar(U'+bV)}{V}\Bigr)'
=\frac{b\Abar(U'+bV)}{V}+\frac{\Abar(U'+bV)U'}{V^2},
\qquad
\frac{\Abar(U'+bV)}{V}\,u_1'
=\frac{2(U'+bV)}{V}u_2,
\end{equation}
and regrouping yields the displayed equality.
The inequality is Cauchy--Schwarz, $|u_2|\le\sqrt{u_1u_3}$ from Equation~\eqref{eq:conservation}, after completing the square.
Positivity of $\Abar$ uses Hypothesis \ref{hyp:subsonic}.
\end{proof}

\begin{lemma}\label{lem:positivity}
Let $(w_\ast,\chibar_\ast)$ be the limit solution of \cref{prop:limit-solutions}, with $w(\ha)=0$ and $\chibar(\ha)=\Abar(\ha)V(\ha)$, so that $u_1(\ha)=u_2(\ha)=\Wc(\ha)=0$ and $u_3(\ha)>0$.
Then $u_1\ge0$ vanishes at most at isolated points, and $u_1(z_m)>0$.
If $m\ge2$ and $Q(z_m)\neq0$, then also $u_1(z_{m-1})>0$.
Moreover, if all $Q(z_j)$ have the same sign and $Q(z_m)\neq0$, then $u_1>0$ on $[0,z_m)$.
In particular $u_1(0)>0$, i.e. Hypothesis \ref{hyp:nonres} holds automatically, and $\Wc(0)\neq0$.
\end{lemma}

\begin{proof}
Consider the limiting solution at the real phase speed $c_r$.

We first show that $u_1$ vanishes at most at isolated points on any critical-level-free interval.
If $u_1(z^\flat)=0$ at an interior point $z^\flat$, then $w(z^\flat)=0$, and $w'(z^\flat)\neq0$, for otherwise $\chibar(z^\flat)=\Abar\bigl[Vw'-(U'+bV)w\bigr](z^\flat)=0$ as well and the solution would be trivial by uniqueness for the regular system in Equation~\eqref{eq:wchi-recall}.
Hence $z^\flat$ is an isolated zero of $w$.
If $w_\ast(z_j)=0$ at a critical level, \cref{prop:local}\ref{it:quadratic} shows that the nontrivial local solution is a nonzero multiple of the regular basis solution $w^{(1)}=VE\,(1+O(\varsigma))$.
Since $U'(z_j)\neq0$, $z_j$ is a simple, hence isolated, zero as well.

To prove $u_1(z_m)>0$, suppose $u_1(z_m)=0$.
By \cref{prop:local}\ref{it:quadratic}, on $I_m$ the solution is a complex multiple of the real regular solution $(w^{(1)},\chibar^{(1)})$.
The quantities $(u_1,u_2,u_3,\Wc)$ extend continuously, and in fact $C^1$, across $z_m$, and $\Wc\equiv0$ on $I_m$.
The function $H$ of Equation~\eqref{eq:H-def} also extends continuously to $z_m$, because in $H=u_2-(\Abar(U'+bV)/V)\,u_1$ the pole of $1/V$ is canceled by $u_1=|\alpha w^{(1)}|^2=O((z-z_m)^2)$.
Since $u_2(z_m)=\Rec\bigl(w^{[1]}w^{\ast}\bigr)(z_m)=0$ as well, we get $H(z_m)=0$.
On the other hand $H(\ha)=0$, from $u_1(\ha)=u_2(\ha)=0$, and by \cref{lem:H-monotone} together with the isolated-zeros property, $H(\ha)-H(z_m)\ge k^2\int_{z_m}^{\ha}\rhobar\,u_1>0$ for a nontrivial solution.
This contradicts $H(z_m)=H(\ha)=0$.
So $u_1(z_m)>0$.

On $(z_m,\ha)$ the values at $\ha$ and the coefficients are real, so the solution is real there and $\Wc\equiv0$ on $(z_m,\ha)$.
By Equation~\eqref{eq:Wjump},
\begin{equation}
\Wc\big|_{(z_{m-1},z_m)}
=\Wc(z_m^-)
=-\pi\,\frac{\rhobar(z_m)Q(z_m)}{|U'(z_m)|}\,u_1(z_m)\;\neq\;0
\quad\text{when }Q(z_m)\neq0 ,
\end{equation}
using $\Wc'=0$ between levels, as in Equation~\eqref{eq:u-system-real}.
Since $u_1(z)=0$ forces $\Wc(z)=0$ at that point, as $\Wc=\Imc(w^{[1]}w^{\ast})$, we get $u_1>0$ on all of $(z_{m-1},z_m)$.
Moreover $u_1(z_{m-1})=0$ is impossible because, by \cref{prop:local}\ref{it:quadratic} applied at $z_{m-1}$, it would make the solution locally a multiple of the real regular basis solution and hence $\Wc(z_{m-1}^+)=0$, contradicting $\Wc(z_{m-1}^+)=\Wc(z_m^-)\neq0$.
Thus $u_1(z_{m-1})>0$.

If all $Q(z_j)$ have the same sign and $Q(z_m)\neq0$, the jumps in Equation~\eqref{eq:Wjump} accumulate monotonically as $z$ decreases.
The flux $\Wc$ is constant between levels and $\Wc(z_j^-)=\Wc(z_j^+)-\pi(\rhobar Q/|U'|)(z_j)u_1(z_j)$ with every increment of the same sign (or zero), and the first increment, at $z_m$, is nonzero because $Q(z_m)\neq0$ and $u_1(z_m)>0$.
Hence $\Wc\neq0$ on $[0,z_m)\setminus\{z_j\}$, which as above forces $u_1>0$ there and at each $z_j$, $j<m$, and at $z=0$.
The inequality $u_1(0)>0$ is Hypothesis \ref{hyp:nonres}, and $\Wc(0)\neq0$.
When $Q(z_m)=0$ the flux vanishes identically on $(z_{m-1},\ha)$ and no conclusion about $u_1(z_{m-1})$ can be drawn from this argument.
This is why the assumption $Q(z_m)\neq0$ is needed here.
\end{proof}
\section{The compressible Miles instability}
\label{sec:instability}

In this section, we prove the instability theorem.
We first show that without a critical layer the root of the dispersion relation near $c_k$ stays real, so that the corresponding mode is neutral (\cref{thm:necessity}).
This isolates the critical layer as the only source of growth near $c_k$.
With a critical layer, we compare the dispersion function with its linearization at $c_k$ on a neighborhood of $c_k$ in the upper half-plane.
A boundary version of Rouch\'e's theorem (\cref{lem:rouche}), applied with the boundary value of $\Za$ from \cref{thm:limabs}, yields the unstable root, its expansion, and its growth rate (\cref{thm:main}).
The positivity result of \cref{lem:positivity} then turns the residue formula in Equation~\eqref{eq:residue} into the sign criterion (\cref{cor:sufficient,cor:monotone}).

\subsection{Neutrality and persistence without a critical layer}

Assume that $c_k$ lies outside the range of $U$.
Then Equation~\eqref{eq:Xform}, the self-adjoint form of the displacement equation, is coercive near $c_k$.
Coercive means that its quadratic form is bounded below by positive multiples of the $L^2$ norms of $X'$ and $X$ (\cref{lem:coercive}).
This excludes zeros of the interface trace $w(0;k,c)$ and gives $|\Imc\Za(k,c)|\le C|\Imc c|$.
Consequently the root of the dispersion relation near $c_k$ remains real and simple.
This extends \cite[Lemma 3.1]{BSWZ2016} to compressible air.

\begin{lemma}\label{lem:coercive}
Assume Hypotheses \ref{hyp:reg}, \ref{hyp:noslip}, and \ref{hyp:subsonic}.
Suppose that $c_k\notin U([0,\ha])$ and set $d:=\operatorname{dist}\bigl(c_k,U([0,\ha])\bigr)>0$, the distance from $c_k$ to the range $U([0,\ha])$ of $U$.
There are $r_\ast\in(0,\min\{d/4,c_k/2\})$ and $m_F,M_F,M_F',C_Z>0$, depending only on the profiles, $d$, $k$, and $\delta_0$, such that for all $c\in\C$ with $|c-c_k|\le r_\ast$,
\begin{equation}\label{eq:F-bounds}
\begin{gathered}
\Rec F\ge m_F>0,\qquad
\Rec V^2\ge\tfrac12d^2,\qquad
|F|+|V^2/\rhobar|\le M_F,\\
|\Imc F|+k^2|\Imc(V^2/\rhobar)|\le M_F'\,|\Imc c| .
\end{gathered}
\end{equation}
Moreover, the following hold.
\begin{enumerate}[label=\textup{(\roman*)}]
\item\label{it:coercive-nozero} The top-normalized solution satisfies $w(0;k,c)\neq0$.
In particular Hypothesis \ref{hyp:nonres} holds on the real interval $[c_k-r_\ast,c_k+r_\ast]$, and $\Za(k,\cdot)$ is holomorphic on the whole neighborhood $\{|c-c_k|<r_\ast\}$ of $c_k$ in $\C$.
\item\label{it:coercive-real} $\Za(k,c)\in\R$ for real $c$ with $|c-c_k|\le r_\ast$.
\item\label{it:coercive-slope} $|\Imc\Za(k,c)|\le C_Z\,|\Imc c|$.
\end{enumerate}
\end{lemma}

\begin{proof}
\emph{Bounds in Equation~\eqref{eq:F-bounds}.} For $|c-c_k|\le r_\ast\le d/4$ we have $|U-\Rec c|\ge3d/4$ and $|\Imc c|\le r_\ast$, so $\Rec V^2=(U-\Rec c)^2-(\Imc c)^2\ge\tfrac9{16}d^2-\tfrac1{16}d^2 =\tfrac12d^2$, and with $F=a^2V^2/(\rhobar(a^2-V^2))$,
\begin{equation}
\Rec\bigl[V^2\,(a^2-V^2)^{\ast}\bigr]
=a^2\Rec V^2-|V|^4
\ge a^2\bigl[(U-\Rec c)^2-r_\ast^2\bigr]
-\bigl[(U-\Rec c)^2+r_\ast^2\bigr]^2 ,
\end{equation}
and since $(U-\Rec c)^2+r_\ast^2\le(1-\tfrac12\delta_0)a^2$ for $r_\ast$ small by Hypothesis \ref{hyp:subsonic}, the right-hand side is bounded below by $a^2\bigl[\tfrac12\delta_0(U-\Rec c)^2-2r_\ast^2\bigr] \ge a^2\bigl[\tfrac{9}{32}\delta_0d^2-2r_\ast^2\bigr]>0$ for $r_\ast$ small.
Dividing by $\rhobar|a^2-V^2|^2$ gives $\Rec F\ge m_F>0$.
The upper bounds follow from compactness of $[0,\ha]\times\{|c-c_k|\le r_\ast\}$ and the positive lower bounds for $\rhobar$ and $|a^2-V^2|$.
The last bound in Equation~\eqref{eq:F-bounds} holds because $\Imc V^2=-2(U-\Rec c)\Imc c$ and $\Imc F=-2\Imc c\,(U-\Rec c)\,a^4/(\rhobar|a^2-V^2|^2)$.
Note also $\Rec(V^2/\rhobar)=\Rec V^2/\rhobar\ge d^2/(2\sup\rhobar)=:m_V>0$.
In particular $V$, $a^2-V^2$, and $F$ have no zeros on $[0,\ha]\times\{|c-c_k|\le r_\ast\}$, so the coefficients of the system in Equation~\eqref{eq:wchi} are continuous in $z$ and holomorphic in $c$ on $\{|c-c_k|<r_\ast\}$, including at real phase speeds.
Hence the top-normalized solution and $\Za$ extend holomorphically there whenever $w(0;k,\cdot)\neq0$.

For \ref{it:coercive-nozero}, suppose $w(0;k,c)=0$.
Then $X=\rhobar\,w/(\im kV)$ solves $(FX')'=k^2(V^2/\rhobar)X$ with $X(0)=X(\ha)=0$.
The reduction of \cref{prop:SLform}\ref{it:X} is valid for the present $c$, because its computation uses only $V\neq0$ and $a^2-V^2\neq0$, which hold for every $c$ with $|c-c_k|\le r_\ast$, including the real phase speeds, for which there is no critical level.
When we test the equation with $X^{\ast}$ over $(0,\ha)$ and integrate by parts, we obtain
\begin{equation}
\int_0^{\ha}\Bigl[F|X'|^2+k^2\tfrac{V^2}{\rhobar}|X|^2\Bigr]=0 ,
\end{equation}
whose real part is at least $m_F\|X'\|_{L^2}^2+k^2m_V\|X\|_{L^2}^2$ by Equation~\eqref{eq:F-bounds}, so $X\equiv0$, hence $(w,\chibar)\equiv0$, contradicting the values at $z=\ha$ given in Equation~\eqref{eq:canonical}.
For real $c_r$ in the interval $[c_k-r_\ast,c_k+r_\ast]$ there are no critical levels, so the limit solution of \cref{prop:limit-solutions} is the top-normalized solution at $c=c_r$ itself and Hypothesis \ref{hyp:nonres} follows.

For \ref{it:coercive-real}, when $c$ is real the coefficients of the system in Equation~\eqref{eq:wchi} and the values at $z=\ha$ given in Equation~\eqref{eq:canonical} are real, so $(w,\chibar)$ is real-valued and the quotient in Equation~\eqref{eq:Za-formula} is real.

For \ref{it:coercive-slope}, normalize $\xi(0)=1$, which is possible by \ref{it:coercive-nozero}, so $X(0)=1$, $X(\ha)=0$.
Write $X=\Upsilon+\tilde X$ with the lift $\Upsilon(z):=(\ha-z)/\ha$, a function with the same boundary values as $X$, and $\tilde X(0)=\tilde X(\ha)=0$.
When we test the equation with $\tilde X^{\ast}$ and integrate by parts, we obtain
\begin{equation}
\int_0^{\ha}\Bigl[F|\tilde X'|^2+k^2\tfrac{V^2}{\rhobar}|\tilde X|^2\Bigr]
=-\int_0^{\ha}\Bigl[F\Upsilon'(\tilde X')^{\ast}
+k^2\tfrac{V^2}{\rhobar}\Upsilon\tilde X^{\ast}\Bigr].
\end{equation}
Taking real parts, using Equation~\eqref{eq:F-bounds}, and absorbing the right-hand side by Young's inequality gives $\|\tilde X\|_{H^1}\le C$, hence $\|X\|_{H^1}\le C$, with $C$ independent of $c$.
When we test with $X^{\ast}$ and use $X(\ha)=0$ and $X(0)=1$, we obtain
\begin{equation}
\qbar(0)=F(0)X'(0)
=-\int_0^{\ha}
\Bigl[F|X'|^2+k^2\tfrac{V^2}{\rhobar}|X|^2\Bigr]\dd z ,
\end{equation}
so by Equation~\eqref{eq:F-bounds}, $|\Imc\Za(k,c)|=|\Imc\qbar(0)|\le CM_F'\,|\Imc c|=:C_Z|\Imc c|$.
\end{proof}

\begin{theorem}\label{thm:necessity}
Assume Hypotheses \ref{hyp:reg}, \ref{hyp:noslip}, and \ref{hyp:subsonic}, and suppose that $c_k\notin U([0,\ha])$.
With $r_\ast$ as in \cref{lem:coercive}, there is $\eps_k>0$ such that the following hold for all $\eps\in(0,\eps_k)$.
\begin{enumerate}[label=\textup{(\roman*)}]
\item\label{it:nec-neutral} Every normal mode of the linearized problem (\cref{prop:reduction}) with wavenumber $k$ and $|c-c_k|\le r_\ast$ has $\Imc c=0$.
\item\label{it:nec-persist} The dispersion relation $D(k,\cdot,\eps)=0$ has exactly one root $c_\eps$ in the closed neighborhood $\{|c-c_k|\le r_\ast\}$ of $c_k$.
It is real and simple, and it is real-analytic in $\eps$, that is, given by a convergent power series in $\eps$, with
\begin{equation}\label{eq:neutral-branch}
c_\eps=c_k+\eps\,\frac{\Za(k,c_k)}{2kc_k\coth(k\hw)}+O(\eps^{2}),
\qquad
\Za(k,c_k)\in\R.
\end{equation}
Without a critical layer, the root near $c_k$ persists but remains real, so that the corresponding mode is neutral.
\end{enumerate}
\end{theorem}

\begin{proof}
For \ref{it:nec-neutral}, let $(w,c)$ be a normal mode with $|c-c_k|\le r_\ast$.
Suppose $\Imc c\neq0$.
Then \cref{prop:reduction} gives $\hat\zeta\neq0$, and we normalize $\hat\zeta=1$.
When we take imaginary parts of the dispersion relation in Equation~\eqref{eq:dispersion} and use \cref{lem:coercive}\ref{it:coercive-slope}, we obtain
\begin{equation}
2k\coth(k\hw)\,\Rec c\,|\Imc c|
=|\Imc\Dw(k,c)|=\eps\,|\Imc\Za(k,c)|\le\eps\,C_Z\,|\Imc c| .
\end{equation}
Since $\Rec c\ge c_k-r_\ast\ge c_k/2>0$, choosing $\eps_k:=kc_k\coth(k\hw)/C_Z$ forces $\Imc c=0$.

For \ref{it:nec-persist}, Equations~\eqref{eq:Dw-def} and \eqref{eq:ck} give the exact quadratic form $\Dw(k,c)=k\coth(k\hw)\,(c^2-c_k^2)$.
Hence, on the neighborhood $\{|c-c_k|\le r_\ast\}$, the function $D(k,c,\eps)=k\coth(k\hw)\,(c^2-c_k^2)-\eps\Za(k,c)$ is holomorphic in $c$ by \cref{lem:coercive}\ref{it:coercive-nozero} and affine in $\eps$, with $D(k,c_k,0)=0$ and $\partial_cD(k,c_k,0)=2kc_k\coth(k\hw)\neq0$.
The holomorphic implicit function theorem gives $\eps_1>0$ and a unique analytic function $\eps\mapsto c_\eps$, defined for $|\eps|<\eps_1$, with $c_\eps=c_k$ and $\partial_\eps c_\eps=\Za(k,c_k)/(2kc_k\coth(k\hw))$ at $\eps=0$, which gives Equation~\eqref{eq:neutral-branch}.
Uniqueness in the whole neighborhood, and simplicity, follow from Rouch\'e's theorem on the circle $|c-c_k|=r_\ast$.
On this circle, $|\Dw(k,c)|=k\coth(k\hw)\,|c-c_k|\,|c+c_k|\ge k\coth(k\hw)\,r_\ast c_k$, since $|c+c_k|\ge2c_k-r_\ast\ge c_k$, while $\eps|\Za|\le\eps M_\ast$ with $M_\ast:=\max_{|c-c_k|\le r_\ast}|\Za(k,c)|<\infty$.
For $\eps<\eps_k':=k\coth(k\hw)r_\ast c_k/M_\ast$ the functions $D$ and $\Dw$ have equally many zeros in this neighborhood, namely the one simple zero $c_k$ of $\Dw$.
The other root $-c_k$ of $\Dw$ lies outside this neighborhood because $r_\ast\le c_k/2$.
The function $\Za(k,\cdot)$ is real on the real interval $[c_k-r_\ast,c_k+r_\ast]$ (\cref{lem:coercive}\ref{it:coercive-real}) and holomorphic on this neighborhood, so the Schwarz reflection principle gives $\Za(k,c^{\ast})=\Za(k,c)^{\ast}$ and $D(k,c^{\ast},\eps)=D(k,c,\eps)^{\ast}$.
The unique zero in this neighborhood is therefore real.
After shrinking $\eps_k$ to $\min\{\eps_k,\eps_1,\eps_k'\}$, all preceding estimates hold simultaneously.
\end{proof}

\subsection{The unstable root and the sign criterion}

The unstable root is found by comparing the dispersion function with its linearization at $c_k$ on a small set in the upper half-plane.
Because the boundary values of $\Za$ at real phase speeds enter, we need Rouch\'e's theorem for a region whose boundary contains a real interval, on which the functions are only continuous.
We use the following boundary version of Rouch\'e's theorem.

\begin{lemma}
\label{lem:rouche}
Let $c_0\in\R$, $r>0$, and $D^{+}:=\{c\in\C:\,|c-c_0|<r,\ \Imc c>0\}$, with closure $\overline{D^{+}}$ and boundary $\partial D^{+}$.
Let $\varphi,\psi$ be holomorphic in $D^{+}$ and continuous on $\overline{D^{+}}$, with $|\varphi-\psi|<|\psi|$ on $\partial D^{+}$.
Then $\varphi$ and $\psi$ have finitely many zeros in $D^{+}$ and equally many, counted with multiplicity, and neither vanishes on $\partial D^{+}$.
\end{lemma}

\begin{proof}
Set $\mathfrak m:=\min_{\partial D^{+}}\bigl(|\psi|-|\varphi-\psi|\bigr)>0$.
Since $\varphi$ and $\psi$ are uniformly continuous on the compact set $\overline{D^{+}}$, they have a common modulus of continuity $\omega_0$, that is, a nondecreasing function on $[0,\infty)$ with $\omega_0(t)\to0$ as $t\downarrow0$ such that $|\varphi(c)-\varphi(c')|\le\omega_0(|c-c'|)$ and $|\psi(c)-\psi(c')|\le\omega_0(|c-c'|)$ for all $c,c'\in\overline{D^{+}}$.
For $t\in(0,r/2)$ let $D_t:=\{c:\,|c-c_0|<r-t,\ \Imc c>t\}$, a convex set with $\overline{D_t}\subset D^{+}$ whose boundary is a piecewise-smooth simple closed curve.
Every $c\in\overline{D^{+}}\setminus D_t$ lies within distance $t$ of $\partial D^{+}$.
If $\Imc c\le t$, its vertical projection to the real diameter lies in $\partial D^{+}$ at distance $\le t$.
If $|c-c_0|\ge r-t$, its radial projection to the semicircle does, because the projection preserves the sign of the imaginary part.
Hence, for such $c$ with nearest boundary point $c'$,
\begin{equation}
|\psi(c)|\ge|\psi(c')|-\omega_0(t)\ge\mathfrak m+|\varphi-\psi|(c')-\omega_0(t),
\qquad
|\varphi-\psi|(c)\le|\varphi-\psi|(c')+2\omega_0(t),
\end{equation}
so that $|\varphi-\psi|(c)<|\psi(c)|$ and $|\varphi(c)|\ge|\psi(c)|-|\varphi-\psi|(c)\ge\mathfrak m-3\omega_0(t)>0$ whenever $3\omega_0(t)<\mathfrak m$.
Choose such a $t$.
Then $\varphi$ and $\psi$ have no zeros in $\overline{D^{+}}\setminus D_t$, in particular none on $\partial D^{+}$, the strict inequality $|\varphi-\psi|<|\psi|$ holds on $\partial D_t$, and $\varphi,\psi$ are holomorphic on a neighborhood of $\overline{D_t}\subset D^{+}$.
Rouch\'e's theorem on $D_t$, whose boundary is a simple closed curve, gives equal, finite zero counts there, which are the zero counts in $D^{+}$.
\end{proof}

\begin{theorem}\label{thm:main}
Let $k>0$ and assume Hypotheses \ref{hyp:reg}--\ref{hyp:sigma}.
Suppose
\begin{equation}\label{eq:mainhyp}
\Imc\Za^{+}(k,c_k)>0,
\end{equation}
with $\Za^+$ the boundary value of \cref{thm:limabs}.
Let $r_0$ and $M_Z$ be as in \cref{prop:Zhalfdisc}, with $r_0$ shrunk once more so that $r_0\le c_k$ and $\Imc\Za^{+}(k,c_r)\ge\tfrac12\Imc\Za^{+}(k,c_k)>0$ on the real interval $[c_k-r_0,c_k+r_0]$, and set
\begin{equation}\label{eq:c1}
d_1:=\partial_c\Dw(k,c_k)=2kc_k\coth(k\hw),
\qquad
c_1:=\frac{\Za^{+}(k,c_k)}{d_1},
\qquad
R:=\max\Bigl\{\frac{4M_Z}{d_1},\ |c_1|+c_k\Bigr\}.
\end{equation}
Then there are $\eps_0>0$ and $C_0>0$ such that the following hold for every $\eps\in(0,\eps_0)$.
\begin{enumerate}[label=\textup{(\roman*)}]
\item\label{it:main-exist} The dispersion relation $D(k,\cdot,\eps)=\Dw-\eps\Za$ has exactly one root $c_\eps$ in the set $D^{+}_{r_0}:=\{|c-c_k|<r_0,\ \Imc c>0\}$, whose closure is the set $\mathbb D^{+}_{r_0}$ of \cref{prop:Zhalfdisc}.
It is simple.
Every root in $D^{+}_{r_0}$, hence $c_\eps$ itself, satisfies $|c-c_k|\le(2M_Z/d_1)\,\eps$.
There are no roots on the real interval $[c_k-r_0,c_k+r_0]$ in the sense of boundary values, that is, $\Dw(k,c_r)\neq\eps\Za^{+}(k,c_r)$ there.
Equivalently, the holomorphic dispersion function $\Edet(k,\cdot,\eps)$ of \cref{def:determinant} has exactly one zero in $D^{+}_{r_0}$ (\cref{lem:determinant}\ref{it:det-nonvanish}).
\item\label{it:main-expansion} The root satisfies
  \begin{equation}\label{eq:expansion-rem}
  \bigl|c_\eps-c_k-\eps\,c_1\bigr|
  \le C_0\,\eps^{2}\bigl(1+\log(1/\eps)\bigr).
  \end{equation}
  In particular $\Imc c_\eps\ge\tfrac12\eps\,\Imc c_1>0$ for small $\eps$, and the linearized problem has a normal mode growing like $e^{k\Imc c_\eps\,t}$.
\item\label{it:main-rate} The growth rate satisfies
  \begin{equation}\label{eq:growthrate}
  \gamma_\eps:=k\,\Imc c_\eps
  =-\,\eps\,\frac{\pi\,c_k\tanh(k\hw)}{2}
  \sum_{j=1}^{m}
  \frac{\rhobar(z_j)\,Q(z_j)}{|U'(z_j)|}\,
  \bigl|\hat y(z_j)\bigr|^{2}
  \;+\;O\bigl(\eps^{2}(1+\log(1/\eps))\bigr),
  \end{equation}
  with $Q=U''+(g/a^2)U'=\rhobar(U'/\rhobar)'$ and $\hat y$ the interface-normalized limiting solution of \cref{thm:limabs}.
\end{enumerate}
\end{theorem}

The proof compares $D$ with the affine function $\ell(c)=d_1(c-c_k-\eps c_1)$, a polynomial of degree one in $c$, on the small set $D^{+}_{R\eps}:=\{|c-c_k|<R\eps,\ \Imc c>0\}$.
A preliminary estimate places every root in $D^{+}_{r_0}$ within $(2M_Z/d_1)\eps$ of $c_k$.
The modulus of continuity of $\Za$ in \cref{prop:Zhalfdisc}\ref{it:Z-modulus} bounds $D-\ell$ by $O(\eps^2\log(1/\eps))$ on the boundary of $D^{+}_{R\eps}$.
Then \cref{lem:rouche} shows that $D$, like $\ell$, has exactly one zero in $D^{+}_{R\eps}$.

\begin{proof}
Let $c$ range over the closed set $\mathbb D^+_{r_0}$ of \cref{prop:Zhalfdisc}, the closure of $D^{+}_{r_0}$, on which $\Za(k,\cdot)$ is holomorphic in the interior, continuous up to the boundary, and bounded by $M_Z$.

Set $\ell(c):=d_1\bigl(c-c_k-\eps c_1\bigr)$, whose unique zero $c_k+\eps c_1$ lies in the open upper half-plane ($\Imc c_1=\Imc\Za^+(k,c_k)/d_1>0$) and in $D^{+}_{R\eps}$ ($|c_1|<R$).
By Equations~\eqref{eq:Dw-def} and \eqref{eq:ck}, $\Dw(k,c)=k\coth(k\hw)\,(c^{2}-c_k^{2})$ exactly, so $\Dw(k,c)-d_1(c-c_k)=k\coth(k\hw)(c-c_k)^{2}$, and \cref{prop:Zhalfdisc}\ref{it:Z-modulus} gives, for $|c-c_k|\le R\eps$,
\begin{equation}\label{eq:D-ell}
|D(k,c,\eps)-\ell(c)|
\le k\coth(k\hw)\,R^2\eps^2+\eps\,\omega(R\eps)
\le C_R\,\eps^{2}\bigl(1+\log(1/\eps)\bigr)
=:\varrho(\eps),
\end{equation}
where the constant $C_R$ absorbs $\log(R/V_\ast)$.

On the semicircle $|c-c_k|=R\eps$, $\Imc c\ge0$, we have $|\ell(c)|\ge d_1(R-|c_1|)\eps\ge d_1c_k\eps$.
On the real interval $[c_k-R\eps,c_k+R\eps]$, we have $|\ell(c_r)|\ge d_1\eps\,\Imc c_1$.
Hence for $\eps\le\eps_0$ small,
\begin{equation}\label{eq:boundary-margin}
\varrho(\eps)<d_1\,\min\{c_k,\Imc c_1\}\,\eps\le|\ell|
\qquad\text{on the whole boundary of $D^{+}_{R\eps}$.}
\end{equation}

If $D(k,c,\eps)=0$ with $c\in D^{+}_{r_0}$, then $|\Dw(k,c)|=\eps|\Za(k,c)|\le\eps M_Z$.
Since $r_0\le c_k$ gives $|c+c_k|\ge2c_k-|c-c_k|\ge c_k$, we have $|\Dw(k,c)|=k\coth(k\hw)|c-c_k||c+c_k|\ge\tfrac{d_1}2|c-c_k|$, whence $|c-c_k|\le(2M_Z/d_1)\,\eps\le\tfrac12R\eps$.
On the real interval $[c_k-r_0,c_k+r_0]$, a root in the sense of boundary values would require $\Imc\Za^{+}(k,c_r)=0$, since $\Dw$ is real there, and this is excluded by $\Imc\Za^{+}\ge\tfrac12\Imc\Za^{+}(k,c_k)>0$.

Apply \cref{lem:rouche} with $D^{+}=D^{+}_{R\eps}$, $\varphi=D(k,\cdot,\eps)$, and $\psi=\ell$.
By Equations~\eqref{eq:D-ell}--\eqref{eq:boundary-margin} its assumption holds, and $\ell$ has exactly one zero there, which is simple because $\ell$ is affine.
Hence $D$ has exactly one root $c_\eps$ in $D^{+}_{R\eps}$, and it is simple.
By the localization, every root in $D^{+}_{r_0}$ satisfies $|c-c_k|\le\tfrac12R\eps$, hence lies in $D^{+}_{R\eps}$, so $c_\eps$ is the only root in $D^{+}_{r_0}$.
The statement about $\Edet$ follows from \cref{lem:determinant}\ref{it:det-nonvanish}, since $w(0;k,\cdot)\neq0$ on $\mathbb D^+_{r_0}$ (\cref{prop:Zhalfdisc}\ref{it:Z-nopole}).

Since $D(k,c_\eps,\eps)=0$ and $|c_\eps-c_k|\le R\eps$,
\begin{equation}
d_1\,|c_\eps-c_k-\eps c_1|=|\ell(c_\eps)|
=|\ell(c_\eps)-D(k,c_\eps,\eps)|\le\varrho(\eps),
\end{equation}
which is Equation~\eqref{eq:expansion-rem} with $C_0:=C_R/d_1$.
Each root with $\Imc c>0$ is a normal mode by \cref{lem:determinant}\ref{it:det-modes}.
Item \ref{it:main-rate} follows by inserting the residue formula in Equation~\eqref{eq:residue} at $c_r=c_k$ into the definition of $c_1$ in Equation~\eqref{eq:c1}.

The choice of $\eps_0$ must satisfy three requirements.
(a) $R\eps_0\le r_0/2$, so that $D^{+}_{R\eps}$ and the point $c_\eps$ used in \cref{cor:simple} lie in the region where every estimate of \cref{prop:Zhalfdisc}, including the derivative bound \ref{it:Z-deriv}, is available.
(b) The boundary margin in Equation~\eqref{eq:boundary-margin}, $\varrho(\eps_0)<d_1\min\{c_k,\Imc c_1\}\eps_0$, holds.
(c) $\varrho(\eps_0)\le\tfrac12 d_1\eps_0\,\Imc c_1$, which turns Equation~\eqref{eq:expansion-rem} into $\Imc c_\eps\ge\tfrac12\eps\,\Imc c_1$.
Choose $\eps_0\in(0,1)$ small enough that (a), (b), and (c) hold.
Since $\varrho(\eps)/\eps=C_R\eps(1+\log(1/\eps))$ is increasing on $(0,1)$, requirements (b) and (c) then also hold with $\eps_0$ replaced by any $\eps\in(0,\eps_0)$.
\end{proof}

\begin{corollary}\label{cor:simple}
Under the assumptions of \cref{thm:main} there is $\eps_0'\in(0,\eps_0)$ such that
\begin{equation}
\bigl|\partial_cD(k,c_\eps,\eps)\bigr|\ \ge\ kc_k\coth(k\hw)
=\tfrac12\,d_1,
\qquad \eps\in(0,\eps_0'] .
\end{equation}
\end{corollary}

\begin{proof}
The derivative is $\partial_cD(k,c,\eps)=2k\coth(k\hw)\,c-\eps\,\partial_c\Za(k,c)$.
At $c=c_\eps$ we bound the two terms.
The first term satisfies $|2k\coth(k\hw)c_\eps-d_1|=2k\coth(k\hw)|c_\eps-c_k|\le 2k\coth(k\hw)R\eps\le\tfrac14 d_1$ for small $\eps$.
For the second term, by \cref{thm:main}\ref{it:main-expansion} and requirement (a) in the choice of $\eps_0$ there, $\Imc c_\eps\ge\tfrac12\eps\Imc c_1$ and $|c_\eps-c_k|\le R\eps\le r_0/2$, so \cref{prop:Zhalfdisc}\ref{it:Z-deriv} gives
\begin{equation}
\eps\,\bigl|\partial_c\Za(k,c_\eps)\bigr|
\le C\eps\Bigl(1+\bigl|\log\bigl(\tfrac12\eps\Imc c_1/V_\ast\bigr)
\bigr|\Bigr)
\le C'\eps\bigl(1+\log(1/\eps)\bigr)
\ \longrightarrow\ 0,
\end{equation}
so this term is at most $\tfrac14 d_1$ for $\eps\le\eps_0'$.
\end{proof}

\begin{corollary}
\label{cor:spectral}
Under the assumptions of \cref{thm:main}, for every $\eps\in(0,\eps_0)$ the closed, densely defined operator $\Lop_k$ of \cref{def:modal-operator} (\cref{app:operator}) has the eigenvalue
\begin{equation}
\lambda_\eps:=-\im k\,c_\eps,
\end{equation}
with $\Rec\lambda_\eps=k\,\Imc c_\eps\ge\tfrac12\,\eps\,k\,\Imc c_1>0$ and with one-dimensional eigenspace spanned by the normal mode of \cref{thm:main}\ref{it:main-expansion}.
The eigenvalue is an isolated, algebraically simple point of the spectrum of $\Lop_k$ for every $\eps\in(0,\eps_0)$.
For $\eps\in(0,\eps_0']$, the root $c_\eps$ also satisfies the quantitative derivative bound of \cref{cor:simple}.
Thus the linearized problem is unstable at wavenumber $k$ in the sense of \cref{rem:spaces}, since $\Lop_k$ has an eigenvalue with positive real part.
\end{corollary}

\begin{proof}
By \cref{thm:main}\ref{it:main-exist} and \cref{lem:determinant}\ref{it:det-nonvanish}, $\Edet(k,c_\eps,\eps)=0$ with $\Imc c_\eps>0$.
\cref{thm:pointspec} converts this into the eigenvalue statement, including the dimension of the eigenspace.
For algebraic simplicity, \cref{thm:main}\ref{it:main-exist} counts one root with multiplicity, so $c_\eps$ is a simple zero of $D(k,\cdot,\eps)$ for every $\eps\in(0,\eps_0)$, that is, $\partial_cD(k,c_\eps,\eps)\neq0$.
Since $w(0;k,c_\eps)\neq0$ (\cref{prop:Zhalfdisc}\ref{it:Z-nopole}), the factorization in Equation~\eqref{eq:Edet-factor} gives $\partial_c\Edet(k,c_\eps,\eps) =w(0;k,c_\eps)\,\partial_cD(k,c_\eps,\eps)\neq0$, because the derivative of $w(0;k,\cdot)$ multiplies $D(k,c_\eps,\eps)=0$.
Hence \cref{cor:algsimple} applies.
The quantitative bound for $\eps\le\eps_0'$ is \cref{cor:simple}.
\end{proof}

\begin{corollary}
\label{cor:sufficient}
Assume Hypotheses \ref{hyp:reg}--\ref{hyp:subsonic} and \ref{hyp:sigma}, and
\begin{equation}\label{eq:sign-hyp}
c_k\,Q(z_j)\le0\quad\text{for all }j,
\qquad
c_k\,Q(z_m)<0\quad\text{at the highest critical level.}
\end{equation}
Then Hypothesis \ref{hyp:nonres} and the condition in Equation~\eqref{eq:mainhyp} hold automatically, and all conclusions of \cref{thm:main} follow.
If instead the strict inequality in Equation~\eqref{eq:sign-hyp} is available only at some level $j_0<m$, the same conclusions follow under the additional assumption $w_\ast(z_{j_0})\neq0$, a nondegeneracy that is checkable numerically for given profiles.
\end{corollary}

\begin{proof}
By \cref{lem:positivity}, under the sign condition in Equation~\eqref{eq:sign-hyp} the limit solution has $u_1(z_m)>0$ and $u_1>0$ on $[0,z_m)$, where $u_1=|w_\ast|^2$ is the quantity $u_1$ of \cref{prop:local}\ref{it:quadratic} for the limit solution.
In particular $u_1(0)>0$, which is Hypothesis \ref{hyp:nonres}.
Since $c_k>0$, the sign condition in Equation~\eqref{eq:sign-hyp} makes every term of the sum $\sum_j\rhobar Q|\hat y|^2/|U'|$ in Equation~\eqref{eq:residue} nonpositive, and the term $j=m$ strictly negative because $u_1(z_m)>0$.
Hence $\Imc\Za^+(k,c_k)=-\pi c_k^2\sum_j\rhobar Q|\hat y|^2/|U'|>0$, which is the condition in Equation~\eqref{eq:mainhyp}.
Under the alternative assumption $w_\ast(z_{j_0})\neq0$, the flux accumulation in the proof of \cref{lem:positivity} applies with the first nonzero increment at $z_{j_0}$ instead of $z_m$.
Because all jumps have the same sign, they give $\Wc\neq0$ on $[0,z_{j_0})$, hence $u_1>0$ there, including Hypothesis \ref{hyp:nonres} at $z=0$.
The $j_0$-term of the sum in Equation~\eqref{eq:residue} is then strictly negative, which gives the condition in Equation~\eqref{eq:mainhyp} again.
\end{proof}

For a monotone wind, the sufficient condition reduces to the sign of $Q$ at the unique critical level.

\begin{corollary}
\label{cor:monotone}
Let $U\in C^4([0,\ha])$ with $U(0)=0$ and $U'>0$ on $[0,\ha]$.
Let the neutral atmosphere satisfy Hypothesis \ref{hyp:reg}, and let the surface tension satisfy Hypothesis \ref{hyp:sigma}.
Let $k>0$ be such that $c_k\in U((0,\ha))$ and Hypothesis \ref{hyp:subsonic} holds.
Since $U'>0$, there is exactly one critical level $z_c$, and it is regular.
Assume
\begin{equation}\label{eq:criterion}
Q(z_c)
=U''(z_c)+\frac{g}{a^2(z_c)}\,U'(z_c)
=\rhobar(z_c)\Bigl(\frac{U'}{\rhobar}\Bigr)'(z_c)
\;<\;0 .
\end{equation}
Then Hypotheses \ref{hyp:critical}--\ref{hyp:nonres} and the condition in Equation~\eqref{eq:mainhyp} hold automatically, and for $\eps\in(0,\eps_0)$ the dispersion relation has exactly one root in the set $D^{+}_{r_0}$ of \cref{thm:main}, and this root is simple, with
\begin{gather*}
\Imc\Za^{+}(k,c_k)
=-\pi c_k^{2}\,\frac{\rhobar(z_c)\,Q(z_c)}{|U'(z_c)|}\,
\bigl|\hat y(z_c)\bigr|^{2}>0,
\qquad
c_\eps=c_k+\eps\,c_1+O\bigl(\eps^{2}(1+\log(1/\eps))\bigr),\\
\gamma_\eps=k\,\Imc c_\eps
=-\,\eps\,\frac{\pi\,c_k\tanh(k\hw)}{2}\,
\frac{\rhobar(z_c)\,Q(z_c)}{|U'(z_c)|}\,\bigl|\hat y(z_c)\bigr|^{2}
+O\bigl(\eps^{2}(1+\log(1/\eps))\bigr)>0,
\end{gather*}
and $\lambda_\eps=-\im kc_\eps$ is an isolated, algebraically simple eigenvalue of the closed operator $\Lop_k$ with $\Rec\lambda_\eps>0$ (\cref{cor:spectral}).
As $a\to\infty$, the criterion in Equation~\eqref{eq:criterion} reduces to Miles's condition $U''(z_c)<0$.
\end{corollary}

\begin{proof}
The condition $U'>0$ makes every solution of $U=c_k$ regular and unique, which is Hypothesis \ref{hyp:critical} with $m=1$.
With $c_k>0$, the sign condition in Equation~\eqref{eq:sign-hyp} reduces to $Q(z_c)<0$, so \cref{cor:sufficient} supplies Hypothesis \ref{hyp:nonres}, the condition in Equation~\eqref{eq:mainhyp}, and the displayed value of $\Imc\Za^{+}(k,c_k)$.
The remaining statements are \cref{thm:main}, \cref{cor:simple}, and \cref{cor:spectral}.
\end{proof}

\begin{remark}\label{rem:mirror}
Here $c_k>0$ denotes the right-going root of $\Dw$ given by Equation~\eqref{eq:ck}.
The left-going root $-c_k$ is covered by the reflection $(x,U)\mapsto(-x,-U)$.
The same argument then applies to the reflected problem, whose critical levels are the solutions of $U=-c_k$, with the sign condition in Equation~\eqref{eq:sign-hyp} read with $c_k$ replaced by $-c_k$.
The two roots have independent sets of critical levels.
For the non-negative winds of \S\ref{sec:examples} the left-going wave has no critical level, and \cref{thm:necessity}, applied after reflection, keeps it neutral.
\end{remark}

\begin{remark}
\label{rem:reading}
Since $g/a^2>0$, for a wind profile that increases with height with shear decreasing upward ($U'>0>U''$) compressibility shifts the local coefficient $Q=U''+(g/a^2)U'$ upward.
While $Q(z_c)<0$ this reduces $|Q(z_c)|$ below $|U''(z_c)|$, and the criterion in Equation~\eqref{eq:criterion} fails altogether once
\begin{equation}\label{eq:cutoff-height}
\frac{g}{a^2(z_c)}\ \ge\ \frac{|U''(z_c)|}{U'(z_c)} .
\end{equation}
For a logarithmic profile $U=(u_\ast/\kappa_{\mathrm K})\log(1+z/z_0)$, where $u_\ast$ is the friction velocity, $\kappa_{\mathrm K}$ the von K\'arm\'an constant, and $z_0$ the roughness length, the condition in Equation~\eqref{eq:cutoff-height} holds beyond the height at which $z_c+z_0=a^2/g$.
For the bounded shear-layer profile $U=U_\infty(1-e^{-z/L})$ of limiting speed $U_\infty$ and shear depth $L$, it holds uniformly in $z$ when $a^2<gL$.
In the second case, consider a wavenumber that satisfies the assumptions of \cref{prop:suppression} and lies in the incompressible Miles band, which is the set of wavenumbers $k$ for which $c_k$ lies in the range of $U$ with $U''<0$ at the critical levels.
At such a wavenumber the sum in Equation~\eqref{eq:growthrate} is positive, so the leading Miles coefficient $\Imc c_1$ changes sign and the local $O(\eps)$ Miles bifurcation from $c_k$ is absent (\cref{prop:suppression}).
At the same time the growth rate is not a function of $Q(z_c)$ alone.
The amplitude $|\hat y(z_c)|^2$ and the weight $\rhobar(z_c)$ are themselves modified by compressibility, and the coefficient $\Abar/\rhobar=a^{2}/(a^{2}-V^{2})$ also changes the vertical boundary-value problem.
Hence the sign of the net correction is not decided by the local criterion alone.
For the parameters of wind waves over the sea, the compressible corrections are at most about $10^{-3}$ locally and $10^{-2}$ globally (\S\ref{subsec:polytropic-wind}, \cref{rem:two-effects}), and the condition in Equation~\eqref{eq:cutoff-height} would require critical layers $8.4\,\mathrm{km}$ above the sea surface.
Miles's incompressible criterion is therefore insensitive to compressibility for water waves, and \cref{thm:lowmach} makes this precise at each wavenumber.
\end{remark}

\begin{remark}
As in \cite[Rem.~4.3]{BSWZ2016}, \cref{cor:sufficient} applies for every $k$ such that $c_k$ lies in the range of $U$ with $Q<0$ at the critical levels.
With surface tension the set of such $k$ is a bounded band, because large-$k$ waves are too fast, with $c_k\to\infty$.
Its edges are set by $c_k\in U([0,\ha])$ as in the incompressible theory, except near the compressible cutoff in Equation~\eqref{eq:cutoff-height}, where the band is further truncated.
The Kelvin--Helmholtz mechanism is absent by Hypothesis \ref{hyp:noslip}.
The compressible Kelvin--Helmholtz problem for uniform wind is treated in closed form in \S\ref{subsec:KH} for comparison.
\end{remark}
\section{Recovery of the incompressible theory in the low-Mach limit}
\label{sec:lowmach}

In this section, we quantify the incompressible limit at a single wavenumber.
For neutral atmospheres parametrized by the Mach number $M$, we show that, as $M\to0$, the pressure response converges to its incompressible counterpart at rate $O(M^2)$ and that the unstable root converges at rate $O(\eps M^2)$ (\cref{thm:lowmach}).
Let $U$, $g$, the depths $\ha$ and $\hw$, and $\hat a\in C^3([0,\ha])$ with $\hat a\ge\hat a_->0$ be given.
For $M>0$ and a reference speed $a_\ast>0$ chosen below, define the neutral atmosphere
\begin{equation}\label{eq:M-family}
a_M^2(z):=\frac{a_\ast^2\,\hat a(z)}{M^2},
\qquad
\bar\rho_{0,M}(z):=\exp\Bigl(-\frac{gM^2}{a_\ast^2}
\int_0^z\frac{\dd z'}{\hat a(z')}\Bigr),
\end{equation}
so that the coefficient $b$ of Equation~\eqref{eq:neutral} is $b_M=g/a_M^2=gM^2/(a_\ast^2\hat a)$.
We take $\hat a(0)=1$ and
\begin{equation}
a_\ast:=\max\bigl\{\sup_{[0,\ha]}|U|,\ c_k\bigr\},
\end{equation}
so $M=a_\ast/a_M(0)$ is the Mach number based on the wind and wave speeds and the interface sound speed.
The argument also permits $M$-dependent shapes $\hat a_M$ with $\|\hat a_M-\hat a\|_{C^3}=O(M^2)$, since such a change perturbs the coefficients estimated in Equation~\eqref{eq:coef-lip} by a further $O(M^2)$.
Here and below, the $C^k$ norm of a function on $[0,\ha]$, or on a subinterval, is defined as in Equation~\eqref{eq:norms} with $I_j$ replaced by that interval.
This covers the polytropic atmospheres of \cref{ex:polytropic}.
As $M\to0$, the coefficients of Equation~\eqref{eq:SL} for the atmosphere in Equation~\eqref{eq:M-family}, marked by the subscript $M$, satisfy $\bar\rho_{0,M}\to1$, $b_M\to0$, $\Abar_M\to1$, $G_M\to U''$, and $Q_M\to U''$ in $C^2$, all at rate $O(M^2)$.
The limiting equation is
\begin{equation}\label{eq:rayleigh-limit}
w''=\Bigl(k^2+\frac{U''}{V}\Bigr)w,
\qquad w(\ha)=0,
\end{equation}
whose boundary value $\Za^{+,\mathrm{inc}}(k,c_r)$, equal to $c_r^2\,\hat y'(0)+c_rU'(0)-g$ in the normalization of \cite{BSWZ2016}, is the object analyzed there.
We mark the pressure response, its boundary value, and the limiting solution of the atmosphere with Mach number $M$ by a superscript $M$, as in $\Za^{M}$, $\Za^{+,M}$, and $w_\ast^{M}$, and those of the limiting problem in Equation~\eqref{eq:rayleigh-limit} by a superscript $\mathrm{inc}$.
Similarly, $c_1(M)$ and $c_\eps(M)$ denote the coefficient $c_1$ in Equation~\eqref{eq:c1} and the root $c_\eps$ of \cref{thm:main} for this atmosphere, and the value $M=0$ refers to the limiting problem.

\begin{theorem}\label{thm:lowmach}
Let $K$ be a compact set of regular values of $U$ containing $c_k$, on which Hypothesis \ref{hyp:critical} holds and the incompressible problem satisfies Hypothesis \ref{hyp:nonres}.
Then there are $M_0,C>0$ such that for all $M\in(0,M_0]$ and $c_r\in K$ the compressible problem satisfies Hypotheses \ref{hyp:subsonic}--\ref{hyp:nonres} as well, its boundary value exists, and
\begin{equation}\label{eq:lowmach-rate}
\bigl|\Za^{+,M}(k,c_r)-\Za^{+,\mathrm{inc}}(k,c_r)\bigr|
\;\le\;C\,M^{2} .
\end{equation}
Moreover, with $\mathbb D^{+}_{r_0}$ the closed set of \cref{cor:uniformM}, a closed neighborhood of $c_k$ in the closed upper half-plane of the form used in \cref{prop:Zhalfdisc}, the estimate holds uniformly on that set,
\begin{equation}\label{eq:lowmach-complex}
\sup_{c\in\mathbb D^{+}_{r_0}}
\bigl|\Za^{M}(k,c)-\Za^{\mathrm{inc}}(k,c)\bigr|
\;\le\;C\,M^{2} .
\end{equation}
Consequently $c_1(M)=c_1(0)+O(M^2)$, where $c_1(0)$ is the coefficient of the incompressible theory.
If moreover the incompressible problem is unstable at $k$ in the sense of \cref{thm:main}, that is $\Imc\Za^{+,\mathrm{inc}}(k,c_k)>0$, then there is $M_1\in(0,M_0]$ such that for all $M\le M_1$ the compressible problem satisfies the condition in Equation~\eqref{eq:mainhyp} as well, with $\Imc c_1(M)\ge\tfrac12\,\Imc c_1(0)>0$.
Both problems then possess the unstable root of \cref{thm:main}.
The expansion in Equation~\eqref{eq:expansion-rem} holds with constants $(\eps_0,C_0)$ independent of $M\le M_1$,
\begin{equation}
c_\eps(M)=c_k+\eps\,c_1(M)+O\bigl(\eps^{2}(1+\log(1/\eps))\bigr)
\qquad\text{uniformly in }M\le M_1.
\end{equation}
For each $\eps$, the roots converge at the rate
\begin{equation}\label{eq:root-rate}
\bigl|c_\eps(M)-c_\eps(0)\bigr|\;\le\;C\,\eps\,M^{2},
\qquad 0<\eps\le\eps_0,\ 0<M\le M_1 .
\end{equation}
\end{theorem}

The proof has three parts.
First, the constants of the critical-layer theory are uniform in $M$ (\cref{lem:uniformconst}).
Second, the local solutions depend Lipschitz-continuously on their coefficients (\cref{lem:coef-perturb}).
Third, the roots are compared along a line segment on which $\partial_cD^{\mathrm{inc}}$, the derivative of the incompressible dispersion function, stays close to $d_1$.

\begin{lemma}\label{lem:uniformconst}
For $M_0$ small enough, the atmospheres in Equation~\eqref{eq:M-family} satisfy the subsonic gap of Hypothesis \ref{hyp:subsonic} uniformly in $M\in(0,M_0]$.
The coefficient functions of the local representation in Equation~\eqref{eq:rep}, namely $b_M$, $E_M$, $f_M=1/(\Abar_ME_M)$, $g_{1,M}$, $g_{2,M}$, together with the coefficients of the system in Equation~\eqref{eq:wchi-recall} on intervals containing no critical level, are bounded in $C^2$, except $g_{1,M}$, which is bounded in $C^1$, uniformly for $M\in[0,M_0]$, with
\begin{equation}\label{eq:coef-lip}
\|b_M\|_{C^2}+\|f_M-1\|_{C^2}+\|g_{1,M}-g_1^{\mathrm{inc}}\|_{C^1}
+\|g_{2,M}-g_2^{\mathrm{inc}}\|_{C^1}\le C M^2 ,
\end{equation}
where the coefficients at $M=0$ are those of the Rayleigh problem in Equation~\eqref{eq:rayleigh-limit}, which we mark with the superscript $\mathrm{inc}$, namely $f^{\mathrm{inc}}=1$, $g_1^{\mathrm{inc}}=(1/U')'$, and $g_2^{\mathrm{inc}}=k^2/U'$.
Consequently the constants of \cref{lem:fixedpoint}, \cref{prop:local}, \cref{prop:convergence}, \cref{prop:limit-solutions}, and \cref{thm:limabs}, including $\varsigma_0$ and the contraction constant $\tfrac12$, can be chosen uniformly in $M\in[0,M_0]$, and so can $\varsigma$ and $\delta_1$.
\end{lemma}

\begin{proof}
The gap holds uniformly once $M_0$ is small, because $\Rec(a_M^2-V^2)\ge a_M^2-|U-c|^2$, $a_M^2\ge a_\ast^2\hat a_-/M^2$, and $|U-c|^2\le(2a_\ast+r_0)^2$ for every complex $c$ with $|c-c_k|\le r_0$.
The bounds in Equation~\eqref{eq:coef-lip} follow from $b_M=gM^2/(a_\ast^2\hat a)$, from $\Abar_M=\bar\rho_{0,M}\bigl(1-M^2V^2/(a_\ast^2\hat a)\bigr)^{-1}$, and from $\bar\rho_{0,M}$ and $E_M$, each of which is $1+O(M^2)$ in $C^2$ by Equation~\eqref{eq:M-family} and $\hat a\ge\hat a_-$, together with the definitions in Equation~\eqref{eq:g-def}, which are linear in $f$.
Every constant in the cited results depends only on these norms, the gap, $|U'(z_j)|^{-1}$, and $k$.
\end{proof}

\begin{proof}[Proof of \cref{thm:lowmach}]
By \cref{lem:uniformconst}, the critical-layer theory of \S\ref{sec:criticallayer} applies to every atmosphere in Equation~\eqref{eq:M-family} with constants uniform in $M\in[0,M_0]$.

On intervals containing no critical level, the coefficients of the system in Equation~\eqref{eq:wchi-recall} for the atmospheres in Equation~\eqref{eq:M-family} differ by $O(M^2)$ in $C^0$.
The values at $z=\ha$ given in Equation~\eqref{eq:canonical} also differ by $O(M^2)$, because $\Abar_M(\ha)=1+O(M^2)$, while $\Abar=1$ in the limiting problem in Equation~\eqref{eq:rayleigh-limit}.
Hence the solution maps on these intervals differ in norm by $O(M^2)$ by Gr\"onwall's inequality.

Across a critical layer we apply \cref{lem:coef-perturb} of \cref{app:proofs}, including its uniform bounds on the endpoint evaluation matrices and their inverses, valid for every $\delta\in[0,\delta_1]$, to the coefficient tuples $\mathfrak a_M:=(b_M,\bar\rho_{0,M},f_M,g_{1,M},g_{2,M})$ and $\mathfrak a^{\mathrm{inc}}:=(0,1,1,g_1^{\mathrm{inc}},g_2^{\mathrm{inc}})$ of the limiting problem.
Both tuples satisfy the first-order-system identities of \cref{lem:coef-perturb}, because both are generated by neutral atmospheres.
These identities state that $g_1$ and $g_2$ are obtained from $b$, $\rhobar$, and $f$ by Equation~\eqref{eq:g-def}.
The two tuples also share $U$, $c$, the base point $z_j$, and the branch of the logarithm in Equation~\eqref{eq:branch}.
Their distance $\|\mathfrak a_M-\mathfrak a^{\mathrm{inc}}\|_{\mathcal X}$ is the sum of the $C^0(I_j)$ norms of the differences of $b$, $\rhobar$, and $g_2$, the $C^1(I_j)$ norm of the difference of $f$, and the $C^{0,1}(I_j)$ norm of the difference of $g_1$.
It satisfies $\|\mathfrak a_M-\mathfrak a^{\mathrm{inc}}\|_{\mathcal X}\le CM^2$ by Equation~\eqref{eq:coef-lip} together with $\|\bar\rho_{0,M}-1\|_{C^1}\le CM^2$.
Hence the maps $T_\delta$ across the critical neighborhoods, their inverses, and the critical values $w_\ast(z_j)$ are $O(M^2)$-close to their incompressible counterparts, uniformly in $\delta\in[0,\delta_1]$ and on $K$.

When we compose the solution operators on intervals containing no critical level with the maps $T_\delta$ across the critical neighborhoods, from $z=\ha$ down to $z=0$, the values $\bigl(w_\ast(0),\chibar_\ast(0),w_\ast(z_j)\bigr)$ of the normalized limit solution differ between $M$ and $0$ by at most $C_1M^2$, uniformly on $K$.
Since $m_0:=\min_K|w_\ast^{\mathrm{inc}}(0;\cdot)|>0$ by Hypothesis \ref{hyp:nonres} for the incompressible problem and continuity on the compact set $K$, we shrink $M_0$ so that $C_1M_0^2\le m_0/2$.
Then $|w_\ast^{M}(0;c_r)|\ge m_0/2$ for all $M\in(0,M_0]$ and $c_r\in K$, which is Hypothesis \ref{hyp:nonres} for the compressible problem.
With this common lower bound on the denominator, the quotient in Equation~\eqref{eq:Za-formula} gives Equation~\eqref{eq:lowmach-rate}.
For the complex version in Equation~\eqref{eq:lowmach-complex}, we also shrink $M_0$ so that the conclusions of \cref{cor:uniformM} hold for every $M\in(0,M_0]$.
\cref{cor:uniformM}\ref{it:uM-trace} gives a common positive lower bound for the interface traces $|w^{M}(0;k,c)|$ on $\mathbb D^{+}_{r_0}$.
Equation~\eqref{eq:lowmach-complex} is then \cref{cor:uniformM}\ref{it:uM-complex}, because the same composition applies at every $c=c_r+\im\delta$ with $\delta\in[0,\delta_1]$, because \cref{lem:coef-perturb} is uniform in $\delta$, and every point of $\mathbb D^{+}_{r_0}$ has $\Imc c\le r_0\le\delta_1$.

The leading coefficient satisfies $c_1(M)-c_1(0) =\bigl(\Za^{+,M}-\Za^{+,\mathrm{inc}}\bigr)(k,c_k)/(2kc_k\coth(k\hw))=O(M^2)$.
Assume $\Imc\Za^{+,\mathrm{inc}}(k,c_k)>0$.
By Equation~\eqref{eq:lowmach-rate} there is $M_1\in(0,M_0]$ with $\Imc\Za^{+,M}(k,c_k)\ge\tfrac12\Imc\Za^{+,\mathrm{inc}}(k,c_k)>0$, hence $\Imc c_1(M)\ge\tfrac12\Imc c_1(0)>0$, for all $M\le M_1$, so \cref{thm:main} applies to every $M\le M_1$.
By \cref{cor:uniformM} of \cref{app:proofs}, the atmospheres in Equation~\eqref{eq:M-family} admit a common closed neighborhood $\mathbb D^{+}_{r_0}$ of $c_k$ in the closed upper half-plane on which no $\Za^{M}(k,\cdot)$ has a pole, a common bound $M_Z$, and a common modulus of continuity $\omega(\cdot)$ at $c_k$.
Together with the uniform lower bound on $\Imc c_1(M)$, every constant in the proof of \cref{thm:main}, including the constants $(\eps_0,R,C_0)$ in Equations~\eqref{eq:c1} and \eqref{eq:expansion-rem}, can be chosen uniformly in $M\le M_1$, and equally at $M=0$.
The uniform choice of $(\eps_0,R,C_0)$ also gives
\begin{equation}\label{eq:Im-floor}
\Imc c_\eps(M)\ \ge\ \tfrac12\,\eps\,\Imc c_1(M)
\ \ge\ \tfrac14\,\eps\,\Imc c_1(0),
\qquad
|c_\eps(M)-c_k|\le R\eps,
\qquad \eps\le\eps_0,\ M\le M_1 .
\end{equation}

For the estimate at a given $\eps$, write $D_M:=\Dw-\eps\Za^{M}$ and $D^{\mathrm{inc}}:=\Dw-\eps\Za^{\mathrm{inc}}$.
Since $D_M(k,c_\eps(M),\eps)=0$,
\begin{equation}
\bigl|D^{\mathrm{inc}}(k,c_\eps(M),\eps)\bigr|
=\bigl|D^{\mathrm{inc}}-D_M\bigr|(k,c_\eps(M),\eps)
=\eps\,\bigl|\Za^{M}-\Za^{\mathrm{inc}}\bigr|(k,c_\eps(M))
\;\le\;C\eps M^{2}
\end{equation}
by Equation~\eqref{eq:lowmach-complex}.
The straight segment from $c_\eps(0)$ to $c_\eps(M)$ stays in $\{|c-c_k|\le R\eps,\ \Imc c\ge\tfrac14\eps\Imc c_1(0)\}$, because both endpoints satisfy both constraints by Equation~\eqref{eq:Im-floor} and both constraints are preserved under convex combination.
On this segment,
\begin{equation}
\bigl|\partial_cD^{\mathrm{inc}}(k,c,\eps)-d_1\bigr|
\le 2k\coth(k\hw)\,R\eps
+\eps\,C\Bigl(1+\bigl|\log\bigl(\tfrac14\eps\Imc
c_1(0)/V_\ast\bigr)\bigr|\Bigr)
\;\le\;\tfrac12\,d_1
\end{equation}
for $\eps\le\eps_0$, after we shrink $\eps_0$ once, independently of $M$, so that in addition $R\eps_0\le r_0/2$.
The segment then lies in the region where \cref{prop:Zhalfdisc}\ref{it:Z-deriv} is available for the incompressible problem, to which the critical-layer theory of \S\ref{sec:criticallayer} applies with the same constants.
Hence
\begin{align*}
&\bigl|D^{\mathrm{inc}}(k,c_\eps(M),\eps)-D^{\mathrm{inc}}(k,c_\eps(0),\eps)\bigr|\\
&\qquad=\Bigl|\int_0^1\partial_cD^{\mathrm{inc}}\bigl(k,(1-t)c_\eps(0)+tc_\eps(M),\eps\bigr)\dd t\Bigr|
\,\bigl|c_\eps(M)-c_\eps(0)\bigr|\\
&\qquad\ge\ \tfrac12\,d_1\,\bigl|c_\eps(M)-c_\eps(0)\bigr|,
\end{align*}
since the integral lies within $\tfrac12 d_1$ of $d_1$.
With $D^{\mathrm{inc}}(k,c_\eps(0),\eps)=0$ this gives $|c_\eps(M)-c_\eps(0)|\le(2C/d_1)\,\eps M^{2}$, which is Equation~\eqref{eq:root-rate}.
\end{proof}

\begin{remark}\label{rem:lowmach-scope}
The estimate in Equation~\eqref{eq:root-rate} controls the unstable root itself for $0<\eps\le\eps_0$, pointwise in $k$.
The constants may depend on $k$, and no uniformity over a wavenumber interval is claimed.
\end{remark}

\begin{remark}\label{rem:two-effects}
The scaling in Equation~\eqref{eq:M-family} sends the acoustic ratio $|V|^2/a^2$ and the hydrostatic ratio $gz_c/a^2$ to zero at order $M^2$.
The former enters through $\Abar=\rhobar a^2/(a^2-V^2)$, while the latter enters through $\rhobar$ and $Q-U''=(g/a^2)U'$.
For ocean--atmosphere parameters, the global Mach parameter satisfies $M^2\lesssim10^{-2}$ for the wavenumbers of wind waves.
Here $\lesssim$ means that the left side is at most of the order of the right side.
The squared local wave Mach number $(c_k/a(z_c))^2$ and the hydrostatic ratio $gz_c/a(z_c)^2$ are each typically at most $10^{-3}$.
\end{remark}
\section{Uniform, logarithmic, and shear-layer winds}
\label{sec:examples}

In this section, we use examples to explain the results physically and to relate them to winds over the sea.
A uniform wind forms a vortex sheet at the interface, as in the classical compressible Kelvin--Helmholtz problem, so it lies outside Hypothesis \ref{hyp:noslip} and the results of \S\ref{sec:instability} and \S\ref{sec:lowmach} do not apply to it.
We include it for comparison, because its air-side pressure response has a closed form (\cref{prop:KH}).
The closed form shows how the sound speed and the density scale height of the air set the rate at which the perturbation decays with height, so in this example an upper boundary several decay lengths above the interface has little effect.
For the logarithmic wind of the neutral surface layer over the sea and for a bounded shear layer, we evaluate the critical-layer coefficient $Q$ (\S\ref{subsec:polytropic-wind}).
For the logarithmic wind, the compressible term $(g/a^2)U'$ in $Q$ is at most about $10^{-3}|U''|$ at critical levels within about $10\,\mathrm m$ of the sea surface, so the criterion in Equation~\eqref{eq:criterion} agrees with Miles's criterion $U''(z_c)<0$ for winds over the sea.
We then consider the shear-layer wind in a model atmosphere with constant sound speed below the gravity speed $\sqrt{gL}$ of the shear layer.
In this model, $Q>0$ everywhere although $U''<0$, and the Miles bifurcation is absent (\cref{prop:suppression}).
This regime does not occur over the sea, but it shows that the sign criterion of the incompressible theory does not extend to a compressible atmosphere in general.

\subsection{Closed-form pressure response for a uniform wind}\label{subsec:KH}

For comparison with the classical compressible Kelvin--Helmholtz problem \cite{Miles1958VortexSheet,McKenzie1972}, we solve the air equation for a uniform wind in the constant-sound-speed atmosphere of \cref{ex:constant-sound-speed}.
This case lies outside Hypothesis \ref{hyp:noslip} and provides an exactly solvable subsonic example.
In this atmosphere the sound speed is constant, $a\equiv a_{\mathrm e}$, and its value is a free parameter, which need not be the value for air.

\begin{proposition}\label{prop:KH}
Let the atmosphere have constant sound speed $a$ and density $\rhobar=e^{-bz}$ with $b=g/a^2$ (\cref{ex:constant-sound-speed}), and let the wind be uniform, $U\equiv U_0$ for a constant $U_0$.
Here $V=U-c$ is the constant $U_0-c$, and we assume $V\neq0$ and the subsonicity condition $|V|^2<a^2$.
Then the mass-weighted displacement $X=\rhobar\,\xi$ satisfies the constant-coefficient equation
\begin{equation}\label{eq:KH-ode}
X''+b\,X'-k^2\Bigl(1-\frac{V^2}{a^2}\Bigr)X=0 ,
\end{equation}
and with the impermeable top boundary at $\ha$ the air-side pressure response is
\begin{equation}\label{eq:KH-Z}
\Za^{\mathrm{KH}}(k,c)
=\frac{a^2V^2}{a^2-V^2}\,
\frac{\mu_-+\mu_+\,\Theta}{1+\Theta}-g,
\qquad
\Theta:=-e^{(\mu_--\mu_+)\ha},
\qquad
\mu_\pm:=-\frac b2\pm\nu,
\end{equation}
where $\nu:=\sqrt{b^2/4+k^2(1-V^2/a^2)}$ is the principal square root, $\sqrt{\zeta}=\exp(\tfrac12\operatorname{Log}\zeta)$ with $\operatorname{Log}$ from Equation~\eqref{eq:Log-def}.
It satisfies $\Rec\nu>0$ and is well defined for $c$ in the complex subsonic neighborhood of Hypothesis \ref{hyp:subsonic}, on which $\Rec(1-V^2/a^2)>0$.
As $\ha\to\infty$, $\Za^{\mathrm{KH}}\to a^2V^2\mu_-/(a^2-V^2)-g$.
As $a\to\infty$, $\Za^{\mathrm{KH}}\to-kV^2\coth(k\ha)-g$, the finite-air-depth incompressible uniform-wind response.
Taking in addition $\ha\to\infty$ gives $-kV^2-g$, the air-side term in the Kelvin--Helmholtz dispersion relation of \cite[Ex.~3.1]{BSWZ2016}.
\end{proposition}

\begin{proof}
The derivation of Equation~\eqref{eq:Xform} in \cref{prop:SLform} and the definition of $\Za$ in Equation~\eqref{eq:Za-def} use only $V\neq0$ and $a^2-V^2\neq0$, not Hypothesis \ref{hyp:noslip}, and both conditions hold here for real and complex $c$ because $V$ is a nonzero constant and $|V|^2<a^2$.
Equation~\eqref{eq:KH-ode} is Equation~\eqref{eq:Xform} evaluated for this base state and divided by $F$, which is nonzero because $V\neq0$.
The boundary condition $w(\ha)=0$ is $X(\ha)=0$, and Equation~\eqref{eq:KH-Z} follows from $\qbar=\rhobar FX'$ and Equation~\eqref{eq:Za-def}.
The limits follow by direct computation.
\end{proof}

Uniform wind gives $U(0)=U_0\neq0$ and hence a vortex sheet, so the coupled reduction of \cref{prop:reduction} does not apply.
The kinematic conditions give the interface values $w_{\mathrm a}(0)=\im k(U_0-c)\hat\zeta$ and $w_{\mathrm w}(0)=-\im kc\hat\zeta$ of the vertical velocity in the air and in the water.
Nevertheless, the air-side displacement remains $\xi(0)=w(0)/(\im kV(0))=\hat\zeta$.
Thus Equation~\eqref{eq:Za-def} still defines the air-side response, whose closed form is given by Equation~\eqref{eq:KH-Z}.
The water side and the dynamic condition are the same as in \S\ref{sec:linearization}, so a normal mode with $\hat\zeta=1$ satisfies $\Dw(k,c)=\eps\Za^{\mathrm{KH}}(k,c)$, which has the form of Equation~\eqref{eq:dispersion} and is the compressible Kelvin--Helmholtz dispersion relation for this configuration.
The only change is that the vertical velocity jumps across the interface.
The operator $\Lop_k$ of \cref{thm:operator} and the results of \S\ref{sec:instability} and \S\ref{sec:lowmach} require this velocity to be continuous, which Hypothesis \ref{hyp:noslip} ensures, so we use this wind only for comparison.

The acoustic factor $(1-V^2/a^2)$ and the scale-height shift $-b/2$ both enter the vertical decay exponent.
Their combined effect is the full expression in Equation~\eqref{eq:KH-Z}.
The term $-b/2$ is the Lamb correction associated with the density scale height \cite{GossardHooke1975}.
Up to a nonzero factor, $1+\Theta=1-e^{-2\nu\ha}$ is the interface trace $w(0;k,c)$ from \cref{subsec:determinant}.
For real subsonic $c$ it is positive, so Hypothesis \ref{hyp:nonres} holds in this uniform-wind example and the vertical mode decays exponentially with height.

Since $|\Theta|=e^{-2\Rec\nu\,\ha}$, the finite-height correction in Equation~\eqref{eq:KH-Z} is $O(e^{-2\Rec\nu\,\ha})$.
If $\Rec(1-V^2/a^2)\ge\delta_0$, the principal branch satisfies $\Rec\nu\ge k\sqrt{\delta_0}$ because $\Rec\sqrt{\zeta}\ge\sqrt{\Rec\zeta}$ for $\Rec\zeta>0$.
For the ocean--atmosphere scales used above, $|V|^2/a^2\lesssim10^{-2}$, so the subsonic factor is at least about $0.99$.
A top boundary several decay lengths $1/\Rec\nu$ above the interface therefore has little effect in this example.
No estimate for removing the top boundary is proved here for variable shear.

\subsection{The Miles criterion for logarithmic and shear-layer winds}\label{subsec:polytropic-wind}

We evaluate $Q$ for logarithmic and bounded shear-layer winds in the hydrostatic atmosphere given by Equation~\eqref{eq:polytropic}.
For the logarithmic wind $U=(u_\ast/\kappa_{\mathrm K})\log(1+z/z_0)$, with friction velocity $u_\ast$, von K\'arm\'an constant $\kappa_{\mathrm K}$, and roughness length $z_0$,
\begin{equation}
Q(z)=\frac{u_\ast}{\kappa_{\mathrm K}(z+z_0)}
\Bigl[\frac{g}{a^2(z)}-\frac1{z+z_0}\Bigr],
\end{equation}
so $Q<0$, the unstable side of the criterion in Equation~\eqref{eq:criterion}, below the height $z+z_0=a^2(z)/g$.
For the polytropic atmosphere in Equation~\eqref{eq:polytropic}, the cutoff is $z_{\mathrm{cut}}=(a_0^2/g-z_0)/\gamma\approx8.4\,\mathrm{km}$ for air.
The constant-sound-speed value is $a_{\mathrm e}^2/g\approx11.8\,\mathrm{km}$, with $a_{\mathrm e}$ from \cref{ex:constant-sound-speed}.
At a critical level $z_c$, the compressible term $(g/a^2)U'$ is the fraction $(z_c+z_0)\,g/a^2(z_c)$ of $|U''|$, which is about $10^{-3}$ for $z_c=10\,\mathrm m$.
For the bounded shear-layer profile $U=U_\infty(1-e^{-z/L})$, with limiting speed $U_\infty$ and shear depth $L$,
\begin{equation}\label{eq:Q-exp}
Q(z)=\frac{U_\infty}{L^{2}}\,e^{-z/L}\Bigl[\frac{gL}{a^2(z)}-1\Bigr],
\end{equation}
so the sign of $Q$ at height $z$ is the sign of $gL-a^2(z)$, and the instability criterion compares the sound speed with the gravity speed $\sqrt{gL}$ of the shear layer.

\subsection{Failure of the Miles bifurcation under strong compressibility}
\label{subsec:suppression}

For the shear-layer wind, $-U''=U'/L$, so in the constant-sound-speed model the condition $a^2<gL$ means that the hydrostatic term $(g/a^2)U'$ in $Q$ dominates the curvature term $U''$ at every height.
This reverses the sign of the critical-layer residue and of the leading Miles coefficient.
The sign condition is independent of $k$, although the threshold $\eps_0$ in \cref{thm:main} may depend on $k$.

\begin{proposition}\label{prop:suppression}
Let the atmosphere have constant sound speed $a$ (\cref{ex:constant-sound-speed}) and wind $U=U_\infty(1-e^{-z/L})$, $U_\infty>0$, and let $k$ be such that $c_k\in U((0,\ha))$ with a single critical level and Hypothesis \ref{hyp:subsonic} holds.
If
\begin{equation}\label{eq:suppression-cond}
a^2<gL ,
\end{equation}
then $Q>0$ on all of $[0,\ha]$ although $U''<0$ everywhere.
The limit solution satisfies $w_\ast(z_c)\neq0$ and Hypothesis \ref{hyp:nonres} by \cref{lem:positivity} with the stated sign condition, and
\begin{equation}
\Imc\Za^{+}(k,c_k)
=-\pi c_k^2\,\frac{\rhobar(z_c)Q(z_c)}{|U'(z_c)|}\,|\hat y(z_c)|^2\;<\;0 .
\end{equation}
Consequently $\Imc c_1<0$ for the coefficient $c_1$ in Equation~\eqref{eq:c1}, and there are $\eps_0,r_0>0$ such that for $\eps\in(0,\eps_0)$ the dispersion relation has no root in the set $\{|c-c_k|\le r_0,\ \Imc c>0\}$, and none on the real interval $[c_k-r_0,c_k+r_0]$ in the sense of boundary values.
At this admissible wavenumber $k$, the incompressible criterion $U''(z_c)<0$ of \cite{Miles1957,BSWZ2016} predicts growth, whereas the compressible dispersion relation has no $O(\eps)$ unstable root bifurcating from $c_k$.
\end{proposition}

\begin{proof}
Equation~\eqref{eq:Q-exp} with $a^2<gL$ gives $Q>0$.
Since all critical-level values $c_kQ(z_j)$ have the same sign, \cref{lem:positivity} applies and yields $w_\ast(z_c)\neq0$.
The same argument holds for every regular value $c_r$ near $c_k$, because $Q>0$ on $[0,\ha]$.
Hence Hypothesis \ref{hyp:nonres} holds on a real interval about $c_k$, and \cref{thm:limabs} gives the displayed formula.
Choose $r_0$ and $M_Z$ from \cref{prop:Zhalfdisc}, then shrink $r_0$ so that $r_0\le c_k$ and $\Imc\Za^{+}(k,c_r)\le\tfrac12\Imc\Za^{+}(k,c_k)<0$ on the real interval $[c_k-r_0,c_k+r_0]$.
Since $\Dw$ is real there, there is no root in the sense of boundary values on this interval.
The localization estimate in the proof of \cref{thm:main} uses only $|\Za|\le M_Z$ and $r_0\le c_k$, which hold on the closed set $\mathbb D^+_{r_0}$ of \cref{prop:Zhalfdisc}.
It therefore places every root in $\{|c-c_k|\le r_0,\ \Imc c>0\}$ inside $\{|c-c_k|\le\tfrac12R\eps\}$, where $R=\max\{4M_Z/d_1,|c_1|+c_k\}$ is given by Equation~\eqref{eq:c1} and $R\eps_0\le r_0$.
On the set $D^{+}_{R\eps}$ defined after \cref{thm:main}, compare $D$ with the affine function $\ell(c)=d_1(c-c_k-\eps c_1)$.
Because $\Imc c_1<0$, the zero of $\ell$ lies in the lower half-plane.
The boundary estimate $|D-\ell|\le\varrho(\eps)<d_1\min\{c_k,|\Imc c_1|\}\eps\le|\ell|$ uses only $\Imc c_1\neq0$, so \cref{lem:rouche} gives no zero of $D$ in $D^{+}_{R\eps}$.
Localization then excludes roots from the larger set $\{|c-c_k|\le r_0,\ \Imc c>0\}$.
\end{proof}

\begin{remark}\label{rem:suppression}
\cref{prop:suppression} excludes an unstable root near $c_k$ but does not produce a decaying eigenvalue.
Because the coefficients are real, complex phase speeds with $\Imc c\neq0$ occur in conjugate pairs (\cref{lem:determinant}\ref{it:det-hol}).
An isolated lower-half-plane partner would require the excluded upper-half-plane root.
The negative value $\Imc\Za^{+}(k,c_k)$ is the signed pressure work on a prescribed neutral wave, not an eigenvalue decay rate.
\end{remark}

\appendix
\section{Thermodynamic identities and the general stratified system}
\label{app:thermo}

In this appendix, we derive the normal-mode equations for a base state whose entropy varies with height, to show where the neutrality assumption of the main text enters.
In the main text, the entropy is constant and the Brunt--V\"ais\"al\"a frequency vanishes (Equation~\eqref{eq:neutral}).
Neutrality enters in two places.
First, it removes the buoyancy term $N^2$ from these equations, so the equation near a critical level is of Rayleigh type and not of Taylor--Goldstein type (\cref{rem:general-strat}).
Second, it gives the density gradient $\rho_0'/\rho_0=-g/a^2$, and with this gradient the critical-layer coefficient is $Q=U''+(g/a^2)U'$ of Equation~\eqref{eq:Gc}.
\cref{rem:leeslin} compares $Q$ with the coefficient $(\rho_0U')'/\rho_0$ of compressible boundary-layer theory.
We also collect the thermodynamic identities used in \S\ref{sec:structure}.

We replace the constant entropy $S_\ast$ of the laminar state in Equation~\eqref{eq:laminar} by an entropy profile $S_0(z)$ and keep the wind $U$.
As in \S\ref{subsec:laminar}, this base state is steady if and only if its density $\rho_0$ and its pressure $p_0=P(\rho_0,S_0)$ satisfy hydrostatic balance $p_0'=-\rho_0g$, the first relation in Equation~\eqref{eq:hydrostatic}.
We write $a^2=P_\rho(\rho_0,S_0)$ and $P_S=\partial_SP(\rho_0,S_0)$ for the derivatives of the equation of state at the base state.
The chain rule $p_0'=a^2\rho_0'+P_SS_0'$ turns hydrostatic balance into
\begin{equation}\label{eq:app-hydro}
a^2\rho_0'+P_SS_0'=-\rho_0g .
\end{equation}
The Brunt--V\"ais\"al\"a frequency $N$ is defined by $N^2=-g\bigl(\rho_0'/\rho_0+g/a^2\bigr)$, as in Equation~\eqref{eq:neutral}.
Dividing Equation~\eqref{eq:app-hydro} by $\rho_0a^2$ gives $N^2=gP_SS_0'/(\rho_0a^2)$.
For an ideal gas, $P_S>0$, so $N^2>0$ exactly when the entropy increases with height.
A constant entropy, $S_0'\equiv0$, gives $\rho_0'/\rho_0=-g/a^2$ and $N^2\equiv0$, as in Equation~\eqref{eq:neutral}.

For this base state, the linearized equations are those of \cref{prop:linearized}, except that the vertical advection of the base entropy adds the term $S_0'w$ to the entropy equation in Equation~\eqref{eq:lin-ent}, which becomes $\partial_ts+U\partial_xs+S_0'w=0$.
For a normal mode with $\Imc c\neq0$, let $\xi=w/(\im kV)$ be the vertical displacement amplitude of \cref{prop:reduction}.
The entropy equation becomes $\im kVs+S_0'w=0$, so $s=-S_0'\xi$, and the Lagrangian entropy variation $\Delta S:=s+S_0'\xi$ vanishes, as noted in the proof of \cref{lem:entropy}.
For the constant-entropy states of the main text, this gives $s=0$, as in \cref{lem:entropy}.
The Lagrangian pressure and density variations $\Delta p=q+p_0'\xi$ and $\Delta\rho=r+\rho_0'\xi$ of Equation~\eqref{eq:lagrangian} therefore satisfy
\begin{equation}
\Delta p-a^2\Delta\rho=(q-a^2r)+(p_0'-a^2\rho_0')\,\xi=P_S\,\Delta S=0,
\end{equation}
by the equation of state $q=a^2r+P_Ss$ and the chain rule above.
Solving $\Delta p=a^2\Delta\rho$ for $r$ with $p_0'=-\rho_0g$ gives
\begin{equation}\label{eq:app-r}
r=\frac{q}{a^2}-\rho_0\Bigl(\frac{\rho_0'}{\rho_0}+\frac g{a^2}\Bigr)\xi
=\frac q{a^2}+\frac{\rho_0N^2}{g}\,\xi .
\end{equation}
The last term is the buoyancy term, and it vanishes for a neutral atmosphere.

The proof of \cref{prop:xiq} eliminates $u$ and uses this expression for $r$, which gives the general first-order system in Equation~\eqref{eq:xiq-general}.
In this system, the buoyancy enters only through the term $-\rho_0N^2\xi$ in the equation for $q'$.
Let $X_{\mathrm{gen}}=e^{-\int_0^zb}\xi$ and $F_{\mathrm{gen}}=e^{2\int_0^zb}\rho_0a^2V^2/(a^2-V^2)$ be as in \cref{rem:leeslin}, where $b=g/a^2$.
The first equation of the system can then be written as $F_{\mathrm{gen}}X_{\mathrm{gen}}'=e^{\int_0^zb}q$, and the second equation gives the self-adjoint form
\begin{equation}\label{eq:app-selfadjoint}
\bigl(F_{\mathrm{gen}}X_{\mathrm{gen}}'\bigr)'
=e^{2\int_0^zb}\,\rho_0\bigl(k^2V^2-N^2\bigr)\,X_{\mathrm{gen}} ,
\end{equation}
which is the equation in \cref{rem:general-strat}.
In the neutral case, $e^{-\int_0^zb}=\rho_0/\rho_0(0)$, so $X_{\mathrm{gen}}=X/\rho_0(0)$ and $F_{\mathrm{gen}}=\rho_0(0)^2F$, and Equation~\eqref{eq:app-selfadjoint} reduces to Equation~\eqref{eq:Xform}.

Now let $c$ be real, and let $z_c$ be a regular critical level, where $V(z_c)=0$ and $U'(z_c)\neq0$.
Away from the critical levels, the derivation above applies without change.
At $z_c$, the leading coefficient vanishes quadratically, $F_{\mathrm{gen}}\sim e^{2\int_0^{z_c}b}\rho_0(z_c)U'(z_c)^2(z-z_c)^2$, where $\sim$ means, as in \S\ref{sec:structure}, that the ratio of the two sides tends to $1$ as $z\to z_c$.
Substituting $X_{\mathrm{gen}}=(z-z_c)^{s}$ into the leading terms of Equation~\eqref{eq:app-selfadjoint} gives the Frobenius exponents $s$, which depend on how the coefficient on the right side behaves at $z_c$.
If $N^2(z_c)\neq0$, this coefficient does not vanish at $z_c$, and the exponents satisfy $s(s+1)=-\mathrm{Ri}_c$ with $\mathrm{Ri}_c=N^2(z_c)/U'(z_c)^2$.
The roots $s=-\tfrac12\pm\sqrt{\tfrac14-\mathrm{Ri}_c}$ are the Taylor--Goldstein exponents of \cref{rem:general-strat}.
If $N^2\equiv0$, this coefficient vanishes quadratically at $z_c$, like $F_{\mathrm{gen}}$, and the exponents satisfy $s(s+1)=0$, so they are $s=0$ and $s=-1$, as for Equation~\eqref{eq:Xform}.
In this case, the leading terms of a solution satisfy $F_{\mathrm{gen}}X_{\mathrm{gen}}'=\mathrm{const}$ near $z_c$.
Dividing by $F_{\mathrm{gen}}$ and integrating shows that the Frobenius solution with $X_{\mathrm{gen}}\sim(z-z_c)^{-1}$ also contains the term $\log|z-z_c|$ with the coefficient $\lim_{z\to z_c}\frac{\dd}{\dd z}\log\bigl(F_{\mathrm{gen}}/(z-z_c)^2\bigr)$.
This coefficient is the first-order coefficient in the expansion of $F_{\mathrm{gen}}/(z-z_c)^2$ about $z_c$ divided by its value at $z_c$.
It equals $Q_{\mathrm{gen}}(z_c)/U'(z_c)$, with $Q_{\mathrm{gen}}$ as in \cref{rem:leeslin}, because
\begin{equation}
U'(z_c)\lim_{z\to z_c}\frac{\dd}{\dd z}
\log\Bigl(\frac{F_{\mathrm{gen}}}{(z-z_c)^2}\Bigr)
=U''(z_c)+\Bigl(\frac{2g}{a^2(z_c)}+\frac{\rho_0'(z_c)}{\rho_0(z_c)}\Bigr)U'(z_c)
=Q_{\mathrm{gen}}(z_c),
\end{equation}
where the terms with $(a^2)'$ cancel because $(V^2)'(z_c)=0$.
For a neutral atmosphere, $\rho_0'/\rho_0=-g/a^2$, so $Q_{\mathrm{gen}}=U''+(g/a^2)U'$ is the coefficient $Q$ of Equation~\eqref{eq:Gc}.
This agrees with the logarithmic term in Equation~\eqref{eq:frobenius}, because $X_{\mathrm{gen}}$ is then proportional to $\rho_0w/V$.

\section{Technical estimates for limiting absorption and the
low-Mach limit}
\label{app:proofs}

In this appendix, we prove the estimates that were deferred from \S\ref{sec:criticallayer} and \S\ref{sec:lowmach}.
As in \S\ref{sec:criticallayer}, $C$ denotes a positive constant that may change from line to line.
The first three subsections complete the local analysis near a critical level in \S\ref{subsec:local}.
Section~\ref{subsec:B-equiv} proves that a solution of the integral representation of \cref{lem:representation} solves the differential system, \S\ref{subsec:logextract} proves the logarithmic decomposition of the singular integral, and \S\ref{subsec:log-compare} compares this decomposition for $\delta>0$ and $\delta=0$, which gives the convergence rate in \cref{prop:convergence}.
Section~\ref{subsec:coef-perturb} shows that the local solutions depend Lipschitz-continuously on the coefficients of the equations, which the low-Mach limit of \S\ref{sec:lowmach} requires.
Section~\ref{subsec:pullback} proves that the local solutions depend $C^1$ on the real part $c_r$ of the phase speed, as stated in \cref{prop:convergence,prop:limit-solutions}.
Section~\ref{subsec:uniformM} combines these results into the statements uniform in the Mach number that the proof of \cref{thm:lowmach} uses.

\subsection{Equivalence of the integral representation}
\label{subsec:B-equiv}

In this subsection, we verify in detail the converse part of \cref{lem:representation}, which states that a pair satisfying the integral representation in Equation~\eqref{eq:rep} solves the system in Equation~\eqref{eq:wchi-recall} wherever $V\neq0$.
By this converse, the fixed points of \cref{lem:fixedpoint} solve the differential system, also at $\delta=0$ away from the critical level.

Let $(w,\chibar)$ satisfy the system in Equation~\eqref{eq:rep}, and let $z\in I_j$ with $V(z)\neq0$, that is, any $z\in I_j$ when $\delta>0$ and any $z\neq z_j$ when $\delta=0$.
Differentiating Equation~\eqref{eq:rep-chi} gives the second equation of the system in Equation~\eqref{eq:wchi-recall}, $\chibar'=-b\chibar+\rhobar k^2Vw$, directly.
When we differentiate Equation~\eqref{eq:rep-w} and use $(VE)'=(V'/V+b)\,VE$, which follows from $E'=bE$, we obtain
\begin{equation}
w'=\Bigl(\frac{V'}V+b\Bigr)
\Bigl[\alpha VE+VE\!\int_{z_j+\varsigma}^{z}\bigl(\tfrac{g_1\chibar}{V}+g_2w\bigr)\Bigr]
-\Bigl(\frac{Ef\chibar}{U'}\Bigr)'
+E\,\bigl[g_1\chibar+Vg_2w\bigr].
\end{equation}
We substitute $\chibar'$ from the second equation into the term $-(Ef\chibar/U')'$ and compare the result with the first equation, $w'=(V'/V+b)w+\chibar/(\Abar V)$, in which we insert Equation~\eqref{eq:rep-w} for $w$.
The terms containing $w$ cancel because $-(Ef/U')\rhobar k^2V+VEg_2=0$ by the definition of $g_2$ in Equation~\eqref{eq:g-def}.
The terms containing $\chibar$ cancel because
\begin{equation}
-\Bigl(\frac{Ef}{U'}\Bigr)'+\frac{bEf}{U'}+Eg_1
+\Bigl(\frac{U'}{V}+b\Bigr)\frac{Ef}{U'}-\frac1{\Abar V}=0 .
\end{equation}
To verify this identity, we use $E'=bE$, $g_1=(f/U')'-bf/U'$, and $Ef=1/\Abar$.
The first three terms then add up to $-bEf/U'$, and the last two add up to $bEf/U'$.
Hence $(w,\chibar)$ solves the system in Equation~\eqref{eq:wchi-recall} wherever $V\neq0$.

\subsection{Logarithmic decomposition of the singular integral}
\label{subsec:logextract}

In this subsection, we prove the decomposition in Equation~\eqref{eq:logextract} and the regularity of its remainder in Equation~\eqref{eq:remainder-C1}.
We write $\dd s=\dd\tau/U'$ and use the function $\mathfrak g=(\psi/U')\circ\tau^{-1}$ of \S\ref{subsec:local}, which is $\psi/U'$ written in the variable $\tau$, so that $\mathfrak g(0)=\psi(z_j)/U'(z_j)$.
The idea is to split $\mathfrak g$ into its value $\mathfrak g(0)$ at the critical level, where $\tau=0$, and the difference $\mathfrak g-\mathfrak g(0)$.
The constant part integrates to an explicit logarithm, and the difference vanishes at the critical level, so its quotient by $\tau-\im\delta$ is bounded.

Let $\tau_+=\tau(z_j+\varsigma)$, as in the proof of \cref{prop:convergence}.
For $\delta>0$, the variable $\tau'-\im\delta$ stays in the open lower half-plane, where the principal logarithm is holomorphic, so
\begin{equation}
\begin{aligned}
\int_{z_j+\varsigma}^{z}\frac{\psi}{V}\,\dd s
&=\int_{\tau_+}^{\tau(z)}\frac{\mathfrak g(\tau')}{\tau'-\im\delta}\,\dd\tau'\\
&=\mathfrak g(0)\bigl[\operatorname{Log}(\tau(z)-\im\delta)
-\operatorname{Log}(\tau_+-\im\delta)\bigr]
+\int_{\tau_+}^{\tau(z)}
\frac{\mathfrak g(\tau')-\mathfrak g(0)}{\tau'-\im\delta}\,\dd\tau' .
\end{aligned}
\end{equation}
This is Equation~\eqref{eq:logextract}, and the last integral is the remainder $\mathcal R[\psi](z)$.
Since $|\mathfrak g(\tau')-\mathfrak g(0)|/|\tau'-\im\delta| \le|\mathfrak g(\tau')-\mathfrak g(0)|/|\tau'| \le\mathrm{Lip}(\mathfrak g)$ and the range of integration has length at most $C\varsigma$, the remainder is bounded by $C\varsigma\,\mathrm{Lip}(\mathfrak g)$, uniformly in $\delta\in[0,\delta_1]$.
At $\delta=0$, the logarithms are the boundary values from below given by Equation~\eqref{eq:branch}, which are their limits as $\delta\downarrow0$.
The remainder at $\delta=0$ is the limit of the remainders for $\delta>0$ by the dominated convergence theorem, because the integrands converge for $\tau'\neq0$ and are bounded by $\mathrm{Lip}(\mathfrak g)$.
Hence Equation~\eqref{eq:logextract} at $\delta=0$ is the limit of the identities for $\delta>0$.

For Equation~\eqref{eq:remainder-C1}, we differentiate $\mathcal R[\psi]$ at $\delta=0$ and $z\neq z_j$ and obtain
\begin{equation}
\mathcal R[\psi]'(z)
=U'(z)\,\frac{\mathfrak g(\tau(z))-\mathfrak g(0)}{\tau(z)}
=\frac{\psi(z)-\psi(z_j)\,U'(z)/U'(z_j)}{\tau(z)} .
\end{equation}
If $\psi\in C^1(I_j)$, then $\mathfrak g\in C^1$ by Hypothesis \ref{hyp:reg}, so the difference quotient $\bigl(\mathfrak g(\tau)-\mathfrak g(0)\bigr)/\tau$ converges to $\mathfrak g'(0)$ as $\tau\to0$ from either side.
Hence $\mathcal R[\psi]'(z)\to U'(z_j)\,\mathfrak g'(0) =[\psi'(z_j)U'(z_j)-\psi(z_j)U''(z_j)]/U'(z_j)^2$, the same limit from both sides.
Since $\mathcal R[\psi]$ is continuous at $z_j$, it is therefore continuously differentiable across $z_j$.

\subsection{Comparison of the logarithmic decompositions for \texorpdfstring{$\delta>0$}{delta>0} and \texorpdfstring{$\delta=0$}{delta=0}}
\label{subsec:log-compare}

In this subsection, we prove the two estimates that the proof of \cref{prop:convergence} uses to compare the local solutions for $\delta>0$ with their limits at $\delta=0$.
The kernels $1/V_\delta$ and $1/V_0$ of the singular integral cannot be compared directly, because their pointwise difference $|1/V_\delta-1/V_0|=\delta/(|\tau||\tau-\im\delta|)$ is not integrable across the critical level.
We therefore compare the two parts of the decomposition in Equation~\eqref{eq:logextract} separately.
The logarithmic terms, which appear multiplied by $V$ in Equation~\eqref{eq:rep-w}, are compared in Equation~\eqref{eq:log-model}, and the remainders in Equation~\eqref{eq:remainder-compare}.

The logarithmic terms satisfy the model bound
\begin{equation}\label{eq:log-model}
\sup_{z\in I_j}
\bigl|(\tau-\im\delta)\operatorname{Log}(\tau-\im\delta)
-\tau\operatorname{Log}(\tau-\im0)\bigr|
\;\le\;C\,\delta\bigl(1+\bigl|\log(\delta/V_\ast)\bigr|\bigr),
\qquad 0<\delta\le\delta_1 .
\end{equation}
To prove it, let $\tau$ be given.
The function $\delta'\mapsto(\tau-\im\delta')\operatorname{Log}(\tau-\im\delta')$ is continuous on $[0,\delta]$, with the boundary value given by Equation~\eqref{eq:branch} at $\delta'=0$, and differentiable for $\delta'>0$ with
\begin{equation}
\partial_{\delta'}
\bigl[(\tau-\im\delta')\operatorname{Log}(\tau-\im\delta')\bigr]
=-\im\bigl[\operatorname{Log}(\tau-\im\delta')+1\bigr].
\end{equation}
The modulus of this derivative is at most $|\log(\delta'/V_\ast)|+C$, because $\delta'\le|\tau-\im\delta'|\le CV_\ast$ on $I_j$ and $V_\ast=m_U\varsigma$.
This bound is integrable at $\delta'=0$, and for $\delta\le\delta_1\le V_\ast$ we have $\int_0^\delta|\log(\delta'/V_\ast)|\,\dd\delta' =\delta\bigl(1+|\log(\delta/V_\ast)|\bigr)$.
Integrating the derivative over $\delta'\in[0,\delta]$ therefore gives Equation~\eqref{eq:log-model}.

For the remainders, let $\psi$ be a Lipschitz function, the same for $\delta>0$ and $\delta=0$.
Subtracting the two decompositions, written in the variable $\tau$, gives
\begin{equation}\label{eq:remainder-compare}
\mathcal R_\delta[\psi](z)-\mathcal R_0[\psi](z)
=\int_{\tau_+}^{\tau(z)}
\bigl(\mathfrak g(t)-\mathfrak g(0)\bigr)
\Bigl(\frac1{t-\im\delta}-\frac1{t}\Bigr)\dd t ,
\end{equation}
where the integral at $\delta=0$ is the boundary value.
Since $|\mathfrak g(t)-\mathfrak g(0)|\le\mathrm{Lip}(\mathfrak g)\,|t|$ and $|1/(t-\im\delta)-1/t|=\delta/(|t|\sqrt{t^2+\delta^2})$, the integrand is bounded pointwise by $\mathrm{Lip}(\mathfrak g)\,\delta/\sqrt{t^2+\delta^2}$, which is integrable across $t=0$.
Since $|\tau|\le CV_\ast$ on $I_j$, it follows that
\begin{equation}
\bigl\|\mathcal R_\delta[\psi]-\mathcal R_0[\psi]\bigr\|_{C^0(I_j)}
\le\mathrm{Lip}(\mathfrak g)\,\delta
\int_{-CV_\ast}^{CV_\ast}\frac{\dd t}{\sqrt{t^2+\delta^2}}
\le C\,\mathrm{Lip}(\mathfrak g)\,\delta\bigl(1+|\log(\delta/V_\ast)|\bigr).
\end{equation}
Both estimates are uniform for $c_r$ in compact sets of regular values.

\subsection{A coefficient-perturbation lemma for the local solutions}
\label{subsec:coef-perturb}

In this subsection, we show that the local solutions and the maps built from them depend Lipschitz-continuously on the coefficient functions of the representation in Equation~\eqref{eq:rep} (\cref{lem:coef-perturb}).
The low-Mach limit of \S\ref{sec:lowmach} needs this estimate, because there the coefficients depend on the Mach number and we compare the local solutions of two atmospheres.
We therefore regard the coefficient functions as independent functions.

Let $I_j$ and $\delta\in[0,\delta_1]$ be given.
We consider tuples $\mathfrak a=(b,\rhobar,f,g_1,g_2)$ of functions on $I_j$ that share $U$ and $c$, and hence $V$, $\tau$, the base point $z_j$, and the branch of the logarithm.
We assume that each tuple satisfies the bounds of \S\ref{subsec:local} with a common constant, that $\varsigma$ is small in terms of this constant, as in \S\ref{subsec:local}, and that the map $\mathcal B$ of \cref{lem:fixedpoint}, built from the tuple, is a $\tfrac12$-contraction, that is, it reduces the distance between any two functions at least by the factor $\tfrac12$.
We also assume that each tuple belongs to a one-parameter family $(\mathfrak a_\delta)_{\delta\in[0,\delta_1]}$ that obeys the same bounds uniformly and depends on $\delta$ as in Equation~\eqref{eq:coef-delta}.
In \S\ref{subsec:local}, the tuple is computed from one atmosphere, with $f=1/(\Abar E)$ and with $g_1$ and $g_2$ given by Equation~\eqref{eq:g-def}.
We say that a tuple satisfies the first-order-system identities when $g_1$ and $g_2$ are obtained from $b$, $\rhobar$, and $f$ by Equation~\eqref{eq:g-def}, that is, $g_1=(f/U')'-bf/U'$ and $g_2=\rhobar k^2f/U'$.
For such a tuple, we compute $E$ from $b$ as in \S\ref{subsec:local} and set $\Abar=1/(fE)$.
Then, by \S\ref{subsec:B-equiv}, every solution of the system in Equation~\eqref{eq:rep} solves the system in Equation~\eqref{eq:wchi-recall} with coefficients $(\Abar,b,\rhobar)$ wherever $V\neq0$.
We measure the distance between two tuples $\mathfrak a$ and $\tilde{\mathfrak a}$ by
\begin{equation}
\|\mathfrak a-\tilde{\mathfrak a}\|_{\mathcal X}
:=\|b-\tilde b\|_{C^0}+\|\rhobar-\tilde\rhobar\|_{C^0}
+\|f-\tilde f\|_{C^1}+\|g_1-\tilde g_1\|_{C^{0,1}}
+\|g_2-\tilde g_2\|_{C^0},
\end{equation}
where all norms are taken on $I_j$.

\begin{lemma}\label{lem:coef-perturb}
In this setting, the following bounds hold uniformly in $\delta\in[0,\delta_1]$ and in $c_r$ on compact sets of regular values.
\begin{enumerate}[label=\textup{(\roman*)}]
\item\label{it:cp-fixed} The fixed points of \cref{lem:fixedpoint} with the same constants $(\alpha,\beta)$ satisfy
\begin{equation}
\|w_{\mathfrak a}-w_{\tilde{\mathfrak a}}\|_{C^0(I_j)}
+\|\chibar_{\mathfrak a}-\chibar_{\tilde{\mathfrak a}}\|_{C^0(I_j)}
\le C\,(|\alpha|+|\beta|)\,
\|\mathfrak a-\tilde{\mathfrak a}\|_{\mathcal X}\, .
\end{equation}
In particular, the critical values satisfy $|w_{\mathfrak a}(z_j)-w_{\tilde{\mathfrak a}}(z_j)|\le C(|\alpha|+|\beta|)\|\mathfrak a-\tilde{\mathfrak a}\|_{\mathcal X}$.
\item\label{it:cp-matrices} Suppose that both families $(\mathfrak a_\delta)$ and $(\tilde{\mathfrak a}_\delta)$ satisfy the first-order-system identities for every $\delta\in[0,\delta_1]$.
Then the endpoint evaluation matrices
\begin{equation}
\mathcal E^{\pm}_{\mathfrak a,\delta}:\
(\alpha,\beta)\longmapsto(w,\chibar)(z_j\pm\varsigma)
\end{equation}
satisfy
\begin{equation}
\|\mathcal E^{\pm}_{\mathfrak a,\delta}\|
+\|(\mathcal E^{\pm}_{\mathfrak a,\delta})^{-1}\|\le C,
\qquad
\|\mathcal E^{\pm}_{\mathfrak a,\delta}-\mathcal E^{\pm}_{\tilde{\mathfrak a},\delta}\|
+\|(\mathcal E^{\pm}_{\mathfrak a,\delta})^{-1}
-(\mathcal E^{\pm}_{\tilde{\mathfrak a},\delta})^{-1}\|
\le C\,\|\mathfrak a-\tilde{\mathfrak a}\|_{\mathcal X}\, .
\end{equation}
\item\label{it:cp-transfer} Under the assumption of \ref{it:cp-matrices}, the maps $T=\mathcal E^{-}\circ(\mathcal E^{+})^{-1}$ of \cref{prop:convergence}, which carry the values at $z_j+\varsigma$ to the values at $z_j-\varsigma$, and their inverses are bounded and satisfy the same Lipschitz bound in $\mathfrak a$.
\end{enumerate}
\end{lemma}

\begin{proof}
For \ref{it:cp-fixed}, let $\mathcal B_{\mathfrak a}$ denote the map of \cref{lem:fixedpoint} built from $\mathfrak a$, and set $w:=w_{\tilde{\mathfrak a}}$.
The uniform $\tfrac12$-contraction gives $\|w_{\mathfrak a}-w\|_\infty \le2\,\|\mathcal B_{\mathfrak a}[w] -\mathcal B_{\tilde{\mathfrak a}}[w]\|_\infty$, and $\|w\|_\infty\le C(|\alpha|+|\beta|)$ by \cref{lem:fixedpoint}.
We compare the terms of the two maps applied to this $w$.
The functions $E$ and $\tilde E$ computed from $b$ and $\tilde b$ obey $\|E-\tilde E\|_{C^1}\le C\|b-\tilde b\|_{C^0}$.
The functions $\chibar_{\mathfrak a}[w]$ and $\chibar_{\tilde{\mathfrak a}}[w]$ defined by Equation~\eqref{eq:rep-chi} take the same value $\beta$ at $z_j$ and obey $\|\chibar_{\mathfrak a}[w]-\chibar_{\tilde{\mathfrak a}}[w]\|_\infty +\mathrm{Lip}\bigl(\chibar_{\mathfrak a}[w] -\chibar_{\tilde{\mathfrak a}}[w]\bigr) \le C\,(|\beta|+\|w\|_\infty)\, \|\mathfrak a-\tilde{\mathfrak a}\|_{\mathcal X}$.
The regular terms of $\mathcal B$, namely the prefactor $\alpha VE$, the term $-Ef\chibar/U'$, and the $g_2$-integral, therefore contribute $C(|\alpha|+|\beta|)\|\mathfrak a-\tilde{\mathfrak a}\|_{\mathcal X}$.
A singular integral $V\int_{z_j+\varsigma}^{z}\varphi/V\,\dd s$ with a Lipschitz function $\varphi$ is bounded by $C\|\varphi\|_{C^{0,1}}$, by the decomposition in Equation~\eqref{eq:logextract} and $\sup_{I_j}|V|\bigl(1+\bigl|\log|V|\bigr|\bigr) \le C\varsigma\log(1/\varsigma)$.
The singular terms of the two maps differ by $V(E-\tilde E)\int g_1\chibar_{\mathfrak a}[w]/V$ and $V\tilde E\int\psi/V$, with the difference $\psi:=g_1\chibar_{\mathfrak a}[w] -\tilde g_1\chibar_{\tilde{\mathfrak a}}[w]$.
The first is bounded by $C\|b-\tilde b\|_{C^0}(|\alpha|+|\beta|)$.
The function $\psi$ is Lipschitz with $\|\psi\|_{C^{0,1}}\le C(|\alpha|+|\beta|) \|\mathfrak a-\tilde{\mathfrak a}\|_{\mathcal X}$, and its value at the critical level is $\psi(z_j)=(g_1-\tilde g_1)(z_j)\,\beta$, so the second is bounded by the same quantity.
This proves the bound for $w_{\mathfrak a}-w_{\tilde{\mathfrak a}}$.
A linear map between normed spaces is bounded if the norm of the image of each element is at most a constant times the norm of the element, and a bounded linear functional is a bounded linear map into $\C$.
The bound for $\chibar_{\mathfrak a}-\chibar_{\tilde{\mathfrak a}}$ follows, because the map $w\mapsto\chibar_{\mathfrak a}[w]$ is affine with a bounded linear part and the functions $\chibar_{\mathfrak a}[w]$ and $\chibar_{\tilde{\mathfrak a}}[w]$ obey the estimate above.
Evaluation at $z_j$ is a bounded linear functional on $C^0(I_j)$, which gives the bound for the critical values.
At $\delta=0$, the explicit formula in Equation~\eqref{eq:wc-beta} exhibits the same Lipschitz dependence through $(f,\rhobar)(z_j)$.

For \ref{it:cp-matrices}, the columns of $\mathcal E^{\pm}_{\mathfrak a,\delta}$ are the values at $z_j\pm\varsigma$ of the fixed points with $(\alpha,\beta)=(1,0)$ and $(0,1)$.
They are bounded by \cref{lem:fixedpoint} and Lipschitz in $\mathfrak a$ by \ref{it:cp-fixed}, so $\|\mathcal E^{\pm}_{\mathfrak a,\delta}\|\le C$ and $\|\mathcal E^{\pm}_{\mathfrak a,\delta} -\mathcal E^{\pm}_{\tilde{\mathfrak a},\delta}\|\le C\|\mathfrak a-\tilde{\mathfrak a}\|_{\mathcal X}$.
For tuples satisfying the first-order-system identities, the fixed points solve the system in Equation~\eqref{eq:wchi-recall} by \S\ref{subsec:B-equiv}, and the identity $\Xi'=(V'/V)\Xi$ for the modified Wronskian of \cref{prop:local}\ref{it:basis-wronskian} applies.
The proof of \cref{prop:convergence}\ref{it:conv-matrices}, with constants that depend only on the bounds of \S\ref{subsec:local}, then gives $\det\mathcal E^{\pm}_{\mathfrak a,\delta} =\varkappa_{\mathfrak a,\delta}V_\delta(z_j\pm\varsigma)$ with $|\varkappa_{\mathfrak a,\delta}|\ge\tfrac12$.
Hence $|\det\mathcal E^{\pm}_{\mathfrak a,\delta}|\ge\tfrac12m_U\varsigma$, and the formula for the inverse of a $2\times2$ matrix gives $\|(\mathcal E^{\pm}_{\mathfrak a,\delta})^{-1}\| \le C\|\mathcal E^{\pm}_{\mathfrak a,\delta}\|/|\det\mathcal E^{\pm}_{\mathfrak a,\delta}|\le C$.
The Lipschitz bound for the inverses follows from the identity $\mathcal E^{-1}-\tilde{\mathcal E}^{-1} =\mathcal E^{-1}(\tilde{\mathcal E}-\mathcal E)\tilde{\mathcal E}^{-1}$.

For \ref{it:cp-transfer}, the maps $T=\mathcal E^{-}\circ(\mathcal E^{+})^{-1}$ and $T^{-1}=\mathcal E^{+}\circ(\mathcal E^{-})^{-1}$ are products of matrices that are bounded and Lipschitz in $\mathfrak a$ by \ref{it:cp-matrices}, so they have both properties as well.
\end{proof}

\subsection{Parameter dependence in moving coordinates}
\label{subsec:pullback}

In this subsection, we prove that the local solutions depend $C^1$ on the real part $c_r$ of the phase speed, as stated in \cref{prop:convergence,prop:limit-solutions}.
Two features of the original coordinates prevent direct differentiation.
The base point $z_j(c_r)$ and the interval $I_j(c_r)=[z_j(c_r)-\varsigma,z_j(c_r)+\varsigma]$ move with $c_r$, and the derivative $\partial_{c_r}[\tau\operatorname{Log}(\tau-\im0)]$ at fixed $z$ is singular at the critical level.
We therefore pull each critical neighborhood back to a fixed interval, on which neither the base point nor the endpoints move.
In the new variable, the derivative of $\tau$ with respect to $c_r$ vanishes at the critical level, and this factor cancels the logarithmic singularity.

\begin{lemma}\label{lem:pullback}
Let $K$ be a compact set of regular values of $U$ with the same number of critical levels, and pull each critical neighborhood back to the common interval $[-\varsigma,\varsigma]$ by $z=z_j(c_r)+x$.
Assume that the coefficient tuple $\mathfrak a=(b,\rhobar,f,g_1,g_2)$ is, in the pulled-back variables, $C^1$ in $c_r$ with values in $C^1([-\varsigma,\varsigma])$, and that it ranges in a bounded set on which the map $\mathcal B$ of \cref{lem:fixedpoint} is a $\tfrac12$-contraction.
Then the following hold uniformly for $\delta\in[0,\delta_1]$ and $c_r\in K$.
\begin{enumerate}[label=\textup{(\roman*)}]
\item\label{it:pb-map} The pulled-back map $\mathcal B$ is $C^1$ in $c_r$ as a map into $C^0([-\varsigma,\varsigma])$, with uniformly bounded derivative.
\item\label{it:pb-fixed} The pulled-back fixed points, as maps into $C^0([-\varsigma,\varsigma])$, the endpoint evaluation matrices $\mathcal E^{\pm}$, and the critical values $w(z_j(c_r))$ are $C^1$ in $c_r$ with uniformly bounded derivatives.
\item\label{it:pb-inverse} If, in addition, the tuple satisfies the first-order-system identities for every $c_r\in K$ and every $\delta\in[0,\delta_1]$, then the endpoint matrices are uniformly invertible, and their inverses and the maps $T=\mathcal E^{-}(\mathcal E^{+})^{-1}$ are $C^1$ in $c_r$ with uniformly bounded derivatives.
\end{enumerate}
\end{lemma}

The assumptions hold for the coefficients of \S\ref{subsec:local} under Hypotheses \ref{hyp:reg} and \ref{hyp:subsonic}.
Under these hypotheses, $b$, $\rhobar$, $f$, and $g_2$ lie in $C^3$ and $g_1$ lies in $C^2$ as functions of $z$, with derivatives that depend continuously on $c$, because $c$ enters only through $V^2$ in $\Abar$.
Differentiating the pullback $\varphi(z_j(c_r)+x)$ of such a function $\varphi$ in $c_r$ uses one derivative in $z$ and the $C^1$ dependence of $z_j(c_r)$, so the pulled-back tuple is $C^1$ in $c_r$ with values in $C^1([-\varsigma,\varsigma])$.
The proof treats the three items in order.
For \ref{it:pb-map}, we differentiate the logarithmic term, the remainder integral, and the regular terms of the pulled-back map separately.
For \ref{it:pb-fixed}, we differentiate the fixed point of this affine map, and for \ref{it:pb-inverse}, we bound the determinants of the endpoint matrices from below with the Wronskian.

\begin{proof}
Set $I:=[-\varsigma,\varsigma]$ and write $z=z_j(c_r)+x$ with $x\in I$.
A tilde marks a pullback, $\tilde\varphi(x,c_r):=\varphi(z_j(c_r)+x)$.
Every bound is uniform in $(c_r,\delta)\in K\times[0,\delta_1]$.
Differentiating $U(z_j(c_r))=c_r$ gives $z_j'(c_r)=1/U'(z_j(c_r))$.

In the variable $x$, the interval, the base point, and both endpoints do not move.
Only the functions of $x$ in the formulas below depend on $c_r$.
We write $\tilde\tau(x,c_r):=U(z_j(c_r)+x)-c_r$, $\tilde V:=\tilde\tau-\im\delta$, and $\tilde B(x):=\int_0^x\tilde b$, and we denote by $\tilde\chi[w]$ the pullback of the map $w\mapsto\chibar[w]$ of the proof of \cref{lem:fixedpoint} and by $\mathcal B_{c_r,\delta}$ the pullback of the map $\mathcal B$.
On the common space $C^0(I)$, these maps read
\begin{align*}
\tilde\chi[w](x)&=e^{-\tilde B(x)}\beta
+\int_0^x e^{\tilde B(x')-\tilde B(x)}\,
\tilde\rhobar\,k^2\,\tilde V(x')\,w(x')\,\dd x',\\
\mathcal B_{c_r,\delta}[w](x)&=\alpha\,\tilde V\tilde E
-\frac{\tilde E\tilde f\,\tilde\chi[w]}{\tilde U'}
+\tilde V(x)\tilde E(x)\int_{\varsigma}^{x}
\Bigl[\frac{\tilde g_1\tilde\chi[w]}{\tilde V}
+\tilde g_2\,w\Bigr]\dd x' .
\end{align*}

For \ref{it:pb-map}, let $w\in C^0(I)$ and set $\psi:=\tilde g_1\tilde\chi[w]$.
When we change variables in Equation~\eqref{eq:logextract} back to $x$ through $t=\tilde\tau(x',c_r)$ and $\dd t=\tilde U'\dd x'$, the singular integral splits as
\begin{equation}
\int_\varsigma^x\frac{\psi(x')}{\tilde V(x')}\,\dd x'
=\frac{\psi(0)}{\tilde U'(0)}
\Bigl[\operatorname{Log}\bigl(\tilde\tau(x)-\im\delta\bigr)
-\operatorname{Log}\bigl(\tilde\tau(\varsigma)-\im\delta\bigr)\Bigr]
+\int_\varsigma^x\frac{\mathfrak h(x',c_r)}{\tilde V(x')}\,\dd x',
\end{equation}
where
\begin{equation}
\mathfrak h(x',c_r):=\psi(x')-\psi(0)\,
\frac{\tilde U'(x')}{\tilde U'(0)} .
\end{equation}
Then $\mathfrak h(0,c_r)=0$, and $|\mathfrak h(x')|\le C\,|x'|\,\bigl(|\beta|+\|w\|_\infty\bigr)$ because $\psi$ and $\tilde U'$ are Lipschitz, with $\mathrm{Lip}(\psi)\le C(|\beta|+\|w\|_\infty)$ by Equation~\eqref{eq:chi-of-w}.

Set $\theta(x,c_r):=\tilde U'(x)/\tilde U'(0)-1$.
Then $\partial_{c_r}\tilde\tau=\theta$, $\theta(0,c_r)=0$, $|\theta(x)|\le C|x|$, and $\theta$ is uniformly Lipschitz in $x$.
Moreover, $\partial_{c_r}\tilde V=\theta$ for every $\delta$.
The pulled-back coefficients, $\tilde E$, and $e^{\pm\tilde B}$ are $C^1$ in $c_r$ with the stated function-space regularity.
Differentiating the formula for $\tilde\chi[w]$ above, whose integrand has no singularity, shows that $\tilde\chi[w]$ is $C^1$ in $c_r$ and that $\partial_{c_r}\tilde\chi[w]$ is Lipschitz in $x$, with norm bounded by $C(|\beta|+\|w\|_\infty)$.
Consequently $\partial_{c_r}\psi$ is Lipschitz in $x$, and $\psi(0)/\tilde U'(0)=\tilde g_1(0)\beta/\tilde U'(0)$ is $C^1$ in $c_r$.

For the logarithmic term, we have, for every $\delta\in[0,\delta_1]$ and $x\neq0$,
\begin{equation}
\partial_{c_r}\bigl[\tilde V
\operatorname{Log}(\tilde\tau-\im\delta)\bigr]
=\theta\,\bigl[\operatorname{Log}(\tilde\tau-\im\delta)+1\bigr],
\qquad
\bigl|\theta\,[\operatorname{Log}(\tilde\tau-\im\delta)+1]\bigr|
\le C|x|\Bigl(1+\bigl|\log(m_U|x|/V_\ast)\bigr|\Bigr),
\end{equation}
which extends continuously by $0$ across $x=0$, uniformly in $\delta$.
Differentiation in $c_r$ at fixed $x$ thus produces the factor $\theta=O(x)$, which cancels the logarithmic singularity.
The endpoint value $\operatorname{Log}(\tilde\tau(\varsigma)-\im\delta)$ is $C^1$ in $c_r$ because $|\tilde\tau(\varsigma)|\ge m_U\varsigma$.
Since the factors $\tilde E$ and $\psi(0)/\tilde U'(0)$ are $C^1$ in $c_r$, the logarithmic term of $\mathcal B_{c_r,\delta}[w]$ is $C^1$ in $c_r$ as a map into $C^0(I)$, with uniformly bounded, uniformly continuous derivative.

For the remainder integral, we differentiate under the integral sign and use
\begin{equation}
\partial_{c_r}\Bigl[\frac{\mathfrak h}{\tilde V}\Bigr]
=\frac{\partial_{c_r}\mathfrak h}{\tilde V}
-\frac{\mathfrak h\,\theta}{\tilde V^{2}} .
\end{equation}
For the second piece, $|\mathfrak h|\le C|x'|$, $|\theta|\le C|x'|$, and $|\tilde V|\ge|\tilde\tau|\ge m_U|x'|$ give the bound $C/m_U^2$.
For the first piece,
\begin{equation}
\partial_{c_r}\mathfrak h
=\partial_{c_r}\psi
-\partial_{c_r}\psi(0)\,\frac{\tilde U'}{\tilde U'(0)}
-\psi(0)\,\partial_{c_r}\Bigl[\frac{\tilde U'(x')}{\tilde
U'(0)}\Bigr],
\end{equation}
which vanishes at $x'=0$, because the first two terms cancel there, and $\partial_{c_r}[\tilde U'(x')/\tilde U'(0)] =\bigl[U''(z_j+x')\,U'(z_j)-U'(z_j+x')\,U''(z_j)\bigr] z_j'/U'(z_j)^{2}$ vanishes at $x'=0$ as well.
Since $\partial_{c_r}\mathfrak h$ is Lipschitz in $x'$ uniformly, $|\partial_{c_r}\mathfrak h(x')|\le C|x'|$, and the first piece is bounded by $C/m_U$.
The mean value theorem therefore bounds the $c_r$ difference quotients by an integrable function that does not depend on the increment of $c_r$, uniformly in $\delta$.
The dominated convergence theorem yields
\begin{equation}
\partial_{c_r}\int_\varsigma^x\frac{\mathfrak h}{\tilde V}\,\dd x'
=\int_\varsigma^x\partial_{c_r}
\Bigl[\frac{\mathfrak h}{\tilde V}\Bigr]\dd x',
\end{equation}
which is continuous in $(x,c_r)$ and bounded by $C\varsigma$, uniformly in $\delta\in[0,\delta_1]$.
No endpoint or base-point terms arise because both are stationary in the $x$-coordinate.

The remaining terms of $\mathcal B_{c_r,\delta}[w]$, namely $\alpha\tilde V\tilde E$, $-\tilde E\tilde f\tilde\chi[w]/\tilde U'$, and the $\tilde g_2$-integral, are $C^1$ in $c_r$ as maps into $C^0(I)$ by the preceding coefficient estimates.
The logarithmic and remainder estimates show that, for each $w$, $c_r\mapsto\mathcal B_{c_r,\delta}[w]\in C^0(I)$ is $C^1$ and
\begin{equation}
\|\partial_{c_r}\mathcal B_{c_r,\delta}[w]\|_{C^0(I)}
\le C\bigl(|\alpha|+|\beta|+\|w\|_\infty\bigr),
\end{equation}
and $\partial_{c_r}\mathcal B_{c_r,\delta}[w]$ is continuous in $c_r$ uniformly in $\delta$.
Since $\mathcal B_{c_r,\delta}$ is affine in $w$, we write $\mathcal B_{c_r,\delta}[w]=\mathcal A_{c_r,\delta}w+\mathcal B_{c_r,\delta}[0]$, and the linear part $\mathcal A_{c_r,\delta}w$ is $\mathcal B_{c_r,\delta}[w]$ with $\alpha=\beta=0$.
In this case $\psi(0)=0$, so the logarithmic term is absent, and each term of the derivative computed above is an integral operator applied to $w$, whose integrand is bounded by the integrable functions above and depends continuously on $c_r$.
The dominated convergence theorem then shows that the same statements hold for $\mathcal A_{c_r,\delta}$ in the operator norm on $C^0(I)$.
Here the operator norm $\|\mathcal G\|_{C^0\to C^0}$ of a linear map $\mathcal G$ from $C^0(I)$ to itself is the supremum of $\|\mathcal Gw\|_{C^0(I)}$ over the $w$ with $\|w\|_{C^0(I)}\le1$.
The uniform continuity of the derivative on the compact set $K$ strengthens differentiability to the uniform Taylor estimate
\begin{equation}
\bigl\|\mathcal A_{c_r+h,\delta}-\mathcal A_{c_r,\delta}
-h\,\partial_{c_r}\mathcal A_{c_r,\delta}\bigr\|_{C^0\to C^0}=o(h),
\end{equation}
uniformly in $\delta\in[0,\delta_1]$ and $c_r\in K$, and likewise for the affine parts.
Here $o(h)$ denotes a quantity whose ratio to $h$ tends to zero as $h\to0$.
This proves \ref{it:pb-map}.

For \ref{it:pb-fixed}, we use the decomposition $\mathcal B_{c_r,\delta}[w]=\mathcal A_{c_r,\delta}w+\mathcal B_{c_r,\delta}[0]$ of the affine map.
Both $\mathcal A_{c_r,\delta}$ and $\mathcal B_{c_r,\delta}[0]$ are $C^1$ in $c_r$ by \ref{it:pb-map}, and we denote their derivatives with respect to $c_r$ by $\dot{\mathcal A}_{c_r,\delta}$ and $\dot{\mathcal B}_{c_r,\delta}[0]$.
The fixed point satisfies $w_{c_r,\delta}=(\mathrm{Id}-\mathcal A_{c_r,\delta})^{-1}\mathcal B_{c_r,\delta}[0]$, where $\mathrm{Id}$ is the identity map of $C^0(I)$, and $\|(\mathrm{Id}-\mathcal A_{c_r,\delta})^{-1}\|_{C^0\to C^0}\le2$ by the uniform $\tfrac12$-contraction.
Define the candidate derivative $\dot w:=(\mathrm{Id}-\mathcal A_{c_r,\delta})^{-1} \bigl(\dot{\mathcal A}_{c_r,\delta}\,w_{c_r,\delta}+\dot{\mathcal B}_{c_r,\delta}[0] \bigr)$.
We subtract the fixed-point identities at $c_r+h$ and $c_r$, write $\mathcal A_{c_r+h,\delta}w_{c_r+h,\delta}-\mathcal A_{c_r,\delta}w_{c_r,\delta} =(\mathcal A_{c_r+h,\delta}-\mathcal A_{c_r,\delta})w_{c_r+h,\delta} +\mathcal A_{c_r,\delta}(w_{c_r+h,\delta}-w_{c_r,\delta})$, and obtain
\begin{multline*}
(\mathrm{Id}-\mathcal A_{c_r,\delta})
\bigl(w_{c_r+h,\delta}-w_{c_r,\delta}-h\,\dot w\bigr)
=\bigl(\mathcal A_{c_r+h,\delta}-\mathcal A_{c_r,\delta}-h\dot{\mathcal A}_{c_r,\delta}\bigr)
w_{c_r+h,\delta}\\
+\bigl(\mathcal B_{c_r+h,\delta}[0]-\mathcal B_{c_r,\delta}[0]-h\dot{\mathcal B}_{c_r,\delta}[0]\bigr)
+h\,\dot{\mathcal A}_{c_r,\delta}\bigl(w_{c_r+h,\delta}-w_{c_r,\delta}\bigr).
\end{multline*}
Its norm in $C^0(I)$ is $o(h)$, the first two brackets by the uniform Taylor estimates above and the last because $\|\dot{\mathcal A}_{c_r,\delta}\|_{C^0\to C^0}$ is bounded and $\|w_{c_r+h,\delta}-w_{c_r,\delta}\|_{C^0(I)}\le Ch$, which follows from the same identity without the terms containing $\dot w$, $\dot{\mathcal A}$, and $\dot{\mathcal B}$.
Hence $\partial_{c_r}w_{c_r,\delta}=\dot w$ exists, is continuous in $c_r$, and is bounded uniformly in $\delta$.
Let $(w_1,\tilde\chi_1)$ and $(w_2,\tilde\chi_2)$ denote the fixed points with $(\alpha,\beta)=(1,0)$ and $(0,1)$.
The endpoint matrices $\mathcal E^{\pm}$ have the columns $(w_1,\tilde\chi_1)(\pm\varsigma)$ and $(w_2,\tilde\chi_2)(\pm\varsigma)$, and evaluation at the points $x=\pm\varsigma$ and $x=0$, which do not move, is a bounded linear functional on $C^0(I)$.
Hence $\mathcal E^{\pm}$ and the critical values $w(z_j(c_r))=w_{c_r,\delta}(0)$ are $C^1$ in $c_r$ with uniformly bounded derivatives, which proves \ref{it:pb-fixed}.

For \ref{it:pb-inverse}, we first bound $\det\mathcal E^{+}$ from below.
We evaluate the representation in Equation~\eqref{eq:rep} at the endpoint $x=\varsigma$, where the singular integral is empty.
The cross terms containing $\tilde E\tilde f/\tilde U'$ cancel in the determinant, and
\begin{equation}
\det\mathcal E^{+}
=\tilde V(\varsigma)\,\tilde E(\varsigma)\,
\tilde\chi_2(\varsigma),
\qquad
\tilde E(\varsigma)\,\tilde\chi_2(\varsigma)
=1+O(\varsigma^{2})
\end{equation}
by Equation~\eqref{eq:rep-chi} and the fixed-point bounds, so $|\det\mathcal E^{+}|\ge\tfrac12|\tilde V(\varsigma)| \ge\tfrac12m_U\varsigma$ for $\varsigma$ small, uniformly in $(c_r,\delta)$.
To transfer this bound to $\mathcal E^{-}$, we use the first-order-system identities, under which the fixed points solve the system in Equation~\eqref{eq:wchi-recall} by \S\ref{subsec:B-equiv}.
The coefficient matrix of this system has trace $V'/V$, so the modified Wronskian $\Xi=w_1\tilde\chi_2-w_2\tilde\chi_1$ of the two fixed points, as in \cref{prop:local}\ref{it:basis-wronskian}, satisfies $\Xi=\varkappa\tilde V$ with $\varkappa$ constant on each interval where $\tilde V\neq0$.
For $\delta>0$ the interval is connected and one constant applies.
At $\delta=0$, the representation in Equation~\eqref{eq:rep} identifies the two one-sided constants, since the cross terms containing $\tilde E\tilde f/\tilde U'$ cancel in $\Xi$, leaving
\begin{equation}
\frac{\Xi}{\tilde V}
=\tilde E\bigl[(1+\mathscr I_1)\tilde\chi_2
-\mathscr I_2\tilde\chi_1\bigr],
\qquad
\mathscr I_\ell(x):=\int_{\varsigma}^{x}
\Bigl(\frac{\tilde g_1\tilde\chi_\ell}{\tilde V}
+\tilde g_2w_\ell\Bigr)\dd x' ,
\end{equation}
in which $\tilde\chi_1=O(x^2)$ by Equation~\eqref{eq:rep-chi} with $\beta=0$ and $|\tilde V|=O(|x|)$.
Since also $|\tilde V|\ge m_U|x|$, the integrand of $\mathscr I_1$ is bounded and $\mathscr I_1$ is continuous across $x=0$, while $\mathscr I_2=O(|\log|x||)$ and $\mathscr I_2\tilde\chi_1=O(x^2|\log|x||)\to0$.
Hence, using $\tilde E(0)=\tilde\chi_2(0)=1$, $\Xi/\tilde V$ extends continuously through $x=0$ with the common value $1+\mathscr I_1(0)$.
Consequently $|\det\mathcal E^{-}| =|\det\mathcal E^{+}|\, |\tilde V(-\varsigma)|/|\tilde V(\varsigma)| \ge\tfrac12|\tilde V(-\varsigma)|\ge\tfrac12m_U\varsigma$ in either case.
With these lower bounds and the $C^1$ dependence of $\mathcal E^{\pm}$ from \ref{it:pb-fixed}, the identity $\partial_{c_r}(\mathcal E^{\pm})^{-1} =-(\mathcal E^{\pm})^{-1}(\partial_{c_r}\mathcal E^{\pm}) (\mathcal E^{\pm})^{-1}$ gives the claimed $C^1$ bounds for the inverses, and the maps $T=\mathcal E^{-}(\mathcal E^{+})^{-1}$ inherit them.
At $\delta=0$ this proves the $C^1$ statement in \cref{prop:convergence}.
All bounds are uniform for $\delta\in[0,\delta_1]$.
\end{proof}

\subsection{Uniformity in the Mach number}
\label{subsec:uniformM}

In this subsection, we collect the statements, uniform in the Mach number, that the proof of \cref{thm:lowmach} uses.
They follow from \cref{lem:coef-perturb}, which compares the local solutions of two atmospheres, from \cref{lem:pullback}, which gives their $C^1$ dependence on $c_r$, and from \cref{lem:uniformconst}, which makes the constants of the critical-layer theory uniform in $M$.

\begin{corollary}\label{cor:uniformM}
Assume the hypotheses of \cref{thm:lowmach}, and let $\mathfrak a_M$, $M\in[0,M_0]$, be the coefficient tuples of the atmospheres in Equation~\eqref{eq:M-family}, and write $\mathfrak a^{\mathrm{inc}}$ for the tuple with $M=0$, which belongs to the limiting problem, as in the proof of \cref{thm:lowmach}.
Since these tuples are generated by atmospheres, they satisfy the first-order-system identities, and $\delta$ enters them only through $V^2$.
By Equation~\eqref{eq:coef-lip}, $\|\mathfrak a_M-\mathfrak a^{\mathrm{inc}}\|_{\mathcal X}\le CM^2$.
Then there are $M_1\in(0,M_0]$, $r_0>0$, $\delta_1>0$, and $C$ such that the following hold uniformly in $M\le M_1$, with $\mathbb D^+_{r_0}=\{c\in\C:|c-c_k|\le r_0,\ \Imc c\ge0\}$ as in \cref{prop:Zhalfdisc}.
\begin{enumerate}[label=\textup{(\roman*)}]
\item\label{it:uM-transfer} For $\delta\in[0,\delta_1]$ and $c_r\in[c_k-r_0,c_k+r_0]$, the maps $T_\delta$ across the critical neighborhoods and their inverses are bounded and differ from their incompressible counterparts by $O(M^2)$.
\item\label{it:uM-trace} With $m_0:=\min_{|c_r-c_k|\le r_0} |w_\ast^{\mathrm{inc}}(0;c_r)|>0$,
  \begin{equation}
  \inf_{M\le M_1}\ \min_{|c_r-c_k|\le r_0}
  \bigl|w_\ast^{M}(0;c_r)\bigr|\ \ge\ \tfrac12 m_0,
  \qquad
  \inf_{M\le M_1}\ \min_{c\in\mathbb D^+_{r_0}}
  \bigl|w^{M}(0;k,c)\bigr|\ \ge\ \tfrac14 m_0 ,
  \end{equation}
  so each $\Za^{M}(k,\cdot)$ is holomorphic in the interior of $\mathbb D^+_{r_0}$, and the denominator in Equation~\eqref{eq:Za-formula} does not vanish on $\mathbb D^+_{r_0}$.
\item\label{it:uM-rate} For $|c_r-c_k|\le r_0$ and $0<\delta\le\delta_1$, $|\Za^{M}(k,c_r+\im\delta)-\Za^{+,M}(k,c_r)| \le C\delta(1+|\log(\delta/V_\ast)|)$, and $|\partial_{c_r}\Za^{+,M}(k,c_r)|\le C$.
\item\label{it:uM-modulus} Consequently, for all $c\in\mathbb D^+_{r_0}$,
  \begin{equation}
  \bigl|\Za^{M}(k,c)-\Za^{+,M}(k,c_k)\bigr|
  \le\omega\bigl(|c-c_k|\bigr),
  \qquad
  \omega(r):=C\,r\bigl(1+|\log(r/V_\ast)|\bigr),
  \end{equation}
  that is, \cref{prop:Zhalfdisc}\ref{it:Z-modulus} holds with constants uniform in $M\le M_1$.
  If, in addition, $\Imc c_1(M)\ge\tfrac12\Imc c_1(0)>0$, then the constants $(\eps_0,R,C_0)$ in the proof of \cref{thm:main} can be chosen uniformly in $M\le M_1$.
\item\label{it:uM-complex} The pressure response is $O(M^2)$-close to its incompressible counterpart,
  \begin{equation}
  \sup_{c\in\mathbb D^+_{r_0}}
  \bigl|\Za^{M}(k,c)-\Za^{\mathrm{inc}}(k,c)\bigr|\ \le\ C\,M^{2},
  \end{equation}
  where $\Za^{\mathrm{inc}}(k,\cdot)$ is the pressure response of the limiting problem in Equation~\eqref{eq:rayleigh-limit}.
  On the real interval $[c_k-r_0,c_k+r_0]$, $\Za^{M}$ and $\Za^{\mathrm{inc}}$ denote their boundary values.
\end{enumerate}
\end{corollary}

\begin{proof}
We first choose $r_0>0$ as in the proof of \cref{prop:Zhalfdisc}, so that every $c_r\in[c_k-r_0,c_k+r_0]$ is a regular value of $U$ with the same number of critical levels and the incompressible problem satisfies Hypothesis \ref{hyp:nonres} for these $c_r$.
The smaller values of $r_0$ chosen below keep these properties.
Part \ref{it:uM-transfer} follows from \cref{lem:coef-perturb}\ref{it:cp-transfer}, applied on each critical neighborhood to the tuples $\mathfrak a_M$ and $\mathfrak a^{\mathrm{inc}}$.
On the intervals between the critical neighborhoods, the coefficients and the values at $z=\ha$ given in Equation~\eqref{eq:canonical} differ by $O(M^2)$, so Gr\"onwall's inequality gives the same bound for the solution maps.

For \ref{it:uM-trace}, $m_0>0$ by Hypothesis \ref{hyp:nonres} for the incompressible problem and the continuity of $w_\ast^{\mathrm{inc}}(0;\cdot)$ on the compact interval $[c_k-r_0,c_k+r_0]$.
When we compose the maps of \ref{it:uM-transfer} and the solution maps between the critical neighborhoods from $z=\ha$ down to $z=0$, we obtain $|w_\ast^{M}(0;c_r)-w_\ast^{\mathrm{inc}}(0;c_r)|\le CM^2$ uniformly on this interval.
Choose $M_1$ with $CM_1^2\le\tfrac12 m_0$.
By \cref{lem:uniformconst}, the constants of \cref{prop:convergence,prop:limit-solutions} are uniform in $M$, and by \cref{lem:pullback} they are uniform in $c_r$, so $|w^{M}(0;k,c_r+\im\delta)-w_\ast^{M}(0;c_r)|\le C\delta(1+|\log(\delta/V_\ast)|)$ for $\delta\in(0,\delta_1]$, with a constant independent of $M$.
We shrink $r_0$ so that $r_0\le\delta_1$ and $Cr_0(1+|\log(r_0/V_\ast)|)\le\tfrac14 m_0$.
This gives the second bound, because every $c\in\mathbb D^+_{r_0}$ has $\Imc c\le r_0\le\delta_1\le V_\ast$ and $\delta\mapsto\delta(1+|\log(\delta/V_\ast)|)$ is increasing on $(0,V_\ast]$.
Since the poles of $\Za^{M}(k,\cdot)$ are the zeros of $w^{M}(0;k,\cdot)$ (\cref{lem:regular-solve}), this proves \ref{it:uM-trace}.

For \ref{it:uM-rate}, we compose the estimates of \cref{prop:convergence} from $z=\ha$ down to $z=0$.
\cref{lem:uniformconst} makes their constants uniform in $M$, and \ref{it:uM-trace} bounds the denominator of the quotient in Equation~\eqref{eq:Za-formula} from below.
The derivative estimate follows from \cref{lem:pullback}, applied uniformly to the tuples $\mathfrak a_M$, and the same lower bound.

For \ref{it:uM-modulus}, let $c=c_r+\im\delta\in\mathbb D^+_{r_0}$.
We split $|\Za^{M}(k,c)-\Za^{+,M}(k,c_k)| \le|\Za^{M}(k,c)-\Za^{+,M}(k,c_r)|+|\Za^{+,M}(k,c_r)-\Za^{+,M}(k,c_k)|$ and apply \ref{it:uM-rate} with $\delta\le|c-c_k|$ and $|c_r-c_k|\le|c-c_k|$.
The map $r\mapsto r(1+|\log(r/V_\ast)|)$ is increasing for $r\le V_\ast$, so both terms are dominated by $\omega(|c-c_k|)$ after shrinking $r_0\le V_\ast$.
In the proof of \cref{thm:main}, Equation~\eqref{eq:D-ell} uses this modulus of continuity, and the lower bounds for the affine function $\ell$ of that proof use $\Imc c_1(M)\ge\tfrac12\Imc c_1(0)$ and $c_k$.
The same composition as in \ref{it:uM-trace} bounds the values $(w^{M},\chibar^{M})(0;k,c)$ uniformly on $\mathbb D^+_{r_0}$, while \ref{it:uM-trace} bounds the denominator in Equation~\eqref{eq:Za-formula} away from zero.
Hence the localization bound and all constants in the proof are uniform for $M\le M_1$.

For \ref{it:uM-complex}, let $c=c_r+\im\delta\in\mathbb D^+_{r_0}$.
Then $\delta\le r_0\le\delta_1$, so \cref{lem:coef-perturb}\ref{it:cp-fixed} and~\ref{it:cp-transfer} apply at $(c_r,\delta)$ to the tuples $\mathfrak a_M$ and $\mathfrak a^{\mathrm{inc}}$.
Hence the maps $T_\delta$ across the critical neighborhoods and the critical values differ by $O(M^2)$, uniformly in $(c_r,\delta)$.
On the intervals containing no critical level, the coefficients of the system in Equation~\eqref{eq:wchi-recall} and the values at $z=\ha$ given in Equation~\eqref{eq:canonical} are $O(M^2)$-close uniformly on $\mathbb D^+_{r_0}$.
Gr\"onwall's inequality gives the same bound for the solution operators.
Composing from $z=\ha$ down to $z=0$ yields $|w^{M}(0;k,c)-w^{\mathrm{inc}}(0;k,c)| +|\chibar^{M}(0;k,c)-\chibar^{\mathrm{inc}}(0;k,c)|\le CM^2$ on $\mathbb D^+_{r_0}$.
With the denominators bounded below by \ref{it:uM-trace}, including its case $M=0$, the quotient formula in Equation~\eqref{eq:Za-formula} gives \ref{it:uM-complex}.
\end{proof}

\section{Fr\'echet derivative and spectrum of the linearized operator}
\label{app:operator}

In this appendix, we prove \cref{thm:operator}.
The theorem settles two points that the derivation in \S\ref{sec:linearization} leaves open.

The first point is the sense in which the system of \cref{prop:linearized} is the linearization of the nonlinear two-phase problem.
In \S\ref{sec:linearization}, we obtained this system by differentiating a one-parameter family of solutions.
No well-posedness theory is available for the two-phase problem (\cref{rem:spaces}), so such families are hypothetical.
The Fr\'echet derivative of the nonlinear equations at the laminar state does not require solutions, but it requires a map between fixed Banach spaces, that is, complete normed vector spaces, whereas the air and water domains move with the interface.
In \S\ref{subsec:flattening}, we therefore transform the equations to fixed reference domains and show that the transformed equations are continuously Fr\'echet differentiable (\cref{prop:residual}).
In \S\ref{subsec:conjugation}, we show that a change of variables to a new unknown, the good unknown, turns this derivative into the linearized system of \cref{prop:linearized} (\cref{prop:goodunknown}).
Hence that system is the linearization of the nonlinear problem even if the family of solutions used in \S\ref{sec:linearization} does not exist.

The second point is the space on which the linearized operator acts.
A differential operator on $L^2$ is defined only on a subspace of $L^2$, its domain.
Its spectrum depends on the choice of the space and of the domain, and so do the statements that an eigenvalue is isolated and algebraically simple.
The linearized problem also has two features that the domain must respect.
The water pressure has no evolution equation, and the interface and boundary conditions constrain the state at every time.
In \S\ref{subsec:modal-operator}, we define $\Lop_k$ on an $L^2$ space at one wavenumber, with the constraints in its domain and the water pressure recovered by solving Laplace's equation in the water.
We then prove that its eigenvalues with $\Rec\lambda\neq0$ correspond to the zeros of the holomorphic dispersion function (\cref{thm:pointspec}).
In \S\ref{subsec:realization}, we show that $\Lop_k$ is the graph closure of the linearized evolution on the $H^3$ core (\cref{prop:realization}), so that $\Lop_k$ is determined by the linearized equations and the constraints.
In \S\ref{subsec:resolvent}, we bound the resolvent and show that a simple zero of $\Edet$ gives an isolated, algebraically simple eigenvalue (\cref{prop:resolvent,cor:algsimple}).
\cref{cor:spectral} applies these results to the unstable eigenvalue.

In \S\ref{subsec:flattening}, \S\ref{subsec:conjugation}, and \cref{prop:realization}, we assume, in addition to Hypotheses \ref{hyp:reg} and \ref{hyp:noslip}, the thermodynamic regularity stated in \cref{thm:operator}.
It requires an open set $\mathscr U$ of states $(\rho,S)$ whose closure $\overline{\mathscr U}$ is a compact subset of $(0,\infty)\times\R$, with
\begin{equation}\label{eq:thermo-domain}
\bigl\{(\rho_0(z),S_\ast):0\le z\le\ha\bigr\}\subset\mathscr U,
\qquad
P\in C^5(\overline{\mathscr U}),
\qquad
\inf_{\overline{\mathscr U}}P_\rho>0 .
\end{equation}
The first condition states that $\mathscr U$ contains the base range, the set of thermodynamic states of the laminar state.
The second condition means that the partial derivatives of $P$ of order at most five are continuous on $\mathscr U$ and extend continuously to the closure $\overline{\mathscr U}$.
The constant-sound-speed atmosphere of \cref{ex:constant-sound-speed}, whose equation of state is smooth on $(0,\infty)\times\R$ with $P_\rho=a_{\mathrm e}^2$, satisfies Equation~\eqref{eq:thermo-domain} for every such $\mathscr U$.
The polytropic atmosphere of \cref{ex:polytropic} satisfies it for every such $\mathscr U$ in a small neighborhood of the base range, provided $K\in C^5$ near $S_\ast$, because $a^2(0)=\gamma K(S_\ast)\rho_0(0)^{\gamma-1}>0$ gives $K>0$ near $S_\ast$.
The $C^5$ regularity of $P$ is used only in the composition estimates for the Fr\'echet derivative (\cref{lem:composition}).
The definition of $\Lop_k$ and its spectral theory for $\Imc c\neq0$ use only Hypotheses \ref{hyp:reg} and \ref{hyp:noslip}.
Neutrality gives $\rho_0'=-(g/a^2)\rho_0$, and hence $\rho_0\in C^4([0,\ha])$ because $a\in C^3([0,\ha])$ by Hypothesis \ref{hyp:reg}.
The good-unknown map of \S\ref{subsec:conjugation} requires this regularity, because it adds the product of $\rho_0'$ with a function in $H^3$ of the air strip to a density variation in $H^3$, and this product lies in $H^3$ when $\rho_0'\in C^3([0,\ha])$.

\subsection{Fixed-domain formulation}
\label{subsec:flattening}

Let $E_1$ and $E_2$ be Banach spaces with the norms $\|\cdot\|_{E_1}$ and $\|\cdot\|_{E_2}$.
A linear map $\mathcal G:E_1\to E_2$ is bounded if its operator norm $\|\mathcal G\|_{E_1\to E_2}:=\sup\{\|\mathcal G\mathbb V\|_{E_2}:\mathbb V\in E_1,\ \|\mathbb V\|_{E_1}\le1\}$ is finite, and then $\|\mathcal G\mathbb V\|_{E_2}\le\|\mathcal G\|_{E_1\to E_2}\|\mathbb V\|_{E_1}$ for all $\mathbb V\in E_1$.
We write $\mathcal L(E_1,E_2)$ for the space of bounded linear maps from $E_1$ to $E_2$ with the operator norm, $\mathcal L(E_1):=\mathcal L(E_1,E_1)$, and $\|\mathcal G\|$ for $\|\mathcal G\|_{E_1\to E_2}$ when the spaces are clear from the context.
A map $\mathcal F$ from an open subset of $E_1$ into $E_2$ is Fr\'echet differentiable at a point $\mathbb U$ if there is $\mathrm d\mathcal F(\mathbb U)\in\mathcal L(E_1,E_2)$, whose value at $\mathbb V\in E_1$ we write $\mathrm d\mathcal F(\mathbb U)[\mathbb V]$, with $\|\mathcal F(\mathbb U+\mathbb V)-\mathcal F(\mathbb U)-\mathrm d\mathcal F(\mathbb U)[\mathbb V]\|_{E_2}=o(\|\mathbb V\|_{E_1})$ as $\mathbb V\to0$, that is, the ratio of the left side to $\|\mathbb V\|_{E_1}$ tends to zero.
We call $\mathrm d\mathcal F(\mathbb U)$ the Fr\'echet derivative of $\mathcal F$ at $\mathbb U$ and $\mathrm d\mathcal F(\mathbb U)[\mathbb V]$ the derivative of $\mathcal F$ at $\mathbb U$ in the direction $\mathbb V$.
The map $\mathcal F$ is continuously Fr\'echet differentiable, or $C^1$, if it is Fr\'echet differentiable at every point and $\mathrm d\mathcal F$ depends continuously on the point in the operator norm.
A Fr\'echet derivative thus requires a map between fixed Banach spaces, but the fluid domains move with $\eta$.
We therefore pull every field back to fixed reference domains, the horizontal strips $\mathcal S_{\mathrm a}$ and $\mathcal S_{\mathrm w}$ defined below, by a change of variables that depends on $\eta$, a Hanzawa-type transformation \cite{Hanzawa1981}.
The transformed equations then contain $\eta$ and its derivatives in their coefficients, and the interface conditions become conditions on a line that does not move.
However, a perturbation of $\eta$ alone then changes every pulled-back field, even when the physical fields do not change.
The good unknown of \cref{prop:goodunknown} removes this apparent perturbation.

The sets $\mathcal S_{\mathrm a}:=\T_{k_0}\times(0,\ha)$ and $\mathcal S_{\mathrm w}:=\T_{k_0}\times(-\hw,0)$ are the horizontal strips that the air and the water occupy when the interface is flat, and we call them the reference strips.
We write $(x,\tilde z)$ for the coordinates on them.
For a strip $\mathcal S$, the space $L^2(\mathcal S)$ consists of the square-integrable functions on $\mathcal S$, with the norm $\|f\|_{L^2(\mathcal S)}:=\bigl(\int_{\mathcal S}|f|^2\,\dd x\,\dd\tilde z\bigr)^{1/2}$.
For an integer $m\ge1$, the Sobolev space $H^m(\mathcal S)$ consists of the functions in $L^2(\mathcal S)$ whose weak partial derivatives of order at most $m$, defined as in \cref{rem:spaces} with the integrals taken over $\mathcal S$ and with smooth functions that vanish near the two boundary lines of $\mathcal S$ in place of those that vanish near the endpoints of an interval, lie in $L^2(\mathcal S)$, with the norm $\|f\|_{H^m(\mathcal S)}:=\bigl(\sum_{i+j\le m}\|\partial_x^i\partial_{\tilde z}^jf\|_{L^2(\mathcal S)}^2\bigr)^{1/2}$, where the sum runs over the integers $i,j\ge0$ with $i+j\le m$.
For $1\le p<\infty$, the space $L^p$ on a strip, on an interval, or on $\T_{k_0}$ consists of the functions $f$ with $\int|f|^p<\infty$, with the norm $\|f\|_{L^p}:=\bigl(\int|f|^p\bigr)^{1/p}$, where the integral on $\T_{k_0}$ is taken over one period.
The space $L^\infty$ consists of the functions $f$ that are bounded outside a set of zero area, or of zero length on an interval or on $\T_{k_0}$, and $\|f\|_{L^\infty}$ is the smallest bound for $|f|$ outside such a set.
For a continuous function on a closed strip, on a closed interval, or on $\T_{k_0}$, this norm is the maximum of $|f|$, which on an interval is the norm $\|f\|_{C^0}$ of Equation~\eqref{eq:norms}.
A function $f\in L^2(\T_{k_0})$ on the circle $\T_{k_0}$ has the Fourier series $f(x)=\sum_{n\in\Z}\hat f_ne^{\im nk_0x}$ with the coefficients $\hat f_n:=\frac{k_0}{2\pi}\int_0^{2\pi/k_0}f(x)e^{-\im nk_0x}\,\dd x$, and for real $s\ge0$ the space $H^s(\T_{k_0})$ consists of the functions with
\begin{equation}
\|f\|_{H^s(\T_{k_0})}^2:=\frac{2\pi}{k_0}\sum_{n\in\Z}\bigl(1+n^2k_0^2\bigr)^s\,|\hat f_n|^2<\infty .
\end{equation}
For an integer $s$, this norm is equivalent to the norm of \cref{rem:spaces} on one period, that is, each of the two norms is at most a constant times the other.
The fractional value $s=\frac52$ describes the traces of $H^3$ functions of a strip on the lines $\{\tilde z=\mathrm{const}\}$, that is, their restrictions to these lines, which lie in $H^{5/2}(\T_{k_0})$ with norms at most a constant times the $H^3$ norm \cite[Ch.~7]{AdamsFournier2003}.
A product of normed spaces, such as $H^3(\mathcal S_{\mathrm a})^4$, the product of four copies of $H^3(\mathcal S_{\mathrm a})$, carries the norm whose square is the sum of the squares of the norms of the components.
As in Appendix~\ref{app:proofs}, $C$ denotes a positive constant that may change from line to line.
Set $h_-:=\min\{\ha,\hw\}$ and let $\phi\in C^\infty(\R)$ be a cutoff function with $\phi\equiv1$ on $[-h_-/3,h_-/3]$, $\phi\equiv0$ outside $(-2h_-/3,2h_-/3)$, and $|\phi'|\le4/h_-$, where $C^\infty$ is the class of smooth functions, those in $C^j$ for every integer $j\ge0$.
For $\eta\in H^{4}(\T_{k_0})$ with $\|\eta\|_{L^\infty}\le h_-/16$ define
\begin{equation}\label{eq:hanzawa}
\Phi_\eta(x,\tilde z):=\bigl(x,\ \tilde z+\phi(\tilde z)\,\eta(x)\bigr),
\qquad
J:=1+\phi'(\tilde z)\,\eta(x)\ \ge\ \tfrac34 .
\end{equation}
The map $\tilde z\mapsto\tilde z+\phi(\tilde z)\eta(x)$ is strictly increasing because $J\ge\frac34$, and it fixes $\tilde z=\ha$ and $\tilde z=-\hw$, where $\phi$ vanishes.
Hence $\Phi_\eta$ is a change of variables that maps $\mathcal S_{\mathrm a}$ and $\mathcal S_{\mathrm w}$ bijectively onto the air and water domains $\Omega_{\mathrm a}(\eta)=\{\eta(x)<z<\ha\}$ and $\Omega_{\mathrm w}(\eta)=\{-\hw<z<\eta(x)\}$ of \S\ref{sec:formulation}, that is, every point of these domains is the image of exactly one point of the strips.
It leaves the top and the bottom in place and maps the line $\{\tilde z=0\}$ onto the graph of $\eta$.
For a function $f$ on a moving domain with pullback $\tilde f:=f\circ\Phi_\eta$,
\begin{equation}\label{eq:flatten-deriv}
(\partial_xf)\circ\Phi_\eta
=\partial_x\tilde f-\frac{\phi\,\eta_x}{J}\,\partial_{\tilde z}\tilde f,
\qquad
(\partial_zf)\circ\Phi_\eta=\frac1J\,\partial_{\tilde z}\tilde f,
\qquad
(\partial_tf)\circ\Phi_\eta
=\partial_t\tilde f-\frac{\phi\,\eta_t}{J}\,\partial_{\tilde z}\tilde f
\end{equation}
when $\eta$ depends on $t$.
The unknown is the tuple
\begin{equation}
\mathbb U:=\bigl(\tilde\rho,\tilde u,\tilde w,\tilde S;\
\tilde v_x,\tilde v_z,\tilde p_{\mathrm w};\ \eta\bigr),
\end{equation}
which consists of the pullbacks of the air fields $(\rho_{\mathrm a},u_{\mathrm a},w_{\mathrm a},S_{\mathrm a})$, the pullbacks of the water fields $\mathbf u_{\mathrm w}=(u_{\mathrm w},w_{\mathrm w})$ and $p_{\mathrm w}$, and the interface elevation.
We introduce the spaces
\begin{equation}\label{eq:XY-spaces}
\mathbb X:=H^{3}(\mathcal S_{\mathrm a})^{4}\times
H^{3}(\mathcal S_{\mathrm w})^{3}\times H^{4}(\T_{k_0}),
\qquad
\mathcal O:=\Bigl\{\mathbb U\in\mathbb X:\
\operatorname{ran}(\tilde\rho,\tilde S)\subset\mathscr U_{d_0},\
\|\eta\|_{L^\infty}<\tfrac{h_-}{16}\Bigr\},
\end{equation}
where $\operatorname{ran}(\tilde\rho,\tilde S)$, the range of $(\tilde\rho,\tilde S)$, is the set of values of $(\tilde\rho,\tilde S)$ on the closed strip, $d_0>0$ is smaller than the distance from the base range to the boundary $\partial\mathscr U$, and, with $\operatorname{dist}(y,\partial\mathscr U)$ the distance from a point $y$ to $\partial\mathscr U$,
\begin{equation}
\mathscr U_{d_0}:=\bigl\{y\in\mathscr U:\
\operatorname{dist}(y,\partial\mathscr U)>d_0\bigr\} .
\end{equation}
The set $\mathscr U_{d_0}$ is open, and its closure is contained in the set $\{y\in\overline{\mathscr U}:\operatorname{dist}(y,\partial\mathscr U)\ge d_0\}$, which is a compact subset of $\mathscr U$.
The set $\mathcal O$ is open in $\mathbb X$, for the following reason.
By the Sobolev embedding theorem in two dimensions \cite[Ch.~4]{AdamsFournier2003}, every function $f\in H^3$ of a strip is continuous on the closed strip, with $\|f\|_{L^\infty}\le C\|f\|_{H^3}$, and the Sobolev embedding theorem on $\T_{k_0}$ gives $\|\eta\|_{L^\infty}\le C\|\eta\|_{H^4(\T_{k_0})}$.
The range of a continuous field on the closed strip is compact, and both defining conditions of $\mathcal O$, the range in the open set $\mathscr U_{d_0}$ and the strict bound on $\|\eta\|_{L^\infty}$, are preserved under perturbations with small $L^\infty$ norms, hence under perturbations with small norms in $\mathbb X$.
On $\mathcal O$, the density satisfies $\tilde\rho\ge\rho_->0$, where $\rho_-$ is the minimum of $\rho$ over the closure of $\mathscr U_{d_0}$.
On this compact set, $P$, $P_\rho$, and $P_S$ are $C^4$ with bounded derivatives, and $P_\rho$ is bounded below by a positive constant.
The laminar state in Equation~\eqref{eq:laminar} is $\mathbb U_0:=(\rho_0,U,0,S_\ast;\,0,0,p_{\mathrm w,0};\,0)$ in these variables.
It lies in $\mathbb X$ by Hypothesis \ref{hyp:reg}, and in $\mathcal O$ by the choice of $d_0$.

Pulled back to $\tilde z=0$, the air kinematic condition in Equation~\eqref{eq:kinematic-nonlinear} reads
\begin{equation}\label{eq:etat-sub}
\eta_t=\mathcal K_\eta(\mathbb U)
:=\bigl(\tilde w-\tilde u\,\eta_x\bigr)\big|_{\tilde z=0} .
\end{equation}
When we substitute $\mathcal K_\eta(\mathbb U)$ for $\eta_t$ in the pulled-back time derivatives in Equation~\eqref{eq:flatten-deriv}, the two-phase system in Equations~\eqref{eq:air-nonlinear}--\eqref{eq:rigid} on the moving domains becomes equivalent, for curves $t\mapsto\mathbb U(t)\in\mathcal O$ that are $C^1$ in time, to the system
\begin{equation}\label{eq:flat-system}
\partial_t\bigl(\tilde\rho,\tilde u,\tilde w,\tilde S,
\tilde v_x,\tilde v_z,\eta\bigr)
=\mathcal N_{\mathrm{dyn}}(\mathbb U),
\qquad
\mathcal N_{\mathrm{con}}(\mathbb U)=0 ,
\end{equation}
with $\mathcal N_{\mathrm{dyn}}$ and $\mathcal N_{\mathrm{con}}$ given below.
No time derivative of $\tilde p_{\mathrm w}$ occurs in this system.
The water pressure has no evolution equation and is determined at each time by the incompressibility constraint.
A system in which some equations contain no time derivative is called differential-algebraic.
We call the equations with time derivatives the evolution equations and the others the constraint equations.
To write $\mathcal N_{\mathrm{dyn}}$ completely, we introduce the transformed differential operators and the pulled-back advection operator, which is the material derivative without its time derivative,
\begin{equation}\label{eq:eta-operators}
\partial_x^{\eta}:=\partial_x-\frac{\phi\eta_x}{J}
\partial_{\tilde z},
\qquad
\partial_z^{\eta}:=\frac1J\,\partial_{\tilde z},
\qquad
\mathfrak M^{\eta}_{(\tilde u,\tilde w)}
:=\tilde u\,\partial_x^{\eta}+\tilde w\,\partial_z^{\eta}
-\frac{\phi\,\mathcal K_\eta(\mathbb U)}{J}\,
\partial_{\tilde z} .
\end{equation}
By Equation~\eqref{eq:flatten-deriv}, the first two operators are the pullbacks of $\partial_x$ and $\partial_z$.
The last term of $\mathfrak M^{\eta}$ comes from the time derivative in Equation~\eqref{eq:flatten-deriv}.
The point $\Phi_\eta(x,\tilde z)$ moves vertically with the velocity $\phi\eta_t$, the mesh velocity of the flattening, and $\eta_t$ is replaced by $\mathcal K_\eta(\mathbb U)$.
The subscript of $\mathfrak M^{\eta}$ is the advecting velocity, and $\mathfrak M^{\eta}_{(\tilde v_x,\tilde v_z)}$ is defined by the same formula with $(\tilde v_x,\tilde v_z)$ in place of $(\tilde u,\tilde w)$.
With $\tilde p_{\mathrm a}:=P(\tilde\rho,\tilde S)$, the components of $\mathcal N_{\mathrm{dyn}}$ are
\begin{equation}\label{eq:Ndyn-mass}
\begin{aligned}
\bigl[\mathcal N_{\mathrm{dyn}}\bigr]_{\tilde\rho}
&=-\mathfrak M^{\eta}_{(\tilde u,\tilde w)}\tilde\rho
-\tilde\rho\bigl(\partial_x^{\eta}\tilde u
+\partial_z^{\eta}\tilde w\bigr),
&\bigl[\mathcal N_{\mathrm{dyn}}\bigr]_{\tilde u}
&=-\mathfrak M^{\eta}_{(\tilde u,\tilde w)}\tilde u
-\frac1{\tilde\rho}\,\partial_x^{\eta}\tilde p_{\mathrm a},\\
\bigl[\mathcal N_{\mathrm{dyn}}\bigr]_{\tilde w}
&=-\mathfrak M^{\eta}_{(\tilde u,\tilde w)}\tilde w
-\frac1{\tilde\rho}\,\partial_z^{\eta}\tilde p_{\mathrm a}-g,
&\bigl[\mathcal N_{\mathrm{dyn}}\bigr]_{\tilde S}
&=-\mathfrak M^{\eta}_{(\tilde u,\tilde w)}\tilde S,\\
\bigl[\mathcal N_{\mathrm{dyn}}\bigr]_{\tilde v_x}
&=-\mathfrak M^{\eta}_{(\tilde v_x,\tilde v_z)}\tilde v_x
-\frac1{\rhow}\,\partial_x^{\eta}\tilde p_{\mathrm w},
&\bigl[\mathcal N_{\mathrm{dyn}}\bigr]_{\tilde v_z}
&=-\mathfrak M^{\eta}_{(\tilde v_x,\tilde v_z)}\tilde v_z
-\frac1{\rhow}\,\partial_z^{\eta}\tilde p_{\mathrm w}-g,\\
\bigl[\mathcal N_{\mathrm{dyn}}\bigr]_{\eta}
&=\mathcal K_\eta(\mathbb U).&&
\end{aligned}
\end{equation}
Each component comes from the corresponding equation of Equations~\eqref{eq:air-nonlinear}, \eqref{eq:water-nonlinear}, and \eqref{eq:etat-sub}, pulled back by $\Phi_\eta$ and solved for the time derivative, as one checks term by term with Equation~\eqref{eq:flatten-deriv}.
The constraint equations are $\mathcal N_{\mathrm{con}}(\mathbb U)=0$, with
\begin{equation}\label{eq:Ncon}
\begin{aligned}
\mathcal N_{\mathrm{con}}(\mathbb U)
:=\Bigl(
&\partial_x\tilde v_x-\frac{\phi\eta_x}{J}\partial_{\tilde z}\tilde v_x
+\frac{\partial_{\tilde z}\tilde v_z}{J};\ \
\bigl(\tilde w-\tilde u\eta_x\bigr)\big|_{\tilde z=0}
-\bigl(\tilde v_z-\tilde v_x\eta_x\bigr)\big|_{\tilde z=0};\\
&\bigl(\tilde p_{\mathrm w}-P(\tilde\rho,\tilde S)\bigr)\big|_{\tilde z=0}
-\sigma\kappa(\eta);\ \
\tilde w\big|_{\tilde z=\ha};\ \
\tilde v_z\big|_{\tilde z=-\hw}
\Bigr).
\end{aligned}
\end{equation}
They are the divergence constraint, the kinematic compatibility of the two phases, the dynamic condition with $\kappa(\eta)=-\eta_{xx}(1+\eta_x^2)^{-3/2}$, and the impermeability conditions at the top and bottom.
Because $\Phi_\eta$ maps the line $\{\tilde z=0\}$ onto the interface, the interface conditions are evaluated exactly, without a Taylor expansion in $\eta$.
We call $\mathcal N:=(\mathcal N_{\mathrm{dyn}},\mathcal N_{\mathrm{con}})$ the residual of the flattened system.
Its target space is
\begin{equation}\label{eq:Yspace}
\mathbb Y:=
\underbrace{H^{2}(\mathcal S_{\mathrm a})^{4}
\times H^{2}(\mathcal S_{\mathrm w})^{2}
\times H^{5/2}(\T_{k_0})}_{\text{evolution equations}}
\ \times\
\underbrace{H^{2}(\mathcal S_{\mathrm w})
\times H^{5/2}(\T_{k_0})\times H^{2}(\T_{k_0})
\times H^{5/2}(\T_{k_0})^{2}}_{\text{constraint equations}} .
\end{equation}

The residual is built from products, quotients, and nonlinear functions of $H^3$ fields, namely $P(\tilde\rho,\tilde S)$, $1/\tilde\rho$, $1/J$, and the curvature of $\eta$.
Its differentiability in Sobolev spaces rests on the following standard composition estimates.
In the lemma, $C_{\mathrm{emb}}$ denotes the constant of the Sobolev embedding inequality $\|u\|_{L^\infty}\le C_{\mathrm{emb}}\|u\|_{H^m}$, which holds on a strip for $m\ge2$ and on $\T_{k_0}$ for $m\ge1$.
For a compact set $K_1$ in an open set $W$ of $\R$ or of $\R^2$, we set $r(K_1):=\min\{1,\operatorname{dist}(K_1,\partial W)\}/(2C_{\mathrm{emb}})$, where $\operatorname{dist}(K_1,\partial W)$ is the smallest distance between a point of $K_1$ and a point of the boundary $\partial W$ of $W$, and $+\infty$ when $\partial W$ is empty.

\begin{lemma}\label{lem:composition}
Let $\mathcal S$ be one of the strips above or the circle $\T_{k_0}$, and let $m=2$ or $m=3$.
\begin{enumerate}[label=\textup{(\roman*)}]
\item\label{it:comp-algebra} $H^m(\mathcal S)$ is a Banach algebra, that is, $\|uv\|_{H^m}\le C\|u\|_{H^m}\|v\|_{H^m}$, and the multiplier inequality $\|uv\|_{H^{m-1}}\le C\|u\|_{H^{m}}\|v\|_{H^{m-1}}$ holds.
On $\T_{k_0}$, the space $H^s(\T_{k_0})$ is a Banach algebra for every real $s>\frac12$.
\item\label{it:comp-nemytskii} Let $W$ be an open subset of $\R$ or of $\R^2$, and let $F\in C^{m+2}(W)$.
The Nemytskii map $u\mapsto F(u)$ is the composition with $F$, taken at each point, so that $F(u)(x,\tilde z):=F\bigl(u(x,\tilde z)\bigr)$ on a strip and $F(u)(x):=F\bigl(u(x)\bigr)$ on $\T_{k_0}$.
It is defined and $C^1$ on the set of the $u\in H^m(\mathcal S)$, or $u\in H^m(\mathcal S)^2$ when $W\subset\R^2$, with $\operatorname{ran}(u)\subset W$, where $\operatorname{ran}(u)$ is the set of values of $u$ on $\overline{\mathcal S}$.
This set of fields $u$ is open, because $\operatorname{ran}(u)$, the range of a continuous field on the compact set $\overline{\mathcal S}$, is a compact subset of $W$ and so has a positive distance to $\partial W$, and because the embedding inequality $\|h\|_{L^\infty}\le C_{\mathrm{emb}}\|h\|_{H^m}$ gives $\operatorname{ran}(u+h)\subset W$ when $C_{\mathrm{emb}}\|h\|_{H^m}$ is smaller than this distance.
For $y\in W$ and a vector $h$ with as many components as $y$, we write $\mathrm dF(y)[h]:=\sum_i\partial_iF(y)\,h_i$ and $\mathrm d^2F(y)[h,h]:=\sum_{i,j}\partial_{ij}F(y)\,h_ih_j$ for the first and second derivatives of $F$, where the sums run over the components $y_i$ of $y$, $\partial_i$ is the partial derivative with respect to $y_i$, and $\partial_{ij}:=\partial_i\partial_j$.
For fields $u$ and $h$ on $\mathcal S$, the functions $\mathrm dF(u)[h]$ and $\mathrm d^2F(u)[h,h]$ are given by these formulas at each point.
The Fr\'echet derivative of the Nemytskii map at $u$ is $h\mapsto \mathrm dF(u)[h]$, and in the multivariable Taylor formula
\begin{equation}
F(u+h)-F(u)-\mathrm dF(u)[h]
=\int_0^1(1-t)\,\mathrm d^2F(u+th)[h,h]\,\dd t
\end{equation}
the right side has $H^m$ norm at most $C\|h\|_{H^m}^{2}$.
For every compact set $K_1\subset W$ and every $R>0$, this constant is uniform over the $u$ with $\|u\|_{H^m}\le R$ and $\operatorname{ran}(u)\subset K_1$ and the $h$ with $\|h\|_{H^m}\le r(K_1)$.
For the same $u$, the norm $\|F(u)\|_{H^m}$ is bounded, and $\|F(u_1)-F(u_2)\|_{H^m}\le C\|u_1-u_2\|_{H^m}$ whenever $\|u_1-u_2\|_{H^m}\le r(K_1)$.
\item\label{it:comp-inverse} If, in addition, $F$ has no zero in $W$, then the same holds for $u\mapsto1/F(u)$.
In particular, $\tilde\rho\mapsto1/\tilde\rho$ is $C^1$ on $\{\operatorname{ran}(\tilde\rho,\tilde S)\subset\mathscr U_{d_0}\}$, where $\tilde\rho\ge\rho_->0$.
\item\label{it:comp-extension} Let $\mathcal S$ be a strip, and let $\operatorname{ext}$ denote the extension that is constant in $\tilde z$, $(\operatorname{ext}f)(x,\tilde z):=f(x)$.
Then
\begin{equation}
\|\operatorname{ext}f\|_{H^m(\mathcal S)}\le C\,\|f\|_{H^m(\T_{k_0})},
\qquad
\|(\operatorname{ext}f)\,u\|_{H^m(\mathcal S)}
\le C\,\|f\|_{H^m(\T_{k_0})}\,\|u\|_{H^m(\mathcal S)},
\end{equation}
the first because $\operatorname{ext}f$ has no vertical derivatives and the strip has finite height, the second by \ref{it:comp-algebra}.
For the use below, $\|\operatorname{ext}\mathcal K_\eta\|_{H^2(\mathcal S)}\le C\|\mathcal K_\eta\|_{H^2(\T_{k_0})}\le C\|\mathcal K_\eta\|_{H^{5/2}(\T_{k_0})}$.
\end{enumerate}
\end{lemma}

\begin{proof}
For \ref{it:comp-algebra}, we use the Sobolev embedding inequalities $\|u\|_{L^\infty}\le C\|u\|_{H^2}$ and $\|u\|_{L^4}\le C\|u\|_{H^1}$, which hold on the two-dimensional strips and on $\T_{k_0}$.
We write $\partial^j$ for any partial derivative of order $j$ with respect to $(x,\tilde z)$, or with respect to $x$ on $\T_{k_0}$, and $\partial:=\partial^1$.
For $m=2$, expand $\partial^2(uv)$ by the Leibniz rule and estimate the three types of terms by $\|u\|_{L^\infty}\|\partial^2v\|_{L^2}$, $\|\partial u\|_{L^4}\|\partial v\|_{L^4}$, and $\|\partial^2u\|_{L^2}\|v\|_{L^\infty}$.
For $m=3$, the additional terms $\partial^3u\,v$ and $\partial^2u\,\partial v$, and those with $u$ and $v$ exchanged, are estimated in the same way by $\|\partial^3u\|_{L^2}\|v\|_{L^\infty}$ and $\|\partial^2u\|_{L^4}\|\partial v\|_{L^4}$.
The multiplier inequality follows the same way, placing the factor from $H^{m-1}$ in $L^2$ or $L^4$ and the factor from $H^m$ in $L^\infty$ or $L^4$, and for $m=3$ it uses that the first derivative of the factor in $H^3$ lies in $H^2$ and is therefore bounded.
On $\T_{k_0}$, if $u$ and $v$ have the Fourier coefficients $\hat u_n$ and $\hat v_n$, then $uv$ has the coefficients $\sum_{n_1+n_2=n}\hat u_{n_1}\hat v_{n_2}$, the convolution of the two sequences.
The algebra property for real $s>\frac12$ then follows from the Fourier norm, the bound $(1+n^2k_0^2)^{s/2}\le2^s\bigl[(1+n_1^2k_0^2)^{s/2}+(1+n_2^2k_0^2)^{s/2}\bigr]$ for $n=n_1+n_2$, and Young's inequality for sequences, which states that the sum of the squared absolute values of a convolution is at most the sum of the squared absolute values of one sequence times the square of the sum of the absolute values of the other.
The last sum is finite for the Fourier coefficients of a function in $H^s(\T_{k_0})$ by the Cauchy--Schwarz inequality, because $\sum_{n\in\Z}(1+n^2k_0^2)^{-s}<\infty$.

For \ref{it:comp-nemytskii}, the function $F(u)$ is bounded, and bounded functions lie in $L^2$ because $\mathcal S$ has finite area or length.
By the chain rule, every term of $\partial^{j}[F(u)]$, $1\le j\le m$, is a partial derivative of $F$ of some order $i$ with $1\le i\le j$, evaluated at $u$ and bounded since $F\in C^m$ on a compact set, times a product of $i$ derivatives of components of $u$ whose orders add up to $j$.
For $m=2$, this product is a single derivative of order at most two, which lies in $L^2$, or the product of two first derivatives, which lies in $L^2$ because first derivatives lie in $L^4$ by the embedding inequality $\|u\|_{L^4}\le C\|u\|_{H^1}$.
For $m=3$, the first derivatives lie in $H^2$ and are therefore bounded.
Since the orders add up to at most three, at most one factor has order two or three, and this factor lies in $L^2$, so the product is bounded or lies in $L^2$.
This uses only $F\in C^m$.
Let $u_1$ and $u_2$ have ranges in $K_1$ and $H^m$ norms at most $R$, with $\|u_1-u_2\|_{H^m}\le r(K_1)$.
Then $\|u_1-u_2\|_{L^\infty}\le\frac12\min\{1,\operatorname{dist}(K_1,\partial W)\}$, so the ranges of $u_2+t(u_1-u_2)$, $t\in[0,1]$, lie in the compact set $K_2$ of points whose distance to $K_1$ is at most this number, and $K_2\subset W$.
The Lipschitz bound follows by writing $F(u_1)-F(u_2)=\int_0^1\mathrm dF\bigl(u_2+t(u_1-u_2)\bigr)[u_1-u_2]\,\dd t$ and applying \ref{it:comp-algebra} together with the boundedness statement for the partial derivatives $\partial_iF\in C^{m+1}$ on $K_2$.
This uses only $F\in C^{m+1}$.
For the $C^1$ statement, the multivariable Taylor formula displayed in the statement holds pointwise.
By definition, $\mathrm d^2F(u+th)[h,h]$ is the sum of the products $h_ih_j\,\partial_{ij}F(u+th)$, and the ranges of $u+th$ lie in $K_2$ for $\|h\|_{H^m}\le r(K_1)$ by the same argument.
Hence \ref{it:comp-algebra} and the boundedness statement applied to $\partial_{ij}F\in C^{m}$, which is where $F\in C^{m+2}$ enters, bound the remainder by $C\|h\|_{H^m}^2$, uniformly for $t\in[0,1]$.
Continuity of $u\mapsto\mathrm dF(u)$ in the operator norm of linear maps from $H^m(\mathcal S)$, or from $H^m(\mathcal S)^2$, to $H^m(\mathcal S)$ follows from the Lipschitz bound for the partial derivatives $\partial_iF$ and \ref{it:comp-algebra}.
The openness of the domain was shown in the statement.
Item \ref{it:comp-inverse} is \ref{it:comp-nemytskii} applied to $1/F\in C^{m+2}(W)$.
Item \ref{it:comp-extension} was proved in the statement.
\end{proof}

\begin{proposition}\label{prop:residual}
The residual $\mathcal N=(\mathcal N_{\mathrm{dyn}},\mathcal N_{\mathrm{con}})$ is a well-defined map $\mathcal O\to\mathbb Y$, and it is continuously Fr\'echet differentiable.
\end{proposition}

\begin{proof}
By Equations~\eqref{eq:eta-operators} and \eqref{eq:Ndyn-mass}, each component of $\mathcal N_{\mathrm{dyn}}$ on a strip is a finite sum of products of factors of the following kinds.
The first kind is an unknown or a first derivative of an unknown, which lies in $H^2$.
The second kind is one of the functions $\phi\eta_x$ and $\phi'\eta$, which lie in $H^3$ of the strip, since $\eta\in H^4(\T_{k_0})$ and $\phi$ is smooth.
The third kind is $1/J$ with $J=1+\phi'\eta$, or the inverse density $1/\tilde\rho$ from the momentum components.
Both lie in $H^2$ by \cref{lem:composition}\ref{it:comp-inverse} with $m=2$, since $J\ge\tfrac34$ and $\tilde\rho\ge\rho_->0$ on $\mathcal O$.
The fourth kind is one of the compositions $P_\rho(\tilde\rho,\tilde S)$ and $P_S(\tilde\rho,\tilde S)$, which appear when the chain rule is applied to $\partial_x^{\eta}\tilde p_{\mathrm a}$ and $\partial_z^{\eta}\tilde p_{\mathrm a}$.
They lie in $H^2$ by \cref{lem:composition}\ref{it:comp-nemytskii} with $m=2$ and $W=\mathscr U$, because $P_\rho,P_S\in C^4$ by $P\in C^5$.
The fifth kind is $\operatorname{ext}\mathcal K_\eta(\mathbb U)$, from the last term of $\mathfrak M^{\eta}$.
The function $\mathcal K_\eta(\mathbb U)$ of Equation~\eqref{eq:etat-sub} is built from the traces of $\tilde u$ and $\tilde w$ at $\tilde z=0$, which lie in $H^{5/2}(\T_{k_0})$ because the trace map $H^3(\mathcal S_{\mathrm a})\to H^{5/2}(\T_{k_0})$ is bounded, and from $\eta_x\in H^3(\T_{k_0})$, multiplied in the algebra $H^{5/2}(\T_{k_0})$.
By \cref{lem:composition}\ref{it:comp-extension}, multiplication by its extension is a bounded map on $H^2$ of the strip.
Every product of two $H^2$ factors lies in $H^2$ by \cref{lem:composition}\ref{it:comp-algebra}, so each component of $\mathcal N_{\mathrm{dyn}}$ on a strip maps $\mathcal O$ into $H^2$, and the component $[\mathcal N_{\mathrm{dyn}}]_\eta$ maps $\mathcal O$ into $H^{5/2}(\T_{k_0})$.
Each factor is a $C^1$ function of $\mathbb U$ with values in $H^2$, either because it is linear in $\mathbb U$ or by \cref{lem:composition}, and finite sums and products of $C^1$ maps into a Banach algebra are $C^1$.
The components of $\mathcal N_{\mathrm{con}}$ are treated in the same way, since the traces are bounded linear maps $H^3(\mathcal S)\to H^{5/2}(\T_{k_0})$ and products on $\T_{k_0}$ use the algebra $H^{5/2}(\T_{k_0})$.
In the dynamic condition, $P(\tilde\rho,\tilde S)$ lies in $H^3(\mathcal S_{\mathrm a})$ by \cref{lem:composition}\ref{it:comp-nemytskii} with $m=3$, since $P\in C^5$.
The curvature $\kappa(\eta)=-\eta_{xx}(1+\eta_x^2)^{-3/2}$ is the product of $\eta_{xx}\in H^2(\T_{k_0})$ and $(1+\eta_x^2)^{-3/2}$.
The second factor lies in $H^3(\T_{k_0})$ by \cref{lem:composition}\ref{it:comp-nemytskii} with $m=3$ and the smooth function $t\mapsto(1+t^2)^{-3/2}$, since $\eta_x\in H^3(\T_{k_0})$.
Hence $\kappa(\eta)\in H^2(\T_{k_0})$ by the multiplier inequality of \cref{lem:composition}\ref{it:comp-algebra}, and $\kappa$ is $C^1$ in $\eta$ for the same reasons.
Therefore $\mathcal N$ maps $\mathcal O$ into the space $\mathbb Y$ of Equation~\eqref{eq:Yspace} and is $C^1$.
\end{proof}

\subsection{The Fr\'echet derivative in terms of the good unknown}
\label{subsec:conjugation}

In this subsection, we compute the derivative $\mathrm d\mathcal N(\mathbb U_0)$ of the residual at the laminar state and compare it with the linearized system of \cref{prop:linearized}.
The variation of a pulled-back field at a fixed point of the reference strip differs from the Eulerian variation at the corresponding physical point.
The reason is that a displacement $\zeta$ of the interface moves the physical point $\Phi_\eta(x,\tilde z)$ vertically by $\phi(\tilde z)\zeta(x)$, so the pulled-back field changes in addition by $\phi\zeta$ times the base-state gradient.
The good-unknown map below adds exactly this term, as in Alinhac's good unknown for free-boundary problems \cite{Alinhac1989}, which is used for compressible free boundaries in \cite{Trakhinin2009}.
After this change of variables, the derivative of the flattened equations is the Eulerian linearized system of \S\ref{sec:linearization}.

Let $\mathbb V=(r,u,w,s;\,v_x,v_z,q_{\mathrm w};\,\zeta)\in\mathbb X$ be a variation of the Eulerian fields and of the interface, where we identify the reference strips with the undisturbed domains $\{0<z<\ha\}$ and $\{-\hw<z<0\}$ of \S\ref{sec:linearization}.
We write $\mathbb A\mathbb V$ for the tuple of right-hand sides of Equations~\eqref{eq:linear-air}, \eqref{eq:linear-water}, and \eqref{eq:lin-kinematic}, solved for the time derivatives of $(r,u,w,s,v_x,v_z,\zeta)$, with $q=a^2r+P_Ss$.
We denote its components by $[\mathbb A\mathbb V]_r,\dots,[\mathbb A\mathbb V]_\zeta$, so that, for example, $[\mathbb A\mathbb V]_\zeta=w|_{z=0}$.
For $\mathbb V\in\mathbb X$, the tuple $\mathbb A\mathbb V$ lies in $H^2(\mathcal S_{\mathrm a})^4\times H^2(\mathcal S_{\mathrm w})^2\times H^{5/2}(\T_{k_0})$, the part of the space $\mathbb Y$ in Equation~\eqref{eq:Yspace} that belongs to the evolution equations.

\begin{proposition}\label{prop:goodunknown}
Let $\mathbb U_0\in\mathcal O$ be the laminar state of \S\ref{subsec:flattening}.
Define the \emph{good-unknown map} $\mathcal T\in\mathcal L(\mathbb X)$,
\begin{equation}\label{eq:goodunknown}
\mathcal T\,(r,u,w,s;\,v_x,v_z,q_{\mathrm w};\,\zeta)
:=\bigl(r+\phi\zeta\rho_0',\ u+\phi\zeta U',\ w,\ s;\
v_x,\ v_z,\ q_{\mathrm w}+\phi\zeta p_{\mathrm w,0}';\ \zeta\bigr),
\end{equation}
a bounded linear map of $\mathbb X$ onto itself with a bounded inverse, which replaces $\zeta$ by $-\zeta$ in the correction terms.
Let $\mathcal T_{\mathrm{ev}}$ be the map induced by $\mathcal T$ on the components $(r,u,w,s,v_x,v_z,\zeta)$ with time derivatives, which adds $\phi\zeta\rho_0'$ to $r$ and $\phi\zeta U'$ to $u$.
Then the following hold for every $\mathbb V\in\mathbb X$.
\begin{enumerate}[label=\textup{(\roman*)}]
\item\label{it:gu-dyn} For the evolution equations,
\begin{equation}
\mathrm d\mathcal N_{\mathrm{dyn}}(\mathbb U_0)\bigl[\mathcal T\mathbb
V\bigr]
=\mathcal T_{\mathrm{ev}}\bigl[\mathbb A\mathbb V\bigr].
\end{equation}
Thus the right side adds $\phi\,[\mathbb A\mathbb V]_\zeta\,\rho_0'$ to the $r$-component and $\phi\,[\mathbb A\mathbb V]_\zeta\,U'$ to the $u$-component of $\mathbb A\mathbb V$.
\item\label{it:gu-con} For the constraint equations,
\begin{multline*}
\mathrm d\mathcal N_{\mathrm{con}}(\mathbb U_0)\bigl[\mathcal T\mathbb
V\bigr]
=\Bigl(\nabla\cdot\mathbf v;\ \
w|_{z=0}-v_z|_{z=0};\\
\bigl(q_{\mathrm w}|_{z=0}-\rhow g\zeta\bigr)
-\bigl(q|_{z=0}-\rho_0(0)g\zeta\bigr)+\sigma\zeta_{xx};\ \
w|_{z=\ha};\ v_z|_{z=-\hw}\Bigr),
\end{multline*}
with $\mathbf v=(v_x,v_z)$ and $q=a^2r+P_Ss$ the Eulerian pressure variation of \S\ref{sec:linearization}.
\end{enumerate}
Consequently a $C^1$ curve $t\mapsto\mathbb V(t)\in\mathbb X$ solves the linearized system of \cref{prop:linearized} if and only if $t\mapsto\mathcal T\mathbb V(t)$ solves the linearization at $\mathbb U_0$ of the flattened system in Equation~\eqref{eq:flat-system}.
In other words, after the change of variables $\mathcal T$, the Fr\'echet derivative $\mathrm d\mathcal N(\mathbb U_0)$ is the linearized system of \cref{prop:linearized}.
\end{proposition}

We call $\mathbb V$ the good unknown.
It is the Eulerian variation, whereas $\mathcal T\mathbb V$ is the variation of the flattened fields.

\begin{proof}
Write $(\tilde r,\tilde u^1,\tilde w^1,\tilde s;\,\tilde{\mathbf v}^1;\,\tilde q_{\mathrm w};\,\zeta):=\mathcal T\mathbb V$, so that $\tilde r=r+\phi\zeta\rho_0'$, $\tilde u^1=u+\phi\zeta U'$, $\tilde w^1=w$, $\tilde s=s$, $\tilde{\mathbf v}^1=\mathbf v$, and $\tilde q_{\mathrm w}=q_{\mathrm w}+\phi\zeta p_{\mathrm w,0}'$.
These are the chain-rule relations between the variation of a pullback and the Eulerian variation at fixed points.
For example, for a family of densities $\rho(\vartheta)$ with $\rho(0)=\rho_0$ and $\partial_\vartheta\rho|_{\vartheta=0}=r$, one has $\partial_\vartheta\bigl[\rho(\vartheta)\circ\Phi_{\vartheta\zeta}\bigr] \big|_{\vartheta=0}=r+\phi\zeta\,\rho_0'$.
The superscript $1$ marks a variation whose letter alone would denote a field of $\mathbb U$, as for $\tilde u^1$, $\tilde w^1$, and $\tilde{\mathbf v}^1$, while $\tilde r$, $\tilde s$, and $\tilde q_{\mathrm w}$ need no mark because the corresponding fields of $\mathbb U$ are $\tilde\rho$, $\tilde S$, and $\tilde p_{\mathrm w}$.
Since $\mathcal N$ is $C^1$ by \cref{prop:residual}, its derivative in the direction $\mathcal T\mathbb V$ can be computed term by term.
We display the mass and momentum components of the air, in which the cancellations occur, and indicate the remaining components.

At $\mathbb U_0$, we have $\tilde\rho=\rho_0(\tilde z)$, $\tilde u=U(\tilde z)$, $\tilde w=0$, $\eta=0$, $J=1$, and $\mathcal K_\eta(\mathbb U_0)=0$.
The derivative of $\mathcal K_\eta(\mathbb U)$ at $\mathbb U_0$ in the direction $\mathcal T\mathbb V$, which we denote by $\mathcal K^1$, is $\mathcal K^1=\tilde w^1|_{\tilde z=0}-U(0)\zeta_x=w|_{z=0}$, because $U(0)=0$.
For the mass equation, the derivative of the component $[\mathcal N_{\mathrm{dyn}}]_{\tilde\rho}$ of Equation~\eqref{eq:Ndyn-mass} at $\mathbb U_0$ in the direction $\mathcal T\mathbb V$ is
\begin{equation}
\mathrm d\bigl[\mathcal N_{\mathrm{dyn}}\bigr]_{\tilde\rho}
\bigl[\mathcal T\mathbb V\bigr]
=\phi\rho_0'\,w|_{z=0}
-U\partial_x\tilde r+U\phi\zeta_x\rho_0'
-\tilde w^1\rho_0'
-\rho_0\bigl(\partial_x\tilde u^1-\phi\zeta_xU'
+\partial_{\tilde z}\tilde w^1\bigr).
\end{equation}
When we substitute $\tilde r=r+\phi\zeta\rho_0'$ and $\tilde u^1=u+\phi\zeta U'$, the terms $-U\partial_x(\phi\zeta\rho_0')=-U\phi\zeta_x\rho_0'$ and $-\rho_0\partial_x(\phi\zeta U')=-\rho_0\phi\zeta_xU'$ cancel the two terms containing $\zeta_x$, and we obtain
\begin{equation}
\mathrm d\bigl[\mathcal N_{\mathrm{dyn}}\bigr]_{\tilde\rho}
\bigl[\mathcal T\mathbb V\bigr]
=\phi\rho_0'\,w|_{z=0}
+\Bigl[-U\partial_xr-\rho_0'w-\rho_0(\partial_xu+\partial_zw)\Bigr].
\end{equation}
The bracket is $[\mathbb A\mathbb V]_r$, the right-hand side of Equation~\eqref{eq:lin-cont} solved for $\partial_tr$, and the first term is $\phi\rho_0'[\mathbb A\mathbb V]_\zeta$.
Together they form the $r$-component of $\mathcal T_{\mathrm{ev}}[\mathbb A\mathbb V]$ in \ref{it:gu-dyn}.

For vertical momentum, we differentiate the component $[\mathcal N_{\mathrm{dyn}}]_{\tilde w}$ of Equation~\eqref{eq:Ndyn-mass}.
At $\mathbb U_0$, the factor $\tilde w=0$ removes every coupling of $\mathfrak M^{\eta}\tilde w$ to $\eta$, so this term contributes $-U\partial_x\tilde w^1$.
The variation of the pressure $\tilde p_{\mathrm a}=P(\tilde\rho,\tilde S)$ in the direction $\mathcal T\mathbb V$ is the flattened counterpart $\tilde q:=a^2\tilde r+P_S\tilde s$ of $q=a^2r+P_Ss$, and $\tilde q=q+\phi\zeta\,a^2\rho_0'=q+\phi\zeta\,p_0'$ by $S_0'=0$ and Equation~\eqref{eq:hydrostatic}.
Hence
\begin{equation}
\mathrm d\bigl[\mathcal N_{\mathrm{dyn}}\bigr]_{\tilde w}
\bigl[\mathcal T\mathbb V\bigr]
=-U\partial_x\tilde w^1
-\frac1{\rho_0}\Bigl[\partial_{\tilde z}\bigl(q+\phi\zeta p_0'\bigr)
-\phi'\zeta\,p_0'
+\tilde r\,g\Bigr].
\end{equation}
The term $-\phi'\zeta p_0'$ is the variation of $1/J$ against the base-state gradient $\partial_{\tilde z}p_0=p_0'$, and the term $\tilde rg$ is the variation of $1/\tilde\rho$ against $\partial_{\tilde z}p_0=-\rho_0g$.
Now $\partial_{\tilde z}(\phi\zeta p_0')-\phi'\zeta p_0' =\phi\zeta p_0''$ and $\tilde rg=rg+\phi\zeta\rho_0'g$, and hydrostatic balance gives $p_0''+g\rho_0'=(-\rho_0g)'+g\rho_0'=0$.
So the corrections containing $\zeta$ cancel because the base state solves the equations, and this component equals $-U\partial_xw-(\partial_zq+gr)/\rho_0=[\mathbb A\mathbb V]_w$, the right-hand side of Equation~\eqref{eq:lin-zmom} solved for $\partial_tw$.
Since $\mathcal T$ leaves the $\tilde w$-component unchanged, this is the corresponding component of the identity in \ref{it:gu-dyn}.

For horizontal momentum, we differentiate the component $[\mathcal N_{\mathrm{dyn}}]_{\tilde u}$ of Equation~\eqref{eq:Ndyn-mass}.
The variation of $1/\tilde\rho$ multiplies $\partial_x^{\eta}p_0$, which vanishes at the laminar state because the base pressure $p_0$ does not depend on $x$ and $\eta=0$, so no terms with $1/\rho_0^2$ remain.
This gives
\begin{equation}
\mathrm d\bigl[\mathcal N_{\mathrm{dyn}}\bigr]_{\tilde u}
\bigl[\mathcal T\mathbb V\bigr]
=\phi\,U'\,\mathcal K^1
-U\partial_x\tilde u^1+U\phi\zeta_xU'
-\tilde w^1U'
-\frac1{\rho_0}\bigl[\partial_x(q+\phi\zeta p_0')
-\phi\zeta_x\,p_0'\bigr],
\end{equation}
with $\mathcal K^1=w|_{z=0}$.
The mesh term of $\mathfrak M^{\eta}$ enters with the plus sign, since it is subtracted inside $\mathfrak M^{\eta}$.
When we substitute $\tilde u^1=u+\phi\zeta U'$, the term $-U\partial_x(\phi\zeta U')=-U\phi\zeta_xU'$ cancels $+U\phi\zeta_xU'$, and $\partial_x(\phi\zeta p_0')=\phi\zeta_xp_0'$ cancels the term $-\phi\zeta_xp_0'$ in the last bracket, which leaves
\begin{equation}
\mathrm d\bigl[\mathcal N_{\mathrm{dyn}}\bigr]_{\tilde u}
\bigl[\mathcal T\mathbb V\bigr]
=\Bigl[-U\partial_xu-U'w-\frac1{\rho_0}\partial_xq\Bigr]
+\phi\,U'\,w|_{z=0}
=[\mathbb A\mathbb V]_u+\phi\,U'\,[\mathbb A\mathbb V]_\zeta .
\end{equation}
This is the $u$-component of $\mathcal T_{\mathrm{ev}}[\mathbb A\mathbb V]$, and the bracket is Equation~\eqref{eq:lin-xmom} solved for $\partial_tu$.
Thus the mesh velocity supplies the good-unknown correction to the time derivative of $u$.
The entropy component is pure transport, because $S_0'=0$.
The water momentum components repeat these computations with $\rhow$ in place of $\rho_0$, zero base velocity, and $p_{\mathrm w,0}''=0$ in place of hydrostatic balance, which gives Equation~\eqref{eq:linear-water}.
The divergence constraint has no terms containing $\eta$ at $\mathbb U_0$, because these terms multiply derivatives of the base water velocity, which is zero.

For the constraint equations in Equation~\eqref{eq:Ncon}, the kinematic compatibility condition linearizes to $w|_{z=0}-U(0)\zeta_x-v_z|_{z=0}=w|_{z=0}-v_z|_{z=0}$ by Hypothesis \ref{hyp:noslip}.
The component $[\mathcal N_{\mathrm{dyn}}]_\eta$ linearizes to $\mathcal K^1=w|_{z=0}$, which together with the compatibility condition reproduces Equation~\eqref{eq:lin-kinematic}.
The dynamic condition linearizes to $\tilde q_{\mathrm w}|_{\tilde z=0}-\tilde q|_{\tilde z=0}-\sigma \mathrm d\kappa(0)[\zeta]$, where $\mathrm d\kappa(0)$ is the Fr\'echet derivative of the curvature at $\eta=0$.
Here $\tilde q_{\mathrm w}|_{\tilde z=0}$ and $\tilde q|_{\tilde z=0}$ are the traces at $\tilde z=0$ of the variations of $\tilde p_{\mathrm w}$ and $P(\tilde\rho,\tilde S)$, which equal the variations of their traces because the trace map is linear and bounded.
Since $\phi(0)=1$, these traces are $\tilde q_{\mathrm w}|_{\tilde z=0}=q_{\mathrm w}|_{z=0}+\zeta p_{\mathrm w,0}'(0)=q_{\mathrm w}|_{z=0}-\rhow g\zeta$ and $\tilde q|_{\tilde z=0}=q|_{z=0}+\zeta p_0'(0) =q|_{z=0}-\rho_0(0)g\zeta$, and $\mathrm d\kappa(0)[\zeta]=-\zeta_{xx}$, which gives the Lagrangian combination of Equation~\eqref{eq:lin-dynamic}.
Thus, in terms of the good unknown, the variations of the flattened pressures at $\tilde z=0$ are the Lagrangian pressure variations.
The impermeability conditions at the top and bottom are unchanged by $\mathcal T$ because $\phi$ vanishes there.
The final equivalence statement follows because $\mathcal T$ is invertible, with a bounded inverse, and does not depend on time.
\end{proof}

\subsection{The closed linearized operator and its eigenvalues}
\label{subsec:modal-operator}

In this subsection, we define the operator $\Lop_k$ and show that its eigenvalues with $\Rec\lambda\neq0$ are the values $\lambda=-\im kc$ at the zeros of $\Edet(k,\cdot,\eps)$.
The coefficients depend only on $z$, so the linearized evolution decomposes into spatial Fourier components, and we consider one wavenumber $k>0$.
Negative wavenumbers are recovered by the reflection symmetry of \cref{lem:determinant}\ref{it:det-hol}.
We define $\Lop_k$ on the Eulerian Fourier amplitudes and use pressure rather than density as the thermodynamic variable, because the vertical momentum equation and the dynamic condition involve the pressure and its derivative.
The water pressure is not a component of the state, because it is determined by the other components.

\begin{definition}\label{def:modal-operator}
Let $k>0$ and $\eps>0$, with the density scaling of \S\ref{subsec:eps-family}, so that $\rho_0(0)=\eps\rhow$.
Set
\begin{equation}
\mathcal V_k:=\bigl\{\hat{\mathbf v}=(\hat v_x,\hat v_z)\in
L^2(-\hw,0)^2:\ \im k\hat v_x+\hat v_z'=0,\ \hat v_z(-\hw)=0\bigr\},
\end{equation}
where the divergence condition holds in the sense of distributions, that is, $\int_{-\hw}^0(\im k\hat v_x\psi-\hat v_z\psi')\,\dd z=0$ for every $\psi\in C_c^\infty(-\hw,0)$.
Here $C_c^\infty(-\hw,0)$ is the space of smooth functions with compact support in $(-\hw,0)$, the smooth functions on $(-\hw,0)$ that vanish near $-\hw$ and near $0$.
For such $\hat{\mathbf v}$, $\hat v_z'=-\im k\hat v_x\in L^2$, so $\hat v_z\in H^1$ and its value at $z=-\hw$ is defined.
The set $\mathcal V_k$ is a closed subspace of $L^2(-\hw,0)^2$, because the divergence condition passes to limits in $L^2$, and then $\hat v_z$ converges in $H^1$.
Set
\begin{gather*}
\mathcal H_k:=L^2(0,\ha)^{4}\times\mathcal V_k\times\C\ \ni\
\hat{\mathbb U}=(\hat u,\hat w,\hat q,\hat s,\hat{\mathbf v},\hat\zeta),\\
\mathcal D(\Lop_k):=\bigl\{\hat{\mathbb U}\in\mathcal H_k:\
\hat w,\hat q\in H^1(0,\ha),\ \hat w(\ha)=0,\
\hat w(0)=\hat v_z(0)\bigr\},
\end{gather*}
where $\mathcal H_k$ carries the norm given by $\|\hat{\mathbb U}\|_{\mathcal H_k}^2:=\|\hat u\|_{L^2}^2+\|\hat w\|_{L^2}^2+\|\hat q\|_{L^2}^2+\|\hat s\|_{L^2}^2+\|\hat{\mathbf v}\|_{L^2}^2+|\hat\zeta|^2$.
For $\hat{\mathbb U}\in\mathcal D(\Lop_k)$, define $\hat r:=(\hat q-P_S\hat s)/a^2$ and the water pressure amplitude
\begin{equation}\label{eq:qw-solve}
\hat q_{\mathrm w}(z):=\bigl[\hat q(0)+G_\sigma\hat\zeta\bigr]\,
\frac{\cosh k(z+\hw)}{\cosh k\hw},
\qquad
G_\sigma:=\bigl(\rhow-\rho_0(0)\bigr)g+\sigma k^2 .
\end{equation}
It is the unique solution of $\hat q_{\mathrm w}''=k^2\hat q_{\mathrm w}$ on $(-\hw,0)$ with $\hat q_{\mathrm w}'(-\hw)=0$ and $\hat q_{\mathrm w}(0)=\hat q(0)+G_\sigma\hat\zeta$, which is the linearized dynamic condition in Equation~\eqref{eq:lin-dynamic} for these amplitudes.
Define
\begin{equation}\label{eq:Lop-def}
\Lop_k\hat{\mathbb U}:=
\begin{pmatrix}
-\im kU\hat u-U'\hat w-\im k\hat q/\rho_0\\[2pt]
-\im kU\hat w-\bigl(\hat q'+g\hat r\bigr)/\rho_0\\[2pt]
-\im kU\hat q+\rho_0g\hat w-\rho_0a^2\bigl(\im k\hat u+\hat w'\bigr)\\[2pt]
-\im kU\hat s\\[2pt]
-\bigl(\im k\hat q_{\mathrm w},\ \hat q_{\mathrm w}'\bigr)/\rhow\\[2pt]
\hat v_z(0)
\end{pmatrix}.
\end{equation}
\end{definition}

The six components of $\Lop_k\hat{\mathbb U}$ in Equation~\eqref{eq:Lop-def} are the amplitudes of the parts proportional to $e^{\im kx}$ of Equations~\eqref{eq:linear-air}--\eqref{eq:lin-kinematic}, solved for the time derivatives.
From top to bottom, they come from the horizontal and vertical momentum equations, the equation for the pressure, the entropy equation, the water momentum equation, and the kinematic condition.
The $\hat q$-component is $a^2$ times the mass equation plus $P_S$ times the entropy equation, where we used $a^2\rho_0'=p_0'=-\rho_0g$.
The water component lies in $\mathcal V_k$, because $\im k(\im k\hat q_{\mathrm w})+(\hat q_{\mathrm w}')'=0$ and $\hat q_{\mathrm w}'(-\hw)=0$.
The dynamic condition is built into Equation~\eqref{eq:qw-solve}, and the two kinematic conditions appear as the $\hat\zeta$-component and as the constraint $\hat w(0)=\hat v_z(0)$ in the domain.
Unlike the condition $w(0)=\im kV(0)\hat\zeta$ of a normal mode in Equation~\eqref{eq:mode-kinematic}, this constraint does not involve the phase speed, so the domain of $\Lop_k$ is the same for all eigenvalues.

\begin{lemma}\label{lem:closed}
The operator $\Lop_k$ is densely defined and closed on $\mathcal H_k$.
That is, $\mathcal D(\Lop_k)$ is dense in $\mathcal H_k$, and the graph $\{(\hat{\mathbb U},\Lop_k\hat{\mathbb U}):\hat{\mathbb U}\in\mathcal D(\Lop_k)\}$ is closed in $\mathcal H_k\times\mathcal H_k$.
\end{lemma}

\begin{proof}
To prove density, let $\hat{\mathbb U}\in\mathcal H_k$ and approximate $\hat u,\hat q,\hat s$ in $L^2$ by smooth functions.
We keep $\hat{\mathbf v}$ and $\hat\zeta$, and we approximate $\hat w$ by $\hat w_n:=\check w_n+\hat v_z(0)\,\theta_n$.
Here $\check w_n\in C_c^\infty(0,\ha)$ with $\check w_n\to\hat w$ in $L^2$, which is possible because $C_c^\infty(0,\ha)$ is dense in $L^2(0,\ha)$, and $\theta_n(z):=\max\{0,1-nz/\ha\}$ is a piecewise linear function with $\theta_n(0)=1$, $\theta_n=0$ on $[\ha/n,\ha]$, and $\|\theta_n\|_{L^2}\to0$.
Then $\hat w_n\in H^1$, $\hat w_n(0)=\hat v_z(0)$, $\hat w_n(\ha)=0$, and $\hat w_n\to\hat w$ in $L^2$, so the approximants lie in $\mathcal D(\Lop_k)$ and converge to $\hat{\mathbb U}$ in $\mathcal H_k$.

For closedness, let $\hat{\mathbb U}_n\in\mathcal D(\Lop_k)$ with $\hat{\mathbb U}_n\to\hat{\mathbb U}$ and $\Lop_k\hat{\mathbb U}_n\to\mathbb F$ in $\mathcal H_k$.
By the $\hat q$-component of Equation~\eqref{eq:Lop-def}, $\rho_0a^2\hat w_n'=-[\Lop_k\hat{\mathbb U}_n]_{\hat q} -\im kU\hat q_n+\rho_0g\hat w_n-\im k\rho_0a^2\hat u_n$ converges in $L^2$.
Since $\rho_0a^2$ is continuous and positive on $[0,\ha]$, $\hat w_n'$ converges in $L^2$, so $\hat w\in H^1$, $\hat w_n\to\hat w$ in $H^1$, and the values $\hat w_n(0)$ and $\hat w_n(\ha)$ converge.
By the $\hat w$-component, $\hat q_n'=-\rho_0[\Lop_k\hat{\mathbb U}_n]_{\hat w} -\im k\rho_0U\hat w_n-g\hat r_n$ converges in $L^2$, because $\hat r_n=(\hat q_n-P_S\hat s_n)/a^2$ converges in $L^2$, as $1/a^2$ and $P_S/a^2$ are bounded.
Hence $\hat q\in H^1$ and $\hat q_n(0)\to\hat q(0)$.
On the water side, $\hat v_{z,n}'=-\im k\hat v_{x,n}$ converges in $L^2$, so $\hat v_{z,n}\to\hat v_z$ in $H^1$ and $\hat v_{z,n}(0)\to\hat v_z(0)$.
The boundary conditions $\hat w(\ha)=0$ and $\hat w(0)=\hat v_z(0)$ pass to the limit, so $\hat{\mathbb U}\in\mathcal D(\Lop_k)$.
The water pressure in Equation~\eqref{eq:qw-solve} is $\hat q(0)+G_\sigma\hat\zeta$ times a smooth function of $z$ that does not depend on $n$, so it converges with all its derivatives, uniformly on $[-\hw,0]$, when $(\hat q_n(0),\hat\zeta_n)\to(\hat q(0),\hat\zeta)$.
Hence $\Lop_k\hat{\mathbb U}_n$ converges to $\Lop_k\hat{\mathbb U}$ in $\mathcal H_k$, component by component, and $\mathbb F=\Lop_k\hat{\mathbb U}$.
\end{proof}

\begin{theorem}\label{thm:pointspec}
Let $k>0$, $\eps>0$, and $c\in\C$ with $\Imc c\neq0$, and set $\lambda:=-\im kc$.
Then $\lambda$ is an eigenvalue of $\Lop_k$ if and only if $\Edet(k,c,\eps)=0$.
In that case the eigenspace $\ker(\Lop_k-\lambda)$, the set of $\hat{\mathbb U}\in\mathcal D(\Lop_k)$ with $\Lop_k\hat{\mathbb U}=\lambda\hat{\mathbb U}$, is one-dimensional, and it is spanned by the amplitudes $(\hat u,\hat w,\hat q,\hat s,\hat{\mathbf v},\hat\zeta)$ of the normal mode of \cref{prop:reduction}.
In particular every root $c$ of the dispersion relation in Equation~\eqref{eq:dispersion} with $\Imc c>0$ gives the eigenvalue $\lambda=-\im kc$ of $\Lop_k$ with $\Rec\lambda=k\Imc c>0$.
\end{theorem}

\begin{proof}
Let $\hat{\mathbb U}\in\mathcal D(\Lop_k)$, $\hat{\mathbb U}\neq0$, satisfy $\Lop_k\hat{\mathbb U}=\lambda\hat{\mathbb U}$.
The $\hat s$-component of this equation gives $\im k(c-U)\hat s=0$, and $V=U-c\neq0$, so $\hat s\equiv0$ and $\hat r=\hat q/a^2$.
The $\hat u$-component contains no derivative, $\im kV\hat u=-(U'\hat w+\im k\hat q/\rho_0)$, and determines $\hat u\in L^2$ from $(\hat w,\hat q)$.
When we substitute this expression into the $\hat q$-component, multiply by $V$, and rearrange, we obtain
\begin{equation}
\im k\hat q
=\frac{\rho_0a^2}{a^2-V^2}\bigl[V\hat w'-(U'+bV)\hat w\bigr]
=A\bigl[V\hat w'-(U'+bV)\hat w\bigr],
\end{equation}
where the division by $a^2-V^2$ is permitted for $\Imc c\neq0$ by \cref{lem:regular-solve}, and $A$ is defined in Equation~\eqref{eq:A-def}.
Thus $\im k\hat q=\chi$ with $\chi$ as in Equation~\eqref{eq:A-def}, and the $\hat w$-component with $\hat r=\hat q/a^2$ becomes $\chi'=-b\chi+\rho_0k^2V\hat w$.
Hence the pair $(\hat w,\chibar)$ with $\chibar=\im k\hat q/(\eps\rhow)$ solves the system in Equation~\eqref{eq:wchi-recall}.
The coefficients of this system are continuous, so $\hat w,\hat q\in C^1$.
The condition $\hat w(\ha)=0$ makes the air component a constant multiple $\mu\,(w,\chibar)(\cdot;k,c)$ of the top-normalized solution (\cref{lem:regular-solve}).
On the water side, $\lambda\neq0$, and the water component gives $\hat{\mathbf v}=-\nabla_k\hat q_{\mathrm w}/(\lambda\rhow)$ with $\nabla_k:=(\im k,\partial_z)$.
Hence $\hat v_z(0)=-\hat q_{\mathrm w}'(0)/(\lambda\rhow) =-k\tanh(k\hw)\,\hat q_{\mathrm w}(0)/(\lambda\rhow)$ by Equation~\eqref{eq:qw-solve}.
The $\hat\zeta$-component $\lambda\hat\zeta=\hat v_z(0)$ and the domain constraint $\hat w(0)=\hat v_z(0)$ then give $\hat w(0)=\lambda\hat\zeta=\im kV(0)\hat\zeta$ and
\begin{equation}
\lambda^2\rhow\,\hat\zeta=-k\tanh(k\hw)\,\hat q_{\mathrm w}(0)
\quad\Longleftrightarrow\quad
\rhow\Bigl[kc^2\coth(k\hw)-g-\frac{\sigma k^2}{\rhow}\Bigr]\hat\zeta
=\hat q(0)-\rho_0(0)g\,\hat\zeta ,
\end{equation}
where the equivalence uses $\lambda^2=-k^2c^2$ and Equation~\eqref{eq:qw-solve}.
The second relation is the dispersion relation in Equation~\eqref{eq:dispersion0}, written for general $\hat\zeta$.
If $\hat\zeta=0$, this relation and Equation~\eqref{eq:qw-solve} force $\hat q_{\mathrm w}(0)=\hat q(0)=0$, and $\hat w(0)=\lambda\hat\zeta=0$.
Then $(\hat w,\chibar)(0)=(0,0)$, so $\mu=0$ by uniqueness for the system in Equation~\eqref{eq:wchi-recall}, since the top-normalized solution is not identically zero.
Consequently $\hat q_{\mathrm w}\equiv0$, $\hat{\mathbf v}=0$, and $\hat{\mathbb U}=0$, a contradiction.
This is the form at a single wavenumber of the exclusion of modes with $\hat\zeta=0$ in \cref{prop:reduction}.
So $\hat\zeta\neq0$, and, as in the proof of \cref{lem:determinant}\ref{it:det-modes}, the displayed dispersion relation with $\hat\zeta=\mu w(0;k,c)/(\im kV(0))$ and $\hat q(0)=\eps\rhow\,\mu\,\chibar(0;k,c)/(\im k)$ is equivalent to $\Edet(k,c,\eps)=0$.

Conversely, if $\Edet(k,c,\eps)=0$, then $w(0;k,c)\neq0$ by \cref{lem:determinant}\ref{it:det-nonvanish}.
Let $\hat w:=w(\cdot;k,c)$ and $\hat q:=\eps\rhow\chibar(\cdot;k,c)/(\im k)$, and set $\hat s:=0$, $\hat u:=-(U'\hat w+\im k\hat q/\rho_0)/(\im kV)$, $\hat\zeta:=\hat w(0)/(\im kV(0))\neq0$, $\hat q_{\mathrm w}$ by Equation~\eqref{eq:qw-solve}, and $\hat{\mathbf v}:=-\nabla_k\hat q_{\mathrm w}/(\lambda\rhow)$.
All regularity and domain conditions hold.
Indeed, $\hat w,\hat q\in C^1$, the field $\hat{\mathbf v}$ is smooth and divergence-free with $\hat v_z(-\hw)=0$, and $\hat w(0)=\hat v_z(0)$ and the $\hat\zeta$-component of $\Lop_k\hat{\mathbb U}=\lambda\hat{\mathbb U}$ hold because $\Edet=0$ gives the displayed relation, read in the reverse direction.
Finally $(\Lop_k-\lambda)\hat{\mathbb U}=0$ by construction.

For the dimension, the argument above shows that every eigenfunction has the air component $\mu\,(w,\chibar)(\cdot;k,c)$ for some $\mu\in\C$, and that its other components are then determined by the formulas above.
Hence the eigenspace is spanned by the eigenfunction with $\mu=1$.
\end{proof}

\subsection{The linearized evolution on the constraint space and its closure}
\label{subsec:realization}

In this subsection, we show that $\Lop_k$ is the operator on $L^2$ that the linearized equations of \S\ref{subsec:conjugation} define at wavenumber $k$, and not only an operator with the same eigenvalues.
The two descriptions differ in three ways.
The linearized equations act on the good unknown, which contains the density and the water pressure, whereas $\Lop_k$ acts on the pressure and does not contain the water pressure.
The linearized equations include the constraint equations of Equation~\eqref{eq:Ncon}, whereas $\Lop_k$ builds them into its space, its domain, and the formula for the water pressure in Equation~\eqref{eq:qw-solve}.
In addition, the linearized equations act on $H^3$ states, whereas $\Lop_k$ acts on $L^2$ states.
\cref{prop:realization} connects the two descriptions.
A curve of $H^3$ states solves the linearized equations and satisfies the constraints exactly when the corresponding curve of pressure-based states is a classical trajectory of $\Lop_k$ and the water pressure is given by Equation~\eqref{eq:qw-solve}.
A classical trajectory of $\Lop_k$ is a curve $\hat{\mathbb U}(t)$ in the domain of $\Lop_k$ that is continuous in the graph norm $\|\hat{\mathbb U}\|_{\mathcal H_k}+\|\Lop_k\hat{\mathbb U}\|_{\mathcal H_k}$ of \S\ref{subsec:operator}, continuously differentiable in $\mathcal H_k$, and satisfies $\frac{\dd}{\dd t}\hat{\mathbb U}=\Lop_k\hat{\mathbb U}$, as in \cref{prop:realization}\ref{it:real-traj}.
The operator $\Lop_k$ is also the graph closure of the linearized evolution restricted to the $H^3$ states that satisfy the constraints.
We first define the spaces and maps that connect the two descriptions.

Let $\Pi_{\mathrm{ev}}$ be the map that removes the water-pressure component of a tuple, so that $\Pi_{\mathrm{ev}}\mathbb V=(r,u,w,s,v_x,v_z,\zeta)$ consists of the components with time derivatives.
A $C^1$ curve $t\mapsto\mathcal T\mathbb V(t)$ of flattened variations satisfies the evolution equations of the linearized flattened system when $\frac{\dd}{\dd t}\Pi_{\mathrm{ev}}(\mathcal T\mathbb V)=\mathrm d\mathcal N_{\mathrm{dyn}}(\mathbb U_0)[\mathcal T\mathbb V]$.
Since $\mathcal T$ does not depend on time, $\frac{\dd}{\dd t}\Pi_{\mathrm{ev}}(\mathcal T\mathbb V)=\mathcal T_{\mathrm{ev}}\bigl[\frac{\dd}{\dd t}\Pi_{\mathrm{ev}}\mathbb V\bigr]$, and $\mathrm d\mathcal N_{\mathrm{dyn}}(\mathbb U_0)[\mathcal T\mathbb V]=\mathcal T_{\mathrm{ev}}[\mathbb A\mathbb V]$ by \cref{prop:goodunknown}\ref{it:gu-dyn}.
Since $\mathcal T_{\mathrm{ev}}$ is invertible, because its inverse subtracts the terms $\phi\zeta\rho_0'$ and $\phi\zeta U'$ that it adds, the curve satisfies these equations if and only if the good unknown solves the Eulerian evolution
\begin{equation}\label{eq:euler-evolution}
\frac{\dd}{\dd t}\,\Pi_{\mathrm{ev}}\mathbb V=\mathbb A\mathbb V,
\end{equation}
in which $\mathbb A\mathbb V$ depends on the water pressure $q_{\mathrm w}$ contained in $\mathbb V$.
By \cref{prop:goodunknown}\ref{it:gu-con}, the constraint equations of the flattened system for $\mathcal T\mathbb V$ are the linearized constraints for $\mathbb V$.
We therefore compare $\Lop_k$ with the Eulerian evolution in Equation~\eqref{eq:euler-evolution}.

At wavenumber $k$, the space $\mathbb X_k$ consists of the tuples $\hat{\mathbb V}=(\hat r,\hat u,\hat w,\hat s,\hat v_x,\hat v_z,\hat q_{\mathrm w},\hat\zeta)$ with amplitudes in $H^3$ of their intervals, $[0,\ha]$ in the air and $[-\hw,0]$ in the water, and with $\hat\zeta\in\C$.
Such a tuple corresponds to the tuple of its components multiplied by $e^{\im kx}$.
These functions have complex values and have the period $2\pi/k_0$ in $x$ because $k=nk_0$ with a positive integer $n$, so their real and imaginary parts are elements of $\mathbb X$.
The map $\mathbb A$ has real coefficients and acts on such a tuple through its real and imaginary parts, and $\mathbb A_k$ denotes the corresponding restriction of $\mathbb A$, in which $\partial_x$ becomes multiplication by $\im k$.
The maps $\Pi_{\mathrm{ev}}$ and $\mathcal T$ act on $\mathbb X_k$ by the same formulas.
The operator $\Lop_k$ uses the pressure $\hat q=a^2\hat r+P_S\hat s$ instead of the density, and we define the corresponding change of variables
\begin{equation}\label{eq:Jk-def}
\mathcal Q_k(\hat r,\hat u,\hat w,\hat s,\hat{\mathbf v},\hat\zeta)
:=(\hat u,\hat w,\ a^2\hat r+P_S\hat s,\ \hat s,\hat{\mathbf v},
\hat\zeta).
\end{equation}
It is a bounded linear map of $\Pi_{\mathrm{ev}}\mathbb X_k$ onto $H^3(0,\ha)^4\times H^3(-\hw,0)^2\times\C$ with a bounded inverse, and it extends to a bounded linear map of $L^2(0,\ha)^4\times L^2(-\hw,0)^2\times\C$, a space that contains $\mathcal H_k$, onto itself with a bounded inverse.
Both statements hold because multiplication by $a^2$, by $P_S$, or by $1/a^2$ is bounded on $H^3$ and on $L^2$, and the inverse is given by $\hat r=(\hat q-P_S\hat s)/a^2$.

The water pressure is recovered from its value at the interface by the Poisson operator $\mathcal P_k$.
For $(\mathfrak q,\hat\zeta)\in\C^2$, $\mathcal P_k(\mathfrak q,\hat\zeta)$ is the solution of $\hat q_{\mathrm w}''=k^2\hat q_{\mathrm w}$ on $(-\hw,0)$ with $\hat q_{\mathrm w}'(-\hw)=0$ and $\hat q_{\mathrm w}(0)=\mathfrak q+G_\sigma\hat\zeta$, with $G_\sigma$ as in Equation~\eqref{eq:qw-solve}.
The equation states that $\hat q_{\mathrm w}(z)e^{\im kx}$ is harmonic, that is, it solves Laplace's equation, and in general a Poisson operator maps the boundary values of a harmonic function to the function itself.
In this notation, Equation~\eqref{eq:qw-solve} states that $\hat q_{\mathrm w}=\mathcal P_k(\hat q(0),\hat\zeta)$.
We write $\mathbb A_k^{\mathrm{red}}$ for the evolution in which the water pressure is replaced by this solution,
\begin{equation}\label{eq:Ared-def}
\mathbb A_k^{\mathrm{red}}\hat{\mathbb V}
:=\mathbb A_k\bigl(\hat r,\hat u,\hat w,\hat s,\hat v_x,\hat v_z,\
\mathcal P_k(\hat q(0),\hat\zeta),\ \hat\zeta\bigr),
\end{equation}
where $\hat q(0)$ is the value at $z=0$ of $\hat q=a^2\hat r+P_S\hat s$.
Thus $\mathbb A_k^{\mathrm{red}}\hat{\mathbb V}$ does not depend on the water-pressure component of $\hat{\mathbb V}$.

The linearized constraints at wavenumber $k$ are given by the constraint map $\mathcal C_k$, the restriction to the Fourier component $e^{\im kx}$ of the linearized constraint map $\mathrm d\mathcal N_{\mathrm{con}}(\mathbb U_0)\circ\mathcal T$ of \cref{prop:goodunknown}\ref{it:gu-con},
\begin{multline*}
\mathcal C_k\hat{\mathbb V}
:=\Bigl(\im k\hat v_x+\hat v_z';\ \
\hat w(0)-\hat v_z(0);\ \
\hat q_{\mathrm w}(0)-\hat q(0)-G_\sigma\hat\zeta;\ \
\hat w(\ha);\ \hat v_z(-\hw)\Bigr)\\
\in\ H^2(-\hw,0)\times\C^4
=:\mathbb Y_{\mathrm{con},k} .
\end{multline*}
Its null space $\mathfrak T_k:=\ker\mathcal C_k$, the set of $\hat{\mathbb V}\in\mathbb X_k$ with $\mathcal C_k\hat{\mathbb V}=0$, is the linearized constraint space at wavenumber $k$.
The $H^3$ core is its subspace
\begin{equation}\label{eq:core-def}
\mathfrak D_k:=\bigl\{\hat{\mathbb V}\in\mathfrak T_k:\
\hat q_{\mathrm w}=\mathcal P_k(\hat q(0),\hat\zeta)\bigr\} .
\end{equation}
The elements of $\mathfrak T_k$ have the correct interface value of the water pressure, because the dynamic condition is one of the constraint equations.
The elements of $\mathfrak D_k$ also have the harmonic profile of the water pressure in the interior.
This distinction comes from the differential-algebraic structure of Equation~\eqref{eq:flat-system}.
The constraint equations determine only the interface value of $\hat q_{\mathrm w}$, and the interior profile is determined by the requirement that the water evolution preserve the divergence and bottom constraints, as part \ref{it:real-traj} of \cref{prop:realization} shows.

\begin{proposition}
\label{prop:realization}
Let $k>0$.
Then the following hold.
\begin{enumerate}[label=\textup{(\roman*)}]
\item\label{it:real-split} The map $\mathcal C_k:\mathbb X_k\to\mathbb Y_{\mathrm{con},k}$ is split surjective, that is, it is bounded, its range is all of $\mathbb Y_{\mathrm{con},k}$, and it has a bounded linear right inverse $\mathcal C_k^{\dagger}:\mathbb Y_{\mathrm{con},k}\to\mathbb X_k$, so that $\mathcal C_k\mathcal C_k^{\dagger}=\mathrm{Id}$, the identity map of $\mathbb Y_{\mathrm{con},k}$.
Hence $\mathfrak T_k=\ker\mathcal C_k$ is a closed complemented subspace of $\mathbb X_k$, that is, every element of $\mathbb X_k$ is, in exactly one way, the sum of an element of $\mathfrak T_k$ and an element of the closed subspace $\operatorname{ran}\mathcal C_k^{\dagger}$, the range of $\mathcal C_k^{\dagger}$.
A right inverse is obtained as follows.
Given $(\hat f,\mathfrak c_1,\mathfrak c_2,\mathfrak c_3,\mathfrak c_4)\in\mathbb Y_{\mathrm{con},k}$, set $\hat\zeta:=0$, $\hat r:=\hat s:=\hat u:=0$, and
\begin{equation}
\hat v_z(z):=\mathfrak c_4+\int_{-\hw}^{z}\hat f(z')\,\dd z',
\qquad
\hat v_x:=0,
\end{equation}
so that $\im k\hat v_x+\hat v_z'=\hat f$ and $\hat v_z(-\hw)=\mathfrak c_4$ exactly.
Choose $\hat w\in H^3$ with $\hat w(0)=\mathfrak c_1+\hat v_z(0)$ and $\hat w(\ha)=\mathfrak c_3$, and $\hat q_{\mathrm w}\in H^3$ with $\hat q_{\mathrm w}(0)=\mathfrak c_2$.
Every step is a bounded linear map of $(\hat f,\mathfrak c_1,\dots,\mathfrak c_4)$.
No compatibility condition arises, because the value $\hat v_z(0)$ is not prescribed.
\item\label{it:real-traj} Let $t_1>0$, and let $t\mapsto\hat{\mathbb V}(t)\in\mathbb X_k$ be continuously differentiable on $[0,t_1]$.
Then the following two statements are equivalent.
\begin{enumerate}[label=\textup{(\alph*)}]
\item $\hat{\mathbb V}(t)\in\mathfrak T_k$ for all $t$, and $\hat{\mathbb V}$ solves the Eulerian evolution $\frac{\dd}{\dd t}\Pi_{\mathrm{ev}}\hat{\mathbb V} =\mathbb A_k\hat{\mathbb V}$ of Equation~\eqref{eq:euler-evolution}.
Equivalently, by \cref{prop:goodunknown}, the flattened variation $\mathcal T\hat{\mathbb V}$ solves the linearized flattened system.
\item The curve $\hat{\mathbb U}(t):=\mathcal Q_k\Pi_{\mathrm{ev}} \hat{\mathbb V}(t)$ satisfies $\hat{\mathbb U}\in C([0,t_1];\mathcal D(\Lop_k))\cap C^1([0,t_1];\mathcal H_k)$ and $\frac{\dd}{\dd t}\hat{\mathbb U}=\Lop_k\hat{\mathbb U}$, and the water pressure of $\hat{\mathbb V}(t)$ is $\mathcal P_k(\hat q(0),\hat\zeta)$ for every $t$.
Here $C([0,t_1];\mathcal D(\Lop_k))$ is the space of maps from $[0,t_1]$ into $\mathcal D(\Lop_k)$ that are continuous in the graph norm, and $C^1([0,t_1];\mathcal H_k)$ is the space of continuously differentiable maps from $[0,t_1]$ into $\mathcal H_k$, so $\hat{\mathbb U}$ is a classical trajectory of $\Lop_k$ exactly when it lies in both spaces and satisfies $\frac{\dd}{\dd t}\hat{\mathbb U}=\Lop_k\hat{\mathbb U}$.
\end{enumerate}
\item\label{it:real-core} On the $H^3$ core,
\begin{equation}\label{eq:realization-identity}
\Lop_k\,\mathcal Q_k\Pi_{\mathrm{ev}}\hat{\mathbb V}
=\mathcal Q_k\,
\mathbb A_k^{\mathrm{red}}\hat{\mathbb V},
\qquad
\hat{\mathbb V}\in\mathfrak D_k,
\qquad
\mathcal Q_k\Pi_{\mathrm{ev}}\mathfrak D_k\subset\mathcal D(\Lop_k),
\end{equation}
and $\mathcal Q_k\Pi_{\mathrm{ev}}\mathfrak D_k$ is a core for $\Lop_k$, that is, a subspace of the domain on which the closure of the restriction of $\Lop_k$ is $\Lop_k$ itself.
In other words, $\Lop_k$ is the graph closure of $\mathbb A_k^{\mathrm{red}}$ on the $H^3$ core $\mathfrak D_k$, after the change of variables $\mathcal Q_k$.
\end{enumerate}
\end{proposition}

\begin{proof}
For \ref{it:real-split}, the displayed choices give $\im k\hat v_x+\hat v_z'=\hat f$, and $\hat v_z\in H^3$ because an antiderivative of an $H^2$ function on an interval lies in $H^3$.
The functions $\hat w$ and $\hat q_{\mathrm w}$ can be taken polynomial in $z$, with coefficients linear in the prescribed values.
Applying $\mathcal C_k$ to the constructed tuple returns $(\hat f,\mathfrak c_1,\mathfrak c_2,\mathfrak c_3,\mathfrak c_4)$, the third component because $\hat q=0$ and $\hat\zeta=0$, and each step is bounded and linear.
With $\mathrm{Id}$ now the identity map of $\mathbb X_k$, the map $\mathrm{Id}-\mathcal C_k^{\dagger}\mathcal C_k$ is a projection, that is, it equals its own square, because $\mathcal C_k\mathcal C_k^{\dagger}$ is the identity map of $\mathbb Y_{\mathrm{con},k}$.
It maps $\mathbb X_k$ onto $\ker\mathcal C_k$ and vanishes on $\operatorname{ran}\mathcal C_k^{\dagger}$, which is closed because $\mathcal C_k$ is a bounded left inverse of $\mathcal C_k^{\dagger}$.

For \ref{it:real-traj}, we first prove that (a) implies (b).
Membership in $\mathfrak T_k$ gives, for every $t$, that $\im k\hat v_x+\hat v_z'=0$ and $\hat v_z(-\hw)=0$, so $(\hat v_x,\hat v_z)\in\mathcal V_k$, that $\hat w(\ha)=0$ and $\hat w(0)=\hat v_z(0)$, so $\mathcal Q_k\Pi_{\mathrm{ev}}\hat{\mathbb V}(t)\in \mathcal D(\Lop_k)$, and that the dynamic condition gives the interface value of the water pressure in Equation~\eqref{eq:qw-solve}.
For the interior profile, we differentiate the divergence and bottom constraints in time and insert the water components of $\mathbb A_k$, which gives
\begin{equation}
0=\frac{\dd}{\dd t}\bigl(\im k\hat v_x+\hat v_z'\bigr)
=-\frac1{\rhow}\bigl(\im k(\im k\hat q_{\mathrm w})
+\hat q_{\mathrm w}''\bigr)
=\frac1{\rhow}\bigl(k^2\hat q_{\mathrm w}
-\hat q_{\mathrm w}''\bigr),
\qquad
0=\frac{\dd}{\dd t}\hat v_z(-\hw)
=-\frac{\hat q_{\mathrm w}'(-\hw)}{\rhow}.
\end{equation}
So $\hat q_{\mathrm w}$ solves $\hat q_{\mathrm w}''=k^2\hat q_{\mathrm w}$ with $\hat q_{\mathrm w}'(-\hw)=0$ and the interface value above, and therefore $\hat q_{\mathrm w}=\mathcal P_k(\hat q(0),\hat\zeta)$.
It remains to compare the components of $\mathbb A_k$ with those of $\Lop_k$ in Equation~\eqref{eq:Lop-def}.
Under $\mathcal Q_k$, the $\hat q$-component is $a^2$ times the $\hat r$-component plus $P_S$ times the $\hat s$-component of $\mathbb A_k$, which is the third component of Equation~\eqref{eq:Lop-def} by $p_0'=a^2\rho_0'$, as in the computation after \cref{def:modal-operator}.
The $\hat u$-, $\hat w$-, and $\hat s$-components are identical.
The water and $\hat\zeta$-components are identical once $\hat q_{\mathrm w}=\mathcal P_k(\hat q(0),\hat\zeta)$.
Hence $\frac{\dd}{\dd t}\hat{\mathbb U}=\Lop_k\hat{\mathbb U}$.
Since $\hat{\mathbb V}$ is $C^1$ with values in $\mathbb X_k$, the curve $\hat{\mathbb U}$ is $C^1$ with values in $\mathcal H_k$, and $\Lop_k\hat{\mathbb U}$ is continuous in $\mathcal H_k$ by the formulas in Equation~\eqref{eq:Lop-def}, so $\hat{\mathbb U}$ is continuous in the graph norm.
This is (b).
For the converse, let (b) hold.
The constraint equations hold for every $t$, because the divergence and bottom conditions are built into $\mathcal V_k$, the top and kinematic conditions into $\mathcal D(\Lop_k)$, and the dynamic condition is the interface value in Equation~\eqref{eq:qw-solve}.
Since $\hat r=(\hat q-P_S\hat s)/a^2$ by the definition of $\mathcal Q_k$, the evolution equations in Equation~\eqref{eq:euler-evolution} are Equation~\eqref{eq:Lop-def} read backward through $\mathcal Q_k$, which gives (a).

For \ref{it:real-core}, the identity in Equation~\eqref{eq:realization-identity} is the comparison of the components just performed, now stated for single states, so that no time derivative enters.
For $\hat{\mathbb V}\in\mathfrak D_k$, we have $\hat q_{\mathrm w}=\mathcal P_k(\hat q(0),\hat\zeta)$ by definition.
Hence $\mathbb A_k^{\mathrm{red}}\hat{\mathbb V}=\mathbb A_k\hat{\mathbb V}$, and each component of $\mathcal Q_k\mathbb A_k\hat{\mathbb V}$ equals the corresponding component of $\Lop_k\mathcal Q_k\Pi_{\mathrm{ev}}\hat{\mathbb V}$.
The domain membership was checked in the proof of \ref{it:real-traj}.

For the core property, let $\hat{\mathbb U}=(\hat u,\hat w,\hat q,\hat s,\hat{\mathbf v}, \hat\zeta)\in\mathcal D(\Lop_k)$.
We approximate $\hat{\mathbb U}$ in the graph norm by elements of $\mathcal Q_k\Pi_{\mathrm{ev}}\mathfrak D_k$.
Since $\hat v_x=\im\hat v_z'/k$ on $\mathcal V_k$, the water component is determined by $\hat v_z\in H^1$ with $\hat v_z(-\hw)=0$.
Choose smooth $\check u_n\to\hat u$ and $\check s_n\to\hat s$ in $L^2$, and smooth $\check q_n\to\hat q$, $\check w_n\to\hat w$, and $\check v_{z,n}\to\hat v_z$ in $H^1$.
Choose smooth functions $\psi_0,\psi_1$ on $[0,\ha]$ with $\psi_0(0)=1$, $\psi_0(\ha)=0$, $\psi_1(0)=0$, $\psi_1(\ha)=1$, and $\psi_2$ on $[-\hw,0]$ with $\psi_2(-\hw)=1$, $\psi_2(0)=0$.
We correct the boundary values by adding multiples of these functions,
\begin{equation}
\hat v_{z,n}:=\check v_{z,n}-\check v_{z,n}(-\hw)\,\psi_2,
\qquad
\hat w_n:=\check w_n-\check w_n(\ha)\,\psi_1
-\bigl[\check w_n(0)-\hat v_{z,n}(0)\bigr]\,\psi_0 .
\end{equation}
Then $\hat w_n(\ha)=0$, $\hat v_{z,n}(-\hw)=0$, and $\hat w_n(0)=\hat v_{z,n}(0)$ hold exactly.
Every coefficient of $\psi_0$, $\psi_1$, and $\psi_2$ tends to $0$, because values at the endpoints converge under $H^1$ convergence and the limits satisfy $\hat v_z(-\hw)=0$, $\hat w(\ha)=0$, and $\hat w(0)=\hat v_z(0)$.
Hence $\hat w_n\to\hat w$ and $\hat v_{z,n}\to\hat v_z$ in $H^1$.
We complete the water component by
\begin{equation}
\hat v_{x,n}:=\frac{\im}{k}\,\hat v_{z,n}',
\qquad
\hat{\mathbf v}_n:=(\hat v_{x,n},\hat v_{z,n}),
\end{equation}
so that $\im k\hat v_{x,n}+\hat v_{z,n}'=0$ exactly and $\hat v_{x,n}\to\hat v_x=\im\hat v_z'/k$ in $L^2$.
The approximating tuple is then
\begin{equation}
\hat{\mathbb V}_n
:=\Bigl(\frac{\check q_n-P_S\check s_n}{a^2},\
\check u_n,\ \hat w_n,\ \check s_n,\
\frac{\im}{k}\hat v_{z,n}',\ \hat v_{z,n},\
\mathcal P_k\bigl(\check q_n(0),\hat\zeta\bigr),\
\hat\zeta\Bigr) .
\end{equation}
It lies in $\mathfrak D_k$, because the divergence and bottom conditions hold exactly by the construction of $\hat{\mathbf v}_n$, the kinematic and top conditions by the corrected boundary values, and the dynamic condition and the harmonic profile by the definition of $\mathcal P_k$.
Its image $\hat{\mathbb U}_n:=\mathcal Q_k\Pi_{\mathrm{ev}}\hat{\mathbb V}_n=(\check u_n,\hat w_n,\check q_n,\check s_n,\hat{\mathbf v}_n,\hat\zeta)$ lies in $\mathcal D(\Lop_k)$.
The components of Equation~\eqref{eq:Lop-def} involve $\hat w$ and $\hat q$ through at most one derivative and the values at the endpoints, $\hat u$ and $\hat s$ through multiplication by bounded coefficients, and $\hat{\mathbf v}$ through the water pressure and the value $\hat v_z(0)$.
The water pressure $\mathcal P_k(\check q_n(0),\hat\zeta)$ converges to $\mathcal P_k(\hat q(0),\hat\zeta)$ with all its derivatives, uniformly on $[-\hw,0]$, because $\check q_n(0)\to\hat q(0)$.
Since $\hat w_n$, $\check q_n$, and $\hat v_{z,n}$ converge in $H^1$ and $\check u_n$ and $\check s_n$ converge in $L^2$, the images $\hat{\mathbb U}_n$ converge to $\hat{\mathbb U}$ in the graph norm of $\Lop_k$.
\end{proof}

\begin{remark}
\label{rem:realisation-scope}
\cref{prop:realization} identifies $\Lop_k$ as the graph closure of $\mathbb A_k^{\mathrm{red}}$ on the $H^3$ core, not of the full evolution $\mathbb A_k$.
On a state that satisfies the constraints, the constraint equations determine only the interface value of the water pressure, and the harmonic profile is selected by the requirement that the evolution preserve the divergence and bottom constraints, as in the differential-algebraic structure of Equation~\eqref{eq:flat-system}.
The image $\mathbb A_k^{\mathrm{red}}\hat{\mathbb V}$ of a state $\hat{\mathbb V}\in\mathfrak D_k$ need not itself satisfy the constraints.
The conditions that the image also satisfies the constraints are called higher compatibility conditions, and they concern the time derivative along a classical trajectory.
They are part of the definition of the domains of the powers of $\Lop_k$, since a state lies in the domain of $\Lop_k^2$ exactly when it lies in $\mathcal D(\Lop_k)$ and its image under $\Lop_k$ lies in $\mathcal D(\Lop_k)$ again.
Together with \cref{prop:goodunknown}, \cref{prop:realization} proves \cref{thm:operator}\ref{it:op-realized}.
\end{remark}

\subsection{The resolvent and the algebraic simplicity of the unstable
eigenvalue}
\label{subsec:resolvent}

In this subsection, we solve the equation $(\lambda-\Lop_k)\hat{\mathbb U}=\mathbb F$ for $\lambda=-\im kc$ with $\Imc c\neq0$ and $\Edet(k,c,\eps)\neq0$, and we bound the solution by $|\Edet(k,c,\eps)|^{-1}$ (\cref{prop:resolvent}).
Near a simple zero of $\Edet(k,\cdot,\eps)$, this bound shows that the eigenvalue is isolated and algebraically simple (\cref{cor:algsimple}).
We use $\lambda=-\im kc$, so $\Rec\lambda=k\Imc c$.
The resolvent set of $\Lop_k$ is the set of $\lambda\in\C$ for which $\lambda-\Lop_k$ is a bijection of $\mathcal D(\Lop_k)$ onto $\mathcal H_k$, that is, for which the equation $(\lambda-\Lop_k)\hat{\mathbb U}=\mathbb F$ has exactly one solution $\hat{\mathbb U}\in\mathcal D(\Lop_k)$ for every $\mathbb F\in\mathcal H_k$.
The spectrum of $\Lop_k$ is the complement of the resolvent set.
For $\lambda$ in the resolvent set, the inverse $(\lambda-\Lop_k)^{-1}$ is the resolvent.
It is bounded on $\mathcal H_k$ by the closed graph theorem, which states that a linear map from a Banach space into a Banach space that is defined on the whole first space is bounded when its graph is closed.
The graph of $(\lambda-\Lop_k)^{-1}$ is the graph of $\lambda-\Lop_k$ with the two entries of each pair exchanged, and this graph is closed because $\Lop_k$ is closed.

\begin{proposition}\label{prop:resolvent}
Let $k>0$, $\eps>0$, and a compact set $\Sigma\subset\{c\in\C:\Imc c\neq0\}$ be given.
There is a constant $C$, which depends on $\Sigma$, such that for every $c\in\Sigma$ with $\Edet(k,c,\eps)\neq0$, the operator $\lambda-\Lop_k$ with $\lambda=-\im kc$ is a bijection of $\mathcal D(\Lop_k)$ onto $\mathcal H_k$, so $\lambda$ lies in the resolvent set of $\Lop_k$, and
\begin{equation}
\bigl\|(\lambda-\Lop_k)^{-1}\bigr\|_{\mathcal H_k\to\mathcal H_k}
\ \le\ C\Bigl(1+\frac1{|\Edet(k,c,\eps)|}\Bigr).
\end{equation}
\end{proposition}

\begin{proof}
Since $\Edet(k,c,\eps)\neq0$, \cref{thm:pointspec} shows that $\lambda$ is not an eigenvalue of $\Lop_k$, so $(\lambda-\Lop_k)\hat{\mathbb U}=0$ only for $\hat{\mathbb U}=0$, and the equation $(\lambda-\Lop_k)\hat{\mathbb U}=\mathbb F$ has at most one solution.
To construct a solution and bound it, let $\mathbb F\in\mathcal H_k$ be given, with the components $\hat F_u$, $\hat F_w$, $\hat F_q$, $\hat F_s$, $\hat{\mathbf F}_v=(\hat F_{v,x},\hat F_{v,z})\in\mathcal V_k$, and $\hat F_\zeta$, in the order of the components of $\hat{\mathbb U}$.
We solve $(\lambda-\Lop_k)\hat{\mathbb U}=\mathbb F$ by the elimination in the proof of \cref{thm:pointspec}, now with these source terms.
The $\hat s$-component of this equation is $\im kV\hat s=\hat F_s$, so $\hat s=\hat F_s/(\im kV)\in L^2$.
The $\hat u$-component gives $\im kV\hat u=\hat F_u-U'\hat w-\im k\hat q/\rho_0$.
When we substitute these expressions into the $\hat q$- and $\hat w$-components as in the proof of \cref{thm:pointspec}, the pair $(\hat w,\chibar)$ with $\chibar=\im k\hat q/(\eps\rhow)$ satisfies the inhomogeneous system
\begin{equation}\label{eq:inhomog}
\hat w'=\Bigl(\frac{U'}{V}+b\Bigr)\hat w
+\frac{\chibar}{\Abar V}+\mathfrak s_w,
\qquad
\chibar'=-b\,\chibar+\rhobar k^2V\hat w+\mathfrak s_\chi,
\end{equation}
with the source terms
\begin{equation}
\mathfrak s_w
:=\frac{\hat F_q}{\rho_0a^2}-\frac{\hat F_u}{V},
\qquad
\mathfrak s_\chi
:=\frac{\im k\,\rho_0\hat F_w}{\eps\rhow}
+\frac{b\,P_S\hat F_s}{\eps\rhow\,V},
\end{equation}
which are linear in $\mathbb F$.
The first comes from the elimination of $\hat u$ in the $\hat q$-component, and the second from the $\hat w$-component with $\hat r=(\hat q-P_S\hat s)/a^2$ and the value of $\hat s$.
On $\Sigma$, all denominators are bounded below, because $|V(z)|=|U(z)-c|\ge|\Imc c|\ge\min_{c'\in\Sigma}|\Imc c'|>0$, and $|a^2-V^2|$ is bounded below on the compact set $\Sigma$ by the argument of \cref{lem:regular-solve}.
Hence $\|\mathfrak s_w\|_{L^2}+\|\mathfrak s_\chi\|_{L^2}\le C\|\mathbb F\|_{\mathcal H_k}$ uniformly on $\Sigma$.
Let $Y:=(w,\chibar)(\cdot;k,c)$ be the top-normalized solution, and let $\breve Y=(\breve w,\breve\chi)$ be the solution of the homogeneous system with the values $(\breve w,\breve\chi)(\ha)=(1,0)$, where $\breve\chi$ is the component that corresponds to $\chibar$.
Their modified Wronskian, as in \cref{prop:local}\ref{it:basis-wronskian}, is $\Xi=w\breve\chi-\breve w\chibar=\varkappa_{\mathrm{top}}V$ with the constant $\varkappa_{\mathrm{top}}=\Xi(\ha)/V(\ha)=-\Abar(\ha)\neq0$, by the computation in the proof of \cref{prop:convergence}\ref{it:conv-matrices}, which is valid for $\Imc c\neq0$.
Variation of constants gives the particular solution that vanishes at $z=\ha$,
\begin{equation}\label{eq:vop}
Y_p(z):=Y(z)\int_{\ha}^{z}
\frac{\breve\chi\,\mathfrak s_w-\breve w\,\mathfrak s_\chi}
{\varkappa_{\mathrm{top}}V}\,\dd z'
\;+\;\breve Y(z)\int_{\ha}^{z}
\frac{w\,\mathfrak s_\chi-\chibar\,\mathfrak s_w}
{\varkappa_{\mathrm{top}}V}\,\dd z' .
\end{equation}
Every solution of Equation~\eqref{eq:inhomog} has the form $\mu Y+\breve\mu\breve Y+Y_p$ with constants $\mu$ and $\breve\mu$, and the condition $\hat w(\ha)=0$ gives $\breve\mu=0$, because $w(\ha)=0$, $\breve w(\ha)=1$, and $Y_p(\ha)=0$.
On $\Sigma$, the particular solution obeys
\begin{equation}\label{eq:Xp-bound}
\|Y_p\|_{H^1(0,\ha)\times H^1(0,\ha)}\ \le\
C\,\|\mathbb F\|_{\mathcal H_k},
\end{equation}
since the homogeneous solutions and $1/(\varkappa_{\mathrm{top}}V)$ are bounded on the compact set $\Sigma$, the integrals are bounded by the $L^1$ norms of the source terms, which are at most $C\|\mathbb F\|_{\mathcal H_k}$, and the derivatives of $Y_p$ are recovered from Equation~\eqref{eq:inhomog} itself.

On the water side, the water component gives $\hat{\mathbf v}=(\hat{\mathbf F}_v -\nabla_k\hat q_{\mathrm w}/\rhow)/\lambda$, which lies in $\mathcal V_k$ because $\hat{\mathbf F}_v$ does and $\hat q_{\mathrm w}$ solves $\hat q_{\mathrm w}''=k^2\hat q_{\mathrm w}$ with $\hat q_{\mathrm w}'(-\hw)=0$.
The $\hat\zeta$-component and the domain constraint $\hat w(0)=\hat v_z(0)$ give $\hat\zeta=(\hat w(0)+\hat F_\zeta)/\lambda$.
When we combine these two components, the domain constraint, and the interface value $\hat q_{\mathrm w}(0)=\hat q(0)+G_\sigma\hat\zeta$ from Equation~\eqref{eq:qw-solve}, as in the proof of \cref{thm:pointspec}, we obtain the single relation
\begin{equation}\label{eq:interface-functional}
\mathcal I(\hat w,\hat q)
:=\bigl[\lambda^2\rhow+k\tanh(k\hw)\,G_\sigma\bigr]\hat w(0)
+\lambda\,k\tanh(k\hw)\,\hat q(0)
=\mathfrak d(\mathbb F),
\end{equation}
where
\begin{equation}
\mathfrak d(\mathbb F)
:=\lambda\rhow\,\hat F_{v,z}(0)-k\tanh(k\hw)\,G_\sigma\,
\hat F_\zeta
\end{equation}
is a bounded linear functional of $\mathbb F$, that is, a bounded linear map from $\mathcal H_k$ to $\C$.
It is bounded because $\hat F_{v,z}'=-\im k\hat F_{v,x}\in L^2$ on $\mathcal V_k$, so $\hat F_{v,z}\in H^1$ and its value at $z=0$ is bounded by its $H^1$ norm.
The same substitutions as in the proof of \cref{lem:determinant} give the identity
\begin{equation}\label{eq:ell-Edet}
\mathcal I\bigl(w(\cdot;k,c),\ \eps\rhow\chibar(\cdot;k,c)/(\im k)\bigr)
=-k\tanh(k\hw)\,\rhow\;\Edet(k,c,\eps).
\end{equation}
Let $\mathcal I(Y_p)$ denote the value of $\mathcal I$ at the first component of $Y_p$ and at $\eps\rhow/(\im k)$ times its second component.
Imposing Equation~\eqref{eq:interface-functional} on $\mu Y+Y_p$ then determines
\begin{equation}
\mu=\frac{\mathfrak d(\mathbb F)-\mathcal I(Y_p)}
{-k\tanh(k\hw)\,\rhow\,\Edet(k,c,\eps)} .
\end{equation}
The bounds now follow in order,
\begin{equation}
|\mathcal I(Y_p)|+|\mathfrak d(\mathbb F)|\le C\,\|\mathbb F\|_{\mathcal H_k},
\qquad
|\mu|\le\frac{C}{|\Edet(k,c,\eps)|}\,\|\mathbb F\|_{\mathcal H_k} ,
\end{equation}
the first by Equation~\eqref{eq:Xp-bound}, the bound on the value at $z=0$ of an $H^1$ function, and the boundedness of the coefficients.
We reconstruct $\hat q,\hat u,\hat s,\hat\zeta,\hat q_{\mathrm w},\hat{\mathbf v}$ from $\mu Y+Y_p$ and the formulas above, and obtain
\begin{equation}
\|\hat w\|_{H^1}+\|\hat q\|_{H^1}
+\|\hat u\|_{L^2}+\|\hat s\|_{L^2}
+\|\hat{\mathbf v}\|_{L^2}+|\hat\zeta|
\ \le\ C\Bigl(1+\frac1{|\Edet(k,c,\eps)|}\Bigr)\|\mathbb F\|_{\mathcal H_k} .
\end{equation}
The bounds for $\hat w$ and $\hat q$ come from $\mu Y+Y_p$ and Equation~\eqref{eq:Xp-bound}, those for $\hat u$ and $\hat s$ from the formulas for them above, with $|V|$ bounded below, and those for $\hat\zeta$, $\hat q_{\mathrm w}$, and $\hat{\mathbf v}$ from the interface formulas with $|\lambda|=k|c|\ge k|\Imc c|$ bounded below on $\Sigma$.
This yields $\hat{\mathbb U}\in\mathcal D(\Lop_k)$ with $(\lambda-\Lop_k)\hat{\mathbb U}=\mathbb F$, each component holding by construction, and the stated bound on the resolvent.
\end{proof}

By \cref{lem:determinant}\ref{it:det-modes}, a zero $c$ of $\Edet(k,\cdot,\eps)$ with $\Imc c>0$ is the phase speed of a growing normal mode, whose amplitude grows like $e^{k\Imc c\,t}$.
In the next corollary, we mark such a zero and its eigenvalue with the subscript $\mathrm u$, for unstable.

\begin{corollary}\label{cor:algsimple}
Let $c_{\mathrm u}$ with $\Imc c_{\mathrm u}>0$ be a zero of $\Edet(k,\cdot,\eps)$ with $\partial_c\Edet(k,c_{\mathrm u},\eps)\neq0$, equivalently, by Equation~\eqref{eq:Edet-factor}, a simple root of $D(k,\cdot,\eps)$ at which $w(0;k,c_{\mathrm u})\neq0$.
Then $\lambda_{\mathrm u}=-\im kc_{\mathrm u}$ is an isolated point of the spectrum of $\Lop_k$, and the resolvent has a pole of first order there, that is, $\|(\lambda-\Lop_k)^{-1}\|_{\mathcal H_k\to\mathcal H_k}\le C/|\lambda-\lambda_{\mathrm u}|$ for $\lambda\neq\lambda_{\mathrm u}$ near $\lambda_{\mathrm u}$.
For every sufficiently small $r>0$, the Riesz projection
\begin{equation}
\mathbb P:=\frac1{2\pi\im}\oint_{|\lambda-\lambda_{\mathrm u}|=r}
(\lambda-\Lop_k)^{-1}\,\dd\lambda ,
\end{equation}
which is the spectral projection onto the eigenvalue $\lambda_{\mathrm u}$, has rank one.
Hence the eigenvalue is algebraically simple, and $\ker\bigl((\Lop_k-\lambda_{\mathrm u})^{m}\bigr)=\ker(\Lop_k-\lambda_{\mathrm u})$ is one-dimensional for every integer $m\ge1$.
In particular, under the assumptions of \cref{thm:main}, the unstable root $c_\eps$ of \cref{thm:main} is a simple zero of $\Edet(k,\cdot,\eps)$ for every $\eps\in(0,\eps_0)$ by \cref{thm:main}\ref{it:main-exist} and Equation~\eqref{eq:Edet-factor}, so the unstable eigenvalue $\lambda_\eps=-\im kc_\eps$ is algebraically simple for every $\eps\in(0,\eps_0)$.
\end{corollary}

\begin{proof}
The function $\Edet(k,\cdot,\eps)$ is holomorphic near $c_{\mathrm u}$ (\cref{lem:determinant}) and has a simple zero there.
By Taylor's theorem, there is $r_1\in(0,\Imc c_{\mathrm u})$ with $|\Edet(k,c,\eps)|\ge\tfrac12|\partial_c\Edet(k,c_{\mathrm u},\eps)|\,|c-c_{\mathrm u}|$ for $|c-c_{\mathrm u}|\le r_1$, and these $c$ lie in the upper half-plane.
We apply \cref{prop:resolvent} with $\Sigma=\{c\in\C:|c-c_{\mathrm u}|\le r_1\}$ and use $|\lambda-\lambda_{\mathrm u}|=k|c-c_{\mathrm u}|$.
Then every $\lambda\neq\lambda_{\mathrm u}$ with $|\lambda-\lambda_{\mathrm u}|\le kr_1$ lies in the resolvent set of $\Lop_k$, and $\|(\lambda-\Lop_k)^{-1}\|\le C/|\lambda-\lambda_{\mathrm u}|$.
Since $\lambda_{\mathrm u}$ is an eigenvalue by \cref{thm:pointspec}, it is an isolated point of the spectrum.
The resolvent is holomorphic for $0<|\lambda-\lambda_{\mathrm u}|\le kr_1$, and the coefficients of the negative powers $(\lambda-\lambda_{\mathrm u})^{-n}$, $n\ge1$, in its Laurent expansion about $\lambda_{\mathrm u}$, the expansion in positive and negative powers of $\lambda-\lambda_{\mathrm u}$, are
\begin{equation}
\mathbb P_n:=\frac1{2\pi\im}\oint_{|\lambda-\lambda_{\mathrm u}|=r}
(\lambda-\lambda_{\mathrm u})^{n-1}(\lambda-\Lop_k)^{-1}\,\dd\lambda,
\qquad 0<r\le kr_1,
\end{equation}
so that $\mathbb P=\mathbb P_1$.
The circle has length $2\pi r$, the factor $|\lambda-\lambda_{\mathrm u}|^{n-1}$ equals $r^{n-1}$ on it, and the operator norm of the resolvent is at most $C/r$ on it, so $\|\mathbb P_n\|\le Cr^{\,n-1}$ for every such $r$.
Letting $r\to0$ gives $\mathbb P_n=0$ for $n\ge2$, so the pole has order one.
For a closed operator with an isolated spectral point, $\mathbb P$ is a projection that commutes with $\Lop_k$ on $\mathcal D(\Lop_k)$ \cite[Ch.~III, \S6.5]{Kato1995}.
A first-order pole makes $\lambda_{\mathrm u}$ semisimple, that is, every generalized eigenvector, a vector $\hat{\mathbb U}$ with $(\Lop_k-\lambda_{\mathrm u})^m\hat{\mathbb U}=0$ for some $m\ge1$, is an eigenvector.
We verify this and the rank directly.
For the resolvent $(\lambda-\Lop_k)^{-1}$, the Laurent coefficients satisfy $(\Lop_k-\lambda_{\mathrm u})\mathbb P_n=\mathbb P_{n+1}$, so $(\Lop_k-\lambda_{\mathrm u})\mathbb P=\mathbb P_2=0$ \cite[Ch.~III, \S6.5]{Kato1995}.
Hence the range $\operatorname{ran}\mathbb P$ of $\mathbb P$ is contained in $\ker(\Lop_k-\lambda_{\mathrm u})$, which is one-dimensional by \cref{thm:pointspec}.
The projection $\mathbb P$ acts as the identity on this null space, because every $\hat{\mathbb U}\in\ker(\Lop_k-\lambda_{\mathrm u})$ satisfies $(\lambda-\Lop_k)^{-1}\hat{\mathbb U}=(\lambda-\lambda_{\mathrm u})^{-1}\hat{\mathbb U}$, so $\mathbb P\hat{\mathbb U}=\hat{\mathbb U}$.
Hence $\operatorname{ran}\mathbb P=\ker(\Lop_k-\lambda_{\mathrm u})$, and $\mathbb P$ has rank one.
It remains to exclude generalized eigenvectors that are not eigenvectors.
In particular, the domain $\mathcal D\bigl((\Lop_k-\lambda_{\mathrm u})^2\bigr)$ of the square is the set of $\hat{\mathbb U}\in\mathcal D(\Lop_k)$ with $(\Lop_k-\lambda_{\mathrm u})\hat{\mathbb U}\in\mathcal D(\Lop_k)$.
If $\hat{\mathbb U}$ lies in this set and $(\Lop_k-\lambda_{\mathrm u})^2\hat{\mathbb U}=0$, put $\hat{\mathbb W}:=(\Lop_k-\lambda_{\mathrm u})\hat{\mathbb U}\in\ker(\Lop_k-\lambda_{\mathrm u})$.
Then $\hat{\mathbb W}=\mathbb P\hat{\mathbb W}=\mathbb P(\Lop_k-\lambda_{\mathrm u})\hat{\mathbb U}=(\Lop_k-\lambda_{\mathrm u})\mathbb P\hat{\mathbb U}=0$, because $\mathbb P$ commutes with $\Lop_k$ on its domain and $(\Lop_k-\lambda_{\mathrm u})\mathbb P=0$.
So $\ker\bigl((\Lop_k-\lambda_{\mathrm u})^2\bigr)=\ker(\Lop_k-\lambda_{\mathrm u})$, and by induction the same holds for every power.
\end{proof}

The same argument applies at a simple zero $c$ of $\Edet(k,\cdot,\eps)$ with $\Imc c<0$.
Alternatively, such zeros are the complex conjugates of zeros in the upper half-plane, by the reflection symmetry of \cref{lem:determinant}\ref{it:det-hol}.

This completes the proof of \cref{thm:operator}.
Part \ref{it:op-frechet} is \cref{prop:residual,prop:goodunknown}, part \ref{it:op-realized} is \cref{lem:closed,prop:realization} with \cref{prop:goodunknown}, and part \ref{it:op-spectrum} is \cref{thm:pointspec,prop:resolvent,cor:algsimple}.

\begin{remark}\label{rem:fulloperator}
The linearized system on $\T_{k_0}$ decomposes over the Fourier modes $k=nk_0$, $n\in\Z$, and the components with $n<0$ are given by the same formulas.
The mean-flow component $n=0$ is not constructed in this paper.
An operator for the full linearized system, which acts on each Fourier mode separately, can be assembled from these components once the component $n=0$ is defined on its space of states that satisfy the constraints.
For every such operator, the eigenfunction of $\Lop_k$ constructed above, multiplied by $e^{\im kx}$, is an $L^2$ eigenfunction, and under the assumptions of \cref{thm:main} it has the eigenvalue $\lambda_\eps=-\im kc_\eps$ with $\Rec\lambda_\eps>0$ (\cref{cor:spectral}).
\end{remark}

\begin{remark}\label{rem:generation}
The resolvent is constructed only for $\Imc c\neq0$, that is, for $\Rec\lambda\neq0$.
The set $\{-\im kU(z):0\le z\le\ha\}$ of values of the coefficient $-\im kU$ in $\Lop_k$, which corresponds to the real phase speeds in the range of $U$, is the candidate for the continuous spectrum, the set of $\lambda$ that are not eigenvalues and for which the equation $(\lambda-\Lop_k)\hat{\mathbb U}=\mathbb F$ can be solved for a dense set of $\mathbb F$ but not for every $\mathbb F$.
A limiting-absorption estimate for $(\lambda-\Lop_k)^{-1}$ as $\lambda$ approaches this set would require an inhomogeneous version of \cref{thm:limabs} and is not established here.
Semigroup generation is not addressed (\cref{rem:spaces}).
\end{remark}

\section*{Acknowledgments}
This work was performed under the auspices of the U.S.\ Department of Energy by Lawrence Livermore National Laboratory under Contract DE-AC52-07NA27344, and was supported by the Laboratory Directed Research and Development Program under Project No.\ 25-LW-087. The LLNL release number is \mbox{LLNL-JRNL-2024377}.

\section*{Disclosure of generative AI use}
The author used the Anthropic models Claude Fable 5, Claude Fable 5.1, and Claude Opus 5.5 and the OpenAI models GPT-5.6 Sol and GPT-6 Astra to assist with research development and manuscript preparation.
The models also guided the Lean 4 formalization used to verify identities and theorem-level deductions.
The formalization is available from the author on request.
The author takes full responsibility for the content of this paper.

\bibliographystyle{abbrv}
\bibliography{references}

\end{document}